\documentclass[reqno]{amsart}
\usepackage{amssymb,amsmath,amsfonts,bm}
\usepackage{dsfont}
\usepackage{epsfig}
\usepackage{graphicx}
\usepackage{enumerate}
\usepackage{verbatim}
\usepackage{color}
\usepackage{hyperref}
\usepackage{caption}
\usepackage[margin=2.7 cm]{geometry}
\usepackage{geometry}
\usepackage{tikz}  
\usetikzlibrary{arrows}
\usepackage{rotating}
\usetikzlibrary{shapes,quotes,calc,angles,arrows,through,intersections,shadings}
\usepackage{tkz-euclide}
\usetkzobj{all}
\usetkzobj{polygons} 
\tikzstyle{int}=[draw, fill=blue!20, minimum size=2em]
\tikzstyle{init} = [pin edge={to-,thin,black}]
\DeclareMathSizes{10}{9}{6}{5}


\date{} 

 \DeclareMathOperator\supp{supp}
 \DeclareMathOperator\jac{Jac}

\DeclareMathOperator\diverg{\textbf{div}}
\theoremstyle{plain}
\newtheorem{theorem}{Theorem}[section]
\newtheorem{definition}[theorem]{Definition}
\newtheorem{proposition}[theorem]{Proposition}

\newtheorem{lemma}[theorem]{Lemma}
\newtheorem{corollary}[theorem]{Corollary}
\newtheorem{remark}[theorem]{Remark}

\numberwithin{theorem}{section}
\numberwithin{equation}{section}
\numberwithin{figure}{section}

\let\oldtocsection=\tocsection
 
\let\oldtocsubsection=\tocsubsection
 
\let\oldtocsubsubsection=\tocsubsubsection
 
\renewcommand{\tocsection}[2]{\hspace{0em}\oldtocsection{#1}{#2}}
\renewcommand{\tocsubsection}[2]{\hspace{1em}\oldtocsubsection{#1}{#2}}
\renewcommand{\tocsubsubsection}[2]{\hspace{2em}\oldtocsubsubsection{#1}{#2}}

\begin{document}

\parskip=1pt

\vspace*{1.5cm}
\title[Rigorous derivation of a binary-ternary Boltzmann equation]
{Rigorous derivation of a binary-ternary Boltzmann equation for a dense gas of hard spheres}
\author[Ioakeim Ampatzoglou]{Ioakeim Ampatzoglou}
\address{Ioakeim Ampatzoglou,  
Department of Mathematics, The University of Texas at Austin.}
\email{ioakampa@math.utexas.edu}

\author[Nata\v{s}a Pavlovi\'{c}]{Nata\v{s}a Pavlovi\'{c}}
\address{Nata\v{s}a Pavlovi\'{c},  
Department of Mathematics, The University of Texas at Austin.}
\email{natasa@math.utexas.edu}

\vspace*{-3cm}
\begin{abstract}
This paper provides the first rigorous derivation of a binary-ternary Boltzmann equation describing the kinetic properties of a dense hard-spheres gas, where particles undergo either binary or ternary instantaneous interactions, while preserving momentum and energy. An important challenge we overcome in deriving this equation is related to providing  a mathematical framework that allows  us to detect both binary and ternary interactions. Furthermore, this paper introduces new algebraic and geometric techniques in order to eventually decouple binary and ternary interactions and understand the way they could succeed one another in time.
\end{abstract}
\maketitle

\tableofcontents
\section{Introduction}
The Boltzmann equation, introduced by L. Boltzmann \cite{boltzmann} and J.C. Maxwell \cite{maxwell}  describes the time evolution of the probability density of a rarefied, monoatomic gas in thermal non-equilibrium in $\mathbb{R}^d$, for $d\geq 2$. The Boltzmann equation  accurately describes very dilute gases since only \textbf{binary} interactions between particles are taken into account.
However, when the gas is dense enough, higher order interactions are much more likely to happen, therefore they significantly affect   time evolution of the gas.  A relevant example is a colloid, which is a homogeneous non-crystalline substance consisting of either large molecules or ultramicroscopic particles of one substance dispersed through a second substance. In \cite{ref 26}, authors pointed out importance of including higher order interactions among particles in a colloidal gas. In particular, they show that in addition to binary interactions, interactions among three particles   significantly contribute to the grand potential of the colloid. A surprising result of \cite{ref 26}, but of invaluable computational importance in numerical simulations, is that  interactions  among three  particles are actually characterized by the sum of the distances between particles, as opposed to depending on different geometric configurations among interacting particles.  The results of \cite{ref 26} have been further verified experimentally e.g. \cite{imp3} and numerically e.g. \cite{imp4}.

Motivated by the observations of \cite{ref 26}, in \cite{ternary} we suggested a model which goes beyond binary interactions incorporating sums of higher order interaction terms.  In particular, we introduced the generalized equation
\begin{equation}\label{intro:generalized Boltzmann}
\begin{cases}
\partial_t f+v\cdot\nabla_xf=\displaystyle\sum_{k=2}^m Q_k(\underbrace{f,f,\cdots,f\,}_\text{$k$-times}),\quad (t,x,v)\in(0,\infty)\times\mathbb{R}^{d}\times\mathbb{R}^d,\\
f(0,x,v)=f_0(x,v),\quad (x,v)\in\mathbb{R}^{d}\times\mathbb{R}^d,
\end{cases}
\end{equation}
where, for $k=1,...,m$, the expression $Q_k(\underbrace{f,f,\cdots,f\,}_\text{$k$-times})$ is the $k$-th order collisional operator and $m\in\mathbb{N}$ is the accuracy of the approximation depending on the density of the gas. We note that equations similar to \eqref{intro:generalized Boltzmann} were studied for Maxwell molecules in the works of Bobylev, Gamba and Cercignani \cite{multiple gamba,multiple gamba 2} using Fourier transform methods. Notice that for $m=2$, equation \eqref{intro:generalized Boltzmann} reduces to the classical Boltzmann equation. 

The task of rigorously deriving an equation of the form \eqref{intro:generalized Boltzmann} from a classical many particle system, even for the case $m=2$ (i.e. the Boltzmann equation), is a challenging problem that has been settled for short times only in certain situations, see e.g. \cite{lanford,king,cercignani paper,gallagher,pulvirenti-simonella,spohn,uchiya} for results in this direction. A relevant step towards rigorously deriving \eqref{intro:generalized Boltzmann} for $m=3$ has been recently obtained in \cite{ternary}, where we considered a certain type of three particle interactions that lead us to derive a purely ternary kinetic equation, which we called a ternary Boltzmann equation. However, the derivation of \eqref{intro:generalized Boltzmann} for $m=3$ has not been addressed yet, and that is exactly what we do in this paper.

 \subsection{Challenges of detecting both binary and ternary interactions}\label{subsec challenges}
The first challenge we face in deriving \eqref{intro:generalized Boltzmann} for $m=3$ is to provide a mathematical framework  allowing us to detect both binary and ternary interactions among particles. We achieve that by assuming the following: 
\begin{itemize}
\item  Binary interactions are modeled as  elastic collisions of hard spheres of diameter $\epsilon$ i.e. two particles interact when the distance of their centers defined as
$$d_2(x_i,x_j):=|x_i-x_j|$$ 
becomes equal to the diameter $\epsilon$.
\item Ternary interactions are of  interaction zone type as in \cite{ternary}, by which we mean  that the particle $i$ interacts with the particles $j$ and $k$ when the non-symmetric ternary distance  
$$d_3(x_i;x_j,x_k):=\sqrt{|x_i-x_j|^2+|x_i-x_k|^2 }$$
 becomes $\sqrt{2}\epsilon$.
\end{itemize}

Simultaneous consideration of both binary and ternary interactions brings us closer  to  understanding  our first obstacle. In particular, in the works on the derivation of the binary Boltzmann equation for hard spheres, pioneered by Lanford \cite{lanford} and recently completed by Gallagher, Saint-Raymond, Texier \cite{gallagher}, the relevant scaling is the Boltzmann-Grad scaling \cite{Grad 1, Grad 2} 
 \begin{equation}\label{Boltzmann-Grad intro}
 N\epsilon^{d-1}\simeq 1,
 \end{equation} 
 as the number of particles $N\to\infty$ and their diameter $\epsilon\to 0^+$.
 
 On the other hand, the scaling used in \cite{ternary} to control ternary interactions is a different scaling:
 \begin{equation}\label{ternary scaling intro}
 N\epsilon^{d-1/2}\simeq 1,
 \end{equation} 
A crucial, conceptual obstacle is the apparent incompatibility of the Boltzmann-Grad scaling \eqref{Boltzmann-Grad intro} dictated by binary interactions and the scaling \eqref{ternary scaling intro} of ternary interactions, if both of them are of order $\epsilon$. This incompatibility creates major difficulties even at the formal level.
We  overcome this scaling obstacle by assuming that, at the $N$-particle level, hard spheres are of diameter $\epsilon_2$ and that particles interact as triplets via an interaction zone $\epsilon_3$.  Imposing scalings \eqref{Boltzmann-Grad intro} with $\epsilon:=\epsilon_2$ and      \eqref{ternary scaling intro} with $\epsilon:=\epsilon_3$, we obtain the common scaling
\begin{equation}\label{common scaling intro}
N\epsilon_2^{d-1}\simeq N\epsilon_3^{d-1/2}\simeq 1,
\end{equation}
as $N\to\infty$ and $\epsilon_2,\epsilon_3\to 0^+$. Notice that the scaling \eqref{common scaling intro} implies that for sufficiently large $N$, we have
\begin{equation}\label{inequality on epsilons intro}
\epsilon_2<<\epsilon_3,
\end{equation}
which will have a prominent role in this paper.

The next challenge we address is the need to decouple binary and ternary interactions for a system of finitely many particles. More precisely, our framework a-priori allows i.e. that particles $i$ and $j$ interact as hard spheres:
$$d_2(x_i,x_j)=\epsilon_2,$$
while at the same time there is another particle $k$ such that the particle $i$ interacts with the particles $j$ and $k$:
$$d_3(x_i;x_j,x_k)=\sqrt{2}\epsilon_3.$$
Such a configuration is illustrated in Figure \ref{both binary and ternary}.
\begin{figure}[htp]
\centering
\captionsetup{justification=centering}
\begin{tikzpicture}[scale=1]
\coordinate (a) at (0,0);
\coordinate (a') at (0,-0.7);
\coordinate (b) at (0.25,0.25);
\coordinate (c) at (2,0);
\coordinate (b') at (0,0.5);
\coordinate (c') at (0.5,1.1);
\coordinate (e) at (-0.25,0.25);
\coordinate (f) at (-0.5,0.5);
\node (A) at (a) {\small{$\bullet$}};
\node (C) at (c) {\small{$\bullet$}};
\node (F) at (f) {\small{$\bullet$}};
\draw (a)--(c);
\draw (a)--(f);
\coordinate (d) at (2.25,0.25);
\node(c7) at (a)[draw,circle through=(b),color=red] {};
\node(c8) at (c)[draw,circle through=(d),color=blue] {};
\node(c9) at (f)[draw,circle through=(e),color=teal] {};
\coordinate (x) at (-0.55,0.03);
\node (x) at (x) {\small{$\epsilon_2$}};
\coordinate (f1) at (0.4,-0.3);
\node (F1) at (f1){{\color{red}\small{$i$}}};
\coordinate (f2) at (2.5,0);
\node (F2) at (f2){{\color{blue}\small{$k$}}};
\coordinate (f3) at (-0.9,0.85);
\node (F3) at (f3){{\color{teal}\small{$j$}}};
\coordinate  (f4) at (1,0.4);
\node (F4) at (f4){\small{$\sqrt{2\epsilon_3^2-\epsilon_2^2}$}};
\end{tikzpicture}
\caption{Both binary and ternary interactions at the same time}
\label{both binary and ternary}
\end{figure}
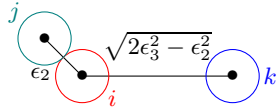
Pathological configurations, including the one we just described, are going to be shown to be negligible. 
 This is far from trivial and for more details on the microscopic dynamics, see Subsection \ref{subsec:intro dynamics} and Section \ref{sec:dynamics}. 
  In particular, we shall show that as long as $0<\epsilon_2<\epsilon_3<1$, only the following two interaction scenarios are possible with non-trivial probability under time evolution:
\begin{enumerate}[(i)]
\item Two particles interact as hard-spheres while all other particles are not involved in any binary or ternary interactions at the same time. This type of configurations generates the binary collisional operator. It is illustrated in Figure \ref{binary interaction mixed}.
\item Three particles interact via an interaction zone, while none of them is involved in a binary interaction with either of the other two particles of the interaction zone at the same time. The rest of the particles are not involved in any binary or ternary interactions. This type of configurations is responsible for generating the ternary collisional operator. It is illustrated in Figure \ref{ternary interaction mixed}.
\end{enumerate}

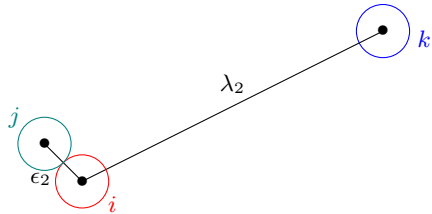
\begin{figure}[htp]
\centering
\captionsetup{justification=centering}
\begin{tikzpicture}[scale=1]
\coordinate (a) at (0,0);
\coordinate (a') at (0,-0.7);
\coordinate (b) at (0.25,0.25);
\coordinate (c) at (4,2);
\coordinate (b') at (0,0.5);
\coordinate (c') at (0.5,1.1);
\coordinate (e) at (-0.25,0.25);
\coordinate (f) at (-0.5,0.5);
\node (A) at (a) {\small{$\bullet$}};
\node (C) at (c) {\small{$\bullet$}};
\node (F) at (f) {\small{$\bullet$}};
\draw (a)--(c);
\draw (a)--(f);
\coordinate (d) at (4.25,2.25);
\node(c7) at (a)[draw,circle through=(b),color=red] {};
\node(c8) at (c)[draw,circle through=(d),color=blue] {};
\node(c9) at (f)[draw,circle through=(e),color=teal] {};
\coordinate (x) at (-0.55,0.03);
\node (x) at (x) {\small{$\epsilon_2$}};
\coordinate (f1) at (0.4,-0.3);
\node (F1) at (f1){{\color{red}\small{$i$}}};
\coordinate (f2) at (4.55,1.9);
\node (F2) at (f2){{\color{blue}\small{$k$}}};
\coordinate (f3) at (-0.9,0.85);
\node (F3) at (f3){{\color{teal}\small{$j$}}};
\coordinate  (f4) at (2,1.3);
\node (F4) at (f4){\small{$\lambda_2$}};
\end{tikzpicture}
\caption{Binary interaction: $\epsilon_2^2+\lambda_2^2>2\epsilon_3^2, \quad\lambda_2>\epsilon_2$.}
\label{binary interaction mixed}
\end{figure}

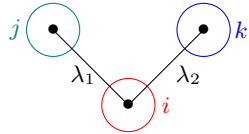
\begin{figure}[htp]
\centering
\captionsetup{justification=centering}
\begin{tikzpicture}[scale=1]
\coordinate (a) at (0,0);
\coordinate (a') at (0.25,0.25);
\node (A) at (a) {\small{$\bullet$}};
\node(c7) at (a)[draw,circle through=(a'),color=red] {};
\coordinate (A1) at (0.5,0);
\node (A1) at (A1){{\color{red}\small{$i$}}};

\coordinate (b) at (1,1);
\coordinate (b') at (1.25,1.25);
\node (B) at (b) {\small{$\bullet$}};
\node(c8) at (b)[draw,circle through=(b'),color=blue] {};
\coordinate (B1) at (1.5,1);
\node (B1) at (B1){{\color{blue}\small{$k$}}};

\coordinate (c) at (-1,1);
\coordinate (c') at (-1.25,1.25);
\node (C) at (c) {\small{$\bullet$}};
\node(c9) at (c)[draw,circle through=(c'),color=teal] {};
\coordinate (C1) at (-1.5,1);
\node (C1) at (C1){{\color{teal}\small{$j$}}};

\draw (a)--(c);
\draw (a)--(b);

\coordinate  (f4) at (0.8,0.4);
\node (F4) at (f4){\small{$\lambda_2$}};

\coordinate  (f5) at (-0.6,0.4);
\node (F5) at (f5){\small{$\lambda_1$}};
\end{tikzpicture}
\caption{Ternary interaction: $\lambda_1^2+\lambda_2^2=2\epsilon_3^2,\quad\lambda_1,\lambda_2>\epsilon_2$.}
\label{ternary interaction mixed}
\end{figure}

Finally, since we will eventually let the number of particles $N\to\infty$, the main challenge we need to address is the stability of a good configuration\footnote{by which we mean a configuration which does not run into any kind of interactions under backwards time evolution.} under the adjunction of one or two collisional particles. Assume, for a moment, that we have a good configuration of $m$-particles and we add $\sigma$ particles to the system, where $\sigma\in\{1,2\}$, such that a binary or ternary interaction is formed among one of the existing particles and the $\sigma$ new particles. In general, under backwards time evolution, the system could run into another binary or ternary interaction, see e.g. Figure \ref{recollision}, which illustrates the mathematically most difficult case where the newly formed $(m+2)$-configuration runs into a binary interaction. To the best of our knowledge, this is the first time there was the need to address the possibility of a newly formed interacting  configuration running into an interaction of a different type (binary to ternary or ternary to binary) backwards in time.  However, in Section \ref{sec:geometric} and Section \ref{sec:stability}, we develop novel algebraic and geometric tools which help us  eliminate pathological scenarios, including the one described in Figure \ref{recollision}, by showing that outside of a small measure set, negligible in the limit,  the newly formed   configuration does not run into  any additional interactions backwards in time.  For more details on the technical difficulties faced, see Subsection \ref{subsec:difficulties}.

\begin{figure}[htp]
\centering
\captionsetup{justification=centering}
\begin{tikzpicture}
\coordinate (a) at (0,0);
\coordinate (a') at (0.25,0.25);
\node (A) at (a) {\small{$\bullet$}};
\node(c7) at (a)[draw,circle through=(a'),color=red] {};
\coordinate (A1) at (0.5,0);
\node (A1) at (A1){{\color{red}\small{$i$}}};

\coordinate (b) at (1,1);
\coordinate (b') at (1.25,1.25);
\node (B) at (b) {\small{$\bullet$}};
\node(c8) at (b)[draw,circle through=(b'),color=blue] {};
\coordinate (B1) at (1.8,1);
\node (B1) at (B1){{\color{blue}\small{$m+2$}}};

\coordinate (c) at (-1,1);
\coordinate (c') at (-1.25,1.25);
\node (C) at (c) {\small{$\bullet$}};
\node(c9) at (c)[draw,circle through=(c'),color=teal] {};
\coordinate (C1) at (-1.8,1);
\node (C1) at (C1){{\color{teal}\small{$m+1$}}};

\draw (a)--(c);
\draw (a)--(b);

\coordinate  (f4) at (0.8,0.4);
\node (F4) at (f4){\small{$\lambda_2$}};

\coordinate  (f5) at (-0.6,0.4);
\node (F5) at (f5){\small{$\lambda_1$}};
\node (G) at (0,-0.65) {\small{$\lambda_1^2+\lambda_2^2=2\epsilon_3^2,\quad\lambda_1,\lambda_2>\epsilon_2$}};

\coordinate (a) at (5,0);
\coordinate (a') at (5.25,0.25);
\node (A) at (a) {\small{$\bullet$}};
\node(c7) at (a)[draw,circle through=(a'),color=red] {};
\coordinate (A1) at (5.5,0);
\node (A1) at (A1){{\color{red}\small{$i$}}};

\coordinate (b) at (5.5,0.5);
\coordinate (b') at (5.25,0.25);
\node (B) at (b) {\small{$\bullet$}};
\node(c7) at (b)[draw,circle through=(b'),color=blue] {};
\coordinate (B1) at (6.3,0.5);
\node (B1) at (B1){{\color{blue}\small{$m+2$}}};
\draw (a)--(b);
\node (D) at (5,0.5) {\small{$\epsilon_2$}};
\coordinate (d1) at (1.5,0.25);
\coordinate (d2) at (4.5,0.25);
\draw [->](d1)--(d2);
\coordinate (d3) at (3,0.5);
\node (D3) at (d3){{\small{backwards in time}}};
\end{tikzpicture}
\caption{}\label{recollision}

\end{figure}

In the next subsection, we  investigate more precisely what happens when a binary or a ternary interactions occurs and describe the time evolution of such a system.
\subsection{Dynamics of finitely many particles}\label{subsec:intro dynamics} Let us describe the evolution in $\mathbb{R}^d$, $d\geq 2$, of a system of $N$ hard spheres of diameter $\epsilon_2$ and interaction zone $\epsilon_3$, where $0<\epsilon_2<\epsilon_3<1$. The assumption $\epsilon_2<\epsilon_3$ is necessary for ternary interactions to be of non trivial probability, see Remark \ref{remark on phase space} for more details.

\subsubsection{Interactions considered} We first define the interactions considered in this paper.
\begin{definition}\label{intro: def of ternary interaction} Let $N\in\mathbb{N}$, with $N\geq 3$,  and $0<\epsilon_2<\epsilon_3<1$. We define  binary and ternary interactions, also referred to  as collisions, as follows:
\begin{itemize}
\item Consider two particles $i,j\in\{1,...,N\}$ with positions $x_i,x_j\in\mathbb{R}^{d}$. We say that the particles $i,j$ are in  an $(i,j)$ binary interaction, if the following geometric condition holds:
\begin{equation}\label{intro-binary collisions}
d_2(x_i,x_j):=|x_i-x_j|=\epsilon_2.
\end{equation}

\item Consider three particles $i,j,k\in\{1,...,N\}$, with positions $x_i,x_j,x_k\in\mathbb{R}^{d}$. We say that the particles $i,j,k$ are in an $(i;j,k)$  interaction\footnote{we use the notation $(i;j,k)$ because the interaction condition is not symmetric. The particle $i$ is the central particle of the interaction i.e. the one interacting with the particles $j$ and $k$ respectively.} if the following geometric condition holds:
\begin{equation}\label{intro-triple collisions}
d_3(x_i;x_j,x_k):=\sqrt{|x_i-x_j|^2+|x_i-x_k|^2}=\sqrt{2}\epsilon_3.
\end{equation}
\end{itemize}
\end{definition}

When an $(i,j)$ interaction occurs, the velocities $v_i,v_j$ of the $i$-th and $j$-th particles  instantaneously transform according to the binary collisional law:
\begin{equation}\label{intro:binary collisional law}
\begin{aligned}
v_i'&=v_i+\langle\omega_1,v_j-v_i\rangle\omega_1,\\
v_j'&=v_j-\langle\omega_1,v_j-v_i\rangle\omega_1,
\end{aligned}
\end{equation}
where 
\begin{equation}\label{Intro:omega def binary}
\omega_1:=\frac{x_j-x_i}{\epsilon_2}.
\end{equation}
Thanks to \eqref{intro-binary collisions}, we have $\omega_1\in\mathbb{S}_1^{d-1}$.
The vector $\omega_1$ is called binary impact direction and it represents the scaled relative position of the colliding particles.
Moreover, one can see that the binary momentum-energy system:
\begin{equation}\label{intro:binary MEC}
\begin{aligned}
v'+v_1'&=v+v_1,\\
|v'|^2+|v_1'|^2&=|v|^2+|v_1|^2,
\end{aligned}
\end{equation}
is satisfied.

When an $(i;j,k)$ interaction happens, the velocities $v_i,v_j,v_k$ of the $i$-th, $j$-th and $k$-th particles instantaneously transform according to the ternary collisional law derived in \cite{ternary}
\begin{equation}\label{intro:ternary collisional law}
\begin{aligned}
v_i^*&=v_i+\frac{\langle\omega_1,v_j-v_i\rangle+\langle\omega_2,v_k-v_i\rangle}{1+\langle\omega_1,\omega_2\rangle}(\omega_{1}+\omega_{2}),\\
v_j^*&=v_j-\frac{\langle\omega_1,v_j-v_i\rangle+\langle\omega_2,v_k-v_i\rangle}{1+\langle\omega_1,\omega_2\rangle}\omega_{1},\\
v_k^*&=v_k-\frac{\langle\omega_1,v_j-v_i\rangle+\langle\omega_2,v_k-v_i\rangle}{1+\langle\omega_1,\omega_2\rangle}\omega_2,
\end{aligned}
\end{equation}
where
\begin{equation}\label{Intro:omega def ternary}
(\omega_1,\omega_2):=\left(\frac{x_j-x_i}{\sqrt{2}\epsilon_3},\frac{x_k-x_i}{\sqrt{2}\epsilon_3}\right).
\end{equation}
Thanks to \eqref{intro-triple collisions}, we have $(\omega_1,\omega_2)\in\mathbb{S}_1^{2d-1}$.
The vectors $(\omega_1,\omega_2)$ are called ternary impact directions and they represents the scaled relative positions of the interacting particles.
Moreover, it has been shown in \cite{ternary} that the ternary momentum-energy system:
\begin{equation}\label{intro:ternary MEC}
\begin{aligned}
v^*+v_1^*+v_2^*&=v+v_1+v_2,\\
|v^*|^2+|v_1^*|^2+|v_2^*|^2&=|v|^2+|v_1|^2+|v_2|^2,
\end{aligned}
\end{equation}
is satisfied.
\subsubsection{Phase space and description of the flow}
Let $N\in\mathbb{N}$, with $N\geq 3$, and $0<\epsilon_2<\epsilon_3<1$.  The natural phase space \footnote{upon symmetrization, one could define the phase space without ordering the particles and obtain a symmetrized version of ternary operator (see \cite{thesis} for more details). For simplicity, we opt to work upon ordering the particles.} to capture both binary and ternary interactions is:
\begin{equation}\label{intro:mixed phase space}
\begin{aligned}
\mathcal{D}_{N,\epsilon_2,\epsilon_3}=\big\{Z_N=(X_N,V_N)\in\mathbb{R}^{2dN}: d_2(x_i,x_j)\geq\epsilon_2,\text{ }\forall (i,j)\in\mathcal{I}_N^2, 
\text{ and }d_3(x_i;x_j,x_k)\geq\sqrt{2}\epsilon_3,\text{ }\forall (i,j,k)\in\mathcal{I}_N^3\big\},
\end{aligned}
\end{equation}
where 
$X_N=(x_1,x_2,...,x_N)$, $V_N=(v_1,v_2,...,v_N),$
represent the positions and velocities of the $N$-particles, and the index sets $\mathcal{I}_N^2,\mathcal{I}_N^3$ are given by 
\begin{align*}
\mathcal{I}_N^2=\{(i,j)\in\{1,...,N\}^2:i<j\},\quad \mathcal{I}_N^3=\{(i,j)\in\{1,...,N\}^3:i<j<k\}.
\end{align*}

Let us  describe the evolution in time of such a system. Consider  an initial configuration $Z_N\in\mathcal{D}_{N,\epsilon_2,\epsilon_3}$. The motion is described as follows:
\begin{enumerate}[(I)]
\item Particles are assumed to perform rectilinear motion as long as there is no interaction
\begin{equation*}
\dot{x}_i=v_i,\quad \dot{v}_i=0,\quad\forall i\in\{1,...,N\}.
\end{equation*}
\item Assume now that an initial configuration $Z_N=(X_N,V_N)$ has evolved until time $t>0$, reaching $Z_N(t)=(X_N(t),V_N(t))$, and that there is an interaction at time $t$. We have the following cases:
\begin{itemize}
\item The interaction is binary: Assuming there is an $(i,j)$ interaction the velocities of the interacting particles instantaneously transform velocities according to the binary collisional law $(v_i(t),v_j(t))\to (v_i'(t),v_j'(t))$ given in \eqref{intro:binary collisional law}.
\item The interaction  is ternary: Assuming there is an $(i;j,k)$  interaction, the velocities of the interacting particles instantaneously transform velocities according to the ternary collisional law $(v_i(t),v_j(t),v_k(t))\to (v_i^*(t),v_j^*(t),v_k^*(t))$ given in \eqref{intro:ternary collisional law}.
\end{itemize}
\end{enumerate}

Let us note that (I)-(II) are not sufficient to generate a global in time flow for the particle system, since the velocity transformations are not smooth. In general pathologies might arise as time evolves, meaning more than one type of interactions happening at the same time, grazing interaction, or infinitely many interactions in finite time. Although, well-defined dynamics was shown to exist  in \cite{alexander} for hard spheres and in \cite{ternary} for the purely ternary case, those results do not imply well-posedness of the flow for the mixed case, where both binary and ternary interactions are taken into account. The reason for that is that a binary interaction can be succeeded by a ternary interaction and vice versa, a situation which was not addressed in \cite{alexander} nor \cite{ternary}. However we are showing that a non-grazing interaction cannot be succeeded by the same interaction. In other words, when two particles $(i,j)$ interact, the next interaction could be anything, binary or ternary, except a binary recollision of the particles $(i,j)$. Similarly, when three particles there is an $(i;j,k)$ interaction, the next interaction can be anything except a ternary  $(i;j,k)$\footnote{any other permutation of the particle $i,j,k$ cannot form an interaction since  $i<j<k$. In case one does not order the particles, a subsequent $(j;i,k)$ interaction, for instance, could possibly happen.} interaction. This observation allows us to define the flow locally a.e. and then run some combinatorial covering arguments to geometrically exclude a  zero Lebesgue measure set such that the flow is globally in time defined on the complement.

Let us informally state this result. For a detailed statement, see Theorem \ref{global flow}.

\textbf{\textit{Existence of a global flow}:}\textit{ Let $N\in\mathbb{N}$ and $0<\epsilon_2<\epsilon_3<1$. There is a global in time measure-preserving flow $(\Psi_m^t)_{t\in\mathbb{R}}:\mathcal{D}_{N,\epsilon_2,\epsilon_3}\to\mathcal{D}_{N,\epsilon_2,\epsilon_3}$ described a.e. by (I)-(II) which preserves kinetic energy.
This flow is called the $N$-particle $(\epsilon_2,\epsilon_3)$-interaction flow.} 

The global measure-preserving  interaction flow  yields the Liouville  equation\footnote{in case $N=2$, the ternary boundary condition is not present in \eqref{intro:Liouville}, while if $N=1$, equation \eqref{intro:Liouville} is just the transport equation.} for the evolution $f_N$ of an initial $N$-particle  probability density $f_{N,0}$. 
\begin{equation}\label{intro:Liouville}
\begin{aligned}
&\partial_tf_N+\sum_{i=1}^Nv_i\nabla_{x_i}f_N=0,\quad (t,Z_N)\in (0,\infty)\times\mathring{D}_{N,\epsilon_2,\epsilon_3},\\
&f_N(t,Z_N')=f(t,Z_N),\quad t\in[0,\infty),\quad Z_N\text{ is a simple binary interaction\footnotemark},\\
&f_N(t,Z_N^*)=f(t,Z_N),\quad t\in[0,\infty),\quad Z_N\text{ is a simple ternary interaction\footnotemark },\\
&f_N(0,Z_N)=f_{N,0}(Z_N),\quad Z_N\in\mathring{D}_{N,\epsilon_2,\epsilon_3}.
\end{aligned}
\end{equation}
\addtocounter{footnote}{-1}
\footnotetext{by simple binary interaction, we mean the only interaction happening is an  $(i,j)$  interaction. In this case, we write $Z_N'=(X_N,V_N')$, where $V_N'=(v_1,...,v_{i-1},v_i',v_{i+1},...,v_{j-1},v_j',v_{j+1},...,v_N)$.}
\addtocounter{footnote}{1}
\footnotetext{by simple ternary interaction, we mean the only interaction happening is an  $(i;j,k)$  interaction. In this case, we write $Z_N^*=(X_N,V_N^*)$, where $V_N^*=(v_1,...,v_{i-1},v_i^*,v_{i+1},...,v_{j-1},v_j^*,v_{j+1},...,v_{k-1},v_k^*,v_{k+1},...,v_N)$.}

The Liouville equation provides a complete deterministic description of the system of $N$-particles. Although Liouville's equation is a linear transport equation, efficiently solving it is almost impossible in case where the particle number $N$ is very large. This is why an accurate kinetic description is welcome, and to obtain it one wants to understand the limiting behavior of it as $N\to\infty$ and $\epsilon_2,\epsilon_3\to 0^+$, with the hope that  qualitative properties will be revealed for a large but finite $N$.

\subsection{The binary-ternary Botzmann equation} 
  To obtain such a kinetic description, we let the number of particles $N\to\infty$ and the diameter and interaction zone of the particles $\epsilon_2,\epsilon_3\to 0^+$ in the \textbf{common scaling} \eqref{common scaling intro}:
\begin{equation*}
N\epsilon_2^{d-1}\simeq N\epsilon_3^{d-\frac{1}{2}}\simeq 1,
\end{equation*}
which will lead the binary-ternary Boltzmann equation
\begin{equation}\label{intro:binary-ternary}
\partial_t f+v\cdot\nabla_x f=Q_2(f,f)+Q_3(f,f,f), \quad (t,x,v)\in (0,\infty)\times\mathbb{R}^d\times\mathbb{R}^d,
\end{equation}
The operator $Q_2(f,f)$, see e.g. \cite{cercignani gases}, is the classical hard sphere binary collisional operator given by
\begin{equation}\label{intro: quad Boltz kernel}
Q_2(f,f)=\int_{\mathbb{S}_1^{d-1}\times\mathbb{R}^d}\langle\omega_1,v_1-v\rangle_+\left(f'f_1'-ff_1\right)\,d\omega_1\,dv_1,
\end{equation}
where
\begin{equation*}
\begin{aligned}
f'=f(t,x,v'),&\quad f=f(x,t,v),\quad f_1'=f_1(t,x,v_1'),&\quad f_1=f(t,x,v_1),
\end{aligned}
\end{equation*}

The operator $Q_3(f,f,f)$, introduced for the first time in \cite{ternary}, is the ternary hard interaction zone operator given by
\begin{equation}\label{intro:kernel}
Q_3(f,f,f)=\int_{\mathbb{S}_1^{2d-1}\times\mathbb{R}^{2d}}b_+\left(f^*f_1^*f_2^*-ff_1f_2\right)\,d\omega_1\,d\omega_2\,dv_1\,dv_2,
\end{equation}
where 
\begin{equation}\label{intro:parameters boltzmann}
\begin{aligned}
b=b(\omega_1,\omega_2,v_1-v,v_2-v):=\langle\omega_1,v_{1}-v\rangle+\langle\omega_2,v_{2}-v_\rangle,\quad b_+=\max\{b,0\},&\\
 f^*=f(t,x,v^*)\quad f=f(x,t,v)\quad f_i^*=f_i^*(t,x,v_i^*),\quad f_i=f(t,x,v_i), \text{ for  }i\in\{1,2\}.&
\end{aligned}
\end{equation}

We should mention that in \cite{gwp}, global well-posedness near vacuum has been shown for \eqref{intro:binary-ternary} for potentials ranging from moderately soft to hard in spaces of functions bounded by Maxwellian. In fact in \cite{gwp}, it is seen that the ternary collisional operator allows consideration of softer potentials that the binary operator. In other words the ternary correction to the Boltzmann equation does not behave worse than the classical Boltzmann equation.

It is  important to point out that, upon symmetrization of the ternary collisional operator (see \cite{thesis}), the corresponding binary-ternary Boltzmann equation  enjoys similar statistical and entropy production properties and conservation laws as the classical Boltzmann equation.  Therefore, such a model could serve as a correction of the classical Boltzmann equation to denser gases. This follows after combining the properties of the classical  binary operator, see e.g. \cite{cercignani gases}, with the properties of the symmetrized  ternary collisional operator investigated for the first time in \cite{thesis}. 

\subsection{Strategy of the derivation and statement of the main result}
In order to pass from the $N$-particle system dynamics to the kinetic equation \eqref{intro:binary-ternary}, we implement the program of constructing linear finite and infinite hierarchies of equations, pioneered
 by Lanford \cite{lanford} and refined by Gallagher, Saint-Raymond, Texier \cite{gallagher},  and connecting them to the new binary-ternary Boltzmann equation. In \cite{ternary}, we extended this program  to include ternary interactions, which led to the rigorous derivation of a purely ternary kinetic equation for particles with hard interaction zone in the scaling \eqref{ternary scaling intro}. However, rigorous derivation of \eqref{intro:binary-ternary} does not follow  from \cite{lanford, gallagher} nor the ternary work \cite{ternary} . As mentioned in Subsection \ref{subsec challenges} the first difficulty is the apparent incompatibility of scalings  \eqref{Boltzmann-Grad intro}-\eqref{ternary scaling intro}, which we overcome by introducing the common scaling \eqref{common scaling intro}. The most challenging task is to make the argument rigorous, though, is the analysis of  all the possible recollisions\footnote{by recollisions we mean the possible divergence of the backwards $(\epsilon_2,\epsilon_3)$-interaction flow from the backwards free flow.} of the backwards $(\epsilon_2,\epsilon_3)$-flow. In contrast to the binary or the ternary case where each binary or ternary interaction is succeeded by a binary or ternary interaction respectively, here we can have any possible interaction sequence of binary or ternary interactions. We keep track of this combinatorics using the set
 \begin{equation}\label{combinatorics intro}
 S_k=\{\sigma=(\sigma_1,...,\sigma_k):\sigma_i\in\{1,2\},\quad\forall i=1,...,k\}.
 \end{equation}
In addition to more involved combinatorics, careful analysis of all the possible interaction sequences requires development of novel geometric and algebraic tools, which we discuss in details in Subsection \ref{subsec:difficulties}. For now, we continue to discuss the process of derivation.

More specifically, we first derive a finite, linear, coupled hierarchy of equations for the marginal densities 
\begin{align*}
f_N^{(s)}(Z_s)&=\int_{\mathbb{R}^{2d(N-s)}}f_N(Z_N)\mathds{1}_{\mathcal{D}_{N,\epsilon_2,\epsilon_3}}(Z_N)\,dx_{s+1}...\,dx_N\,dv_{s+1}...\,dv_N,\quad s\in\{1,...,N-1\},
\end{align*}
of the solution $f_N$ to the Liouville equation,
which we call the  BBGKY\footnote{Bogoliubov, Born, Green, Kirkwood, Yvon}. This hierarchy is given by
 \begin{equation}\label{intro:BBGKY}
 \partial_t f_N^{(s)}+\sum_{i=1}^sv_i\cdot\nabla_{x_i}f_N^{(s)}=\mathcal{C}_{s,s+1}^Nf_N^{(s+1)}+\mathcal{C}_{s,s+2}^Nf_N^{(s+2)},\quad s\in\{1,...,N-1\}
 \end{equation}
 For the precise form of the operators $\mathcal{C}_{s,s+1}^N$, $\mathcal{C}_{s,s+2}^N$, see \eqref{BBGKY operator binary}-\eqref{BBGKY operator triary}. Duhamel's Formula yields that the BBGKY hierarchy can be written in mild form as follows:
 \begin{equation}\label{intro:mild BBGKY}
\hspace{2cm} f_N^{(s)}(t,Z_s)=T_s^tf_{N,0}(Z_s)+\int_0^t T_s^{t-\tau}(\mathcal{C}_{s,s+1}^Nf_N^{(s+1)}+\mathcal{C}_{s,s+2}^Nf_N^{(s+2)})(\tau,Z_s)\,d\tau,\quad s\in\mathbb{N},
 \end{equation}
 where for any continuous function $g_s:\mathcal{D}_{s,\epsilon_2,\epsilon_3}\to\mathbb{R}$, we write
 $T_s^tg_s(Z_s):=g_s(\Psi_s^{-t}Z_s),$
 and $\Psi_s^t$ is the  $(\epsilon_2,\epsilon_3)$-interaction zone  flow of $s$-particles.

We then formally let $N\to\infty$ and $\epsilon_2,\epsilon_3\to 0^+$ in the scaling \eqref{common scaling intro} to obtain an infinite, linear, coupled hierarchy of equations, which we call the Boltzmann hierarchy. This hierarchy is given by
\begin{equation}\label{intro:Boltzmann hierarchy}
\partial_t f^{(s)}+\sum_{i=1}^sv_i\cdot\nabla_{x_i}f^{(s)}=\mathcal{C}_{s,s+1}^\infty f^{(s+1)}+\mathcal{C}_{s,s+2}^\infty f^{(s+2)},\quad s\in\mathbb{N}.
\end{equation}
For the precise form of the operators $\mathcal{C}_{s,s+1}^\infty$, $\mathcal{C}_{s,s+2}^\infty$, see \eqref{boltzmann hiera kernel binary}, \eqref{boltzmann hiera kernel ternary} respectively.
Duhamel's Formula yields that the Boltzmann hierarchy can be written in mild form as follows:
 \begin{equation}\label{intro:mild Boltzmann}
 f^{(s)}(t,Z_s)=S_s^tf_{0}(Z_s)+\int_0^t S_s^{t-\tau}(\mathcal{C}_{s,s+1}^\infty f^{(s+1)}+\mathcal{C}_{s,s+2}^\infty f^{(s+2)})(\tau,Z_s)\,d\tau,\quad s\in\mathbb{N},
 \end{equation}
 where for any continuous function $g_s:\mathbb{R}^{2ds}\to\mathbb{R}$, we write
 $S_s^tg_s(Z_s):=g_s(\Phi_s^{-t}Z_s),$
and $\Phi_s^t$ is the $s$-particle free flow of $s$-particles defined by
 $S_s^t Z_s=S_s^t(X_s,V_s)=(X_s-tV_s,V_s).$

 It can be observed that for factorized initial data and assuming that the solution remains factorized in time\footnote{this is typically called propagation of chaos assumption},  the Boltzmann hierarchy  reduces to the binary-ternary Boltzmann equation \eqref{intro:binary-ternary}.  This observation connects the Boltzmann hierarchy with the binary-ternary Boltzmann equation \eqref{intro:binary-ternary}.  

To make this argument rigorous, we first show that the BBGKY and Boltzmann hierarchy are well-posed in the scaling \eqref{common scaling intro}, at least for short times, and then that the convergence of the BBGKY hierarchy initial data to the Boltzmann hierarchy initial data propagates in the time interval of existence of the solutions.  Showing convergence is a very challenging task, and is the heart of our contribution. We describe details in Subsection \ref{subsec:difficulties}.

Now, we informally state our main result. For a rigorous statement of the result see Theorem \ref{convergence theorem}.
 
 \textbf{\textit{Statement of the main result}:} \textit{Let $F_0$ be  initial data for the Boltzmann hierarchy \eqref{intro:Boltzmann hierarchy}, and $F_{N,0}$ be some BBGKY hierarchy \eqref{intro:Boltzmann hierarchy} initial data  which ``approximate"\footnote{see Subsection   \ref{subseq:approximation} for details} $F_0$ as $N\to\infty$, $\epsilon\to 0^+$ under the scaling \eqref{common scaling intro}. Let $\bm{F_N}$ be the mild solution to the BBGKY hierarchy \eqref{intro:BBGKY} with initial data $F_{N,0}$, and $\bm{F}$ the mild solution to the Boltzmann hierarchy \eqref{intro:Boltzmann hierarchy}, with initial data $F_0$, up to short time $T>0$. Then $\bm{F_N}$ converges in observables\footnote{for a precise definition of convergence in observables, see Subsection \ref{subsec:con in observables}} to $\bm{F}$ in $[0,T]$ as $N\to\infty$, $\epsilon\to 0^+$, under the scaling \eqref{common scaling intro}.} \vspace{0.2cm}

The convergence obtained implies that the solution of the finite hierarchy indeed approximates the solution of the infinite hierarchy in $[0,T]$, as  $N\to\infty$, $\epsilon_2,\epsilon_3\to 0^+$ in the scaling \eqref{common scaling intro}.
For factorized initial data (initial chaotic assumption)  the Boltzmann hierarchy reduces to  equation \eqref{intro:binary-ternary}. 

\subsection{Difficulties faced in the proof of the main result}\label{subsec:difficulties} The main idea to obtain convergence (Theorem \ref{convergence theorem}) is to inductively use mild forms \eqref{intro:mild BBGKY}, \eqref{intro:mild Boltzmann} of the BBGKY hierarchy and  Boltzmann hierarchy respectively, to formally obtain series expansions with respect to the initial data:
\begin{align}
f_N^{(s)}(t,Z_s)=T_s^tf_{N,0}^{(s)}(Z_s)+\sum_{k=1}^\infty\sum_{\sigma\in S_k}\int_0^t&\int_0^{t_1}...\int_0^{t_{k-1}}T_s^{t-t_1}\mathcal{C}_{s,s+\widetilde{\sigma}_1}^NT_{s+\widetilde{\sigma}_1}^{t_1-t_2}...\mathcal{C}_{s+\widetilde{\sigma}_{k-1},s+\widetilde{\sigma}_k}^NT_{s+\widetilde{\sigma}_k}^{t_k}f_{N,0}^{(s+\widetilde{\sigma}_k)}(Z_s)\,dt_k...\,dt_1,\label{intro:BBGKY expansion}\\
f^{(s)}(t,Z_s)=S_s^tf_{0}^{(s)}(Z_s)+\sum_{k=1}^\infty\sum_{\sigma\in S_k}\int_0^t&\int_0^{t_1}...\int_0^{t_{k-1}}S_s^{t-t_1}\mathcal{C}_{s,s+\widetilde{\sigma}_1}^\infty S_{s+\widetilde{\sigma}_1}^{t_1-t_2}...\mathcal{C}_{s+\widetilde{\sigma}_{k-1},s+\widetilde{\sigma}_{k}}^\infty S_{s+\widetilde{\sigma}_k}^{t_k}f_{0}^{(s+\widetilde{\sigma}_k)}(Z_s)\,dt_k...\,dt_1,\label{intro:Boltzmann expansion}
\end{align}
where $S_k$ is defined in \eqref{combinatorics intro}, and given $\sigma\in S_k$, $\ell=1,...,k$, we write $\widetilde{\sigma}_\ell:=\sum_{i=1}^\ell \sigma_i$.
We note that the summation over $S_k$ in \eqref{intro:BBGKY expansion}-\eqref{intro:Boltzmann expansion} allows us to keep track of the possible interaction sequences occurring by ``adding" one or two particles in each time step. For more details, see Section \ref{sec: series expansion}.

Comparing expressions \eqref{intro:BBGKY expansion}-\eqref{intro:Boltzmann expansion}, we expect to obtain the required convergence under the scaling \eqref{common scaling intro} as long as $f_{N,0}^{(s)}$  ``approximates" $f_0^{(s)}$ under the same scaling. However it is not possible to directly compare \eqref{intro:BBGKY expansion}-\eqref{intro:Boltzmann expansion} because of the possible divergence of the backwards interaction flow from the free flow, which we call recollisions. Although recollisions were also faced in \cite{gallagher} and \cite{ternary}, the mixed case, where both binary and ternary interactions are considered,  requires different conceptual treatment in many instances, and is not implied by the results of these works. The reason for that is that a binary interaction can be succeeded by a ternary interaction and vice versa, a situation which was not addressed in \cite{gallagher,ternary}. The key to overcome these difficulties is that the diameter of the particles is much smaller than the interaction zone, as implied by the common scaling \eqref{common scaling intro}. This fact allows us to  develop certain delicate algebraic and geometric arguments   to extract  a small measure set of pathological initial data which lead to recollisions. On the complement of this set,  expansions  \eqref{intro:BBGKY expansion}-\eqref{intro:Boltzmann expansion} are comparable and the required convergence is obtained.

The main idea  for eliminating recollisions is an inductive application in each time step of Proposition \ref{bad set double} and Proposition \ref{bad set double measure}, which treat the binary adjunction, or Proposition \ref{bad set triple} and Proposition \ref{bad set triple measure}, which treat the ternary adjunction.  More precisely we face the following different cases:
\vspace{0.1cm}
\begin{enumerate}[(I)]
\item\textbf{Binary adjunction}: One  particle is added forming a binary interaction with one of the existing particles. The pathological situations  that might arise under backwards time evolution are the following:
\begin{itemize}
\item The  newly formed binary collisional configuration   runs to a binary interaction under time evolution. This pathological situation is eliminated using arguments inspired by \cite{gallagher}. This is actually the only case which is similar to the cases covered  in \cite{gallagher}.
\item The  newly formed binary collisional configuration  runs to a ternary interaction under time evolution. This pathological situation did not appear in any of the previous works since merely binary or ternary interactions were studied. However, due to the fact that $\epsilon_2<<\epsilon_3$, which comes from the scaling \eqref{common scaling intro}, this pathological situation can be treated  using  techniques inspired by \cite{ternary} and adapting them to the binary case.
\end{itemize} 
\vspace{0.2cm}
Proposition \ref{bad set double} and Proposition \ref{bad set double measure} are the relevant results controlling recollisions after a binary adjunction.
\vspace{0.2cm}
\item\textbf{Ternary adjunction}: Two  particles are added forming a ternary interaction with one of the existing particles. The pathological situations that might arise under backwards time evolution are the following:
\begin{itemize}
\item The  newly formed ternary collisional configuration   runs to a ternary interaction under time evolution. This case was studied in depth in \cite{ternary}. We eliminate this pathological situation  using Proposition \ref{bad set tilde triple}. For its proof, we refer to \cite{ternary}.
\item The  newly formed ternary collisional configuration   runs to a binary interaction under time evolution. This is the most challenging case to treat and is the heart of the technical  contribution, because the scaling \eqref{common scaling intro} does not directly help as in the case of the binary adjunction where one of the collisional particles enters an interaction zone. To treat this case, we need to use new algebraic tools (see Proposition \ref{bad set triple}) to exclude sets of initial data which lead to these pathological trajectories and develop elaborate geometric estimates to control its measure. The geometric estimates needed are thoroughly presented in Section \ref{sec:geometric}. In particular, Subsection \ref{subsec:relying} is devoted to developing novel tools which  rely on an appropriate representation of $(2d-1)$-spheres (see \eqref{representation of sphere for fixed omega_1}). More specifically, in  \ref{subsubsec:ball} we perform some initial truncations to the  impact directions,  while in  \ref{subsubsec:conic} we establish certain spherical cap and conic region estimates  needed to control the precollisional case, while \ref{subsubsec:annulus} focuses on developing the necessary annuli estimates  enabling us to control the postcollisional case using precollisional arguments.  After establishing the necessary geometric tools, we employ them in Proposition \ref{bad set triple measure} to show that the corresponding set constructed in Proposition \ref{bad set triple} is negligible.
\end{itemize}
\end{enumerate}

\subsection{Notation} For convenience, we introduce some basic notation which will be frequently used throughout the manuscript:
\begin{itemize}
\item  $d\in\mathbb{N}$ will be a fixed dimension with $d\geq 2$.
\item Given $x,y\in\mathbb{R}$, we write
\begin{align}
x&\lesssim y\Leftrightarrow\exists C_d>0: x\leq C_d y,\label{not ineq}\\
x&\simeq y\Leftrightarrow\exists C_d>0: x=C_d y,\label{not eq}\\
x&\thickapprox y\Leftrightarrow\exists C_{1,d},C_{2,d}>0: C_{1,d}y\leq x\leq C_{d,2}y.\label{not thick}
\end{align}
\item Given $n\in\mathbb{N}$, $\rho>0$ and $w\in\mathbb{R}^n$, we write $B_\rho^n(w)$ for the $n$-closed ball of radius $\rho>0$, centered at $w\in\mathbb{R}^n$.
In particular, we write $B_\rho^n:=B_\rho^n(0),$
for the $\rho$-ball centered at the origin.
\item Given $n\in\mathbb{N}$ and $\rho>0$, we write $\mathbb{S}_\rho^{n-1}$ for the $(n-1)$-sphere of radius $\rho>0$.
\item When we write $x<<y$, we mean that there is a small enough constant $0<c<1$, independent of $x,y$, such that $x<cy$. This constant $c$ is appropriately chosen  for the calculations to make sense.
\end{itemize}
\subsection*{Acknowledgements} I.A. and N.P. acknowledge support  from NSF grants DMS-1516228 and DMS-1840314. Authors are thankful to Irene M. Gamba, Maja Taskovi\'c, Thomas Chen,  Alexis Vasseur and Philip Morrison  for helpful discussions regarding physical and mathematical aspects of the problem. 
\section{Collisional transformations}\label{sec:collisional} In this section we define the collisional transformations of two and three interacting particles respectively. In the two particle case,  particles will interact as regular hard spheres, while in the three particle case, particles will interact as triplets of particles with an interaction zone.
\subsection{Binary interaction} 
 Here, we define the binary collisional tranformation of two interacting hard spheres, induced by an impact direction $\omega_1\in\mathbb{S}_1^{d-1}$. This  will be the law under which the velocities $(v_1,v_2)$ of two interacting hard spheres, with impact direction $\omega_1\in\mathbb{S}_1^{d-1}$, instanteously transform. The impact direction will represent the scaled relative position of the colliding hard spheres.
\begin{definition}
Consider a binary impact direction $\omega_1\in\mathbb{S}_1^{d-1}$. We define the binary  collisional transformation induced by $\omega_1\in\mathbb{S}_1^{d-1}$ as the map
$
T_{\omega_1}:(v_1,v_2)\in\mathbb{R}^{2d}\to (v_1',v_2')\in\mathbb{R}^{2d},
$
where
\begin{equation}\label{binary formulas without}
\begin{aligned}
v_1'&=v_1+\langle\omega_1,v_2-v_1\rangle\omega_1,\\
v_2'&=v_2-\langle\omega_1,v_2-v_1\rangle\omega_1.
\end{aligned}
\end{equation}
\end{definition}
Let us introduce some notation we will be constantly using. We define the binary cross-section
\begin{equation}\label{binary cross}
b_2(\omega_1,\nu_1):=\langle\omega_1,\nu_1\rangle,\quad (\omega_1,\nu_1)\in\mathbb{S}_1^{d-1}\times\mathbb{R}^d.
\end{equation}
Under this notation \eqref{binary formulas without} can be written as :
\begin{equation}\label{binary formulas with}
\begin{aligned}
v_1'&=v_1+b_2(\omega_1,v_2-v_1)\omega_1,\\
v_2'&=v_2-b_2(\omega_1,v_2-v_1)\omega_1.
\end{aligned}
\end{equation}
One can verify that \eqref{binary formulas with} provide the general solution, parametrized by $\omega_1\in\mathbb{S}_1^{d-1}$, of the binary momentum-energy conservation system:
\begin{equation}\label{MEC binary}
\begin{aligned}
v_1'+v_2'&=v_1+v_2,\\
|v_1'|^2+|v_2|^2&=|v_1|^2+|v_2|^2.
\end{aligned}
\end{equation}
Given a binary impact direction $\omega_1\in\mathbb{S}_1^{d-1}$, the binary  collisional transformation $T_{\omega_1}$ satisfies the following properties (see, e.g. \cite{cercignani gases}).
\begin{proposition}\label{binary scattering properties}Consider a binary impact direction $\omega_{1}\in\mathbb{S}_{1}^{d-1}$. The induced binary  collisional transformation $T_{\omega_{1}}$ has the following properties:
\begin{enumerate}[(i)]
  \item Conservation of momentum
  \begin{equation}\label{cons momentum binary}
  v_1'+v_2'=v_1+v_2.
  \end{equation}
  \item Conservation of energy
  \begin{equation}\label{cons energy binary}
  |v_1'|^2+|v_2'|^2=|v_1|^2+|v_2|^2.
  \end{equation}
  \item Conservation of relative velocities magnitude
\begin{equation}\label{relative veloc binary}
|v_1'-v_2'|=|v_1-v_2|.
\end{equation}
  \item Micro-reversibility of the binary cross-section 
\begin{equation}\label{skew symmetry binary}
b_2(\omega_1,v_2'-v_1')=-b_2(\omega_1,v_2-v_1).
\end{equation}
  \item $T_{\omega_{1}}$ is a linear involution i.e. $T_{\omega_{1}}$ is linear and
$
T_{\omega_{1}}^{-1}=T_{\omega_{1}}.
$
In particular, 
$
|\det T_{\omega_{1}}|=1,
$
so $T_{\omega_{1}}$ is measure-preserving.
\end{enumerate}
\end{proposition}
\subsection{Ternary interaction} Now we define the ternary collisional tranformation, induced by a given pair of impact directions, and investigate its properties. The interaction considered will be an instantaneous interaction of three particles with an interaction zone (for more details  see \cite{ternary}).  This will be the law under which the velocities $(v_1,v_2,v_3)$ of three interacting particles, with impact directions $(\omega_1,\omega_2)\in\mathbb{S}_1^{2d-1}$, instanteously transform.  The impact directions will represent the scaled relative positions of the three particles in the interaction zone  setting.

\begin{definition}
Consider a pair of impact directions $(\omega_{1},\omega_{2})\in\mathbb{S}_{1}^{2d-1}$. We  define the ternary collisional transformation induced by $(\omega_1,\omega_2)\in\mathbb{S}_1^{2d-1}$ as the map
$ T_{\omega_1,\omega_2}:(
v_1,
v_2,
v_3
)\in\mathbb{R}^{3d}\longrightarrow(
v_1^*,
v_2^*,
v_3^*)\in\mathbb{R}^{3d},
$
where
\begin{equation}\label{formulas ternary}\begin{cases}
v_1^*=v_1+c_{\omega_{1},\omega_{2}, v_1,v_2,v_3}(\omega_{1}+\omega_{2}),\\
v_2^*=v_2-c_{\omega_{1},\omega_{2}, v_1,v_2,v_3}\omega_{1},\\
v_3^*=v_3-c_{\omega_{1},\omega_{2}, v_1,v_2,v_3}\omega_2.
\end{cases}
\end{equation}
\begin{equation}\label{definition of c}
c_{\omega_1,\omega_2,v_1,v_2,v_3}=\frac{\langle\omega_1,v_2-v_1\rangle+\langle\omega_2,v_3-v_1\rangle}{1+\langle\omega_1,\omega_2\rangle}.
\end{equation}
\end{definition}
 We also define the ternary cross-section 
\begin{equation}\label{cross}
b_3(\omega_1,\omega_2,\nu_1,\nu_2):=\langle\omega_1,\nu_1\rangle+\langle\omega_2,\nu_2\rangle,\quad (\omega_1,\omega_2)\in\mathbb{S}_1^{2d-1},\quad (\nu_1,\nu_2)\in\mathbb{R}^{2d}.
\end{equation}
Notice that, given $(\omega_1,\omega_2,v_1,v_2,v_3)\in\mathbb{S}_1^{2d-1}\times\mathbb{R}^{3d}$, we clearly have
\begin{equation}\label{relation cross-c}
b_3(\omega_1,\omega_2,v_2-v_1,v_3-v_1)=\left(1+\langle\omega_1,\omega_2\rangle\right)c_{\omega_1,\omega_2,v_1,v_2,v_3}.
\end{equation}
\begin{remark}\label{estimate for b and c} Cauchy-Schwartz inequality and the fact that $(\omega_1,\omega_2)\in\mathbb{S}_1^{2d-1}$, yield 
\begin{equation}\label{bound on inverse quotient}
\frac{2}{3}\leq\frac{1}{1+\langle\omega_1,\omega_2\rangle}\leq 2,
\end{equation}
hence for all $(\omega_1,\omega_2,v_1,v_2,v_3)\in\mathbb{S}_1^{2d-1}\times\mathbb{R}^{3d}$, relation \eqref{relation cross-c} implies
\begin{equation}\label{bound for b relative to c}
\frac{2}{3}b_3(\omega_1,\omega_2,v_2-v_1,v_3-v_1)\leq c_{\omega_1,\omega_2,v_1,v_2,v_3}\leq 2b_3(\omega_1,\omega_2,v_2-v_1,v_3-v_1).
\end{equation}
\end{remark}
It has been seen in \cite{ternary,thesis} that \eqref{formulas ternary} provide the general solution, parametrized by $(\omega_1,\omega_2)\in\mathbb{S}_1^{2d-1}$, of the ternary momentum-energy conservation system:
\begin{equation}\label{ternary MEC}
\begin{aligned}
v_1^*+v_2^*+v_3^*&=v_1+v_2+v_3,\\
|v_1^*|^2+|v_2^*|^2+|v_3^*|^2&=|v_1|^2+|v_2|^2+|v_3|^2.
\end{aligned}
\end{equation}
The main properties of the ternary collisional tranformation are summarized in the following Proposition. For the proof, see Proposition 2.3. from \cite{ternary}.
\begin{proposition}\label{triary scattering properties} Consider a pair of impact directions $(\omega_{1},\omega_{2})\in\mathbb{S}_{1}^{2d-1}$. The induced collisional transformation $T_{\omega_{1},\omega_{2}}$ has the following properties:
\begin{enumerate}[(i)]
  \item Conservation of momentum
  \begin{equation}\label{triary cons momentum}
  v_1^*+v_2^*+v_3^*=v_1+v_2+v_3.
  \end{equation}
  \item Conservation of energy
  \begin{equation}\label{triary cons energy}
  |v_1^*|^2+|v_2^*|^2+|v_3^*|^2=|v_1|^2+|v_2|^2+|v_3|^2.
  \end{equation}
  \item Conservation of relative velocities magnitude
\begin{equation}\label{triary relative velocities}
|v_1^*-v_2^*|^2+|v_1^*-v_3^*|^2+|v_2^*-v_3^*|^2=|v_1-v_2|^2+|v_1-v_3|^2+|v_2-v_3|^2.
\end{equation}
  \item Micro-reversibility of the ternary cross-section 
\begin{equation}\label{skew symmetry triary}
b_3(\omega_{1},\omega_{2},v_2^*-v_1^*,v_3^*-v_1^*)=-b_3(\omega_{1},\omega_{2},v_2-v_1,v_3-v_1).
\end{equation}
  \item $T_{\omega_{1},\omega_{2}}$ is a linear involution i.e. $T_{\omega_{1},\omega_{2}}$ is linear and
$
T_{\omega_{1},\omega_{2}}^{-1}=T_{\omega_{1},\omega_{2}}.
$
In particular, 
$
|\det T_{\omega_{1},\omega_{2}}|=1,
$
so $T_{\omega_{1},\omega_{2}}$ is measure-preserving.
\end{enumerate}
\end{proposition}
\section{Dynamics of $m$-particles}\label{sec:dynamics}
In this section we rigorously define the dynamics of $m$ hard spheres of diameter $\sigma_2$ and interaction zone  $\sigma_3$, where $0<\sigma_2<\sigma_3<1$. Heuristically speaking, particles perform rectilinear motion as long as there is no interaction (binary or ternary) and they interact through the binary or ternary collision law when a binary or ternary interaction occurs respectively. However, it is far from obvious that a global dynamics can be defined, since the system might run into pathological configurations e.g. more than one interactions at a time, infinitely many interactions in finite time or interactions which graze under time evolution. The goal of this section is to extract a set of measure zero such that on the complement a global in time, measure preserving flow can be defined.

Throughout this section we consider $m\in\mathbb{N}$ and $0<\sigma_2<\sigma_3<1$.
\subsection{Phase space definitions}
For convenience we define the following index sets:
\begin{align}
\text{For $m\geq 2$: }\mathcal{I}_m^2&=\left\{(i,j)\in\{1,...,m\}^2:i<j\right\}.\label{index 2}\\
\text{For $m\geq 3$: }\mathcal{I}_m^3&=\left\{(i,j,k)\in\{1,...,m\}^3:i<j<k\right\}.\label{index 3}
\end{align}
Given positions $(x_1,x_2)\in\mathbb{R}^{2d}$, we define the binary  distance:
\begin{equation}\label{binary distance}
d_2(x_1,x_2):=|x_1-x_2|,
\end{equation}
and given positions $(x_1,x_2,x_3)\in\mathbb{R}^{3d}$, we define the ternary  distance:
\begin{equation}\label{triary distance}
d_3(x_1;x_2,x_3)=\sqrt{|x_1-x_2|^2+|x_1-x_3|^2}.
\end{equation}
For $m\geq 3$, we define the phase space of $m$-particles of diameter $\sigma_2>0$ and interaction zone $\sigma_3>0$, with $\sigma_2<\sigma_3<1$ as:
\begin{equation}\label{phase space}
\begin{aligned}
\mathcal{D}_{m,\sigma_2,\sigma_3}=\big\{Z_m=(X_m,V_m)\in\mathbb{R}^{2dm}: d_2(x_i,x_j)\geq\sigma_2,\text{ }\forall (i,j)\in\mathcal{I}_m^2,
\text{ and }d_3(x_i;x_j,x_k)\geq\sqrt{2}\sigma_3,\text{ }\forall (i,j,k)\in\mathcal{I}_m^3\big\},
\end{aligned}
\end{equation}
where  $X_m=(x_1,...,x_m)\in\mathbb{R}^{dm}$ represents the positions of the $m$-particles, while $V_m=(v_1,...,v_m)\in\mathbb{R}^{dm}$ represents the velocities of the $m$-particles. 
For convenience we also define 
\begin{equation}\label{phase space m=2}
\mathcal{D}_{2,\sigma_2,\sigma_3}=\left\{Z_2=(X_2,V_2)\in\mathbb{R}^{2d}:|x_1-x_2|\geq\sigma_2\right\},\quad
\mathcal{D}_{1,\sigma_2,\sigma_3}=\mathbb{R}^{2d}.
\end{equation}
For $m\geq 3$, the phase space $\mathcal{D}_{m,\sigma_2,\sigma_3}$ decomposes as:
$\mathcal{D}_{m,\sigma_2,\sigma_3}=\mathring{\mathcal{D}}_{m,\sigma_2,\sigma_3}\cup\partial\mathcal{D}_{m,\sigma_2,\sigma_3},$
where the interior is given by:
\begin{equation}\label{interior phase space}
\begin{aligned}
\mathring{\mathcal{D}}_{m,\sigma_2,\sigma_3}=\big\{Z_m=(X_m,V_m)\in\mathbb{R}^{2dm}: d_2(x_i,x_j)>\sigma_2,\text{ }\forall (i,j)\in\mathcal{I}_m^2, 
\text{ and }d_3(x_i;x_j,x_k)>\sqrt{2}\sigma_3,\text{ }\forall (i,j,k)\in\mathcal{I}_m^3\big\},
\end{aligned}
\end{equation}
and the boundary is given by:
\begin{equation}\label{boundary}
\partial\mathcal{D}_{m,\sigma_2,\sigma_3}=\partial_2\mathcal{D}_{m,\sigma_2,\sigma_3}\cup\partial_3\mathcal{D}_{m,\sigma_2,\sigma_3},
\end{equation}
where $\partial_2\mathcal{D}_{m,\sigma_2,\sigma_3}$ is the binary boundary:
\begin{equation}\label{binary boundary}
\partial_2\mathcal{D}_{m,\sigma_2,\sigma_3}=\left\{Z_m=(X_m,V_m)\in\mathcal{D}_{m,\sigma_2,\sigma_3}:\exists (i,j)\in\mathcal{I}_m^2\text{ with }d_2(x_i,x_j)=\sigma_2\right\},
\end{equation}
and $\partial_3\mathcal{D}_{m,\sigma_2,\sigma_3}$ is the ternary boundary:
\begin{equation}\label{triary boundary}
\partial_3\mathcal{D}_{m,\sigma_2,\sigma_3}=\left\{Z_m=(X_m,V_m)\in\mathcal{D}_{m,\sigma_2,\sigma_3}:\exists (i,j,k)\in\mathcal{I}_m^3\text{ with }d_3(x_i;x_j,x_k)=\sqrt{2}\sigma_3\right\}.
\end{equation}
Elements of $\mathcal{D}_{m,\sigma_2,\sigma_3}$ are called configurations, elements of $\mathring{\mathcal{D}}_{m,\sigma_2,\sigma_3}$ are called noncollisional configurations, and elements of $\partial_2\mathcal{D}_{m,\sigma_2,\sigma_3}$ are called collisional configurations, or just collisions. Elements of $\partial\mathcal{D}_{m,\sigma_2,\sigma_3}$ are called binary collisions, while elements of $\partial_3\mathcal{D}_{m,\sigma_2,\sigma_3}$ are called ternary collisions. When we refer to a collision, it will be either binary or ternary.

Clearly the binary boundary can be written as:
$\partial_2\mathcal{D}_{m,\sigma_2,\sigma_3}=\bigcup_{(i,j)\in\mathcal{I}_m^2}\Sigma_{ij}^2,$
where $\Sigma_{ij}^2$ are the binary collisional surfaces given by
\begin{equation}\label{binary collisional surfaces}
\Sigma_{ij}^2:=\left\{Z_m\in\mathcal{D}_{m,\sigma_2,\sigma_3}:d_2(x_i,x_j)=\sigma_2\right\}.
\end{equation}
In the same spirit the ternary boundary can be written as:
$\partial_3\mathcal{D}_{m,\sigma_2,\sigma_3}=\bigcup_{(i,j,k)\in\mathcal{I}_m^3}\Sigma_{ijk}^3,$
where $\Sigma_{ijk}^3$ are the ternary collisional surfaces given by
\begin{equation}\label{triary collisional surfaces}
\Sigma_{ijk}^3:=\left\{Z_m\in\mathcal{D}_{m,\sigma_2,\sigma_3}:d_3(x_i;x_j,x_k)=\sqrt{2}\sigma_3\right\}.
\end{equation}

We now further decompose collisions to simple binary collisions, simple ternary collisions and multiple collisions. In particular we define simple binary  collisions as:
\begin{equation}\label{simple binary collisions}
\begin{aligned}
\partial_{2,sc}\mathcal{D}_{m,\sigma_2,\sigma_3}:=&\big\{Z_m=(X_m,V_m)\in\mathcal{D}_{m,\sigma_2,\sigma_3}:\exists (i,j)\in\mathcal{I}_m^2 \text{ with }Z_m\in\Sigma_{ij}^2,\\
&Z_m\notin\Sigma_{i'j'}^2,\text{ }\forall (i',j')\in\mathcal{I}_m^2\setminus\{(i,j)\},\text{ }Z_m\notin\Sigma_{i'j'k'}^3,\text{ }\forall (i',j',k')\in\mathcal{I}_m^3
\big\}.
\end{aligned}
\end{equation}
We also define simple ternary collisions as:
\begin{equation}\label{simple triary collisions}
\begin{aligned}
\partial_{3,sc}\mathcal{D}_{m,\sigma_2,\sigma_3}:=&\big\{Z_m=(X_m,V_m)\in\mathcal{D}_{m,\sigma_2,\sigma_3}:\exists (i,j,k)\in\mathcal{I}_m^3 \text{ with }Z_m\in\Sigma_{ijk}^3,\\
&Z_m\notin\Sigma_{i'j'k'}^3,\text{ }\forall (i',j',k')\in\mathcal{I}_m^3\setminus\{(i,j,k)\},
Z_m\notin\Sigma_{i'j'}^2,\text{ }\forall (i',j')\in\mathcal{I}_m^2
\big\}.
\end{aligned}
\end{equation}
\begin{remark}\label{remark on phase space} The assumption $\sigma_2<\sigma_3$ made at the beginning of the section is necessary for $\partial_{3,sc}\mathcal{D}_{m,\sigma_2,\sigma_3}$ to be non-empty. Indeed,  let $\sigma_2\geq\sigma_3$ and assume that $\partial_{3,sc}\mathcal{D}_{m,\sigma_2,\sigma_3}\neq\emptyset$. Consider $Z_m\in \partial_{3,sc}\mathcal{D}_{m,\sigma_2,\sigma_3}$. Then, by \eqref{simple triary collisions}, there is $(i,j,k)\in\mathcal{I}_m^3$ such that
\begin{equation}\label{bigger triple}
|x_i-x_j|^2+|x_i-x_j|^2=2\epsilon_3^2,
\end{equation}
and
\begin{equation}\label{bigger double}
|x_i-x_j|>\epsilon_2,\quad |x_i-x_k|>\epsilon_2.
\end{equation}
By \eqref{bigger triple}, at least one of $|x_i-x_j|$ or $|x_i-x_k|$ has to be smaller  than or equal to  $\epsilon_3$. Assume, without loss of generality, that $|x_i-x_j|\leq\epsilon_3$. Since $\epsilon_2\geq\epsilon_3$, we obtain $|x_i-x_j|\leq\epsilon_2$, which contradicts \eqref{bigger double}. Therefore, if $\sigma_2\geq\sigma_3$, we have  $\partial_{3,sc}\mathcal{D}_{m,\sigma_2,\sigma_3}=\emptyset$.
\end{remark}
A simple collision will be a binary or ternary simple collision i.e.
\begin{equation}\label{simple collision boundary}
\partial_{sc}\mathcal{D}_{m,\sigma_2,\sigma_3}:=\partial_{2,sc}\mathcal{D}_{m,\sigma_2,\sigma_3}\cup \partial_{3,sc}\mathcal{D}_{m,\sigma_2,\sigma_3}.
\end{equation}
Multiple collisions are configurations which are not simple i.e. 
\begin{equation}\label{multiple collisions}
\begin{aligned}
\partial_{mu}\mathcal{D}_{m,\sigma_2,\sigma_3}&:=\partial\mathcal{D}_{m,\sigma_2,\sigma_3}\setminus\partial_{sc}\mathcal{D}_{m,\sigma_2,\sigma_3}.
\end{aligned}
\end{equation}

\begin{remark}
For $m=2$, there is only binary boundary.
\end{remark}

For the binary case, we give the following definitions:

\begin{definition} Let $m\geq 2$ and $Z_m\in\partial_{2,sc} D_{m,\sigma_2,\sigma_3}$. Then there is a unique $(i,j)\in\mathcal{I}_m^2$ such that $Z_m\in\Sigma_{ij}^{2}$ and $Z_m\notin\Sigma_{i'j'k'}^3$, for all $(i',j',k')\in\mathcal{I}_m^3$. In this case we will say $Z_m$ is an $(i,j)$ collision and we will write
\begin{equation}\label{simple collision surfaces binary}
\Sigma_{ij}^{2,sc}=\left\{Z_m\in\mathcal{D}_{m,\sigma_1,\sigma_2}: Z_m\mbox{ is $(i,j)$ collision}\right\}.
\end{equation} 
Clearly $\Sigma_{ij}^{2,sc}\cap\Sigma_{i'j'}^{2,sc}=\emptyset$, for all $ (i,j)\neq(i',j')\in\mathcal{I}_m^2$ and $\partial_{2,sc}\mathcal{D}_{m,\sigma_2,\sigma_3}$ decomposes to:
\begin{equation}\label{decomposition simple binary}
\partial_{2,sc}\mathcal{D}_{m,\sigma_2,\sigma_3}=\bigcup_{(i,j)\in\mathcal{I}_m^2}\Sigma_{ij}^{2,sc}.
\end{equation}
\end{definition}

\begin{remark} Let $m\geq 2$, $(i,j)\in\mathcal{I}_m^2$ and $Z_m\in\Sigma_{ij}^{2,sc}$. Then  
\begin{equation}\label{def of omega binary}
\omega_{1}:=\frac{x_j-x_i}{\sigma_2}\in\mathbb{S}_1^{d-1}.
\end{equation}
 Therefore, each $(i,j)$ collision naturally induces a binary impact direction $\omega_{1}\in\mathbb{S}_1^{d-1}$ and consequently a binary collisional transformation $T_{\omega_{1}}$.
\end{remark}

\begin{definition}
Let $m\geq 2$, $(i,j)\in\mathcal{I}_m^2$ and $Z_m=(X_m,V_m)\in\Sigma_{ij}^{2,sc}$. We write
$Z_m'=(X_m,V_m'),
$
where
\begin{equation*}
V_m'=(v_1,...,v_{i-1},v_i',v_{i+1},...,v_{j-1},v_j',v_{j+1},...,v_{m}),
\end{equation*}
and
$
(v_i',v_j')=T_{\omega_{1}}(v_i,v_j),\quad \omega_{1}\in\mathbb{S}_1^{d-1}\text{ is given by \eqref{def of omega binary}}.
$
\end{definition}

In the same spirit, for the ternary case, we give the following definitions:

\begin{definition} Let $m\geq 3$ and $Z_m\in\partial_{3,sc} D_{m,\sigma_2,\sigma_3}$. Then there is a unique  $(i;j,k)\in\mathcal{I}_m^3$ such that $Z_m\in\Sigma_{ijk}^3$ and $Z_m\notin \Sigma_{i'j'}^2$, for all $(i',j')\in\mathcal{I}_m^2$. In this case we will say $Z_m$ is an $(i;j,k)$ collision and we will write
\begin{equation}\label{simple collision surfaces triary}
\Sigma_{ijk}^{3,sc}=\left\{Z_m\in\mathcal{D}_{m,\sigma_2,\sigma_3}: Z_m\mbox{ is $(i;j,k)$ collision}\right\}.
\end{equation} 
Clearly $\Sigma_{ijk}^{3,sc}\cap\Sigma_{i'j'k'}^{3,sc}=\emptyset$, for all  $(i,j,k)\neq(i',j',k')\in\mathcal{I}_m^3$ and $\partial_{3,sc}\mathcal{D}_{m,\sigma_2,\sigma_3}$ decomposes to:
\begin{equation}\label{decomposition of simple ternary}
\partial_{3,sc}\mathcal{D}_{m,\sigma_2,\sigma_3}=\bigcup_{(i,j,k)\in\mathcal{I}_m^3}\Sigma_{ijk}^{3,sc}.
\end{equation}
\end{definition}

\begin{remark} Let $m\geq 3$, $(i,j,k)\in\mathcal{I}_m^3$ and $Z_m\in\Sigma_{ijk}^{3,sc}$. Then  
\begin{equation}\label{def of omega triary}
(\omega_{1},\omega_2):=\frac{1}{\sqrt{2}\sigma_3}\left(x_j-x_i,x_k-x_i\right)\in\mathbb{S}_1^{2d-1}.
\end{equation}
 Therefore, each $(i;j,k)$ collision naturally induces ternary impact directions $(\omega_{1},\omega_{2})\in\mathbb{S}_1^{2d-1}$ and consequently a collisional transformation $T_{\omega_{1},\omega_{2}}$.
\end{remark}

\begin{definition}
Let $m\geq 3$, $(i,j,k)\in\mathcal{I}_m^3$ and $Z_m=(X_m,V_m)\in\Sigma_{ijk}^{3,s}$. We write
$ Z_m^*=(X_m,V_m^*),
$
where
\begin{equation*}
V_m^*=(v_1,...,v_{i-1},v_i^*,v_{i+1},...,v_{j-1},v_j^*,v_{j+1},...,v_{k-1},v_k^*,v_{k+1},...,v_{m}),
\end{equation*}
and
$
(v_i^*,v_j^*,v_k^*)=T_{\omega_{1},\omega_{2}}(v_i,v_j,v_k),\quad (\omega_{1},\omega_{2})\in\mathbb{S}_1^{2d-1}\text{ are given by \eqref{def of omega triary}}.
$
\end{definition}
\subsection{Classification of simple collisions}
We will now classify simple collisions in order to eliminate collisions which graze in time. For this purpose, we come across the following definitions for the binary and the ternary case respectively.

For the binary case:

\begin{definition}\label{collision class binary}
Let $m\geq 2$, $(i,j)\in\mathcal{I}_m^2$ and $Z_m\in\Sigma_{ij}^{2,s}$. The configuration $Z_m$ is called:
\begin{itemize}
\item binary precollisional when $b_2(\omega_{1},v_j-v_i)<0,$
\item binary postcollisional when $b_2(\omega_{1},v_j-v_i)>0,$
\item binary grazing when $b_2(\omega_{1},v_j-v_i)=0,$
\end{itemize}
where $\omega_1\in\mathbb{S}_1^{d-1}$ is given by \eqref{def of omega binary} and $b_2$ is given by \eqref{binary cross}.
\end{definition}
\begin{remark}
 Let $m\geq 2$, $(i,j)\in\mathcal{I}_m^2$ and $Z_m\in\Sigma_{ij}^{2,s}$. Using \eqref{skew symmetry binary}, we obtain the following:
 \begin{enumerate}[(i)]
\item $Z_m$ is binary precollisional iff $Z_m'$ is binary postcollisional.\vspace{0.2cm}
\item $Z_m$ is binary postcollisional iff $Z_m'$ is binary precollisional.\vspace{0.2cm}
\item $Z_m=Z_m'$ iff $Z_m$ is binary grazing.
\end{enumerate}
\end{remark}

For the ternary case:

\begin{definition}\label{collision class triary}
Let $m\geq 3$, $(i,j,k)\in\mathcal{I}_m^3$ and $Z_m\in\Sigma_{ijk}^{3,s}$. The configuration $Z_m$ is called:
\begin{itemize}
\item ternary precollisional when $b_3(\omega_{1},\omega_{2},v_j-v_i,v_k-v_i)<0,$
\item ternary postcollisional when $b_3(\omega_{1},\omega_{2},v_j-v_i,v_k-v_i)>0,$
\item ternary grazing when $b_3(\omega_{1},\omega_{2},v_j-v_i,v_k-v_i)=0,$
\end{itemize}
where $(\omega_{1},\omega_{2})\in\mathbb{S}_1^{2d-1}$ is given by \eqref{def of omega triary} and $b$ is given by \eqref{cross}. 
\end{definition}
\begin{remark}
 Let $m\geq 3$, $(i,j,k)\in\mathcal{I}_m^3$ and $Z_m\in\Sigma_{ijk}^{3,s}$. Using \eqref{skew symmetry triary}, we obtain the following:
 \begin{enumerate}[(i)]
\item $Z_m$ is ternary precollisional iff $Z_m^*$ is ternary postcollisional.\vspace{0.2cm}
\item $Z_m$ is ternary postcollisional iff $Z_m^*$ is ternary precollisional.\vspace{0.2cm}
\item $Z_m=Z_m^*$ iff $Z_m$ is ternary grazing.
\end{enumerate}
\end{remark}

We will just say precollisional, postcollisional or grazing configuration when it is implied whether a simple collision is binary or ternary.

For $m\geq 2$, we refine the phase space defining
\begin{equation}\label{refined phase space}
\mathcal{D}_{m,\sigma_2,\sigma_3}^*:=\mathring{\mathcal{D}}_{m,\sigma_2,\sigma_3}\cup \partial_{sc,ng}\mathcal{D}_{m,\sigma_2,\sigma_3},
\end{equation}
where $\partial_{sc,ng}\mathcal{D}_{m,\sigma_2,\sigma_3}$ denotes the  part of $\partial\mathcal{D}_{m,\sigma_2,\sigma_3}$ consisting of simple, non-grazing collisions i.e. defined as
\begin{equation}\label{refined phase boundary}
\partial_{sc,ng}\mathcal{D}_{m,\sigma_2,\sigma_3}:=\left\{Z_m\in\partial_{sc}\mathcal{D}_{m,\sigma_2,\sigma_3}: Z_m\text{ is non-grazing}\right\}.
\end{equation}
It is immediate that $\mathcal{D}_{m,\sigma_2,\sigma_3}^*$ is a full measure subset of $\mathcal{D}_{m,\sigma_2,\sigma_3}$ and $\partial_{sc,ng}\mathcal{D}_{m,\sigma_2,\sigma_3}$ is a full surface measure subset of $\partial\mathcal{D}_{m,\sigma_2,\sigma_3}$, since its complement constitutes of lower dimension submanifolds of $\partial\mathcal{D}_{m,\sigma_2,\sigma_3}$ which have zero surface measure.
\subsection{Construction of the local flow} 
Next Lemma shows that the flow can be locally defined for any initial configuration  $Z_m\in\mathcal{D}_{m,\sigma_2,\sigma_3}^*$ up to the time of the first collision.
\begin{lemma}\label{elem dyn step}
Let $m\geq 3$ and $Z_m\in\mathcal{D}_{m,\sigma_2,\sigma_3}^*$. Then there is a time $\tau^1_{Z_m}\in (0,\infty]$ such that defining $Z_m(\cdot):[0,\tau_{Z_m}^1]\to\mathbb{R}^{2dm}$ by:
\begin{equation*}
Z_m(t)=
\begin{cases}
(X_m+tV_m,V_m)\quad \text{if }Z_m\text{ is noncollisional or postcollisional,}\\
(X_m+tV_m',V_m'),\quad\text{if }Z_m\text{ is binary precollisional},\\
(X_m+tV_m^*,V_m^*),\quad\text{if }Z_m\text{ is ternary precollisional},\\
\end{cases}
\end{equation*} the following hold:
\begin{enumerate}[(i)]
\item $Z_m(t)\in\mathring{\mathcal{D}}_{m,\sigma_2,\sigma_3},\quad\forall t\in (0,\tau^1_{Z_m})$.\vspace{0.2cm}
\item if $\tau^1_{Z_m}<\infty$, then $Z_m(\tau^1_{Z_m})\in\partial\mathcal{D}_{m,\sigma_2,\sigma_3}$.\vspace{0.2cm}
\item If $Z_m\in\Sigma_{ij}^{2,sc}$ for some $(i,j)\in\mathcal{I}_m^2$, then $Z_m(\tau^1_{Z_m})\notin\Sigma_{ij}^{2}$.\vspace{0.2cm}
\item If $Z_m\in\Sigma_{ijk}^{3,sc}$ for some $(i,j,k)\in\mathcal{I}_m^3$, then $Z_m(\tau^1_{Z_m})\notin\Sigma_{ijk}^3$.
\end{enumerate}
An analogous statement holds in the case $m=2$, where we just neglect the ternary terms.
\end{lemma}
\begin{proof}
Let us make the convention
$\inf\emptyset=+\infty$.
We define
\begin{equation*}
\begin{aligned}
\tau_{Z_m}^1=
\begin{cases}
\inf\left\{t>0:X_m+tV_m\in\partial\mathcal{D}_{m,\sigma_2,\sigma_3}\right\},\quad\text{if $Z_m$ is noncollisional or postcollisional} , \\
\inf\left\{t>0:X_m+tV_m'\in\partial\mathcal{D}_{m,\sigma_2,\sigma_3}\right\},\quad\text{if $Z_m$ is binary precollisional},\\
\inf\left\{t>0:X_m+tV_m^*\in\partial\mathcal{D}_{m,\sigma_2,\sigma_3}\right\},\quad\text{if $Z_m$ is ternary  precollisional}.
\end{cases}
\end{aligned}
\end{equation*}
Since $\mathring{\mathcal{D}}_{m,\sigma_2,\sigma_3}$ is open, we get $\tau_{Z_m}^1>0,\quad\forall Z_m\in\mathring{\mathcal{D}}_{m,\sigma_2,\sigma_3}$ and claims \textit{(i)-(ii)} follow immediately  for $Z_m\in\mathring{\mathcal{D}}_{m,\sigma_2,\sigma_3}$.

Assume $Z_m\in\partial_{sc,ng}\mathcal{D}_{m,\sigma_2,\sigma_3}$ which yields that $Z_m$ is non-grazing. Therefore we may distinguish the following cases:
\begin{itemize}
 \item $Z_m$ is an $(i,j)$ binary  postcollisional configuration: For any $t>0$, we have
 \begin{equation*}
\begin{aligned}
|x_i-x_j+(v_i-v_j)t|^2&= |x_i-x_j|^2+t^2|v_i-v_j|^2+2t\langle x_i-x_j,v_i-v_j\rangle\\
&\geq \sigma_2^2+2tb_2(x_j-x_i,v_j-v_i)\\
&>\sigma_2^2,
\end{aligned}
\end{equation*}
since $b_2(\omega_1,v_j-v_i)>0$. This inequality and the fact that $Z_m$ is  a simple binary collision imply that $\tau_{Z_m}^1>0$ and claims $(i),(ii),(iii)$ as well. \vspace{0.2cm}
\item $Z_m$ is $(i,j)$ binary precollisional configuration: We use the same argument for $Z_m'$ which is $(i,j)$ binary postcollisional.\vspace{0.2cm}
 \item $Z_m$ is an $(i;j,k)$ ternary postcollisional configuration: For any $t>0$, we have
 \begin{equation*}
\begin{aligned}
&|x_i-x_j+(v_i-v_j)t|^2+|x_i-x_k+(v_i-v_k)t|^2\\
&= |x_i-x_j|^2+|x_i-x_k|^2+t^2\left(|v_i-v_j|^2+|v_i-v_k|^2\right)+2t\left(\langle x_i-x_j,v_i-v_j\rangle +\langle x_i-x_k,v_i-v_k\rangle\right)\\
&\geq 2\sigma_3^2+2tb_3(x_j-x_i,x_k-x_i,v_j-v_i,v_k-v_i)\\
&>2\sigma_3^2,
\end{aligned}
\end{equation*}
since $b_3(\omega_{1},\omega_{2},v_j-v_i,v_k-v_i)>0$. This inequality and the fact that $Z_m$ is a simple ternary collision imply that $\tau_{Z_m}^1>0$ and claims $(i),(ii),(iv)$ as well.\vspace{0.2cm}
\item $Z_m$ is an $(i;j,k)$ ternary precollisional configuration: We use the same argument for $Z_m^*$ which is $(i;j,k)$ ternary postcollisional. 
\end{itemize}
\end{proof}

Let us make an elementary but crucial remark.
\begin{remark}\label{remark on t1-t2} Clearly for configurations with $\tau_{Z_m}^1=\infty$ the flow is globally defined as the free flow. In the case
 where $\tau_{Z_m}^1<\infty$ and $Z_m(\tau_{Z_m}^1)$ is a non-grazing $(i,j)$ collision or non-grazing $(i;j,k)$ collision, we may apply Lemma \ref{elem dyn step} once more and get a corresponding time $\tau_{Z_m}^2$ with the property that $Z_m(\tau_{Z_m}^2)\notin\Sigma_{ij}^2$ or $Z_m(\tau_{Z_m}^2)\notin\Sigma_{ijk}^3$ respectively, if $\tau_{Z_m}^2<\infty$. Therefore, in this case the flow can be defined up to time $\tau_{Z_m}^2$. 
\end{remark}
\begin{remark} Note that Lemma \ref{elem dyn step} implies that given a non-grazing $(i,j)$ collision, the next collision (if it happens) will not be $(i,j)$.  Similarly,  given a non-grazing $(i;j,k)$ collision, the next collision (if it happens) will not be $(i;j,k)$  However, Lemma \ref{elem dyn step} it does not imply that the same particles are not involved in a collision of a different type. For  instance, one could have the sequence of collisions $(i,j)$ and $(i;j,k)$, or $(i;j,k)$ and $(i,j)$ etc. All these cases will be taken into account when establishing a global flow in Subsection \ref{subsec extension to a global flow}.
\end{remark}
\begin{remark}
Similar results hold for the case $m=2$ where there are no ternary interactions.
\end{remark}
\subsection{Extension to a global flow}\label{subsec extension to a global flow}
Now, we extract a zero measure set from $\mathcal{D}_{m,\sigma_2,\sigma_3}^*$ such that the flow is globally defined on the complement. For this purpose, we will first truncate positions and velocities using two parameters $1<<R<\rho$ and then perform  time truncation with a small parameter $\delta$ in the scaling:
\begin{equation}\label{scaling dynamics}
0<\delta R<<\sigma_2<\sigma_3<1<<R<\rho.
\end{equation} 
Throughout this subsection, we consider parameters satisfying the scaling \eqref{scaling dynamics}.

Recall that given $r>0$ we denote the $dm$-ball of radius $r>0$, centered at the origin as $B_r^{dm}$.
We first assume initial positions are in $B_\rho^{dm}$ and initial velocities in $B_R^{dm}$. 
 
For $m\geq 2$, we decompose $D_{m,\sigma_2,\sigma_3}^*\cap(B_\rho^{dm}\times B_R^{dm})$ in the following subsets:
\begin{equation*}
\begin{aligned}
&I_{free}=\big\{Z_m=(X_m,V_m)\in D_{m,\sigma_2,\sigma_3}^*\cap(B_\rho^{dm}\times B_R^{dm}): \tau_{Z_m}^1>\delta\big\},\\
&I_{sc,ng}^1=\big\{Z_m=(X_m,V_m)\in D_{m,\sigma_2,\sigma_3}^*\cap(B_\rho^{dm}\times B_R^{dm}): \tau_{Z_m}^1\leq\delta,\text{ } Z_m(\tau_{Z_m}^1)\in\partial_{sc,ng}\mathcal{D}_{m,\sigma_2,\sigma_3}
\text{ and } \tau_{Z_m}^2>\delta\big\},\\
&I_{sc,g}^1=\big\{Z_m=(X_m,V_m)\in D_{m,\sigma_2,\sigma_3}^*\cap(B_\rho^{dm}\times B_R^{dm}): \tau_{Z_m}^1\leq\delta,\text{ } Z_m(\tau_{Z_m}^1)\in\partial_{sc}\mathcal{D}_{m,\sigma_2,\sigma_3},\text{ and $Z_m(\tau_{Z_m}^1)$ is grazing}\big\},\\
&I_{mu}^1=\big\{Z_m=(X_m,V_m)\in D_{m,\sigma_2,\sigma_3}^*\cap(B_\rho^{dm}\times B_R^{dm}): \tau_{Z_m}^1\leq\delta,\text{ } Z_m(\tau_{Z_m}^1)\in\partial_{mu}\mathcal{D}_{m,\sigma_2,\sigma_3}\big\},\\
&I_{sc,ng}^2=\big\{Z_m=(X_m,V_m)\in D_{m,\sigma_2,\sigma_3}^*\cap(B_\rho^{dm}\times B_R^{dm}): \tau_{Z_m}^1\leq\delta,\text{ } Z_m(\tau_{Z_m}^1)\in\partial_{sc,ng}\mathcal{D}_{m,\sigma_2,\sigma_3},\text{ but } \tau_{Z_m}^2\leq\delta\big\}.
\end{aligned}
\end{equation*}

We remark that there is a well-defined flow  up to time $\delta$ for $Z_m\in I_{free}\cup I^1_{sc,ng}$, since  in such cases one has at most one simple non-grazing collision in $[0,\delta]$. We aim to estimate the measure of the pathological set $I^1_{sc,g}\cup I^1_{mu}\cup I^2_{sc,ng}$, with respect to the truncation parameters.

\begin{lemma}\label{zero measure} Assume $m\geq 2$. Then $I_{sc,g}^1$ is of zero Lebesgue measure. 
\end{lemma}
\begin{proof} Assume first $m\geq 3$.
Clearly 
$I_{sc,g}^1\subseteq \displaystyle\bigcup_{(i,j)\in\mathcal{I}_m^2}M_{ij}^2\cup\displaystyle\bigcup_{(i,j,k)\in\mathcal{I}_m^3}M_{ijk}^3,$
where
\begin{equation*}
\begin{aligned}
M_{ij}^2&=\left\{Z_m\in D_{m,\sigma_2,\sigma_3}^*\cap(B_\rho^{dm}\times B_R^{dm}):Z_m(\tau_{Z_m}^1)\text{ is an $(i,j)$ grazing collision}\right\},\\
M_{ijk}^3&=\left\{Z_m\in D_{m,\sigma_2,\sigma_3}^*\cap(B_\rho^{dm}\times B_R^{dm}):Z_m(\tau_{Z_m}^1)\text{ is an $(i;j,k)$ grazing collision}\right\}.
\end{aligned}
\end{equation*}
The above covering consists of lower dimension submanifolds of the space, so it has zero measure. For $m=2$, we use a similar argument.
\end{proof}

Before proceeding to the next result, let us note that conservation of energy \eqref{cons energy binary}, \eqref{triary cons energy} imply the  following elementary but useful remark: 
\begin{remark}\label{remark on conservation} The following hold:
\begin{itemize}
\item For $m\geq 2$:
$Z_m\in\partial_{2,sc}\mathcal{D}_{m,\sigma_2,\sigma_3}\cap(\mathbb{R}^{dm}\times B_R^{dm})\Leftrightarrow Z_m'\in\partial_{2,sc}\mathcal{D}_{m,\sigma_2,\sigma_3}\cap(\mathbb{R}^{dm}\times B_R^{dm}).$
\item For $m\geq 3$:
$Z_m\in\partial_{3,sc}\mathcal{D}_{m,\sigma_2,\sigma_3}\cap(\mathbb{R}^{dm}\times B_R^{dm})\Leftrightarrow Z_m^*\in\partial_{3,sc}\mathcal{D}_{m,\sigma_2,\sigma_3}\cap(\mathbb{R}^{dm}\times B_R^{dm}).$
\end{itemize}
\end{remark}
\begin{lemma}\label{covering} For $m\geq 3$, the following inclusion holds:
\begin{equation}\label{required inclusion}
I_{mu}^1\cup I_{sc,ng}^2\subseteq U_{22}\cup U_{23}\cup U_{32}\cup U_{33},
\end{equation}
where 
\begin{align}
U_{22}&:=\bigcup_{(i,j)\neq (i',j')\in\mathcal{I}_{m}^2}(U_{ij}^2\cap U_{i'j'}^2),\label{U_22}\\
U_{23}&:=\bigcup_{(i,j)\in\mathcal{I}_{m}^2, (i',j',k')\in\mathcal{I}_m^3}(U_{ij}^2\cap U_{i'j'k'}^3),\label{U_23}\\
U_{32}&:=\bigcup_{(i,j,k)\in\mathcal{I}_{m}^3, (i',j')\in\mathcal{I}_m^2}(U_{ijk}^3\cap U_{i'j'}^2),\label{U_32}\\
U_{33}&:=\bigcup_{(i,j,k)\neq (i',j',k')\in\mathcal{I}_{m}^3}(U_{ijk}^3\cap U_{i'j'k'}^3)\label{U_33},
\end{align}
and given $(i,j)\in\mathcal{I}_m^2$, $(i,j,k)\in\mathcal{I}_m^3$, we denote
\begin{align}
U_{ij}^2&:=\left\{Z_m=(X_m,V_m)\in B_\rho^{dm}\times B_R^{dm}:
\sigma_2\leq d_2(x_i,x_j)\leq \sigma_2+2\delta R\right\}.\label{U_ij^2}\\
U_{ijk}^3&:=\left\{Z_m=(X_m,V_m)\in B_\rho^{dm}\times B_R^{dm}:
2\sigma_3^2\leq d_3^2(x_i;x_j,x_k)\leq (\sqrt{2}\sigma_3+4\delta R)^2\right\}.\label{U_ijk^3}
\end{align}

For $m=2$, we have $I_1^{mu}=I^2_{sc,ng}=\emptyset.$
\end{lemma}
\begin{proof} 
For $m=2$,  we have that $\partial_{mu}\mathcal{D}_{2,\sigma_2,\sigma_3}=\emptyset$, hence $I^{1}_{mu}=\emptyset.$
Also, since $m=2$, we trivially obtain $\mathcal{I}_2=\{(1,2)\}$, hence Remark \ref{remark on t1-t2} implies that $\tau_{Z_m}^2=\infty$ i.e.  $I_{sc,ng}^2=\emptyset$.

Assume now that $m\geq 3$. We first assume that either $Z_m\in\mathring{\mathcal{D}}_{m,\sigma_2,\sigma_3}$ or  $Z_m$ is postcollisional. Therefore, up to time $\tau_{Z_m}^1$, we have free flow i.e.
$Z_m(t)=(X_m+tV_m,V_m),$ for all   $t\in[0,\tau_{Z_m}^1].$

\textbf{Inclusion for $I^1_{mu}$:}

We have $\tau_{Z_m}^1\leq\delta$ and $Z_m(\tau_{Z_m}^1)\in\partial_{mu}\mathcal{D}_{m,\sigma_2,\sigma_3}$. We claim the following which clearly imply  inclusion \eqref{required inclusion} for $I_{mu}^1$:
\begin{enumerate}[(I)]
\item $Z_m(\tau_{Z_m}^1)\in\Sigma_{ij}^2\cap\Sigma_{i'j'}^2\Rightarrow Z_m\in U_{ij}^2\cap U_{i'j'}^2\quad\forall (i,j),(i',j')\in\mathcal{I}_m^2$.
\item $Z_m(\tau_{Z_m}^1)\in\Sigma_{ij}^2\cap\Sigma_{i'j'k'}^3\Rightarrow Z_m\in U_{ij}^2\cap U_{i'j'k'}^3,\quad\forall (i,j)\in\mathcal{I}_m^3,\quad\forall (i',j,k')\in\mathcal{I}_m^3$.
\item $Z_m(\tau_{Z_m}^1)\in\Sigma_{ijk}^3\cap\Sigma_{i'j'}^2\Rightarrow Z_m\in U_{ijk}^3\cap U_{i'j'}^2,\quad\forall (i,j,k)\in\mathcal{I}_m^3,\quad\forall (i',j')\in\mathcal{I}_m^2$.
\item $Z_m(\tau_{Z_m}^1)\in\Sigma_{ijk}^3\cap\Sigma_{i'j'k'}^3\Rightarrow Z_m\in U_{ij}^2\cap U_{i'j'k'}^3,\quad\forall (i,j,k),(i',j',k')\in\mathcal{I}_m^3$.
\end{enumerate}
Without loss of generality, we prove claim \text{(III)}. 
We have $Z_m(\tau_{Z_m}^1)\in\Sigma_{ijk}^3\cap\Sigma_{i'j'}^2$, therefore
\begin{align}
d_3^2\left(x_i\left(\tau_{Z_m}^1\right); x_j\left(\tau_{Z_m}^1\right),x_k\left(\tau_{Z_m}^1\right)\right)&=2\sigma_3^2,\quad d_2\left(x_{i'}\left(\tau_{Z_m}^1\right), x_{j'}\left(\tau_{Z_m}^1\right)\right)=\sigma_2.\label{collision i'j'}
\end{align}
Since there is free motion up to $\tau_{Z_m}^1$, triangle inequality implies
\begin{equation}\label{fremotion step 1}
|x_i-x_j|\leq |x_i(\tau_{Z_m}^1)-x_j(\tau_{Z_m}^1)|+\delta |v_i-v_j|\leq |x_i(\tau_{Z_m}^1)-x_j(\tau_{Z_m}^1)|+2\delta R.
\end{equation}
Since there is an $(i;j,k)$ ternary collision at $\tau_{Z_m}^1$, we have
\begin{equation}\label{sub collision}
|x_i(\tau_{Z_m}^1)-x_j(\tau_{Z_m}^1)|^2+|x_i(\tau_{Z_m}^1)-x_k(\tau_{Z_m}^1)|^2=2\sigma_3^2\Rightarrow |x_i(\tau_{Z_m}^1)-x_j(\tau_{Z_m}^1)|\leq\sqrt{2}\sigma_3
\end{equation}
Combining \eqref{fremotion step 1}-\eqref{sub collision}, we obtain
\begin{equation}\label{fremotion step 2}
|x_i-x_j|^2\leq  |x_i(\tau_{Z_m}^1)-x_j(\tau_{Z_m}^1)|^2 +4\sqrt{2}\sigma_3\delta R+4\delta^2 R^2.
\end{equation}
Using the same argument for the pair $(i,k)$, adding  and recalling the fact that there is $(i;j,k)$ collision at $\tau_{Z_m}^1$, we obtain
\begin{align}
2\sigma_3^2\leq d_3^2(x_i;x_j,x_k)\leq 2\sigma_3^2+8\sqrt{2}\sigma_3 R\delta+8\delta R^2&\leq 2\sigma_3^2+8\sqrt{2}\sigma_3 R\delta+16\delta R^2= (\sqrt{2}\sigma_3+4\delta R)^2\nonumber\\
&\Rightarrow Z_m\in U_{ijk}^3,\label{free motion triary}
\end{align}
where the lower inequality holds trivially since $Z_m\in\mathcal{D}_{m,\sigma_2,\sigma_3}$. 

For the pair $(i',j')$, \eqref{collision i'j'}  and triangle inequality yield
\begin{equation}\label{free motion binary}
\sigma_2\leq |x_i-x_j|=|x_{i'}(\tau_{Z_m}^1)-x_{j'}(\tau_{Z_m}^1)-\tau_{Z_m}^1(v_{i'}-v_{j'})|\leq \sigma_2+2\delta R\Rightarrow Z_m\in U_{i'j'}^2,
\end{equation}
where the lower inequality trivially holds because of the phase space. 

Combining \eqref{free motion triary}-\eqref{free motion binary}, we obtain 
$Z_m\in U_{ijk}^3\cap U_{i'j'}^2,$
and claim \text{(III)} is proved. The rest of the claims are proved by  similar arguments and we obtain the inclusion
\begin{equation}\label{inclusion for multiple}I_{mu}^1\subseteq U_{22}\cup U_{23}\cup U_{32}\cup U_{33}.
\end{equation}

\textbf{Inclusion for $I_{sc,ng}^2$:} 
Remark \ref{remark on t1-t2} guarantees that
\begin{equation}\label{guarantee of remark}
\begin{cases}
Z_m(\tau_{Z_m}^1)\in\Sigma_{ij}^2\Rightarrow Z_m(\tau_{Z_m}^2)\notin \Sigma_{ij}^2,\\
Z_m(\tau_{Z_m}^1)\in\Sigma_{ijk}^3\Rightarrow Z_m(\tau_{Z_m}^2)\notin \Sigma_{ijk}^3.
\end{cases}
\end{equation}

We claim the following:
\begin{enumerate}[(I)]
\item $Z_m(\tau_{Z_m}^1)\in\Sigma_{ij}^2$, $Z_m(\tau_{Z_m}^2)\in\Sigma_{i'j'}^2\Rightarrow Z_m\in U_{ij}^2\cap U_{i'j'}^2,\quad\forall (i,j),(i',j')\in\mathcal{I}_m^2$.
\item $Z_m(\tau_{Z_m}^1)\in\Sigma_{ij}^2$, $Z_m(\tau_{Z_m}^2)\in\Sigma_{i'j'k'}^3\Rightarrow Z_m\in U_{ij}^2\cap U_{i'j'k'}^3,\quad\forall (i,j)\in\mathcal{I}_m^3,\quad\forall (i',j,k')\in\mathcal{I}_m^3$.
\item $Z_m(\tau_{Z_m}^1)\in\Sigma_{ijk}^3$,$Z_m(\tau_{Z_m}^2)\in\Sigma_{i'j'}^2\Rightarrow Z_m\in U_{ijk}^3\cap U_{i'j'}^2,\quad\forall (i,j,k)\in\mathcal{I}_m^3,\quad\forall (i',j')\in\mathcal{I}_m^2$.
\item $Z_m(\tau_{Z_m}^1)\in\Sigma_{ijk}^3$,$Z_m(\tau_{Z_m}^2)\in\Sigma_{i'j'k'}^3\Rightarrow Z_m\in U_{ij}^2\cap U_{i'j'k'}^3,\quad\forall (i,j,k),(i',j',k')\in\mathcal{I}_m^3$.
\end{enumerate}
By \eqref{guarantee of remark}, proving claims \text{(I)-(IV)} implies  inclusion \eqref{required inclusion} for $I_{sc,ng}^2$.  

Without loss of generality, we prove claim \text{(III)}. Clearly all particles perform free motion until $\tau_{Z_m}^1$, so the same argument we used to obtain \eqref{free motion triary} yields
\begin{equation}\label{shell 1}
2\sigma_3^2\leq d_3^2(x_i;x_j,x_k)\leq (\sqrt{2}\sigma_3+4\delta R)^2\Rightarrow Z_m\in U_{ijk}^3.
\end{equation}
Moreover, particles keep performing free motion up to time $\tau_{Z_m}^2$, except particles $i,j,k$ whose velocities instantaneously tranform because of the collision at $\tau_{Z_m}^1$. 

We wish to prove as well $Z_m\in U_{i'j'}^2$ i.e.
\begin{equation}\label{shell 2} \sigma_2\leq d_2(x_{i'},x_{j'})\leq \sigma_2+2\delta R.\end{equation}
The first inequality trivially holds because of the phase space. To prove  the second inequality,  we distinguish the following cases:
\begin{enumerate}[(i)]
\item $i',j'\notin\{i,j,k\}$:  Since particles $(i',j')$ perform free motion up to $\tau_{Z_m}^2$, a similar argument to the one we used to obtain \eqref{free motion binary} yields $Z_m\in U_{i'j'}$. The only difference is that we apply the argument up to time $\tau_{Z_m}^2\leq\delta$, instead of $\tau_{Z_m}^1$ hence claim \eqref{shell 2} is proved.
\vspace{0.3cm}
\item There is at least one recollision i.e. at least one of $i',j'$ belongs to $\{i,j,k\}$: The argument is similar to (i), the only difference being that  velocities of the recolliding particles  transform at $\tau_{Z_m}^1$.

Since the argument is similar for all cases, let us provide a detailed  proof  only for one recollisional case, for instance $(i',j')=(i,k)$.
We have
\begin{align*}
x_i(\tau_{Z_m}^2)&=x_i(\tau_{Z_m}^1)+(\tau_{Z_m}^2-\tau_{Z_m}^1)v_i^*=x_i+\tau_{Z_m}^1v_i+(\tau_{Z_m}^2-\tau_{Z_m}^1)v_i^*,\\
x_k(\tau_{Z_m}^2)&=x_k(\tau_{Z_m}^1)+(\tau_{Z_m}^2-\tau_{Z_m}^1)v_k^*=x_k+\tau_{Z_m}^1v_k+(\tau_{Z_m}^2-\tau_{Z_m}^1)v_k^*,
\end{align*}
so
\begin{equation*}
x_i-x_k=x_i(\tau_{Z_m}^2)-x_k(\tau_{Z_m}^2)-\tau_{Z_m}^1(v_i-v_k)-(\tau_{Z_m}^2-\tau_{Z_m}^1)(v_i^*-v_k^*).
\end{equation*}
Therefore, triangle inequality implies
\begin{align}
|x_i-x_k|&\leq |x_i(\tau_{Z_m}^2)-x_k(\tau_{Z_m}^2)|+\tau_{Z_m}^1|v_i-v_k|+(\tau_{Z_m}^2-\tau_{Z_m}^1)|v_i^*-v_k^*|\nonumber\\
&\leq |x_i(\tau_{Z_m}^2)-x_k(\tau_{Z_m}^2)|+2\tau_{Z_m}^1 R+2(\tau_{Z_m}^2-\tau_{Z_m}^1)R\label{use of remark on cons}\\
&=|x_i(\tau_{Z_m}^2)-x_k(\tau_{Z_m}^2)|+2\tau_{Z_m}^2 R\nonumber\\
&\leq |x_i(\tau_{Z_m}^2)-x_k(\tau_{Z_m}^2)|+2\delta R,\label{tau leq delta ik}
\end{align}
to obtain \eqref{use of remark on cons}, we use triangle inequality and Remark \ref{remark on conservation}, and to obtain \eqref{tau leq delta ik}, we use the assumption $\tau_{Z_m}^2\leq\delta$. Therefore \eqref{shell 2} is proved. 
\end{enumerate}

Combining \eqref{shell 1}, \eqref{shell 2}, we obtain 
$Z_m\in U_{ijk}^3\cap U_{i'j'}^2,$
and claim \text{(III)} follows. 

The remaining claims are proved in a similar way. We obtain
\begin{equation}\label{inclusion for consecutive}
I_{sc,ng}^2\subseteq U_{22}\cup U_{23}\cup U_{32}\cup U_{33}.
\end{equation}
 Inclusions \eqref{inclusion for multiple}, \eqref{inclusion for consecutive} imply inclusion \eqref{required inclusion}.

 Assume now that $Z_m$ is precollisional. Therefore, we obtain
$$Z_m(t)=
\begin{cases}(X_m+tV_m',V_m'),\quad\forall t\in[0,\tau_{Z_m}^1],\mbox{ if } Z_m\in\partial_{2,sc}\mathcal{D}_{m,\sigma_2,\sigma_3}\\
(X_m+tV_m^*,V_m^*),\quad\forall t\in[0,\tau_{Z_m}^1],\mbox{ if } Z_m\in\partial_{3,sc}\mathcal{D}_{m,\sigma_2,\sigma_3}.
\end{cases}$$
where the collisional transformation is taken with respect to the initial collisional particles.
The proof follows the same lines, using Remark \ref{remark on conservation} for the initial collisional particles whenever needed.
\end{proof}

Now we wish to estimate the measure of $I_{sc,g}^1\cup I_{mu}^1\cup I_{sc,ng}^2$ in order to show that outside of a small measure set we have a well defined flow. Let us first introduce some notation. 

For $m\geq 2$, $(i,j)\in\mathcal{I}_m^2$, a permutation $\pi:\{i,j\}\to\{i,j\}$ and $x_{\pi_j}\in\mathbb{R}^{d}$, we define the set
\begin{equation}\label{S perm binary}
S_{\pi_i}(x_{\pi_j})=\{x_{\pi_i}\in\mathbb{R}^d: (x_i,x_j)\in U_{ij}^2\}.
\end{equation}
For $m\geq 3$, $(i,j,k)\in\mathcal{I}_m^3$, a permutation $\pi:\{i,j,k\}\to\{i,j,k\}$ and $(x_{\pi_j},x_{\pi_k})\in\mathbb{R}^{2d}$, we define the set
\begin{equation}\label{S perm ternary}
S_{\pi_i}(x_{\pi_j},x_{\pi_k})=\{x_{\pi_i}\in\mathbb{R}^d: (x_i,x_j,x_k)\in U_{ijk}^3\}.
\end{equation}

\begin{lemma}\label{estimate of S} The following hold
\begin{enumerate}[(i)]
\item  Let $m\geq 2$, $(i,j,k)\in\mathcal{I}_m^2$, a permutation $\pi:\{i,j\}\to\{i,j\}$ and $x_{\pi_j}\in\mathbb{R}^{d}$.   Then 
\begin{equation}\label{estimate of S perm binary}
|S_{\pi_i}(x_{\pi_j})|_d\leq C_{d,R}\delta.
\end{equation}
\item Let $m\geq 3$, $(i,j,k)\in\mathcal{I}_m^3$, a permutation $\pi:\{i,j,k\}\to\{i,j,k\}$ and $(x_{\pi_j},x_{\pi_k})\in\mathbb{R}^{2d}$.   Then 
\begin{equation}\label{estimate of S perm ternary}
|S_{\pi_i}(x_{\pi_j},x_{\pi_k})|_d\leq C_{d,R}\delta.
\end{equation}
\end{enumerate}
\end{lemma}

\begin{proof} For proof of estimate \eqref{estimate of S perm ternary}, we refer to  Lemma 3.10. in \cite{ternary}.

Let us prove \eqref{estimate of S perm binary}. Consider $(i,j)\in\mathcal{I}_m^2$, and assume  without loss of generality that $\pi(i,j)=(i,j)$.  Let $x_j\in\mathbb{R}^d$. Recalling \eqref{S perm binary}, we obtain
$$S_i(x_j)=\left\{x_i\in\mathbb{R}^d:\sigma_2\leq |x_i-x_j|\leq\sigma_2+2\delta R\right\},$$
thus $S_i(x_j)$ is a spherical shell in $\mathbb{R}^d$ of inner radius $\sigma_2$ and outer radius $\sigma_2+2\delta R$.
Therefore, by scaling \eqref{scaling dynamics}, we obtain
\begin{align*}
|S_i(x_j)|_d&\simeq (\sigma_2+2\delta R)^d-\sigma_2^d=2\delta R\sum_{\ell=0}^{d-1} (\sigma_2+2\delta R)^{d-1-\ell}\sigma_2^\ell\leq C_{d,R}\delta.
\end{align*}
\end{proof}

\begin{remark} Estimates of Lemma \ref{estimate of S} are not sufficient to generate a global flow because $\delta$ represents the length of an elementary time step, therefore iterating, we cannot eliminate pathological sets. We will derive a better estimate  of order $\delta ^2$ to achieve this elimination.
\end{remark}
\begin{lemma}\label{ellipse shell measure}
Let $m\geq 2$, $1<R<\rho$ and $0<\delta R<\sigma_2<\sigma_3<1$. Then the following estimate holds:
\begin{equation}\label{measure estimate dynamics}
|I_{sc,g}^1\cup I_{mu}^1\cup I_{sc,ng}^2|_{2dm}\leq C_{m,d,R}\rho^{d(m-2)}\delta ^2. 
\end{equation}
\end{lemma}
\begin{proof} 
 For $m=2$, the result comes trivially from Lemma \ref{zero measure} and Lemma \ref{covering}.
 
 For $m\geq 3$, we recall from Lemma \ref{zero measure} that $I^1_{g}$ is of measure zero and that by Lemma \ref{covering}, we have
 $$I_{mu}^1\cup I_{sc,ng}^2=U_{22}\cup U_{23}\cup U_{32}\cup U_{33},$$
 where $U_{22},U_{23},U_{32},U_{33}$ are given by \eqref{U_22}-\eqref{U_33}. Therefore it suffices to estimate the measure of $U_{22},U_{23},U_{32},U_{33}$. We will strongly rely on Lemma \ref{estimate of S}.
 
 $\bullet$ \textbf{Estimate of $U_{22}$:} By \eqref{U_22}, we have
 $$U_{22}=\bigcup_{(i,j)\neq (i',j')\in\mathcal{I}_m^2}(U_{ij}^2\cap U_{i'j'}^2).$$
 Consider $(i,j)\neq (i',j')\in\mathcal{I}_m^2$. We distinguish the following possible cases:
 \begin{enumerate}[(I)]
 \item $i',j'\notin\{i,j\}$: By \eqref{U_ij^2}, followed by Fubini's Theorem and part {\em (i)} of Lemma \ref{estimate of S}, we have
 \begin{align*}
 |&U_{ij}^2\cap U_{i'j'}^2|_{2dm}\lesssim R^{dm}\rho^{d(m-4)}\int_{B_\rho^{4d}}\mathds{1}_{S_{i}^2(x_j)\cap S_{i'}^2(x_{j'})}\,dx_{i}\,dx_{i'}\,dx_{j}\,dx_{j'}\\
 &\leq R^{dm}\rho^{d(m-4)}\left(\int_{B_\rho^d}\int_{\mathbb{R}^d}\mathds{1}_{S_i^2(x_j)}\,dx_i\,dx_j\right)\left(\int_{B_\rho^d}\int_{\mathbb{R}^d}\mathds{1}_{S_{i'}^2(x_{j'})}\,dx_{i'}\,dx_{j'}\right)\\
 &\leq C_{d,R}\rho^{d(m-2)}\delta^2.
\end{align*}
\item Exactly one of $i',j'$ belongs to $\{i,j\}$: Without loss of generality we consider the case $(i',j')=(j,j')$, for some $j'>j$ and all other cases follow similarly.  Fubini's Theorem and part {\em (i)} of Lemma \ref{estimate of S} imply
\begin{align*}
|&U_{ij}^2\cap U_{jj'}^2|_{2dm}\lesssim R^{dm}\rho^{d(m-3)}\int_{B_\rho^{3d}}\mathds{1}_{S_i^2(x_j)\cap S_{j}^2(x_{j'})}\,dx_{j}\,dx_{j'}\,dx_i\\
&\leq R^{dm}\rho^{d(m-3)}\int_{B_{\rho}^d}\left(\int_{\mathbb{R}^d}\mathds{1}_{S_i^2(x_j)}\,dx_i\right)\left(\int_{\mathbb{R}^d}\mathds{1}_{S_{j'}^2(x_j)}\,dx_{j'}\right)\,dx_j\\
&\leq C_{d,R}\rho^{d(m-2)}\delta^2.
\end{align*}  
 \end{enumerate}
 Combining cases \text{(I)-(II)}, we obtain
\begin{equation}\label{estimate U_22}
|U_{22}|_{2dm}\leq C_{m,d,R}\rho^{d(m-2)}\delta^2.
\end{equation}

 $\bullet$ \textbf{Estimate of $U_{23}$:} By \eqref{U_23}, we have
 $$U_{23}=\bigcup_{(i,j)\in\mathcal{I}_m^2, (i',j',k')\in\mathcal{I}_m^3}(U_{ij}^2\cap U_{i'j'k'}^3).$$
 Consider $(i,j)\in\mathcal{I}_m^2$, $(i',j',k')\in\mathcal{I}_m^3$. We distinguish the following possible cases:
 \begin{enumerate}[(I)]
 \item $i',j',k'\notin\{i,j\}$: By Fubini's Theorem and parts {\em (i)-(ii)} of Lemma \ref{estimate of S}, we obtain
 \begin{align*}
 |&U_{ij}^2\cap U_{i'j'k'}^3|_{2dm}\lesssim R^{dm}\rho^{d(m-5)}\int_{B_\rho^{5d}}\mathds{1}_{S_j^2(x_i)\cap S_{k'}^3(x_{i'},x_{j'})}\,dx_i\,dx_j\,dx_{i'}\,dx_{j'}\,dx_{k'}\\
 &\leq R^{dm}\rho^{d(m-5)}\left(\int_{B_\rho^d}\int_{\mathbb{R}^d}\mathds{1}_{S_j^2(x_i)}\,dx_i\,dx_j\right)\left(\int_{B_\rho^d\times B_\rho^d}\int_{\mathbb{R}^d}\mathds{1}_{S_{k'}^3(x_{i'},x_{j'})}\,dx_{i'}\,dx_{j'}\,dx_{k'}\right)\\
 &\leq C_{d,R}\rho^{d(m-2)}\delta^2.
 \end{align*}
 \item Exactly one of $i',j',k'$ belongs in $\{i,j\}$: Without loss of generality we consider the case $(i',j',k'):=(i',i,k')$, for some $i'<i<k'$ and all other cases follow similarly. Using Fubini's Theorem and parts {\em (i)-(ii)} of Lemma \ref{estimate of S}, we obtain
 \begin{align*}
 |&U_{ij}^2\cap U_{i'ik'}|_{2dm}\lesssim R^{dm}\rho^{d(m-4)}\int_{B_\rho^{4d}}\mathds{1}_{S_j^2(x_i)\cap S_{i'}^3(x_{i},x_{k'})}\,dx_i\,dx_j\,dx_{i'}\,dx_{k'}\\
 &\leq R^{dm}\rho^{d(m-4)}\int_{B_\rho^d}\left(\int_{\mathbb{R}^d}\mathds{1}_{S_j^2(x_i)}\,dx_j\right)\left(\int_{B_\rho^d}\int_{\mathbb{R}^d}\mathds{1}_{S_{i'}^3(x_i,x_{k'})}\,dx_{i'}\,dx_{k'}\right)\,dx_i\\
 &\leq C_{d,R}\rho^{d(m-2)}\delta^2.
 \end{align*}
 \item Exactly two of $i',j',k'$ belongs in $\{i,j\}$: Without loss of generality we consider the case $(i',j',k')=(i',i,j)$, for some $i'<i$ and all other cases follow similarly. Using Fubini's Theorem and parts {\em (i)-(ii)} of Lemma \ref{estimate of S}, we obtain
 \begin{align*}
 |&U_{ij}^2\cap U_{i'ij}^3|_{2dm}\lesssim R^{dm}\rho^{d(m-3)}\int_{B_\rho^{3d}}\mathds{1}_{S_{i}^2(x_j)\cap S_{i'}^3(x_i,x_j)}\,dx_i\,dx_j\,dx_{i'}\\
 &\leq R^{dm}\rho^{d(m-3)}\int_{B_\rho^d\times B_\rho^d}\left(\int_{\mathbb{R}^d}\mathds{1}_{S_{i}^2(x_j)}\mathds{1}_{S_{i'}^3(x_i,x_j)}\,dx_{i'}\right)\,dx_i\,dx_j\\
 &=R^{dm}\rho^{d(m-3)}\int_{B_\rho^d\times B_\rho^d}\mathds{1}_{S_{i}^2(x_j)}(\int_{\mathbb{R}^d}\mathds{1}_{S_{i'}^3(x_i,x_j)}\,dx_{i'})\,dx_i\,dx_j\\
 &\leq C_{d,R}\rho^{d(m-3)}\delta\int_{B_\rho^{d}}\int_{\mathbb{R}^d}\mathds{1}_{S_i(x_j)}\,dx_i\,dx_j\\
 &\leq C_{d,R}\rho^{d(m-2)}\delta^2.
 \end{align*}
 \end{enumerate}
 Combining cases \text{(I)-(III)}, we obtain
\begin{equation}\label{estimate U_23}
|U_{23}|_{2dm}\leq C_{m,d,R}\rho^{d(m-2)}\delta^2.
\end{equation}

 $\bullet$ \textbf{Estimate of $U_{32}$:} We use a similar argument to the estimate for $U_{23}$, to obtain
\begin{equation}\label{estimate U_32}
|U_{32}|_{2dm}\leq C_{m,d,R}\rho^{d(m-2)}\delta^2.
\end{equation}
 
 $\bullet$ \textbf{Estimate of $U_{33}$:} We refer to Lemma 3.11. from \cite{ternary} for a detailed proof. We obtain
\begin{equation}\label{estimate U_33}
|U_{33}|_{2dm}\leq C_{m,d,R}\rho^{d(m-2)}\delta^2.
\end{equation}

Combining \eqref{estimate U_22}-\eqref{estimate U_33}, we obtain \eqref{measure estimate dynamics} and the proof is complete.
\end{proof}

We inductively use Lemma \ref{ellipse shell measure} to define a global flow which preserves energy for almost all configuration. For this purpose, given $Z_m=(X_m,V_m)\in\mathbb{R}^{2dm}$, we define its kinetic energy as:
\begin{equation}\label{kinetic energy}
E_m(Z_m)=\frac{1}{2}\sum_{i=1}^m|v_i|^2
\end{equation}

For convenience, let us define the $m$-particle free flow:
\begin{definition} Let $m\in\mathbb{N}$. We define the $m$-particle free flow as the family of measure-preserving maps $(\Phi_m^t)_{t\in\mathbb{R}}:\mathbb{R}^{2dm}\to\mathbb{R}^{2dm}$, given by
\begin{equation}\label{free flow}
\Phi_m^tZ_m=\Phi_m^t(X_m,V_m)=(X_m+tV_m,V_m).
\end{equation}
\end{definition}

We are now in the position to state the  Existence Theorem of the $m$-particle $(\sigma_2,\sigma_3)$-flow.
\begin{theorem}\label{global flow}
Let $m\in\mathbb{N}$ and $0<\sigma_2<\sigma_3<1$. There exists a family of measure-preserving maps $(\Psi_m^t)_{t\in\mathbb{R}}:\mathcal{D}_{m,\sigma_2,\sigma_3}\to\mathcal{D}_{m,\sigma_2,\sigma_3}$ such that 
\begin{align}
&\Psi_m^{t+s}Z_m=(\Psi_m^t\circ \Psi_m^s)(Z_m)=(\Psi_m^s\circ \Psi_m^t)(Z_m),\quad\text{a.e. in }\mathcal{D}_{m,\sigma_2,\sigma_3},\quad\forall t,s\in\mathbb{R}\label{flow property},\\
&E_m\left(\Psi_m^t Z_m\right)=E_m(Z_m),\quad\mbox{a.e. in }\mathcal{D}_{m,\sigma_2,\sigma_3},\quad\forall t\in\mathbb{R},\text{ where $E_m$ is given by \eqref{kinetic energy}}\label{kinetic energy flow}.
\end{align}
Moreover, for $m\geq 3$, we have
\begin{align}
\Psi_m^t Z_m'&=\Psi_m^t Z_m,\quad\sigma-\text{a.e. on }\partial_{sc,ng}\mathcal{D}_{m,\sigma_2,\sigma_3}\cap\partial_{2,sc}\mathcal{D}_{m,\sigma_2,\sigma_3},\quad\forall t\in\mathbb{R},\label{bc flow binary}\\
\Psi_m^t Z_m^*&=\Psi_m^t Z_m,\quad\sigma-\text{a.e. on }\partial_{sc,ng}\mathcal{D}_{m,\sigma_2,\sigma_3}\cap\partial_{3,sc}\mathcal{D}_{m,\sigma_2,\sigma_3},\quad\forall t\in\mathbb{R},\label{bc flow triary}
\end{align}
while for $m=2$, we have
\begin{align}
\Psi_2^t Z_2'&=\Psi_2^t Z_2,\quad\sigma-\text{a.e. on }\partial_{sc,ng}\mathcal{D}_{2,\sigma_2,\sigma_3}\cap\partial_{2,sc}\mathcal{D}_{2,\sigma_2,\sigma_3},\quad\forall t\in\mathbb{R},\label{bc flow binary m=2}
\end{align}
where $\sigma$ is the surface measure induced on $\partial\mathcal{D}_m$ by the Lebesgue measure. This family of maps is called the $m$-particle $(\sigma_2,\sigma_3)$-flow. 

For $m=1$, we define $\Psi_1^t:=\Phi_1^t\quad\forall t\in\mathbb{R}.$
\end{theorem}
\begin{proof}
The proof follows exactly the same steps as the proof of Theorem 3.14. from  \cite{ternary} (for additional details see Theorem 4.9.1 from \cite{thesis}).
\end{proof}
\begin{remark}\label{remark on dynamics notation}
We have seen that the flow can be defined only a.e. in $\mathcal{D}_{m,\sigma_2,\sigma_3}$. However to simplify the notation, without loss of generality, we may  assume that the flow is well defined on the whole phase space $\mathcal{D}_{m,\sigma_2,\sigma_3}$. 
\end{remark}
\subsection{The Liouville equation}
Here, we formally derive the Liouville equation for $m$ hard spheres of diameter $\sigma_2$ and interaction zone $\sigma_3$, where $0<\sigma_2<\sigma_3<1$. Without loss of generality, we derive the equation for $m\geq 3$, and for $m=2$ we follow a similar argument neglecting the ternary terms. For $m=1$, the Liouville equation will be trivial since the flow coincides with the free flow. We then introduce the $m$-particle $(\sigma_2,\sigma_3)$ interaction flow operator and the $m$-particle free flow operator.

 Let $m\geq 3$ and consider an initial absolutely continuous  Borel probability  measure $P_0$ on $\mathbb{R}^{2dm}$, with a probability density $f_{m,0}$ satisfying the following properties:
\begin{itemize}
\item $f_{m,0}$ is supported in $\mathcal{D}_{m,\sigma_2,\sigma_3}$ i.e. 
\begin{equation}\label{support initial}
\supp f_{m,0}:=\overline{\{Z_m\in\mathbb{R}^{2dm}:f_{m,0}(Z_m)\neq 0\}}\subseteq\mathcal{D}_{m,\sigma_2,\sigma_3}.
\end{equation}
\item $f_{m,0}$ is symmetric i.e. for any permutation $p_m$ of the $m$-particles, there holds:
\begin{equation}\label{symmetry of initial density}
f_{m,0}(Z_{p_m})=f_{m,0}(Z_m),\quad\forall Z_m\in\mathbb{R}^{2dm}.
\end{equation} 
\end{itemize}
 
The probability measure $P_0$ expresses the initial distribution in space and velocities of the $m$-particles. We are interested in the evolution of this measure under the flow. For this purpose, given $t\geq 0$ we define $P_t$ to be the push-forward of $P_0$ under the flow i.e.
\begin{equation*}
P_t(A)=P_0\left(\Psi_m^{-t}\left(A\right)\right),\quad A\subseteq\mathbb{R}^{2dm}\text{ Borel measurable}.
\end{equation*}
Conservation of measure under the flow implies that $P_t$ is absolutely continuous with probability density given by
\begin{equation}\label{density of push-forward}
f_m(t,Z_m)=\begin{cases}
f_{m,0}\circ\Psi_m^{-t},\quad\text{a.e. in }\mathcal{D}_{m,\sigma_2,\sigma_3},\\
0,\quad \text{a.e. in }\mathbb{R}^{2dm}\setminus\mathcal{D}_{m,\sigma_2,\sigma_3}.
\end{cases}
\end{equation}

Clearly $f_m(t,Z_m)$ is symmetric and supported in $\mathcal{D}_{m,\sigma_2,\sigma_3}$, for all $t\geq 0$. Moreover,  we have
\begin{equation}\label{initial condition mild liouville}
f_m(0,Z_m)=f_{m,0}\circ\Psi_m^0(Z_m)=f_{m,0}(Z_m),\quad Z_m\in\mathring{\mathcal{D}}_{m,\sigma_2,\sigma_3}.
\end{equation}
Since $m>2$, \eqref{bc flow binary} implies
\begin{equation}\label{bc liouville binary}
f_m(t,Z_m')=f_{m,0}\circ\Psi_m^{-t}(Z_m')=f_{m,0}\circ\Psi_m^{-t}(Z_m)=f_m(t,Z_m),\quad\sigma-\text{a.e. on }\partial_{2,sc}\mathcal{D}_{m,\sigma_2,\sigma_3},\quad\forall t\geq 0.
 \end{equation}
Additionally, since $m\geq 3$, \eqref{bc flow triary} implies
\begin{equation}\label{bc liouville triary}
f_m(t,Z_m^*)=f_{m,0}\circ\Psi_m^{-t}(Z_m^*)=f_{m,0}\circ\Psi_m^{-t}(Z_m)=f_m(t,Z_m),\quad\sigma-\text{a.e. on }\partial_{3,sc}\mathcal{D}_{m,\sigma_2,\sigma_3},\quad\forall t\geq 0.\\ 
 \end{equation}
Finally, recall from \eqref{density of push-forward} that
\begin{equation}\label{conservation of density}
\begin{aligned}
f_m(t,Z_m)=f_{m,0}\circ\Psi_m^{-t}(Z_m),\quad\text{a.e. in }\mathcal{D}_{m,\sigma_2,\sigma_3},\quad\forall t\geq 0.
\end{aligned}
\end{equation}

 Combining \eqref{initial condition mild liouville}-\eqref{conservation of density}, and formally assuming that $f_m$ is smooth in time, by the chain rule, we obtain that $f_m$ formally satisfies the $m$-particle Liouville equation in $\mathcal{D}_{m,\sigma_2,\sigma_3}$:
\begin{equation}\label{Liouville equation}
\begin{cases}
\partial_tf_m+\displaystyle\sum_{i=1}^m v_i\cdot\nabla_{x_i}f_m=0,\quad(t,Z_m)\in(0,\infty)\times\mathring{\mathcal{D}}_{m,\sigma_2,\sigma_3}\\
f_m(t,Z_m')=f_m(t,Z_m),\quad (t,Z_m)\in[0,\infty)\times\partial_{2,sc}\mathcal{D}_{m,\sigma_2,\sigma_3},\\
f_m(t,Z_m^*)=f_m(t,Z_m),\quad (t,Z_m)\in[0,\infty)\times\partial_{3,sc}\mathcal{D}_{m,\sigma_2,\sigma_3},\\
f_m(0,Z_m)=f_{m,0}(Z_m),\quad Z_m\in\mathring{\mathcal{D}}_{m,\sigma_2,\sigma_3}.
\end{cases}
\end{equation}
With similar arguments, we conclude that, in the case $m=2$, $f_2$ formally satisfies the $2$-particle Liouville equation $\mathcal{D}_{2,\sigma_2,\sigma_3}$:
\begin{equation}\label{Liouville equation binary}
\begin{cases}
\partial_tf_2+\displaystyle v_1\cdot\nabla_{x_1}f_2+v_2\cdot\nabla_{x_2}f_2=0,\quad(t,Z_2)\in(0,\infty)\times\mathring{\mathcal{D}}_{2,\sigma_2,\sigma_3},\\
f_2(t,Z_2')=f_2(t,Z_2),\quad (t,Z_2)\in[0,\infty)\times\partial_{2,sc}\mathcal{D}_{2,\sigma_2,\sigma_3},\\
f_2(0,Z_2)=f_{2,0}(Z_2),\quad Z_2\in\mathring{\mathcal{D}}_{2,\sigma_2,\sigma_3}.
\end{cases}
\end{equation}
In the case $m=1$, we trivially have 
$
f(t,x_1,v_1)=f_0(\Phi^{-t}_1(x_1,v_1))=f_0(x_1-tv_1,v_1).
$

Now, we introduce some notation defining the $m$-particle free flow operator and the $m$-particle  $(\sigma_2,\sigma_3)$-flow operator.
 For convenience, let us denote 
\begin{align}
C^0(\mathcal{D}_{m,\sigma_2,\sigma_3})&:=\{g_m\in C^0(\mathbb{R}^{2dm}):\supp g_m\subseteq \mathcal{D}_{m,\sigma_2,\sigma_3}\}\label{continuous finite}.
\end{align}
\begin{definition}
For $t\in\mathbb{R}$ and $0<\sigma_2<\sigma_3<1$, we define the $m$-particle $(\sigma_2,\sigma_3)$-flow operator $T_m^t:C^0(\mathcal{D}_{m,\sigma_2,\sigma_3})\to C^0(\mathcal{D}_{m,\sigma_2,\sigma_3})$ as:
\begin{equation}\label{liouville operator}
T_m^tg_m(Z_m)=\begin{cases}
g_{m}(\Psi_m^{-t}Z_m),\quad \text{if }Z_m\in\mathcal{D}_{m,\sigma_2,\sigma_3},\\
0,\quad \text{if }Z_m\notin\mathcal{D}_{m,\sigma_2,\sigma_3},
\end{cases}
\end{equation}
where $\Psi_m$ is the $m$-particle $(\sigma_2,\sigma_3)$-flow defined in Theorem \ref{global flow}.
\end{definition}
\begin{remark} Given an initial probability density $f_{m,0}$, satisfying \eqref{support initial}-\eqref{symmetry of initial density}, the function $f_m(t,Z_m)=T_m^tf_{m,0}(Z_m)$ is formally the unique solution to the Liouville equation \eqref{Liouville equation} with initial data $f_{m,0}$.
\end{remark}
We also define the free flow and the $m$-particle free flow operator.
\begin{definition}For $t\in\mathbb{R}$ and $m\in\mathbb{N}$, we define the $m$-particle free flow operator $S_m^t:C^0(\mathbb{R}^{2dm})\to C^0(\mathbb{R}^{2dm})$ as:
\begin{equation}\label{free flow operator}
S_m^tg_m(Z_m)=g_m(\Phi_m^{-t}Z_m)=g_m(X_m-tV_m,V_m).
\end{equation}
\end{definition}
\section{BBGKY hierarchy, Boltzmann hierarchy and the binary-ternary Boltzmann equation}\label{sec:BBGKY}
 \subsection{The BBGKY hierarchy}
Consider $N$-particles of diameter $0<\epsilon_2<1$ and interaction zone $0<\epsilon_3<1$, where $N\geq 3$ and $\epsilon_2<\epsilon_3$. For $s\in\mathbb{N}$, we define the $s$-marginal of a symmetric probability density $f_N$, supported in $\mathcal{D}_{N,\epsilon_2,\epsilon_3}$, as

\begin{equation}\label{def marginals}
f_N^{(s)}(Z_s)=
\begin{cases}
\displaystyle\int_{\mathbb{R}^{2d\left(N-s\right)}}f_N(Z_N)\,dx_{s+1}...\,dx_N\,dv_{s+1}...\,dv_N,\text{ } 1\leq s< N,\\
f_N,\text{ } s=N,\\
0,\text{ } s>N,
\end{cases}
\end{equation}
where for $Z_s=(X_s,V_s)\in\mathbb{R}^{2ds}$, we write $Z_N=(X_s,x_{s+1},...,x_N,V_s,v_{s+1},...,v_N)$.
One can see, for all $1\leq s\leq N$, the marginals $f_N^{(s)}$ are symmetric probability densities, supported in $\mathcal{D}_{s,\epsilon_2,\epsilon_3}$ and 
\begin{equation*}
f_N^{(s)}(Z_s)=\int_{\mathbb{R}^{2d}}f_N^{(s+1)}(X_N,V_N)\,dx_{s+1}\,dv_{s+1},\quad\forall 1\leq s\leq N-1.
\end{equation*}

Assume now that $f_N$ is formally the solution to the $N$-particle Liouville equation \eqref{Liouville equation} with initial data $f_{N,0}$. We seek to formally find a hierarchy of equations satisfied by the marginals of $f_N$. For $s\geq N$, by definition, we have
\begin{equation}\label{BBGKY s=N} 
f_N^{(N)}=f_N,\text{ and } f_N^{(s)}= 0,\text{ for } s>N,
\end{equation}

We observe that $\partial \mathcal{D}_{N,\epsilon_2,\epsilon_3}$ is equivalent up to surface measure zero to $\Sigma^X\times\mathbb{R}^{dN}$ where 
\begin{equation}\label{decomposition of the boundary BBGKY}\Sigma^X:=\bigcup_{(i,j)\in\mathcal{I}_N^2}\Sigma_{ij}^{2,sc,X}\cup\bigcup_{(i,j,k)\in\mathcal{I}_N^3}\Sigma_{ijk}^{3,sc,X},\end{equation}
\begin{align*}
\Sigma_{ij}^{2,sc,X}:=\big\{X_N\in\mathbb{R}^{dN}:&d_2(x_i,x_j)=\epsilon_2,\text{ }d_2(x_{i'},x_{j'})>\epsilon_2,\quad\forall (i',j')\in\mathcal{I}_N^2\setminus\{(i,j)\}\\
&\text{and }d_3(x_{i'};,x_{j'},x_{k'})>\sqrt{2}\epsilon_3,\quad\forall (i',j',k')\in\mathcal{I}_N^3\big\},
\end{align*}
\begin{align*}
\Sigma_{ijk}^{3,sc,X}:=\big\{X_N\in\mathbb{R}^{dN}:&d_3(x_i;x_j,x_k)=\sqrt{2}\epsilon_3,\text{ }d_2(x_{i'},x_{j'})>\epsilon_2,\quad\forall (i',j')\in\mathcal{I}_N^2\\
&\text{and }d_3(x_{i'};,x_{j'},x_{k'})>\sqrt{2}\epsilon_3,\quad\forall (i',j',k')\in\mathcal{I}_N^3\}\setminus\{(i,j,k)\}\big\}.
\end{align*}
 Notice that \eqref{decomposition of the boundary BBGKY} is a pairwise disjoint union.
\begin{remark}\label{remark on ordering epsilons} The assumption $\epsilon_2<\epsilon_3$ made at at the beginning of the section is necessary for the ternary contribution to be visible.  Indeed, if $\epsilon_2\geq\epsilon_3$, Remark \ref{remark on phase space} and \eqref{decomposition of simple ternary} would imply that $\Sigma_{ijk}^{3,sc,X}=\emptyset$ for all $(i,j,k)\in\mathcal{I}_m^3$, therefore there would not be a ternary collisional term.
\end{remark}

The  hierarchy for $s<N$ will come after integrating by parts the Liouville equation \eqref{Liouville equation}. Consider $1\leq s\leq N-1$. The boundary and initial conditions can be easily recovered integrating Liouville's equation boundary and initial conditions respectively i.e. 
\begin{equation}\label{bc of BBGKY}\begin{cases}
f_N^{(s)}(t,Z_s')=f^{(s)}_N(t,Z_s),\quad (t,Z_s)\in[0,\infty)\times\partial_{2,sc}\mathcal{D}_{s,\epsilon_2,\epsilon_3},\quad s\geq 2,\\
 f_N^{(s)}(t,Z_s^*)=f^{(s)}_N(t,Z_s),\quad (t,Z_s)\in[0,\infty)\times\partial_{3,sc}\mathcal{D}_{s,\epsilon_2,\epsilon_3},\quad s\geq 3,\\
f_N^{(s)}(0,Z_s)=f_{N,0}^{(s)}(Z_s),\quad Z_s\in\mathring{\mathcal{D}}_{s,\epsilon_2,\epsilon_3}.
\end{cases}
\end{equation}
Notice that for $s=2$ there is no ternary boundary condition, while for $s=1$ there is no boundary condition at all.

Consider now a smooth test function $\phi_s$ compactly supported in $(0,\infty)\times\mathcal{D}_{s,\epsilon_2,\epsilon_3}$ such that the following hold:
\begin{itemize}
\item For any $(i,j)\in\mathcal{I}_N^2$ with $j\leq s$, we have
\begin{equation}\label{test boundary binary}
\phi_s(t,p_sZ_N')=\phi_s(t,p_sZ_N)=\phi_s(t,Z_s),\quad\forall (t,Z_N)\in (0,\infty)\times\Sigma_{i,j}^{sc,2},
\end{equation}
\item For any $(i,j,k)\in \mathcal{I}_N^3$ with $j\leq s$, we have
\begin{equation}\label{test boundary triary}\phi_s(t,p_sZ_N^*)=\phi_s(t,p_sZ_N)=\phi_s(t,Z_s),\quad\forall (t,Z_N)\in (0,\infty)\times\Sigma_{i,j,k}^{sc,3},\end{equation}
\end{itemize}
where $p_s:\mathbb{R}^{2dN}\to\mathbb{R}^{2ds}$ denotes the natural projection in space and velocities, given by
$p_s(Z_N)=Z_s.$

Multiplying the Liouville equation by $\phi_s$ and integrating,   we obtain its weak form
\begin{equation}\label{weak form initial}
\int_{(0,\infty)\times\mathcal{D}_{N,\epsilon_2,\epsilon_3}}\bigg(\partial_tf_N\left(t,Z_N\right)+\sum_{i=1}^Nv_i\nabla_{x_i}f_N\left(t,Z_N\right)\bigg)\phi_s(t,Z_s)\,dX_N\,dV_N\,dt=0.
\end{equation}
For the time derivative in \eqref{weak form initial}, we use Fubini's Theorem, integration by parts in time,  the fact that $f_N$ is supported in $(0,\infty)\times\mathcal{D}_{N,\epsilon_2,\epsilon_3}$ and the fact that $\phi_s$ is compactly supported in $(0,\infty)\times\mathcal{D}_{s,\epsilon_2,\epsilon_3}$, to obtain
\begin{align}
\int_{(0,\infty)\times\mathcal{D}_{N,\epsilon_2,\epsilon_3}}\partial_tf_N(t,Z_N)\phi_s(t,Z_s)\,dX_N\,dV_N\,dt&=\int_{(0,\infty)\times\mathcal{D}_{s,\epsilon_2,\epsilon_3}}\partial_tf_N^{(s)}(t,Z_s)\phi_s(t,Z_s)\,dX_s\,dV_s\,dt\label{time term BBGKY}.
\end{align}

For the material derivative term in \eqref{weak form initial}, the Divergence Theorem implies that 
\begin{align}
&\int_{\mathcal{D}_{N,\epsilon_2,\epsilon_3}}\sum_{i=1}^Nv_i\nabla_{x_i}f_N\left(t,Z_N\right)\phi_s(t,Z_s)\,dX_N\,dV_N=\int_{\mathcal{D}_{N,\epsilon_2,\epsilon_3}}\diverg_{X_N}\left[f_N\left(t,Z_N\right)V_N\right]\phi_s(t,Z_s)\,dX_N\,dV_N\nonumber\\
&=-\int_{\mathcal{D}_{N,\epsilon_2,\epsilon_3}}V_N\cdot\nabla_{X_N}\phi_s(t,Z_s)f_N(t,Z_N)\,dX_N\,dV_N+\int_{\Sigma^X\times\mathbb{R}^{dN}}\hat{n}\left(X_N\right)\cdot V_Nf_N\left(t,Z_N\right)\phi_s\left(t,Z_s\right)\,dV_N\,d\sigma,\label{first int by parts diverg}
\end{align}
where $\Sigma^X$ is given by \eqref{decomposition of the boundary BBGKY}, $\hat{n}(X_N)$ is the outwards normal vector on $\Sigma^X$ at $X_N\in\Sigma^X$ and $\,d\sigma$ is the surface measure on $\Sigma^X$. Using the fact that $f_N$ is supported in $\mathcal{D}_{N,\epsilon_2,\epsilon_3}$, Divergence Theorem and the fact that $\phi_s$ is compactly supported in $(0,\infty)\times\mathcal{D}_{s,\epsilon_2,\epsilon_3}$, we obtain
\begin{align}
\int_{\mathcal{D}_{N,\epsilon_2,\epsilon_3}}V_N\cdot\nabla_{X_N}\phi_s(t,Z_s)f_N(t,Z_N)\,dX_N\,dV_N
&=-\int_{\mathcal{D}_{s,\epsilon_2,\epsilon_3}}\sum_{i=1}^sv_i\nabla_{x_i}f_N^{(s)}(t,Z_s)\phi_s(t,Z_s)\,dX_s\,dV_s\label{divergence transport term},
\end{align}
Combining \eqref{weak form initial}-\eqref{divergence transport term}, and recalling the space boundary decomposition \eqref{decomposition of the boundary BBGKY}, we obtain 
\begin{align}
&\int_{(0,\infty)\times\mathcal{D}_{s,\epsilon_2,\epsilon_3}}\left(\partial_tf_N^{(s)}\left(t,Z_s\right)+\sum_{i=1}^sv_i\nabla_{x_i}f_N^{(s)}\left(t,Z_s\right)\right)\phi_s\left(t,Z_s\right)\,dX_s\,dV_s\,dt\nonumber\\
&=-\int_{(0,\infty)\times\Sigma^X\times\mathbb{R}^{dN}}\hat{n}\left(X_N\right)\cdot V_Nf_N\left(t,Z_N\right)\phi_s\left(t,Z_s\right)\,dV_N\,d\sigma\,dt,\nonumber\\
&=:\int_0^\infty \sum_{(i,j)\in\mathcal{I}_N^2} C_{ij}^2(t)+\sum_{(i,j,k)\in\mathcal{I}_N^3} C_{ijk}^3(t)\,dt,\label{weak Liouville}
\end{align}
where for $(i,j)\in\mathcal{I}_N^2$, $t>0$, we denote
\begin{equation}\label{c_ijk^2 def}
C_{ij}^2(t)=-\int_{\Sigma_{i,j}^{2,sc,X}\times\mathbb{R}^{dN}}\hat{n}_{ij}^2\left(X_N\right)\cdot V_Nf_N\left(t,Z_N\right)\phi_s\left(t,Z_s\right)\,dV_N\,d\sigma_{ij}^2,
\end{equation}
 for $(i,j,k)\in\mathcal{I}_N^3$, $t>0$, we denote
\begin{equation}\label{c_ijk^3 def}
C_{ijk}^3(t)=-\int_{\Sigma_{i,j,k}^{3,sc,X}\times\mathbb{R}^{dN}}\hat{n}_{ijk}^3\left(X_N\right)\cdot V_Nf_N\left(t,Z_N\right)\phi_s\left(t,Z_s\right)\,dV_N\,d\sigma_{ijk}^3,
\end{equation}
and $\hat{n}_{ij}^2(X_N)$ is the outwards normal vector on $\Sigma_{ij}^{2,sc,X}$ at $X_N\in\Sigma_{ij}^{2,sc,X}$, $\,d\sigma_{ij}^2$ is the surface measure on $\Sigma_{ij}^{2,sc,X}$, while  $\hat{n}_{ijk}^3(X_N)$ is the outwards normal vector on $\Sigma_{ijk}^{3,sc,X}$ at $X_N\in\Sigma_{ijk}^{3,sc,X}$ and $\,d\sigma_{ijk}^3$ is the surface measure on $\Sigma_{ijk}^{3,sc,X}$.

Following similar calculations to \cite{gallagher} which treats the binary case, and \cite{ternary} which treats the ternary case, we formally obtain the BBGKY hierarchy:
\begin{equation}\label{BBGKY}\begin{cases}
\partial_tf_N^{(s)}+\displaystyle\sum_{i=1}^sv_i\nabla_{x_i}f_N^{(s)}=\mathcal{C}_{s,s+1}^Nf_N^{(s+1)}+\mathcal{C}_{s,s+2}^Nf_N^{(s+2)},\quad (t,Z_s)\in (0,\infty)\times\mathring{\mathcal{D}}_{s,\epsilon_2,\epsilon_3},\\
f_N^{(s)}(t,Z_s')=f_N^{(s)}(t,Z_s),\quad(t,Z_s)\in [0,\infty)\times\partial_{2,sc}\mathcal{D}_{s,\epsilon_2,\epsilon_3},\text{ whenever } s\geq 2,\\
f_N^{(s)}(t,Z_s^*)=f_N^{(s)}(t,Z_s),\quad(t,Z_s)\in [0,\infty)\times\partial_{3,sc}\mathcal{D}_{s,\epsilon_2,\epsilon_3},\text{ whenever } s\geq 3,\\
f_N^{(s)}(0,Z_s)=f_{N,0}^{(s)}(Z_s),\quad Z_s\in\mathring{\mathcal{D}}_{s,\epsilon_2,\epsilon_3},
\end{cases}
\end{equation}
where
\begin{align}
\mathcal{C}_{s,s+1}^N&=\mathcal{C}_{s,s+1}^{N,+}-\mathcal{C}_{s,s+1}^{N,-}\label{BBGKY operator binary},\\
\mathcal{C}_{s,s+2}^N&=\mathcal{C}_{s,s+2}^{N,+}-\mathcal{C}_{s,s+2}^{N,-}\label{BBGKY operator triary}.
\end{align}
and we use the following notation: 

$\bullet$ \textbf{Binary notation:}
For $1\leq s\leq N-1$ we denote
\begin{equation}\label{BBGKY operator+ binary}
\begin{aligned}
\mathcal{C}_{s,s+1}^{N,+}f_N^{(s+1)}(t,Z_s)
=A_{N,\epsilon_2,s}^2\sum_{i=1}^s&\int_{\mathbb{S}_1^{d-1}\times\mathbb{R}^{d}}b_2^+(\omega_1,v_{s+1}-v_i) f_N^{(s+1)}\left(t,Z_{s+1,\epsilon_2,i}'\right)\,d\omega_1\,dv_{s+1},
\end{aligned}
\end{equation}

\begin{equation}\label{BBGKY operator- binary}
\begin{aligned}
\mathcal{C}_{s,s+1}^{N,-}f_N^{(s+2)}(t,Z_s)=
A_{N,\epsilon_2,s}^2\sum_{i=1}^s&\int_{\mathbb{S}_1^{d-1}\times\mathbb{R}^{d}}b_2^+(\omega_1,v_{s+1}-v_i)f_N^{(s+1)}\left(t,Z_{s+1,\epsilon_2,i}\right)\,d\omega_1\,dv_{s+1},
\end{aligned}
\end{equation}
where 
 \begin{equation}\label{A binary}
 \begin{aligned}
 &b_2(\omega_1,v_{s+1}-v_i)=\langle\omega_1,v_{s+1}-v_i\rangle,\\
 &b_2^+=\max\{b_2,0\},\\
 &A_{N,\epsilon_2,s}^2=(N-s)\epsilon_2^{d-1},\\
 &Z_{s+1,\epsilon_2,i}=(x_1,...,x_i,...,x_s,x_i-\epsilon_2\omega_1,v_1,...v_{i-1},v_i,v_{i+1},...,v_s,v_{s+1}),\\
 &Z_{s+1,\epsilon_2,i}'=(x_1,...,x_i,...,x_s,x_i+\epsilon_2\omega_1,v_1,...v_{i-1},v_i',v_{i+1},...,v_s,v_{s+1}').
\end{aligned} 
 \end{equation}
For $s\geq N$ we trivially define
$
\mathcal{C}_{s,s+1}^{N}\equiv 0.
$

$\bullet$ \textbf{Ternary notation:}
For $1\leq s\leq N-2$ we denote
\begin{equation}\label{BBGKY operator+ triary}
\begin{aligned}
\mathcal{C}_{s,s+2}^{N,+}f_N^{(s+2)}(t,Z_s)
=A_{N,\epsilon_3,s}^3\sum_{i=1}^s&\int_{\mathbb{S}_1^{2d-1}\times\mathbb{R}^{2d}}\frac{b_3^+(\omega_1,\omega_2,v_{s+1}-v_i,v_{s+2}-v_i)}{\sqrt{1+\langle\omega_1,\omega_2\rangle}}\\
 &\times f_N^{(s+2)}\left(t,Z_{s+2,\epsilon_3,i}^*\right)\,d\omega_1\,d\omega_2\,dv_{s+1}\,dv_{s+2},
\end{aligned}
\end{equation}

\begin{equation}\label{BBGKY operator- triary}
\begin{aligned}
\mathcal{C}_{s,s+2}^{N,-}f_N^{(s+2)}(t,Z_s)=
A_{N,\epsilon_3,s}^3\sum_{i=1}^s&\int_{\mathbb{S}_1^{2d-1}\times\mathbb{R}^{2d}}\frac{b_3^+(\omega_1,\omega_2,v_{s+1}-v_i,v_{s+2}-v_i)}{\sqrt{1+\langle\omega_1,\omega_2\rangle}}\\
 &\times f_N^{(s+2)}\left(t,Z_{s+2,\epsilon_3,i}\right)\,d\omega_1\,d\omega_2\,dv_{s+1}\,dv_{s+2},
\end{aligned}
\end{equation}
where 
 \begin{equation}\label{A triary}
 \begin{aligned}
 &A_{N,\epsilon_3,s}^3=2^{d-2}(N-s)(N-s-1)\epsilon_3^{2d-1},\\
  &b_3(\omega_1,\omega_2,v_{s+1}-v_i,v_{s+2}-v_i)=\langle\omega_1,v_{s+1}-v_i\rangle+\langle\omega_2,v_{s+2}-v_i\rangle,\\
 &b_{3}^+=\max\{b_3,0\},\\
 &Z_{s+2,\epsilon_3,i}=(x_1,...,x_i,...,x_s,x_i-\sqrt{2}\epsilon_3\omega_1,x_i-\sqrt{2}\epsilon_3\omega_2,v_1,...v_{i-1},v_i,v_{i+1},...,v_s,v_{s+1},v_{s+2}),\\
 &Z_{s+2,\epsilon_3,i}^*=(x_1,...,x_i,...,x_s,x_i+\sqrt{2}\epsilon_3\omega_1,x_i+\sqrt{2}\epsilon_3\omega_2,v_1,...v_{i-1},v_i^*,v_{i+1},...,v_s,v_{s+1}^*,v_{s+2}^*).
\end{aligned} 
 \end{equation}
For $s\geq N-1$ we trivially define
$
\mathcal{C}_{s,s+2}^{N}\equiv 0.
$

 Duhamel's formula implies that the BBGKY hierarchy can be written in mild form as follows
 \begin{equation}\label{mild BBGKY}f_N^{(s)}(t,Z_s)=T_s^tf_{N,0}^{(s)}(Z_s)+\int_0^tT_s^{t-\tau}\left(\mathcal{C}_{s,s+1}^Nf_N^{(s+1)}+\mathcal{C}_{s,s+2}^Nf_N^{(s+2)}\right)(\tau,Z_s)\,d\tau,\quad s\in\mathbb{N},\end{equation}
 where $T_s^t$ is the $s$-particle $(\epsilon_2,\epsilon_3)$-flow  operator given in \eqref{liouville operator}.
 \subsection{The Boltzmann hierarchy}
 We will now derive the Boltzmann hierarchy as the formal limit of the BBGKY hierarchy as $N\to\infty$ and $\epsilon_2,\epsilon_3\to 0^+$ under the scaling
 \begin{equation}\label{scaling}
 N\epsilon_2^{d-1}\simeq N\epsilon_3^{d-1/2}\simeq 1.
 \end{equation}
 This scaling implies that $\epsilon_2$,$\epsilon_3$ satisfy
 \begin{equation}\label{relation epsilons}
 \epsilon_2^{d-1}\simeq\epsilon_3^{d-1/2}.
 \end{equation}
 \begin{remark}\label{remark for epsilons} Using the scaling \eqref{scaling}, we obtain
 \begin{align}
 \epsilon_2\simeq N^{-\frac{1}{d-1}}\overset{N\to\infty}\longrightarrow 0,\quad
 \epsilon_3\simeq N^{-\frac{2}{2d-1}}\overset{N\to\infty}\longrightarrow 0,\label{epsilon  with respect to N}
 \end{align}
 thus
 \begin{equation}\label{quotient of epsilon}
 \frac{\epsilon_2}{\epsilon_3}\simeq N^{-\frac{1}{(d-1)(2d-1)}}\overset{N\to\infty}\longrightarrow 0,
 \end{equation}
Therefore, for $N$ large enough, we have $\epsilon_2<<\epsilon_3$.
 \end{remark}
 \begin{remark}\label{remark on growth factors}
 The scaling \eqref{scaling} guarantees that for a fixed $s\in\mathbb{N}$, we have
 \begin{equation*}
 \begin{aligned}
 A_{N,\epsilon_2,s}^2&=(N-s)\epsilon_2^{d-1}\longrightarrow 1,\quad\text{as }N\to\infty,\\
 A_{N,\epsilon_3,s}^3&=2^{d-2}(N-s)(N-s-1)\epsilon_3^{2d-1}\longrightarrow 1,\quad\text{as }N\to\infty.
 \end{aligned}
 \end{equation*}
 \end{remark}
 
Formally taking the limit under the scaling imposed we may define the following collisional operators:

$\bullet$ \textbf{Binary Boltzmann operator:}
\begin{equation}\label{boltzmann hiera kernel binary}
\mathcal{C}_{s,s+1}^{\infty}=\mathcal{C}_{s,s+1}^{\infty,+}-\mathcal{C}_{s,s+1}^{\infty,-},
\end{equation}
where
\begin{equation}\label{Boltzmann operator + binary}
\mathcal{C}_{s,s+1}^{\infty,+}f^{(s+1)}(t,Z_s)=\sum_{i=1}^s\int_{(\mathbb{S}_1^{d-1}\times\mathbb{R}^{d})}b_2^+(\omega_1,v_{s+2}-v_i)f^{(s+1)}\left(t,Z_{s+1,i}'\right)
\times \,d\omega_1\,dv_{s+1},
\end{equation}
\begin{equation}\label{Boltzmann operator - binary}
\mathcal{C}_{s,s+1}^{\infty,-}f^{(s+1)}(t,Z_s)=\sum_{i=1}^s\int_{(\mathbb{S}_1^{d-1}\times\mathbb{R}^{d})}b_2^+(\omega_1,v_{s+2}-v_i)\times f^{(s+1)}\left(t,Z_{s+1,i}\right)
\times\,d\omega_1\,dv_{s+1},
\end{equation}
\begin{equation}\label{boltzmann notation binary}
\begin{aligned}
&b_2(\omega_1,v_{s+1}-v_i)=\langle\omega_1,v_{s+1}-v_i\rangle,\\
&b_2=\max\{0,b_2\},\\
&Z_{s+1,i}=(x_1,...,x_i,...,x_s,x_i,v_1,...v_{i-1},v_i,v_{i+1},...,v_s,v_{s+1}),\\
&Z_{s+1,i}'=(x_1,...,x_i,...,x_s,x_i,v_1,...v_{i-1},v_i',v_{i+1},...,v_s,v_{s+1}').
\end{aligned}
\end{equation}

$\bullet$ \textbf{Ternary Boltzmann operator:}
\begin{equation}\label{boltzmann hiera kernel ternary}
\mathcal{C}_{s,s+2}^{\infty}=\mathcal{C}_{s,s+2}^{\infty,+}-\mathcal{C}_{s,s+2}^{\infty,-},
\end{equation}
where
\begin{equation}\label{Boltzmann operator + triary}
\begin{aligned}
\mathcal{C}_{s,s+2}^{\infty,+}f^{(s+2)}(t,Z_s)=\sum_{i=1}^s\int_{(\mathbb{S}_1^{2d-1}\times\mathbb{R}^{2d})}\frac{b_3^+(\omega_1,\omega_2,v_{s+1}-v_i,v_{s+2}-v_i)}{\sqrt{1+\langle\omega_1,\omega_2\rangle}}f^{(s+2)}\left(t,Z_{s+2,i}^*\right)&\\ 
\times \,d\omega_1\,d\omega_2\,dv_{s+1}\,dv_{s+2}&,
\end{aligned}
\end{equation}
\begin{equation}\label{Boltzmann operator - triary}
\begin{aligned}
\mathcal{C}_{s,s+2}^{\infty,-}f^{(s+2)}(t,Z_s)=\sum_{i=1}^s\int_{(\mathbb{S}_1^{2d-1}\times\mathbb{R}^{2d})}\frac{b_3^+(\omega_1,\omega_2,v_{s+1}-v_i,v_{s+2}-v_i)}{\sqrt{1+\langle\omega_1,\omega_2\rangle}}\times f^{(s+2)}\left(t,Z_{s+2,i}\right)&\\ 
\times\,d\omega_1\,d\omega_2\,dv_{s+1}\,dv_{s+2}&,
\end{aligned}
\end{equation}
\begin{equation}\label{boltzmann notation triary}
\begin{aligned}
&b_3(\omega_1,\omega_2,v_{s+1}-v_i,v_{s+2}-v_i)=\langle\omega_1,v_{s+1}-v_i\rangle+\langle\omega_2,v_{s+2}-v_i\rangle,\\
&b_3^+=\max\{b_3,0\},\\
&Z_{s+2,i}=(x_1,...,x_i,...,x_s,x_i,x_i,v_1,...v_{i-1},v_i,v_{i+1},...,v_s,v_{s+1},v_{s+2}),\\
&Z_{s+2,i}^*=(x_1,...,x_i,...,x_s,x_i,x_i,v_1,...v_{i-1},v_i^*,v_{i+1},...,v_s,v_{s+1}^*,v_{s+2}^*).
\end{aligned}
\end{equation}

Now we are ready to introduce the Boltzmann hierarchy. More precisely, given an initial probability density $f_0$, the Boltzmann hierarchy for $s\in\mathbb{N}$ is given by:
\begin{equation}\label{Boltzmann hierarchy}\begin{cases}
\partial_tf^{(s)}+\displaystyle\sum_{i=1}^sv_i\nabla_{x_i}f^{(s)}=\mathcal{C}_{s,s+1}^\infty f^{(s+1)}+\mathcal{C}_{s,s+2}^\infty f^{(s+2)},\quad(t,Z_s)\in (0,\infty)\times\mathbb{R}^{2ds},\\
f^{(s)}(0,Z_s)=f_0^{(s)}(Z_s),\quad\forall Z_s\in\mathbb{R}^{2ds}.
\end{cases}
\end{equation}

Duhamel's formula implies that the Boltzmann hierarchy can be written in mild form as follows
 \begin{equation}\label{mild Boltzmann} f^{(s)}(t,Z_s)=S_s^tf_0^{(s)}(Z_s)+\int_0^tS_s^{t-\tau}\left(\mathcal{C}_{s,s+1}^\infty f^{(s+1)}+\mathcal{C}_{s,s+2}^\infty f^{(s+2)}\right)(\tau,Z_s)\,d\tau,\quad s\in\mathbb{N},
 \end{equation}
 where $S_s^t$ denotes the $s-$particle free flow operator given in \eqref{free flow operator}.
\subsection{ The binary-ternary Boltzmann equation}
\subsubsection{The binary-ternary Boltzmann equation} In most applications,  particles are initially independently distributed. This translates to tensorized Boltzmann hierarchy initial data i.e. 
\begin{equation}\label{tensorized initial data}
f_0^{(s)}(Z_s)=f_0^{\otimes s}(Z_s)=\prod_{i=1}^sf_0(x_i,v_i),\quad s\in\mathbb{N},
\end{equation}
where $f_0:\mathbb{R}^{d}\times\mathbb{R}^d\to\mathbb{R}$ is a given function. One can easily verify that the anszatz:
\begin{equation}\label{propagation of chaos}
f^{(s)}(t,Z_s)=f^{\otimes s}(t,Z_s)=\prod_{i=1}^sf(t,x_i,v_i),\quad s\in\mathbb{N},
\end{equation}
solves the Boltzmann hierarchy with initial data given by \eqref{tensorized initial data}, if 
 $f:[0,\infty)\times\mathbb{R}^{d}\times\mathbb{R}^d\to\mathbb{R}$ satisfies the following nonlinear integro-differential equation:
\begin{equation}\label{Boltzmann equation}
\begin{cases}
\partial_tf+v\cdot\nabla_xf=Q_2(f,f)+Q_3(f,f,f),\quad (t,x,v)\in (0,\infty)\times\mathbb{R}^{2d},\\
f(0,x,v)=f_0(x,v),\quad (x,v)\in\mathbb{R}^{2d},
\end{cases}
\end{equation}
which we call the binary-ternary Boltzmann equation. 
 The binary collisional operator $Q_2$ is given by
 \begin{equation}\label{Boltzmann operator binary}
Q_2(f,f)(t,x,v)=\int_{\mathbb{S}_1^{d-1}\times\mathbb{R}^{d}}b_2^+(\omega_1,v_{1}-v)\left(f'f_1'-ff_1\right)\,d\omega_1\,dv_1,
\end{equation}
where
\begin{equation}\label{notation binary}
\begin{aligned}
&b_2(\omega_1,v_1-v)=\langle\omega_1,v_1-v\rangle,\\
&b_2^+=\max\{0,b_2\},\\
&f'=f(t,x,v'),\quad f=f(t,x,v),\\
&f_1'=f(t,x,v_1'),\quad f_1=f(t,x,v_1).
\end{aligned}
\end{equation}

 The ternary collisional operator $Q_3$ is given by
\begin{equation}\label{Boltzmann operator triary}
Q_3(f,f,f)(t,x,v)=\int_{\mathbb{S}_1^{2d-1}\times\mathbb{R}^{2d}}\frac{b_3^+(\omega_1,\omega_2,v_1-v,v_2-v)}{\sqrt{1+\langle\omega_1,\omega_2\rangle}}\left(f^*f_1^*f_2^*-ff_1f_2\right)\,d\omega_1\,d\omega_2\,dv_1\,dv_2,
\end{equation}
where
\begin{equation}\label{notation triary}
\begin{aligned}
&b_3(\omega_1,\omega_2,v_{s+1}-v_i,v_{s+2}-v_i)=\langle\omega_1,v_1-v\rangle+\langle\omega_2,v_2-v\rangle,\\
&b_3^+=\max\{0,b_3\},\\
&f^*=f(t,x,v^*),\quad f=f(t,x,v),\\
&f_1^*=f(t,x,v_1^*),\quad f_1=f(t,x,v_1),\\
&f_2^*=f(t,x,v_2^*),\quad f=f(t,x,v_2).
\end{aligned}
\end{equation}

Duhamel's formula implies the binary-ternary Boltzmann equation can be written in mild form as 
\begin{equation}\label{mild Boltzmann equation}
\begin{aligned}
f(t,x,v)&=S_1^tf_0(x,v)+\int_0^tS_1^{t-\tau}Q(f,f,f)(\tau,x,v)\,d\tau,
\end{aligned}
\end{equation}
where
\begin{equation*}
S_1^tg(x,v)=g(x-tv,v),\quad\forall (t,x,v)\in[0,\infty)\times\mathbb{R}^{2d},\quad g:\mathbb{R}^{2d}\to\mathbb{R}.
\end{equation*}

\begin{remark}\label{remark on propagation}
 We will see in Section \ref{sec:local} that both the Boltzmann hierarchy and the binary-ternary Boltzmann equation are well-posed in  appropriate functional spaces. 
It is not hard to see that if $f$ is formally a solution to the binary-ternary Boltzmann equation with initial data $f_0$, then  the tensorized product $F:=(f^{\otimes s})_{s\in\mathbb{N}}$ is a solution to the Boltzmann hierarchy with initial data $F_0:=(f_0^{\otimes s})_{s\in\mathbb{N}}$.
 Therefore, the tensorized product of the  unique solution to the binary-ternary Boltzmann equation with initial data $f_0$ will give the unique mild solution to the Boltzmann hierarchy with initial data $F_0$.
\end{remark}

\begin{remark}\label{remark on equation properties} It is important to point out that in \cite{thesis}, the ternary operator $Q_3$ was symmetrized to an operator $\widetilde{Q}_3$ which shares  similar statistical and entropy production properties with the classical binary Boltzmann operator $Q_2$ (see \cite{cercignani gases}). In particular, it has a weak formulation which yields an $\mathcal{H}$-Theorem and local conservation laws. Hence, the operator $Q_2+\widetilde{Q}_3$ satisfies these statistical properties as well.
 This observation illustrates that the binary-ternary equation we are studying could serve as an extension term of the classical Boltzmann equation in modeling denser gases.  
\end{remark}

\section{Local well-posedness}\label{sec:local} In this section, we show that the BBGKY hierarchy, the Boltzmann hierarchy and the binary-ternary Boltzmann equation are well-posed for short times in Maxwellian weighted $L^\infty$-spaces. To obtain these results, we combine the continuity  estimates on the binary and ternary collisional operators, obtained in \cite{gallagher} and \cite{ternary} respectively.
 \subsection{LWP for the BBGKY hierarchy}\label{sub BBGKY well posedness}
  Consider $(N,\epsilon_2,\epsilon_3)$ in the scaling \eqref{scaling}, with $N\geq 3$.  For $s\in\{1,...,N\}$, recall from \eqref{continuous finite} the space of functions
  \begin{align*}
C^0(\mathcal{D}_{s,\epsilon_2,\epsilon_3})&:=\{g_m\in C^0(\mathbb{R}^{2ds}):\supp g_s\subseteq \mathcal{D}_{s,\epsilon_2,\epsilon_3}\}.
\end{align*}
  For $\beta> 0$ we define  the Banach space
 \begin{equation*}
 X_{N,\beta,s}
 :=\left\{g_{N,s}\in C^0(\mathcal{D}_{m,\epsilon_2,\epsilon_3})\text{ and }  |g_{N,s}|_{N,\beta,s}<\infty\right\},
 \end{equation*}
with norm
$|g_{N,s}|_{N,\beta,s}=\sup_{Z_s\in\mathbb{R}^{2ds}}|g_{N,s}(Z_s)|e^{\beta E_s(Z_s)},$
where $E_s(Z_s)$ is the kinetic energy of the $s$-particles given by \eqref{kinetic energy}.
For $s>N$ we trivially define  
$X_{N,\beta,s}:=\left\{0\right\}.
$
\begin{remark}\label{T_s isometry} Given $t\in\mathbb{R}$ and $s\in\mathbb{N}$, conservation of energy under the flow \eqref{kinetic energy flow} implies that the $s$-particle of $(\epsilon_2,\epsilon_3)$-flow operator $T_s^t:X_{N,\beta,s}\to X_{N,\beta,s}$, given in \eqref{liouville operator} is an isometry i.e.
\begin{equation*}
|T_s^tg_{N,s}|_{N,\beta,s}=|g_{N,s}|_{N,\beta,s},\quad\forall g_{N,s}\in X_{N,\beta,s}.
\end{equation*}
\end{remark}
\begin{proof} Let $g_{N,s}\in X_{N,\beta,s}$ and $Z_s\in\mathbb{R}^{2ds}$. If $Z_s\notin \mathcal{D}_{s,\epsilon_2,\epsilon_3}$, the result is trivial since $ g_{N,s}$ is supported in $\mathcal{D}_{s,\epsilon_2,\epsilon_3}$. Assume $Z_s\in\mathcal{D}_{s,\epsilon_2,\epsilon_3}$. Then Theorem \ref{global flow} yields
\begin{equation*}
\begin{aligned}
e^{\beta E_s(Z_s)}|T_s^tg_{N,s}|&=e^{\beta E_s(Z_s)}|(g_{N,s}\circ\Psi_s^{-t})(Z_s)|=e^{\beta E_s\left(\Psi_s^{-t}Z_s\right)}|g_{N,s}(\Psi_s^{-t}Z_s)|\leq |g_{N,s}|_{N,s,\beta},
\end{aligned}
\end{equation*}
hence $|T_s^tg_{N,s}|_{N,s,\beta}\leq |g_{N,s}|_{N,s,\beta}$.
The other side of the inequality comes similarly using the fact that $Z_s=\Psi_s^{-t}(\Psi_s^tZ_s)$.
\end{proof}
Consider as well $\mu\in\mathbb{R}$. We define the Banach space 
\begin{equation*}
X_{N,\beta,\mu}:=\left\{G_N=(g_{N,s})_{s\in\mathbb{N}}:\|G_N\|_{N,\beta,\mu}<\infty\right\},\end{equation*}
with norm
$\|G_N\|_{N,\beta,\mu}=\sup_{s\in\mathbb{N}}e^{\mu s}|g_{N,s}|_{N,\beta,s}=\max_{s\in\{1,...,N\}}e^{\mu s}|g_{N,s}|_{N,\beta,s}.$
\begin{remark}\label{T_N isometry} Given $t\in\mathbb{R}$, Remark \ref{T_s isometry} implies that the map $\mathcal{T}^t:X_{N,\beta,\mu}\to X_{N,\beta,\mu}$ given by
\begin{equation}\label{T_N definition}
\mathcal{T}^tG_N:=\left(T_s^tg_{N,s}\right)_{s\in\mathbb{N}},
\end{equation}
is an isometry i.e.
$
\|\mathcal{T}^tG_N\|_{N,\beta,\mu}=\|G_N\|_{N,\beta,\mu},$ for any $G_N\in X_{N,\beta,\mu}.
$
\end{remark}
Finally, given $T>0$, $\beta_0> 0$, $\mu_0\in\mathbb{R}$ and $\bm{\beta},\bm{\mu}:[0,T]\to\mathbb{R}$ decreasing functions of time with $\bm{\beta}(0)=\beta_0$, $\bm{\beta}(T)> 0$, $\bm{\mu}(0)=\mu_0$, we define the Banach space 
\begin{equation*}
\bm{X}_{N,\bm{\beta},\bm{\mu}}:=L^\infty\left([0,T],X_{N,\bm{\beta}(t),\bm{\mu}(t)}\right),
\end{equation*}
with norm
$|||\bm{G_N}|||_{N,\bm{\beta},\bm{\mu}}=\sup_{t\in[0,T]}\|\bm{G_N}(t)\|_{N,\bm{\beta}(t),\bm{\mu}(t)}.$
Similarly as  in Proposition 6.2. from \cite{thesis}, one can obtain the following bounds:
\begin{proposition}\label{remark for initial} Let $T>0$, $\beta_0>0$, $\mu_0\in\mathbb{R}$ and $\bm{\beta},\bm{\mu}:[0,T]\to\mathbb{R}$ decreasing functions with $\beta_0=\bm{\beta}(0)$, $\bm{\beta}(T)> 0$ $\mu_0=\bm{\mu}(0)$. Then for any $G_N=\left(g_{N,s}\right)_{s\in\mathbb{N}}\in X_{N,\beta_0,\mu_0}$, the following estimates hold:
\begin{enumerate}[(i)]
\item $|||G_N|||_{N,\bm{\beta},\bm{\mu}}\leq\|G_N\|_{N,\beta_0,\mu_0}$.\vspace{0.2cm}
\item $\left|\left|\left|\displaystyle\int_0^t\mathcal{T}^{\tau}G_N\,d\tau\right|\right|\right|_{N,\bm{\beta},\bm{\mu}}\leq T\|G_N\|_{N,\beta_0,\mu_0}.$
\end{enumerate}
\end{proposition}
From Proposition 5.3.1. in \cite{gallagher} and Lemma 5.1. in \cite{ternary}, we have the following continuity estimates for the binary and ternary collisional operators respectively:
\begin{lemma}\label{a priori lemma for C BBGKY}
Let $m\in\mathbb{N}$, $\beta>0$.  For any $Z_m\in\mathcal{D}_{m,\epsilon_2,\epsilon_3}$ and $k\in\{1,2\}$, the following  estimate holds:
\begin{equation*}
\left|\mathcal{C}_{m,m+k}^{N}g_{N,m+k}(Z_m)\right|\lesssim  \beta^{-kd/2}\left(m\beta^{-1/2}+\sum_{i=1}^m|v_i|\right)e^{-\beta E_m(Z_m)}|g_{N,m+k}|_{N,\beta,m+k},\quad\forall g_{N,m+k}\in X_{N,\beta,m+k}.
\end{equation*}
\end{lemma}

Let us now define mild solutions to the BBGKY hierarchy:
\begin{definition}\label{def of mild bbgky} Consider $T>0$, $\beta_0> 0$, $\mu_0\in\mathbb{R}$ and the decreasing functions $\bm{\beta},\bm{\mu}:[0,T]\to\mathbb{R}$ with $\bm{\beta}(0)=\beta_0$, $\bm{\beta}(T)> 0$, $\bm{\mu}(0)=\mu_0$.
Consider also  initial data $G_{N,0}=\left(g_{N,s,0}\right)\in X_{N,\beta_0,\mu_0}$. A map $\bm{G_N}=\left(g_{N,s}\right)_{s\in\mathbb{N}}\in\bm{X}_{N,\bm{\beta},\bm{\mu}}$ is a mild solution of the BBGKY hierarchy in $[0,T]$, with initial data $G_{N,0}$, if it satisfies:
\begin{equation*}\bm{G_N}(t)=\mathcal{T}^tG_{N,0}+\int_0^t \mathcal{T}^{t-\tau}\mathcal{C}_N\bm{G_N}(\tau)\,d\tau,\end{equation*}
where, given $\beta>0$, $\mu\in\mathbb{R}$ and $G_{N}=(g_{N,s})_{s\in\mathbb{N}}\in X_{N,\beta,\mu}$, we write
\begin{align*}
\mathcal{C}_N G_N:=(\mathcal{C}_N^2 +\mathcal{C}_N^3) G_N,\quad\mathcal{C}_N^2 G_N:=\left(\mathcal{C}_{s,s+1}^N g_{N,s+1}\right)_{s\in\mathbb{N}},\quad\mathcal{C}_N^3  G_N:=\left(\mathcal{C}_{s,s+2}^N g_{N,s+2}\right)_{s\in\mathbb{N}},
\end{align*}
and $\mathcal{T}^t$ is given by \eqref{T_N definition}.
\end{definition}

Using Lemma \ref{a priori lemma for C BBGKY}, we obtain the following a-priori bounds:
\begin{lemma}\label{a priori lemma for T BBGKY} Let $\beta_0> 0$, $\mu_0\in\mathbb{R}$, $T>0$ and $\lambda\in (0,\beta_0/T)$. Consider the functions $\bm{\beta}_\lambda,\bm{\mu}_\lambda:[0,T]\to\mathbb{R}$ given by
 \begin{equation}\label{beta_lambda-mu_lambda}
 \begin{aligned}
 \bm{\beta}_\lambda(t)=\beta_0-\lambda t,\quad\bm{\mu}_\lambda(t)=\mu_0-\lambda t.
 \end{aligned}
\end{equation}  
 Then  for any $\mathcal{F}(t)\subseteq [0,t]$ measurable,  $\bm{G_N}=\left(g_{N,s}\right)_{s\in\mathbb{N}}\in\bm{X}_{N,\bm{\beta}_\lambda,\bm{\mu}_\lambda}$ and $k\in\{1,2\}$ the following bounds hold:
 \begin{align}
\left|\left|\left|\displaystyle\int_{\mathcal{F}(t)}\mathcal{T}^{t-\tau}\mathcal{C}_N^{k+1}\bm{G_N}(\tau)\,d\tau\right|\right|\right|_{N,\bm{\beta}_\lambda,\bm{\mu}_\lambda}&\leq C_{k+1}|||\bm{G_N}|||_{N,\bm{\beta}_\lambda,\bm{\mu}_\lambda},\label{both bbkky with constant}
\end{align}
 \begin{align}
 C_{k+1}=C_{k+1}(d,\beta_0,\mu_0,T,\lambda)&=C_d\lambda^{-1}e^{-k\bm{\mu}_\lambda(T)}\bm{\beta}_\lambda^{-kd/2}(T)\left(1+\bm{\beta}_{\lambda}^{-1/2}(T)\right)\label{constant of WP binary}.
 \end{align}
\end{lemma}
\begin{proof}
For the proof of 
\eqref{both bbkky with constant} for $k=1$, see Lemma 5.3.1. from  \cite{gallagher} and for the proof for $k=2$ see Lemma  6.4. from \cite{thesis}. 
 \end{proof}
  Choosing $\lambda=\beta_0/2T$, 
  Lemma \ref{a priori lemma for T BBGKY} implies well-posedness of the BBGKY hierarchy up to short time. The proof follows similar steps to the proof of Theorem 6 from \cite{gallagher} and Theorem 6.4.1 from \cite{thesis}.
 \begin{theorem}\label{well posedness BBGKY}
 Let $\beta_0> 0$ and $\mu_0\in\mathbb{R}$. Then there is $T=T(d,\beta_0,\mu_0)>0$ such that for any initial datum $F_{N,0}=(f_{N,0}^{(s)})_{s\in\mathbb{N}}\in X_{N,\beta_0,\mu_0}$ there is unique mild solution $\bm{F_N}=(f_N^{(s)})_{s\in\mathbb{N}}\in\bm{X}_{N,\bm{\beta},\bm{\mu}}$ to the BBGKY hierarchy in $[0,T]$ for the functions $\bm{\beta},\bm{\mu}:[0,T]\to\mathbb{R}$ given by
  \begin{equation}\label{beta mu given lambda}
  \begin{aligned}
  \bm{\beta}(t)&=\beta_0-\frac{\beta_0}{2T}t,\quad \bm{\mu}(t)&=\mu_0-\frac{\beta_0}{2T}t.
  \end{aligned}
  \end{equation}
   The solution $\bm{F_N}$ satisfies the bound:
 \begin{equation}
  \label{a priori bound F_N,0}|||\bm{F_N}|||_{N,\bm{\beta},\bm{\mu}}\leq 2\|F_{N,0}\|_{N,\beta_0,\mu_0}.
 \end{equation}
 Moreover, for any $\mathcal{F}(t)\subseteq[0,t]$ measurable and $k\in\{1,2\}$, the following bound holds:
 \begin{align}
 \label{a priori binary bound F_N}\left|\left|\left|\int_{\mathcal{F}(t)}\mathcal{T}^{t-\tau}C_N^{k+1}\bm{G_N}(\tau)\,d\tau\right|\right|\right|_{N,\bm{\beta},\bm{\mu}}&\leq\frac{1}{16}|||G_N|||_{{N,\bm{\beta},\bm{\mu}}},\quad\forall G_N\in\bm{X}_{N,\bm{\beta},\bm{\mu}},\\
 \end{align}
The time $T$ is explicitly given by:
 \begin{equation}\label{time}
 T\simeq\beta_0\left(e^{-\mu_0-\frac{\beta_0}{2}}(\frac{\beta_0}{2})^{-d/2}+e^{-2\mu_0-\beta_0}(\frac{\beta_0}{2})^{-d}\right)^{-1}\left(1+(\frac{\beta_0}{2})^{-1/2}\right)^{-1}.
 \end{equation}
 \end{theorem}
 \subsection{LWP for the Boltzmann hierarchy} Similary to Subsection \ref{sub BBGKY well posedness}, here we establish a-priori bounds and local well-posedness for the Boltzmann hierarchy. Without loss of generality, we will omit the proofs since they are identical to the BBGKY hierarchy case.  Given $s\in\mathbb{N}$ and $\beta> 0$, we define  the Banach space
 \begin{equation*}
 X_{\infty,\beta,s}
 :=\left\{g_{s}\in C^0(\mathbb{R}^{2ds}):|g_{s}|_{\infty,\beta,s}<\infty\right\},
 \end{equation*}
with norm
$|g_{s}|_{\infty,\beta,s}=\sup_{Z_s\in\mathbb{R}^{2ds}}|g_{s}(Z_s)|e^{\beta E_s(Z_s)},$
where $E_s(Z_s)$ is the kinetic energy of the $s$-particles given by \eqref{kinetic energy}. 
\begin{remark}\label{S_s isometry} Given $t\in\mathbb{R}$ and $s\in\mathbb{N}$, conservation of energy under the free flow  implies that the $s$-particle  free  flow operator $S_s^t:X_{\infty,\beta,s}\to X_{\infty,\beta,s}$, given in \eqref{free flow operator}, is an isometry i.e.
\begin{equation*}
|S_s^tg_{s}|_{\infty,\beta,s}=|g_{s}|_{\infty,\beta,s},\quad\forall g_{s}\in X_{\infty,\beta,s}.
\end{equation*}
\end{remark}

Consider as well $\mu\in\mathbb{R}$. We define the Banach space 
\begin{equation*}
X_{\infty,\beta,\mu}:=\left\{G=(g_{s})_{s\in\mathbb{N}}:\|G\|_{\infty,\beta,\mu}<\infty\right\},\end{equation*}
with norm
$\|G\|_{\infty,\beta,\mu}=\sup_{s\in\mathbb{N}}e^{\mu s}|g_{s}|_{\infty,\beta,s}.$
\begin{remark}\label{S isometry} Given $t\in\mathbb{R}$, Remark \ref{S_s isometry} implies that the map $\mathcal{S}^t:X_{\infty,\beta,\mu}\to X_{\infty,\beta,\mu}$ given by
\begin{equation}\label{S definition}
\mathcal{S}^tG:=\left(S_s^tg_{s}\right)_{s\in\mathbb{N}},
\end{equation}
is an isometry i.e.
$
\|\mathcal{S}^tG\|_{\infty,\beta,\mu}=\|G\|_{\infty,\beta,\mu},$ for any  $G\in X_{\infty,\beta,\mu}.$
\end{remark}
Finally, given $T>0$, $\beta_0> 0$, $\mu_0\in\mathbb{R}$ and $\bm{\beta},\bm{\mu}:[0,T]\to\mathbb{R}$ decreasing functions of time with $\bm{\beta}(0)=\beta_0$, $\bm{\beta}(T)> 0$, $\bm{\mu}(0)=\mu_0$, we define the Banach space 
\begin{equation*}
\bm{X}_{\infty,\bm{\beta},\bm{\mu}}:=L^\infty\left([0,T],X_{\infty,\bm{\beta}(t),\bm{\mu}(t)}\right),
\end{equation*}
with norm
$|||\bm{G}|||_{\infty,\bm{\beta},\bm{\mu}}=\sup_{t\in[0,T]}\|\bm{G}(t)\|_{\infty,\bm{\beta}(t),\bm{\mu}(t)}.$
\begin{proposition}\label{remark for initial boltzmann hierarchy} Let $T>0$, $\beta_0>0$, $\mu_0\in\mathbb{R}$ and $\bm{\beta},\bm{\mu}:[0,T]\to\mathbb{R}$ decreasing functions with $\beta_0=\bm{\beta}(0)$, $\bm{\beta}(T)> 0$ $\mu_0=\bm{\mu}(0)$. Then for any $G=\left(g_{s}\right)_{s\in\mathbb{N}}\in X_{\infty,\beta_0,\mu_0}$, the following estimates hold:
\begin{enumerate}[(i)]
\item $|||G|||_{\infty,\bm{\beta},\bm{\mu}}\leq\|G\|_{\infty,\beta_0,\mu_0}$.\vspace{0.2cm}
\item $\left|\left|\left|\displaystyle\int_0^t\mathcal{S}^{\tau}G\,d\tau\right|\right|\right|_{\infty,\bm{\beta},\bm{\mu}}\leq T\|G\|_{\infty,\beta_0,\mu_0}.$
\end{enumerate}
\end{proposition}

Similarly to Lemma \ref{a priori lemma for C BBGKY}, we obtain:
\begin{lemma}\label{a priori lemma for C Boltzmann}
Let $m\in\mathbb{N}$ and $\beta>0$.  For any $Z_m\in\mathbb{R}^{2dm}$ and $k\in\{1,2\}$, the following continuity estimate holds:
\begin{equation}
\left|\mathcal{C}_{m,m+k}^{\infty}g_{m+k}(Z_m)\right|\lesssim  \beta^{-kd/2}\left(m\beta^{-1/2}+\sum_{i=1}^m|v_i|\right)e^{-\beta E_m(Z_m)}|g_{m+k}|_{\infty,\beta,m+k},\quad\forall g_{m+k}\in X_{\infty,\beta,m+k}.\label{cont estimate both boltzmann}
\end{equation}
\end{lemma}
Let us now define mild solutions to the Boltzmann hierarchy:
\begin{definition}\label{def mild solution boltzmann} Consider $T>0$, $\beta_0> 0$, $\mu_0\in\mathbb{R}$ and the decreasing functions $\bm{\beta},\bm{\mu}:[0,T]\to\mathbb{R}$ with $\bm{\beta}(0)=\beta_0$, $\bm{\beta}(T)> 0$, $\bm{\mu}(0)=\mu_0$.
Consider also  initial data $G_{0}=\left(g_{s,0}\right)\in X_{\infty,\beta_0,\mu_0}$. A map $\bm{G}=\left(g_{s}\right)_{s\in\mathbb{N}}\in\bm{X}_{\infty,\bm{\beta},\bm{\mu}}$ is a mild solution of the Boltzmann hierarchy in $[0,T]$, with initial data $G_0$, if it satisfies:
\begin{equation*}\bm{G}(t)=\mathcal{S}^tG_{0}+\int_0^t \mathcal{S}^{t-\tau}\mathcal{C}_\infty\bm{G}(\tau)\,d\tau,\end{equation*}
where, given $\beta>0$, $\mu\in\mathbb{R}$ and $\widetilde{G}=(\widetilde{g}_{s})_{s\in\mathbb{N}}\in X_{\infty,\beta,\mu}$, we write
\begin{align*}
\mathcal{C}_\infty G:=(\mathcal{C}_\infty^2 +\mathcal{C}_\infty^3) G,\quad \mathcal{C}_\infty^2 G:=\left(\mathcal{C}_{s,s+1}^\infty g_{s+1}\right)_{s\in\mathbb{N}},\quad \mathcal{C}_\infty^3 G:=\left(\mathcal{C}_{s,s+2}^\infty g_{s+2}\right)_{s\in\mathbb{N}},
\end{align*}
and $\mathcal{S}^t$ is given by \eqref{S definition}.
\end{definition}
Using Lemma \ref{a priori lemma for C Boltzmann}, we obtain the following a-priori bounds:
\begin{lemma}\label{a priori lemma for S boltzmann} Let $\beta_0> 0$, $\mu_0\in\mathbb{R}$, $T>0$ and $\lambda\in (0,\beta_0/T)$. Consider the functions $\bm{\beta}_\lambda,\bm{\mu}_\lambda:[0,T]\to\mathbb{R}$ given by
\eqref{beta_lambda-mu_lambda}.
 Then  for any $\mathcal{F}(t)\subseteq [0,t]$ measurable,  $\bm{G}=\left(g_{s}\right)_{s\in\mathbb{N}}\in\bm{X}_{\infty,\bm{\beta}_\lambda,\bm{\mu}_\lambda}$ and $k\in\{1,2\}$, the following bound holds:
 \begin{align}
\left|\left|\left|\displaystyle\int_{\mathcal{F}(t)}\mathcal{S}^{t-\tau}\mathcal{C}_\infty^{k+1}\bm{G}(\tau)\,d\tau\right|\right|\right|_{\infty,\bm{\beta}_\lambda,\bm{\mu}_\lambda}&\leq C_{k+1}|||\bm{G}|||_{\infty,\bm{\beta}_\lambda,\bm{\mu}_\lambda},\label{both boltzmann with constant}
\end{align}
 where the constant $C_{k+1}=C_{k+1}(d,\beta_0,\mu_0,T,\lambda)$ is given by \eqref{constant of WP binary}.
\end{lemma}

  Choosing $\lambda=\beta_0/2T$, Lemma \ref{a priori lemma for S boltzmann} directly implies well-posedness of the Boltzmann hierarchy up to short time.
 \begin{theorem}\label{well posedness boltzmann}
 Let $\beta_0> 0$ and $\mu_0\in\mathbb{R}$. Then there is $T=T(d,\beta_0,\mu_0)>0$ such that for any initial datum $F_{0}=(f_{0}^{(s)})_{s\in\mathbb{N}}\in X_{\infty,\beta_0,\mu_0}$ there is unique mild solution $\bm{F}=(f^{(s)})_{s\in\mathbb{N}}\in\bm{X}_{\infty,\bm{\beta},\bm{\mu}}$ to the Boltzmann hierarchy in $[0,T]$ for the functions $\bm{\beta},\bm{\mu}:[0,T]\to\mathbb{R}$ given by \eqref{beta mu given lambda}. The solution $\bm{F}$ satisfies the bound:
 \begin{equation}
 \label{a priori bound F_0 Boltzmann}|||\bm{F}|||_{\infty,\bm{\beta},\bm{\mu}}\leq 2\|F_{0}\|_{\infty,\beta_0,\mu_0}.
 \end{equation}
 Moreover, for any $\mathcal{F}(t)\subseteq[0,t]$ measurable and $k\in\{1,2\}$, the following bound holds:
 \begin{align}
 \label{a priori binary bound F Boltzmann}\left|\left|\left|\int_{\mathcal{F}(t)}\mathcal{S}^{t-\tau}C_\infty^{k+1}\bm{G}(\tau)\,d\tau\right|\right|\right|_{\infty,\bm{\beta},\bm{\mu}}&\leq\frac{1}{16}|||G|||_{{\infty,\bm{\beta},\bm{\mu}}},\quad\forall G\in\bm{X}_{\infty,\bm{\beta},\bm{\mu}},
 \end{align}
 and the time $T$ is explicitly given by \eqref{time}.
 \end{theorem}
 \subsection{LWP for the binary-ternary Boltzmann equation and propagation of chaos}
 Now, we show local well-posedness for the binary-ternary Boltzmann equation and that, for chaotic initial data, their tensorized product produces the unique mild solution of the Boltzmann hierarchy. Therefore uniqueness implies that the mild solution to the Boltzmann hierarchy remains factorized under time evolution, hence chaos is propagated in time.

For $\beta>0$ let us define the Banach space
\begin{equation*}
X_{\beta,\mu}:=\left\{g\in C^0(\mathbb{R}^{2d}):|g|_{\beta,\mu}<\infty\right\},
\end{equation*}
with norm
$
|g|_{\beta,\mu}=\sup_{(x,v)\in\mathbb{R}^{2d}} |g(x,v)|e^{\mu+\frac{\beta}{2} |v|^2}.
$
Notice that for any $t\in[0,T]$, the map $S_1^t:X_{\beta,\mu}\to X_{\beta,\mu}$ is an isometry.

Consider $\beta_0>0$, $\mu_0\in\mathbb{R}$, $T>0$ and $\bm{\beta},\bm{\mu}:[0,T]\to\mathbb{R}$ decreasing functions of time with $\bm{\beta}(0)=\beta_0$, $\bm{\beta}(T)>0$ and $\bm{\mu}(0)=\mu_0$.
We define the Banach space
\begin{equation*}
\bm{X}_{\bm{\beta},\bm{\mu}}:=L^\infty\left([0,T],X_{\bm{\beta}(t),\bm{\mu}(t)}\right),
\end{equation*} 
with norm
$
\|\bm{g}\|_{\bm{\beta},\bm{\mu}}=\sup_{t\in[0,T]}|\bm{g}(t)|_{\bm{\beta}(t),\bm{\mu}(t)}.
$
One can see that the following estimate holds:
\begin{remark}\label{remark for initial equation} Let $T>0$, $\beta_0>0$, $\mu_0\in\mathbb{R}$ and $\bm{\beta},\bm{\mu}:[0,T]\to\mathbb{R}$ decreasing functions with $\beta_0=\bm{\beta}(0)$, $\bm{\beta}(T)> 0$ $\mu_0=\bm{\mu}(0)$. Then for any $g\in X_{\beta_0,\mu_0}$, the following estimate holds:
\begin{equation*}
\|g\|_{\bm{\beta},\bm{\mu}}\leq |g|_{\beta_0,\mu_0}.
\end{equation*}
\end{remark}

To prove LWP for the binary-ternary Boltzmann equation \eqref{Boltzmann equation}, we will need certain continuity estimates on the binary and ternary collisional operators. The binary estimate we provide below is the bilinear analogue of Proposition 5.3.2.  in \cite{gallagher}.
For the ternary operator, continuity estimates  have been derived in \cite{thesis}, Lemma 6.10. Combining these results we  derive continuity estimates for the binary-ternary collisional operator $Q_2+Q_3$: 

\begin{lemma}\label{continuity boltzmann lemma} Let $\beta>0$, $\mu\in\mathbb{R}$. Then for any $g,h\in X_{\beta,\mu}$ and $(x,v)\in\mathbb{R}^{2d}$, the following nonlinear continuity estimate holds:
\begin{align*}
\big|&\left[Q_2(g,g)+Q_3(g,g,g)\right](x,v)-\left[Q_2(h,h)+Q_3(h,h,h)\right](x,v)\big|\\
&\lesssim \left(e^{-2\mu}\beta^{-d/2}+e^{-3\mu}\beta^{-d}\right)\left(\beta^{-1/2}+|v|\right)e^{-\frac{\beta}{2}|v|^2}\left(|g|_{\beta,\mu}+|h|_{\beta,\mu}\right)(1+|g|_{\beta,\mu}+|h|_{\beta,\mu})|g-h|_{\beta,\mu}.
\end{align*}
\end{lemma}

We define mild solutions to the binary-ternary Boltzmann equation \eqref{Boltzmann equation} as follows:
\begin{definition}
Consider $T>0$, $\beta_0> 0$, $\mu_0\in\mathbb{R}$ and  $\bm{\beta},\bm{\mu}:[0,T]\to\mathbb{R}$  decreasing functions of time, with $\bm{\beta}(0)=\beta_0$, $\bm{\beta}(T)> 0$, $\bm{\mu}(0)=\mu_0$.
Consider also initial data $g_{0}\in X_{\beta_0,\mu_0}$. A map $\bm{g}\in\bm{X}_{\bm{\beta},\bm{\mu}}$ is a mild solution to the binary-ternary Boltzmann equation \eqref{Boltzmann equation} in $[0,T]$, with initial data $g_0\in X_{\beta_0,\mu_0}$, if it satisfies
\begin{equation}\label{mild boltzmann equation}
\bm{g}(t)=S_1^tg_0+\int_0^tS_1^{t-\tau}\left[Q_2(\bm{g},\bm{g})+Q_3(\bm{g},\bm{g},\bm{g})\right](\tau)\,d\tau.
\end{equation} 
where $S_1^t$ denotes the free flow of  one particle given in \eqref{free flow operator}.
\end{definition}

A similar proof to Lemma \ref{a priori lemma for T BBGKY} gives the following:
\begin{lemma}\label{Lemma for integral boltzmann}
Let $\beta_0> 0$, $\mu_0\in\mathbb{R}$, $T>0$ and $\lambda\in (0,\beta_0/T)$. Consider the functions $\bm{\beta}_\lambda,\bm{\mu}_\lambda:[0,T]\to\mathbb{R}$ given by \eqref{beta_lambda-mu_lambda}.  
 Then for any $\bm{g},\bm{h}\in\bm{X}_{\bm{\beta}_\lambda,\bm{\mu}_\lambda}$ the following bounds hold:
 \begin{equation*}
\begin{aligned} 
 &\left\|\int_0^tS_1^{t-\tau}\left[Q_2(\bm{g}-\bm{h},\bm{g}-\bm{h})+Q_3(\bm{g}-\bm{h},\bm{g}-\bm{h},\bm{g}-\bm{h})\right](\tau)\,d\tau\right\|_{\bm{\beta}_\lambda,\bm{\mu}_\lambda} \\
 &\hspace{1cm}\leq C\left(|\bm{g}|_{\bm{\beta}_\lambda,\bm{\mu}_\lambda}+|\bm{h}|_{\bm{\beta}_\lambda,\bm{\mu}_\lambda}\right)\left(1+|\bm{g}|_{\bm{\beta}_\lambda,\bm{\mu}_\lambda}+|\bm{h}|_{\bm{\beta}_\lambda,\bm{\mu}_\lambda}\right)
 |\bm{g}-\bm{h}|_{\bm{\beta}_\lambda,\bm{\mu}_\lambda},
\end{aligned} 
 \end{equation*}
  where
$C=C(d,\beta_0,\mu_0,T,\lambda)=C_2+C_3$ and $C_2,C_3$ are given by  \eqref{constant of WP binary} for $k=1,2$ respectively.
\end{lemma}

Choosing $\lambda=\beta_0/2T$, this estimate implies local well-posedness of the binary-ternary Boltzmann equation up to short times. 
Let us write $B_{\bm{X}_{\bm{\beta},\bm{\mu}}}$ for the unit ball of $\bm{X}_{\bm{\beta},\bm{\mu}}$.

\begin{theorem}[LWP for the binary-ternary Boltzmann equation]\label{lwp boltz eq}
 Let $\beta_0> 0$ and $\mu_0\in\mathbb{R}$. Then there is $T=T(d,\beta_0,\mu_0)>0$ such that for any initial data $f_0\in X_{\beta_0,\mu_0}$, with $|f_0|_{\beta_0,\mu_0}\leq 1/2$, there is a unique mild solution $\bm{f}\in B_{\bm{X}_{\bm{\beta},\bm{\mu}}}$ to the binary-ternary Boltzmann equation in $[0,T]$ with initial data $f_0$, where $\bm{\beta},\bm{\mu}:[0,T]\to\mathbb{R}$ are the functions given by \eqref{beta mu given lambda}. The solution $\bm{f}$ satisfies the bound:
  \begin{equation}\label{bound on initial data boltzmann equ}
 \|\bm{f}\|_{\bm{\beta},\bm{\mu}}\leq 4|f_0|_{\beta_0,\mu_0}.
 \end{equation}
 Moreover, for any $\bm{g,h}\in\bm{X}_{\bm{\beta},\bm{\mu}}$, the following estimates hold:
 \begin{align}
 &\left\|\int_0^tS_1^{t-\tau}\left[Q_2(\bm{g}-\bm{h},\bm{g}-\bm{h})+Q_3(\bm{g}-\bm{h},\bm{g}-\bm{h},\bm{g}-\bm{h})\right](\tau)\,d\tau\right\|_{\bm{\beta},\bm{\mu}}\nonumber \\
 &\hspace{1cm}\leq \frac{1}{8}\left(\|\bm{g}\|_{\bm{\beta},\bm{\mu}}+\|\bm{h}\|_{\bm{\beta},\bm{\mu}}\right)\left(1+|\bm{g}|_{\bm{\beta},\bm{\mu}}+|\bm{h}|_{\bm{\beta},\bm{\mu}}\right)
 \|\bm{g}-\bm{h}\|_{\bm{\beta},\bm{\mu}}.\label{a-priori BE 1}
 \end{align}
The time $T$ is explicitly given by \eqref{time}.
\end{theorem}
\begin{proof}
Choosing $T$ as in \eqref{time},
we obtain $C(d,\beta_0,\mu_0,T,\beta_0/2T)=1/8.$
Thus,  Lemma \ref{Lemma for integral boltzmann} implies estimate \eqref{a-priori BE 1}.
Therefore, for any $g\in B_{\bm{X}_{\bm{\beta},\bm{\mu}}}$, using \eqref{a-priori BE 1} for $\bm{h}=0$, we obtain
\begin{equation}\label{inequality for cubic boltzmann 1}
 \left\|\int_0^tS_1^{t-\tau}\left[Q_2(\bm{g},\bm{g})+Q_3(\bm{g},\bm{g},\bm{g})\right](\tau)\,d\tau\right\|_{\bm{\beta}_\lambda,\bm{\mu}_\lambda}\leq \frac{1}{8}(1+\|\bm{g}\|_{\bm{\beta},\bm{\mu}})\|\bm{g}\|^2_{\bm{\beta},\bm{\mu}}\leq \frac{1}{4}\|\bm{g}\|_{\bm{\beta},\bm{\mu}}.
\end{equation}
Let us define the nonlinear operator $\mathcal{L}:\bm{X}_{\bm{\beta},\bm{\mu}}\to\bm{X}_{\bm{\beta},\bm{\mu}}$ by
\begin{equation*}
\mathcal{L}\bm{g}(t)=S_1^tf_0+\int_0^tS_1^{t-\tau}Q(\bm{g},\bm{g},\bm{g})(\tau)\,d\tau.
\end{equation*}
By triangle inequality, the fact that the free flow is isometric, Remark \ref{remark for initial equation}, bound \eqref{inequality for cubic boltzmann 1} and the assumption $|f_0|_{\beta_0,\mu_0}\leq  1/2$, for any $\bm{g}\in B_{\bm{X}_{\bm{\beta},\bm{\mu}}}$ and $t\in[0,T]$, we have
\begin{equation*}
\begin{aligned}
|\mathcal{L}\bm{g}|_{\bm{\beta}(t),\bm{\mu}(t)}\leq |S_1^tf_0|_{\bm{\beta}(t),\bm{\mu}(t)}+\frac{1}{4}\|\bm{g}\|_{\bm{\beta},\bm{\mu}}=|f_0|_{\bm{\beta}(t),\bm{\mu}(t)}+\frac{1}{4}\|\bm{g}\|_{\bm{\beta},\bm{\mu}}\leq |f_0|_{\beta_0,\mu_0}+\frac{1}{4}\|\bm{g}\|_{\bm{\beta},\bm{\mu}}\leq \frac{1}{2}+\frac{1}{4}=\frac{3}{4}.
\end{aligned}
\end{equation*}
 thus $\mathcal{L}:B_{\bm{X}_{\bm{\beta},\bm{\mu}}}\to B_{\bm{X}_{\bm{\beta},\bm{\mu}}}.$
Moreover, for any $\bm{g},\bm{h}\in B_{\bm{X}_{\bm{\beta},\bm{\mu}}}$, using \eqref{a-priori BE 1}, we obtain
\begin{equation}\label{pre-triang-bolt}
\begin{aligned}
\left\|\mathcal{L}\bm{g}-\mathcal{L}\bm{h}\right\|_{\bm{\beta},\bm{\mu}}
&\leq \frac{1}{8}\left(\|\bm{g}\|_{\bm{\beta},\bm{\mu}}+\|\bm{h}\|_{\bm{\beta},\bm{\mu}}\right)\left(1+\|\bm{g}\|_{\bm{\beta},\bm{\mu}}+\|\bm{h}\|_{\bm{\beta},\bm{\mu}}\right)\|\bm{g}-\bm{h}\|_{\bm{\beta},\bm{\mu}}\leq\frac{3}{4}\|\bm{g}-\bm{h}\|_{\bm{\beta},\bm{\mu}}.
\end{aligned}
\end{equation}
Therefore, the operator $\mathcal{L}:B_{\bm{X}_{\bm{\beta},\bm{\mu}}}\to B_{\bm{X}_{\bm{\beta},\bm{\mu}}}$ is a contraction, so it has a unique fixed point $\bm{f}\in B_{\bm{X}_{\bm{\beta},\bm{\mu}}}$ which is clearly the unique mild solution of the binary-ternary Boltzmann equation in $[0,T]$ with initial data $f_0$.

To  prove \eqref{bound on initial data boltzmann equ}, we use the fact that $\bm{f}=\mathcal{L}\bm{f}$. Then for any $t\in[0,T]$, triangle inequality, definition of $\mathcal{L}$, estimate \eqref{pre-triang-bolt}(for $\bm{g}=\bm{f}$ and $\bm{g}=0$), free flow being isometric,  and Remark \ref{remark for initial equation} yield
\begin{align*}
|\bm{f}|_{\bm{\beta}(t),\bm{\mu}(t)}=|\mathcal{L}\bm{f}|_{\bm{\beta}(t),\bm{\mu}(t)}&\leq |\mathcal{L}\bm{0}|_{\bm{\beta}(t),\bm{\mu}(t)}+|\mathcal{L}\bm{f}-\mathcal{L}\bm{0}|_{\bm{\beta}(t),\bm{\mu}(t)} \leq|S_1^tf_0|_{\bm{\beta}(t),\bm{\mu}(t)}+\frac{3}{4}\|\bm{f}\|_{\bm{\beta},\bm{\mu}}\\
&=|f_0|_{\bm{\beta}(t),\bm{\mu}(t)}+\frac{3}{4}\|\bm{f}\|_{\bm{\beta},\bm{\mu}}
\leq |f_0|_{\beta_0,\mu_0}+\frac{3}{4}\|\bm{f}\|_{\bm{\beta},\bm{\mu}},
\end{align*}
thus $\|\bm{f}\|_{\bm{\beta},\bm{\mu}}\leq |f_0|_{\beta_0,\mu_0}+\displaystyle\frac{3}{4}\|\bm{f}\|_{\bm{\beta},\bm{\mu}},$ and \eqref{bound on initial data boltzmann equ} follows.
\end{proof}

We can now prove that chaos is propagated by the Boltzmann hierarchy.

\begin{theorem}[Propagation of chaos]\label{theorem propagation of chaos}
Let $\beta_0>0$, $\mu_0\in\mathbb{R}$, $T>0$ be the time given in \eqref{time}, and $\bm{\beta},\bm{\mu}:[0,T]\to\mathbb{R}$ the functions defined by \eqref{beta mu given lambda}. Consider $f_0\in X_{\beta_0,\mu_0}$ with
$|f_0|_{\beta_0,\mu_0}\leq 1/2$.
Assume $\bm{f}\in B_{\bm{X}_{\bm{\beta},\bm{\mu}}}$ is the corresponding mild solution of the binary-ternary Boltzmann equation in $[0,T]$, with initial data $f_0$ given by Theorem \ref{lwp boltz eq}. Then the following hold:
\begin{enumerate}[(i)]
\item $F_0=(f_0^{\otimes s})_{s\in\mathbb{N}}\in X_{\infty,\beta_0,\mu_0}$.
\item $\bm{F}=(\bm{f}^{\otimes s})_{s\in\mathbb{N}}\in\bm{X}_{\infty,\bm{\beta},\bm{\mu}}$.
\item $\bm{F}$ is the unique mild solution of the Boltzmann hierarchy in $[0,T]$, with initial data $F_0$.
\end{enumerate}

\end{theorem}
\begin{proof}
\textit{(i)} is trivially verified by the bound on the initial data \eqref{bound on initial data boltzmann equ} and the definition of the norms. By the same bound again, we may apply Theorem \ref{lwp boltz eq} to obtain the unique mild solution $\bm{f}\in B_{\bm{X}_{\bm{\beta},\bm{\mu}}}$ of the corresponding binary-ternary Boltzmann equation. Since $\|\bm{f}\|_{\bm{\beta},\bm{\mu}}\leq 1$, the definition of the norms directly imply \textit{(ii)}. It is also staightforward to verify that $\bm{F}$ is a mild solution of the Boltzmann hierarchy in $[0,T]$, with initial data $F_0$. Uniqueness of the mild solution to the Boltzmann hierarchy, obtained by Theorem \ref{well posedness boltzmann}, implies that $\bm{F}$ is  the unique mild solution.
\end{proof}

 \section{Convergence Statement}\label{sec:convergence}
\label{sec_conv statement}

In this section we define an appropriate notion of convergence, namely convergence in observables, and we state the main result of this paper.  

\subsection{Approximation of Boltzmann hierarchy initial data}\label{subseq:approximation}
Here, we approximate Boltzmann hierarchy initial data by BBGKY hierarchy initial data. Let us first introduce some notation we are using from now on.

Given $\theta>0$, we introduce the set of well-separated spatial configurations as follows:\\
 For $m\\in\mathbb{N}$, we define
\begin{equation}\label{separated space data}
\Delta_m^X(\theta):=\left\{\widetilde{X}_m\in\mathbb{R}^{dm}:|\widetilde{x}_i-\widetilde{x}_j|>\theta,\quad\forall 1\leq i<j\leq m\right\},\quad m\geq 2,\quad\Delta_1^X(\theta):=\mathbb{R}^{d}.
\end{equation}
For $m\in\mathbb{N}$, we also define the set of well-separated configurations as:
\begin{equation}\label{separated data}
\Delta_m(\theta):=\Delta_m^X(\theta)\times\mathbb{R}^{dm}=\left\{(\widetilde{X}_m,\widetilde{V}_m)\in\mathbb{R}^{2dm}:|\widetilde{x}_i-\widetilde{x}_j|>\theta,\quad\forall 1\leq i<j\leq m\right\}.
\end{equation}

 Recall we consider $(N,\epsilon_2,\epsilon_3)$ in the scaling
\begin{equation}\label{scaling admisib} N\epsilon_2^{d-1}\simeq N\epsilon_3^{d-\frac{1}{2}}\simeq 1.\end{equation}
Let us write $\epsilon_{2,N}$, $\epsilon_{3,N}$ for the $\epsilon_2,\epsilon_3$ associated to $N$ under \eqref{scaling admisib}. By Remark \ref{remark for epsilons}, for $N$ large enough, we have
$
0<\epsilon_{2,N}<<\epsilon_{3,N}\overset{N\to\infty}\longrightarrow 0.
$

We define the following approximating sequence:
\begin{definition}Let $s\in\mathbb{N}$, $\beta>0$, $\mu\in\mathbb{R}$ and $G=(g_s)_{s\in\mathbb{N}}\in X_{\infty,\beta,\mu}$. We define 
\begin{equation}\label{approximating sequence}
G_N=(g_{N,s})_{s\in\mathbb{N}},\quad\text{where}\quad g_{N,s}=\mathds{1}_{\Delta_s(\epsilon_{3,N})} g_s.
\end{equation}
The sequence $(G_N)_{N\in\mathbb{N}}$ is called approximating BBGKY  hierarchy sequence of $G$. 
\end{definition}

Similarly to  Proposition 7.2. from \cite{thesis}, one obtains the following approximation property:
\begin{proposition}\label{approximation proposition}
Let $s\in\mathbb{N}$, $\beta>0$, $\mu\in\mathbb{R}$, $G=(g_s)_{s\in\mathbb{N}}\in X_{\infty,\beta,\mu}$ and $(G_N)_{N\in\mathbb{N}}$ the approximating BBGKY hierarchy sequence of $G$. Then the following hold:
\begin{enumerate}[(i)]
\item $G_N\in X_{N,\beta,\mu}$ for all $N\in\mathbb{N} $. In particular,
\begin{equation}\label{uniform bound on approximating sequene}
\sup_{N\in\mathbb{N}}\|G_N\|_{N,\beta,\mu}\leq \|G\|_{\infty,\beta,\mu}
\end{equation}\vspace{0.2cm}
 \item For any $s\in\mathbb{N}$ and $\theta>0$, we have
\begin{equation}\label{initial convergence to 0}
\lim_{N\to\infty}\|g_{N,s}-g_s\|_{L^\infty\left(\Delta_s\left(\theta\right)\right)}= 0.
\end{equation}
\end{enumerate}
\end{proposition}
\subsection{Convergence in observables}\label{subsec:con in observables}
Here, we define the convergence in observables. Let us first introduce some notation. Given $s\in\mathbb{N}$, we define the space of test functions
\begin{equation}\label{test functions}
C_c(\mathbb{R}^{ds})=\left\{\phi_s:\mathbb{R}^{ds}\to\mathbb{R}:\phi_s\text{ is continuous and compactly supported}\right\}.
\end{equation} 

\begin{definition} Consider $T>0$, $s\in\mathbb{N}$ and $g_s\in L^\infty\left([0,T],L^\infty\left(\mathbb{R}^{2ds}\right)\right)$. Given a test function $\phi_s\in C_c(\mathbb{R}^{ds})$, we define the $s$-observable functional as
\begin{equation*}
I_{\phi_s}g_s(t)(X_s)=\int_{\mathbb{R}^{ds}}\phi_s(V_s)g_s(t,X_s,V_s)\,dV_s.
\end{equation*}
\end{definition}

Recalling the set of initially good spatial configurations $\Delta_s^X(\theta)$ from \eqref{separated space data}, we give the definition of the convergence in observables:
\begin{definition}
Let $T>0$. For each $N\in\mathbb{N}$, consider $\bm{G_N}=(g_{N,s})_{s\in\mathbb{N}}\in \prod_{s=1}^\infty L^\infty\left([0,T],L^\infty\left(\mathbb{R}^{2ds}\right)\right)$  and $\bm{G}=(g_s)_{s\in\mathbb{N}}\in \prod_{s=1}^\infty L^\infty\left([0,T],L^\infty\left(\mathbb{R}^{2ds}\right)\right)$. We say that the sequence $(\bm{G_N})_{N\in\mathbb{N}}$ converges in observables to $\bm{G}$ if for any $s\in\mathbb{N}$, $\theta>0$ and $\phi_s\in C_c(\mathbb{R}^{ds})$ , we have
\begin{equation*}
\lim_{N\to\infty}\|I_{\phi_s}g_{N,s}(t)-I_{\phi_s}g_s(t)\|_{L^\infty\left(\Delta_s^X\left(\theta\right)\right)}=0,\quad\text{uniformly in }[0,T].
\end{equation*}

\end{definition}
\subsection{Statement of the main result}
We are now in the position to state our main result. The rest of the paper will be devoted to its proof.
\begin{theorem}[Convergence]\label{convergence theorem}
Let $\beta_0> 0$, $\mu_0\in\mathbb{R}$ and $T=T(d,\beta_0,\mu_0)>0$ given by \eqref{time}. Consider some initial Boltzmann hierarchy data $F_0=(f_0^{(s)})_{s\in\mathbb{N}}\in X_{\infty,\beta_0,\mu_0}$ with approximating BBGKY hierarchy sequence $\left(F_{N,0}\right)_{N\in\mathbb{N}}$. Assume that
\begin{itemize}
\item for each $N$, $\bm{F_N}\in\bm{X}_{N,\bm{\beta},\bm{\mu}}$ is the mild solution (given by Theorem \ref{well posedness BBGKY}) of the BBGKY hierarchy in $[0,T]$ with initial data $F_{N,0}$.

\item $\bm{F}\in\bm{X}_{\infty,\bm{\beta},\bm{\mu}}$ is the mild solution (given by Theorem \ref{well posedness boltzmann}) of the Boltzmann hierarchy in $[0,T]$ with initial data $F_0$.

\item $F_0$ satisfies the following uniform continuity growth condition: There is a constant $C>0$ such that, for any $\zeta>0$, there is $q=q(\zeta)>0$ such that for all $s\in\mathbb{N}$, and for all $Z_s,Z_s'\in\mathbb{R}^{2ds}$ with $|Z_s-Z_s'|<q$, we have
\begin{equation}\label{continuity assumption}
|f_0^{(s)}(Z_s)-f_0^{(s)}(Z_s')|<C^{s-1}\zeta.
\end{equation}
\end{itemize}
Then, $\bm{F_N}$ converges in observables to $\bm{F}$.
\end{theorem}
\begin{remark} Using the definition of convergence, proving Theorem \ref{convergence theorem} is equivalent to proving that for any $s\in\mathbb{N}$, $\phi_s\in C_c(\mathbb{R}^{ds})$ and $\theta>0$ we have
\begin{equation*}
\lim_{N\to\infty}\|I_s^N(t)-I_s^\infty(t)\|_{L^\infty\left(\Delta_s^X\left(\theta\right)\right)}=0,\quad\text{uniformly in }[0,T],
\end{equation*}
 where 
\begin{equation}\label{def-I-Ns}
I_s^N(t)(X_s):=I_{\phi_s}f_N^{(s)}(t)(X_s)=\int_{\mathbb{R}^{ds}}\phi_s(V_s)f_N^{(s)}(t,X_s,V_s)\,dV_s,  
\end{equation}
\begin{equation}\label{def-I-s}
I_s^\infty(t)(X_s):=I_{\phi_s}f^{(s)}(t)(X_s)=\int_{\mathbb{R}^{ds}}\phi_s(V_s)f^{(s)}(t,X_s,V_s)\,dV_s.  
\end{equation}
\end{remark}

We also obtain the following Corollary\footnote{which can be proved in a similar way as in Corollary 7.5. from \cite{thesis}} of Theorem \ref{convergence theorem}.
\begin{corollary} Let $\beta_0>0$, $\mu_0\in\mathbb{R}$ and $f_0\in X_{\beta_0,\mu_0}$, with $|f_0|_{\beta_0,\mu_0}\leq 1/2$. Assume as well that $f_0$ is uniformly continuous. Then for any $s\in\mathbb{N}$, $\phi_s\in C_c(\mathbb{R}^{ds})$ and $\theta>0$, the following convergence holds:
\begin{equation}\label{derivation}
\lim_{N\to\infty}\|I_{\phi_s}f^{\otimes s}\mathds{1}_{\Delta_s(\epsilon_{3,N})}-I_{\phi_s}f^{\otimes s}\|_{L^\infty(\Delta_s(\theta))}=0,
\end{equation}
where $f$ is the mild solution to the binary-ternary Boltzmann equation in $[0,T]$, with initial data $f_0$, given by Theorem \ref{lwp boltz eq} and $T$ is given by \eqref{time}. 
\end{corollary}

In order to prove Theorem \ref{convergence theorem}, we will first use the local estimates developed in Section \ref{sec:local} to reduce the proof to finitely many observables of bounded energy, which are also well separated in time. 
Then, we will develop some geometric estimates which will enable us to eliminate recollisions of the backwards $(\epsilon_2,\epsilon_3)$-flow.

\section{Reduction to term by term convergence}\label{sec: series expansion}
In this section we reduce the proof of Theorem \ref{convergence theorem} to term by term convergence after truncating the observables. After introducing the necessary combinatorial notation to take care of all the possible collision sequences occurring, the idea of the truncation is essentially the same as in \cite{gallagher, thesis}, and it relies on the local estimates developed in Section \ref{sec:local}. For this reason, we illustrate the similarities by providing the proof of the first estimate and omit the proofs of the rest of the estimates.

Throughout this section,  we consider $\beta_0>0$, $\mu_0\in\mathbb{R}$, the functions $\bm{\beta},\bm{\mu}:[0,T]\to\mathbb{R}$ defined by \eqref{beta mu given lambda}, $(N,\epsilon_2,\epsilon_3)$ in the scaling \eqref{scaling} and initial data $F_{N,0}\in X_{N,\beta_0,\mu_0}$, $F_0\in X_{\infty,\beta_0,\mu_0}$. Let $\bm{F_N}=(f_N^{(s)})_{s\in\mathbb{N}}\in\bm{X}_{N,\bm{\beta},\bm{\mu}}$, $\bm{F}=(f^{(s)})_{s\in\mathbb{N}}\in\bm{X}_{\infty,\bm{\beta},\bm{\mu}}$ be the mild solutions of the corresponding BBGKY and Boltzmann hierarchies, respectively, in $[0,T]$, given by Theorems \ref{well posedness BBGKY} and Theorem \ref{well posedness boltzmann}. Let us note that by \eqref{beta mu given lambda}, we obtain
\begin{equation}\label{non dependence}
\bm{\beta}(T)=\frac{\beta_0}{2},\quad\bm{\mu}(T)=\mu_0-\frac{\beta_0}{2},
\end{equation}
thus $\bm{\beta}(T),\bm{\mu}(T)$ do not depend on $T$.

For convenience, we introduce the following notation. Given $k\in\mathbb{N}$ and $t\geq 0$, we denote
\begin{equation}\label{collision times}
\mathcal{T}_k(t):=\left\{(t_1,...,t_k)\in\mathbb{R}^k:0\leq t_k<...\leq t_1\leq t\right\}.
\end{equation}
Since the collisions happening can be either binary or ternary we will introduce some additional notation to keep track of the collision sequences. In particular,
given $k\geq 1$, we denote
\begin{equation}\label{S_k}
S_k:=\left\{\sigma=(\sigma_1,...,\sigma_k):\sigma_i\in\left\{1,2\right\},\quad\forall i=1,...,k\right\}.
\end{equation}
Notice that the cardinality of $S_k$ is given by:
\begin{equation}\label{cardinality of S_k}
|S_k|=2^k,\quad\forall k\geq 1.
\end{equation}
Given $k\in\mathbb{N}$ and $\sigma\in S_k$, for any $1\leq \ell\leq k$ we write
\begin{equation}\label{sigma tilde}
\widetilde{\sigma}_\ell=\sum_{i=1}^\ell\sigma_i.
\end{equation}
We also write
$
\widetilde{\sigma}_0:=0.
$
Notice that
\begin{equation}\label{bound on sigma}
k\leq\widetilde{\sigma}_k\leq 2k,\quad\forall k\in\mathbb{N}.
\end{equation}
\subsection{Series expansion}
Now, we make a series expansion for the mild solution $\bm{F_N}=(f_N^{(s)})_{s\in\mathbb{N}}$ of the BBGKY hierarchy with respect to the initial data $F_{N,0}$. By Definition \ref{def of mild bbgky}, for any $\in\mathbb{N}$, we have  Duhamel's formula:
\begin{equation*}
f^{(s)}_N(t)=T_s^tf^{(s)}_{N,0}+\int_0^tT_s^{t-t_1}\left[\mathcal{C}_{s,s+1}^{N}f^{(s+1)}_N+\mathcal{C}_{s,s+2}^{N}f^{(s+2)}_N\right](t_1)\,dt_1.
\end{equation*}
Let $n\in\mathbb{N}$. Iterating $n$-times  Duhamel's formula,  we obtain
\begin{equation}\label{function plus remainder BBGKY}
f^{(s)}_N(t)=\sum_{k=0}^nf^{(s,k)}_N(t)+R_{N}^{(s,n+1)}(t),
\end{equation}
where we use the notation:
\begin{equation}\label{function expansion BBGKY}
\begin{aligned}
  f^{(s,k)}_N(t)&:=\sum_{\sigma\in S_k}f^{(s,k,\sigma)}_N(t),\text{ for } 1\leq k\leq n,\quad f^{(s,0)}_N(t):=T_s^tf^{(s)}_{N,0}.
\end{aligned}
\end{equation}

\begin{equation}\label{function expansion with indeces BBGKY}
\begin{aligned}
 f^{(s,k,\sigma)}_N(t)=\int_{\mathcal{T}_k(t)}T_s^{t-t_1}\mathcal{C}_{s,s+\widetilde{\sigma}_1}^{N}T_{s+\widetilde{\sigma}_1}^{t_1-t_2}\mathcal{C}_{s+\widetilde{\sigma}_1,s+\widetilde{\sigma}_2}^{N}T_{s+\widetilde{\sigma}_2}^{t_2-t_3}...
T_{s+\widetilde{\sigma}_{k-1}}^{t_{k-1}-t_k}\mathcal{C}_{s+\widetilde{\sigma}_{k-1},s+\widetilde{\sigma}_k}^{N}T_{s+\widetilde{\sigma}_k}^{t_k}f_{N,0}^{(s+\widetilde{\sigma}_k)}\,dt_k...\,dt_1,
\end{aligned}
\end{equation}
\begin{equation}\label{remainder BBGKY}
\begin{aligned}
R_N^{(s,n+1)}(t):=\sum_{\sigma\in S_{n+1}}R_N^{(s,n+1,\sigma)}(t),
\end{aligned}
\end{equation}

\begin{equation}\label{remainder BBGKY with indeces}
\begin{aligned}
R_N^{(s,n+1,\sigma)}(t):=\int_{\mathcal{T}_{n+1}(t)}T_s^{t-t_1}\mathcal{C}_{s,s+\widetilde{\sigma}_1}^NT_{s+\widetilde{\sigma}_1}^{t_1-t_2}\mathcal{C}_{s+\widetilde{\sigma}_1,s+\widetilde{\sigma}_2}^NT_{s+\widetilde{\sigma}_2}^{t_2-t_3}...&\\
T_{s+\widetilde{\sigma}_{n-1}}^{t_{n-1}-t_n}\mathcal{C}_{s+\widetilde{\sigma}_{n-1},s+\widetilde{\sigma}_n}^NT_{s+\widetilde{\sigma}_n}^{t_n-t_{n+1}}\mathcal{C}_{s+\widetilde{\sigma}_n,s+\widetilde{\sigma}_{n+1}}^Nf^{(s+\widetilde{\sigma}_{n+1})}_N(t_{n+1})\,dt_{n+1}\,dt_n...\,dt_1.
&
\end{aligned}
\end{equation}

One can make a similar series expansion for the Boltzmann hierarchy. By Definition \ref{def of mild bbgky}, for any $\in\mathbb{N}$, we have  Duhamel's formula:
\begin{equation*}
f^{(s)}(t)=S_s^tf^{(s)}_{0}+\int_0^tS_s^{t-t_1}\left[\mathcal{C}_{s,s+1}^{\infty}f^{(s+1)}+\mathcal{C}_{s,s+2}^{\infty}f^{(s+2)}\right](t_1)\,dt_1.
\end{equation*}
Iterating $n$-times Duhamel's formula, we obtain
\begin{equation}\label{function plus remainder Boltzmann}
f^{(s)}(t)=\sum_{k=0}^nf^{(s,k)}(t)+R^{(s,n+1)}(t),
\end{equation}
where we use the notation:
\begin{equation}\label{function expansion Boltzmann}
  f^{(s,k)}(t):=\sum_{\sigma\in S_k}f^{(s,k,\sigma)}(t),\text{ for } 1\leq k\leq n,\quad f^{(s,0)}(t):=S_s^tf^{(s)}_{0}.
\end{equation}
\begin{equation}\label{function expansion Boltzmann with indeces}
\begin{aligned}
 f^{(s,k,\sigma)}(t):=\int_{\mathcal{T}_k(t)}S_s^{t-t_1}\mathcal{C}_{s,s+\widetilde{\sigma}_1}^{\infty}S_{s+\widetilde{\sigma}_1}^{t_1-t_2}\mathcal{C}_{s+\widetilde{\sigma}_1,s+\widetilde{\sigma}_2}^{\infty}S_{s+\widetilde{\sigma}_2}^{t_2-t_3}...
S_{s+\widetilde{\sigma}_{k-1}}^{t_{k-1}-t_k}\mathcal{C}_{s+\widetilde{\sigma}_{k-1},s+\widetilde{\sigma}_k}^{\infty}S_{s+\widetilde{\sigma}_k}^{t_k}f_{0}^{(s+\widetilde{\sigma}_k)}\,dt_k...\,dt_1,
\end{aligned}
\end{equation}
\begin{equation}\label{remainder boltzmann}
\begin{aligned}
R^{(s,n+1)}(t):=\sum_{\sigma\in S_{n+1}}R^{(s,n+1,\sigma)}(t),
\end{aligned}
\end{equation}
\begin{equation}\label{remainder boltzmann with indeces}
\begin{aligned}
R^{(s,n+1,\sigma)}(t):=\int_{\mathcal{T}_{n+1}(t)}S_s^{t-t_1}\mathcal{C}_{s,s+\widetilde{\sigma}_1}^\infty S_{s+\widetilde{\sigma}_1}^{t_1-t_2}\mathcal{C}_{s+\widetilde{\sigma}_1,s+\widetilde{\sigma}_2}^\infty S_{s+\widetilde{\sigma}_2}^{t_2-t_3}...&\\
S_{s+\widetilde{\sigma}_{n-1}}^{t_{n-1}-t_n}\mathcal{C}_{s+\widetilde{\sigma}_{n-1},s+\widetilde{\sigma}_n}^\infty S_{s+\widetilde{\sigma}_n}^{t_n-t_{n+1}}\mathcal{C}_{s+\widetilde{\sigma}_n,s+\widetilde{\sigma}_{n+1}}^\infty f^{(s+\widetilde{\sigma}_{n+1})}(t_{n+1})\,dt_{n+1}\,dt_n...\,dt_1.
&
\end{aligned}
\end{equation}

Given $\phi_s\in C_c(\mathbb{R}^{ds})$ and $k\in\mathbb{N}$, let us denote 
\begin{equation}\label{bbgky observ k}
I_{s,k}^N(t)(X_s):=
\int_{\mathbb{R}^{ds}}\phi_s(V_s)f_N^{(s,k)}(t,X_s,V_s)\,dV_s,
\end{equation}
\begin{equation}\label{boltz observ k}
I_{s,k}^\infty(t)(X_s):=
\int_{\mathbb{R}^{ds}}\phi_s(V_s)f^{(s,k)}(t,X_s,V_s)\,dV_s.
\end{equation}
We obtain the following estimates:
\begin{lemma}\label{term by term} For any $s,n\in\mathbb{N}$ and $t\in [0,T]$, the following estimates hold:
\begin{equation*}\|I_s^N(t)-\sum_{k=0}^nI_{s,k}^N(t)\|_{L^\infty_{X_s}}\leq C_{s,\beta_0,\mu_0} \|\phi_s\|_{L^\infty_{V_s}}4^{-n}\|F_{N,0}\|_{N,\beta_0,\mu_0},
\end{equation*}
\begin{equation*}\|I_s^\infty(t)-\sum_{k=0}^nI_{s,k}^\infty(t)\|_{L^\infty_{X_s}}\leq C_{s,\beta_0,\mu_0} \|\phi_s\|_{L^\infty_{V_s}}4^{-n}\|F_{0}\|_{\infty,\beta_0,\mu_0},
\end{equation*}
where the observables $I_s^N$, $I_s^\infty$ defined in \eqref{def-I-Ns}-\eqref{def-I-s}.
\end{lemma}
\begin{proof}
Fix $Z_s=(X_s,V_s)\in\mathbb{R}^{2ds}$, $t\in [0,T]$ and $\sigma\in S_{n+1}$. We repeatedly use  estimate \eqref{a priori binary bound  F_N} of Theorem \ref{well posedness BBGKY}, for $k=1$ if $\sigma_i=1$ or for $k=2$ if $\sigma_i=2$, to obtain
\begin{equation*}
e^{\bm{\beta}(t)E_s(Z_s)+s\bm{\mu}(t)}|R_N^{(s,n+1,\sigma)}(t,X_s,V_s)|\leq 8^{-(n+1)}|||\bm{F_N}|||_{N,\bm{\beta},\bm{\mu}},
\end{equation*}
so adding for all $\sigma\in S_{n+1}$, using \eqref{cardinality of S_k}, \eqref{a priori bound F_N,0} and the definition of the norms, we take
\begin{equation*}
\begin{aligned}
|\phi_s(V_s)R_N^{(s,n+1)}(t,X_s,V_s)|&\lesssim 4^{-(n+1)}e^{-s\bm{\mu}(t)}\|\phi_s\|_{L^\infty_{V_s}}|||\bm{F_N}|||_{N,\bm{\beta},\bm{\mu}}e^{-\bm{\beta}(t)E_s(Z_s)}\\
&\leq 4^{-n}e^{-s\bm{\mu}(T)}\|\phi_s\|_{L^\infty_{V_s}}\|F_{N,0}\|_{N,\beta_0,\mu_0}e^{-\bm{\beta}(T)E_s(Z_s)}.
\end{aligned}
\end{equation*}
Thus, integrating with respect to velocities and recalling \eqref{function plus remainder BBGKY}, \eqref{bbgky observ k}, \eqref{non dependence}, we obtain 
\begin{equation*}
\begin{aligned}
|I_s^N(t)(X_s)-\sum_{k=0}^nI_{s,k}^N(t)(X_s)|&\leq C_{s,\mu_0}\|\phi_s\|_{L^\infty_{V_s}}4^{-n}\|F_{N,0}\|_{N,\beta_0,\mu_0}\int_{\mathbb{R}^{ds}}e^{-\bm{\beta}(T)E_s(Z_s)}\,dV_s\\
&\leq C_{s,\beta_0,\mu_0} \|\phi_s\|_{L^\infty_{V_s}}4^{-n}\|F_{N,0}\|_{N,\beta_0,\mu_0}.
\end{aligned}
\end{equation*}
For the Boltzmann hierarchy, we follow a similar argument using estimates \eqref{a priori binary bound F Boltzmann} and  \eqref{a priori bound F_0 Boltzmann}  instead.
\end{proof}
\subsection{High energy truncation}
We will now truncate energies, so that we can focus on bounded energy domains. 
Let us fix $s,n\in\mathbb{N}$ and  $R>1$. As usual we denote $B_R^{2d}$ to be the $2d$-ball of radius $R$ centered at the origin. 

We first define the truncated BBGKY hierarchy and Boltzmann hierarchy  collisional operators. For $\ell\in\mathbb{N}$ we define
\begin{equation}\label{velocity truncation of operators}
\begin{aligned}
\mathcal{C}_{\ell,\ell+1}^{N,R}g_{l+1}&:=\mathcal{C}_{\ell,\ell+1}^N(g_{l+1}\mathds{1}_{[E_{\ell+1}\leq R^2]}),\quad \mathcal{C}_{\ell,\ell+2}^{N,R}g_{l+2}&:=\mathcal{C}_{\ell,\ell+2}^N(g_{l+2}\mathds{1}_{[E_{\ell+2}\leq R^2]}),\\
\mathcal{C}_{\ell,\ell+1}^{\infty,R} g_{l+1}&:=\mathcal{C}_{\ell,\ell+1}^\infty (g_{l+1}\mathds{1}_{[E_{\ell+1}\leq R^2]}),\quad\mathcal{C}_{\ell,\ell+2}^{\infty,R} g_{l+2}&:=\mathcal{C}_{\ell,\ell+2}^\infty (g_{l+2}\mathds{1}_{[E_{\ell+2}\leq R^2]}).
\end{aligned}
\end{equation}
For the BBGKY hierarchy we define
\begin{equation*}
f_{N,R}^{(s,k)}(t,Z_s):=\sum_{\sigma\in S_k}f_{N,R}^{(s,k,\sigma)}(t,Z_s),\text{ for }1\leq k\leq n,\quad f_{N,R}^{(s,0)}(t,Z_s):=T_s^t(f_{N,0}\mathds{1}_{[E_s\leq R^2]})(Z_s),
\end{equation*}
where given $k\geq 1$ and $\sigma\in S_k$, we denote
\begin{equation*}
\begin{aligned}
f_{N,R}^{(s,k,\sigma)}(t,Z_s&):=\int_{\mathcal{T}_k(t)}T_s^{t-t_1}\mathcal{C}_{s,s+\widetilde{\sigma}_1}^{N,R} T_{s+\widetilde{\sigma}_1}^{t_1-t_2}...\mathcal{C}_{s+\widetilde{\sigma}_{k-1},s+\widetilde{\sigma}_k}^{N,R} T_{s+\widetilde{\sigma}_k}^{t_k}f_{N,0}^{(s+\widetilde{\sigma}_k)}(Z_s)\,dt_k...\,dt_{1}.
\end{aligned}
\end{equation*}
For the Boltzmann hierarchy we define
\begin{equation*}
f_{R}^{(s,k)}(t,Z_s):=\sum_{\sigma\in S_k}f_{R}^{(s,k,\sigma)}(t,Z_s),\text{ for }1\leq k\leq n,\quad f_{R}^{(s,0)}(t,Z_s):=S_s^t(f_{0}\mathds{1}_{[E_s\leq R^2]})(Z_s),
\end{equation*}
where given $k\geq 1$ and $\sigma\in S_k$, we denote
\begin{equation*}
\begin{aligned}
f_{R}^{(s,k,\sigma)}(t,Z_s&):=\int_{\mathcal{T}_k(t)}S_s^{t-t_1}\mathcal{C}_{s,s+\widetilde{\sigma}_1}^{\infty,R} S_{s+\widetilde{\sigma}_1}^{t_1-t_2}...\mathcal{C}_{s+\widetilde{\sigma}_{k-1},s+\widetilde{\sigma}_k}^{\infty,R} S_{s+\widetilde{\sigma}_k}^{t_k}f_{0}^{(s+\widetilde{\sigma}_k)}(Z_s)\,dt_k...\,dt_{1}.
\end{aligned}
\end{equation*}
Given $\phi_s\in C_c(\mathbb{R}^{ds})$ and $k\in\mathbb{N}$, let us denote
\begin{equation}\label{BBGKY bounded energy}
I_{s,k,R}^N(t)(X_s):=
\int_{\mathbb{R}^{ds}}\phi_s(V_s)f_{N,R}^{(s,k)}(t,X_s,V_s)\,dV_s=\int_{B_R^{ds}}\phi_s(V_s)f_{N,R}^{(s,k)}(t,X_s,V_s)\,dV_s,
\end{equation}
\begin{equation}\label{Botlzmann bounded energy}
I_{s,k,R}^\infty(t)(X_s):=
\int_{\mathbb{R}^{ds}}\phi_s(V_s)f_R^{(s,k)}(t,X_s,V_s)\,dV_s=\int_{B_R^{ds}}\phi_s(V_s)f_{R}^{(s,k)}(t,X_s,V_s)\,dV_s.
\end{equation}
Recalling the observables $I_{s,k}^N$, $I_{s,k}^\infty$, defined in \eqref{bbgky observ k}-\eqref{boltz observ k}, 
we obtain the following estimates:
\begin{lemma}\label{energy truncation} For any $s,n\in\mathbb{N}$, $R>1$ and $t\in [0,T]$, the following estimates hold:
\begin{equation*}\sum_{k=0}^n\|I_{s,k,R}^N(t)-I_{s,k}^N(t)\|_{L^\infty_{X_s}}\leq C_{s,\beta_0,\mu_0,T} \|\phi_s\|_{L^\infty_{V_s}}e^{-\frac{\beta_0}{3}R^2}\|F_{N,0}\|_{N,\beta_0,\mu_0},
\end{equation*}
\begin{equation*}\sum_{k=0}^n\|I_{s,k,R}^\infty(t)-I_{s,k}^\infty(t)\|_{L^\infty_{X_s}}\leq C_{s,\beta_0,\mu_0,T} \|\phi_s\|_{L^\infty_{V_s}}e^{-\frac{\beta_0}{3}R^2}\|F_{0}\|_{\infty,\beta_0,\mu_0}.
\end{equation*}
\end{lemma}
\begin{proof}
For the proof, we use  the same ideas as in Lemma 8.4. from \cite{thesis}, and we also use \eqref{cardinality of S_k} to sum over all possible collision sequences.
\end{proof}

\subsection{Separation of collision times} 
We will now separate the time intervals we are integrating at, so that collisions occuring are separated in time. For this purpose consider a small time parameter $\delta>0$. 

For convenience, given $t\geq 0$ and $k\in\mathbb{N}$, we define  
\begin{equation}\label{separated collision times}
\mathcal{T}_{k,\delta}(t):=\left\{(t_1,...,t_k)\in\mathcal{T}_k(t):\quad 0\leq t_{i+1}\leq t_i-\delta,\quad\forall i\in [0,k]\right\},
\end{equation}
where we denote $t_{k+1}=0$, $t_0=t$.

For the BBGKY hierarchy, we define
\begin{equation*}
f_{N,R,\delta}^{(s,k)}(t,Z_s):=\sum_{\sigma\in S_k}f_{N,R,\delta}^{(s,k,\sigma)}(t,Z_s),\text{ for } 1\leq k\leq n,\quad f_{N,R,\delta}^{(s,0)}(t,Z_s):=T_s^t(f_{N,0}\mathds{1}_{[E_s\leq R^2]})(Z_s),
\end{equation*}
where, given $k\geq 1$ and $\sigma\in S_k$, we denote
\begin{equation*}
f_{N,R,\delta}^{(s,k,\sigma)}(t,Z_s):=\int_{\mathcal{T}_{k,\delta}(t)}T_s^{t-t_1}\mathcal{C}_{s,s+\widetilde{\sigma}_1}^{N,R} T_{s+\widetilde{\sigma}_1}^{t_1-t_2}
...\mathcal{C}_{s+\widetilde{\sigma}_{k-1},s+\widetilde{\sigma}_k}^{N,R} T_{s+\widetilde{\sigma}_k}^{t_k}f_{N,0}^{(s+\widetilde{\sigma}_k)}(Z_s)\,dt_k,...\,dt_{1}.
\end{equation*}
In the same spirit, for the Boltzmann hierarchy we define
\begin{equation*}
f_{N,R,\delta}^{(s,k)}(t,Z_s):=\sum_{\sigma\in S_k}f_{N,R,\delta}^{(s,k,\sigma)}(t,Z_s),\text{ for } 1\leq k\leq n,\quad f_{R,\delta}^{(s,0)}(t,Z_s):=S_s^t(f_{0}\mathds{1}_{[E_s\leq R^2]})(Z_s),
\end{equation*}
where, given $k\geq 1$ and $\sigma\in S_k$, we denote
\begin{equation*}
f_{R,\delta}^{(s,k,\sigma)}(t,Z_s):=\int_{\mathcal{T}_{k,\delta}(t)}S_s^{t-t_1}\mathcal{C}_{s,s+\widetilde{\sigma}_1}^{\infty,R} S_{s+\widetilde{\sigma}_1}^{t_1-t_2}
...\mathcal{C}_{s+\widetilde{\sigma}_{k-1},s+\widetilde{\sigma}_k}^{\infty,R} S_{s+\widetilde{\sigma}_k}^{t_m}f_{0}^{(s+\widetilde{\sigma}_k)}(Z_s)\,dt_k,...\,dt_{1}.
\end{equation*}
Given $\phi_s\in C_c(\mathbb{R}^{ds})$ and $k\in\mathbb{N}$, we define
\begin{equation}\label{bbgky truncated time}
I_{s,k,R,\delta}^N(t)(X_s):=
\int_{\mathbb{R}^{ds}}\phi_s(V_s)f_{N,R,\delta}^{(s,k)}(t,X_s,V_s)\,dV_s=\int_{B_R^{ds}}\phi_s(V_s)f_{N,R,\delta}^{(s,k)}(t,X_s,V_s)\,dV_s,
\end{equation}
\begin{equation}\label{boltzmann truncated time}
I_{s,k,R,\delta}^\infty(t)(X_s):=
\int_{\mathbb{R}^{ds}}\phi_s(V_s)f_{R,\delta}^{(s,k)}(t,X_s,V_s)\,dV_s=\int_{B_R^{ds}}\phi_s(V_s)f_{R,\delta}^{(s,k)}(t,X_s,V_s)\,dV_s.
\end{equation}
\begin{remark}\label{time functionals trivial} For $0\leq t\leq\delta$, we trivially obtain $\mathcal{T}_{k,\delta}(t)=\emptyset$. In this case the functionals $I_{s,k,R,\delta}^N(t), I_{s,k,R,\delta}^\infty(t)$ are identically zero.
\end{remark}

Recalling the observables $I_{s,k,R}^N$, $I_{s,k,R}^\infty$ defined in \eqref{BBGKY bounded energy}-\eqref{Botlzmann bounded energy},  we obtain the following estimates:
\begin{lemma}\label{time sep}
For any $s,n\in\mathbb{N}$, $R>0$, $\delta>0$ and $t\in[0,T]$, the following estimates hold:
\begin{equation*}
\sum_{k=0}^n\|I_{s,k,R,\delta}^N(t)-I_{s,k,R}^N(t)\|_{L^\infty_{X_s}}\leq \delta\|\phi_s\|_{L^\infty_{V_s}}C_{d,s,\beta_0,\mu_0,T}^n\|F_{N,0}\|_{N,\beta_0,\mu_0},
\end{equation*}
\begin{equation*}\sum_{k=0}^n\|I_{s,k,R,\delta}^\infty(t)-I_{s,k,R}^\infty(t)\|_{L^\infty_{X_s}}\leq \delta\|\phi_s\|_{L^\infty_{V_s}}C_{d,s,\beta_0,\mu_0,T}^n\|F_{0}\|_{\infty,\beta_0,\mu_0}.
\end{equation*}
\end{lemma}
\begin{proof} For the proof, we follow similar ideas as in Lemma 8.7. from \cite{thesis}, and we also use bound \eqref{bound on sigma} to control the combinatorics occurring.
\end{proof}

Combining Lemma \ref{term by term}, Lemma  \ref{energy truncation} and Lemma \ref{time sep}, we obtain
\begin{proposition}\label{reduction}
For any $s,n\in\mathbb{N}$, $R>1$, $\delta>0$ and $t\in[0,T]$, the following estimates hold:
\begin{equation*}
\begin{aligned}
\|I_s^N(t)-\sum_{k=1}^nI_{s,k,R,\delta}^N(t)\|_{L^\infty_{X_s}}\leq C_{s,\beta_0,\mu_0,T}\|\phi_s\|_{L^\infty_{V_s}}\left(2^{-n}+e^{-\frac{\beta_0}{3}R^2}+\delta C_{d,s,\beta_0,\mu_0,T}^n\right)\|F_{N,0}\|_{N,\beta_0,\mu_0},
\end{aligned}
\end{equation*}
\begin{equation*}
\begin{aligned}
\|I_s^\infty(t)-\sum_{k=1}^nI_{s,k,R,\delta}^\infty(t)\|_{L^\infty_{X_s}}\leq C_{s,\beta_0,\mu_0,T}\|\phi_s\|_{L^\infty_{V_s}}\left(2^{-n}+e^{-\frac{\beta_0}{3}R^2}+\delta C_{d,s,\beta_0,\mu_0,T}^n\right)\|F_{0}\|_{\infty,\beta_0,\mu_0}.
\end{aligned}
\end{equation*}
\end{proposition}
Proposition \ref{reduction} implies that, given $0\leq k\leq n$, $R>1$, $\delta>0$, the convergence proof reduces to controlling the differences
$I_{s,k,R,\delta}^N(t)-I_{s,k,R,\delta}^\infty(t),$
where the observables $I_{s,k,R,\delta}^N$, $I_{s,k,R,\delta}^\infty$ are given by \eqref{bbgky truncated time}-\eqref{boltzmann truncated time}. However this is not immediate since the backwards $(\epsilon_2,\epsilon_3)$-flow  and the backwards free flow do not coincide in general. The goal is to eliminate some small measure set of initial data, negligible in the  limit, such that the backwards $(\epsilon_2,\epsilon_3)$-flow and the backwards free flow are comparable.
\section{Geometric estimates}\label{sec:geometric}
In this section we present some geometric results which will be essential for estimating the measure of the pathological sets leading to recollisions of the backwards $(\epsilon_2,\epsilon_3)$ flow (see Section \ref{sec:stability}). First, we review some of the results we used in \cite{ternary} which are useful here as well. We then present certain novel results, namely Lemma \ref{estimate of cubes}, Lemma \ref{estimate of difference in shell}, Lemma \ref{estimate on annulus I_1} and most importantly Lemma \ref{lemma on I_1,2}, which crucially rely on the following symmetric representation  
of the $(2d-1)$ sphere of radius $r>0$:
\begin{align}
\mathbb{S}_r^{2d-1}=\left\{(\omega_1,\omega_2)\in B_r^d\times B_r^d:\omega_2\in\mathbb{S}_{\sqrt{r^2-|\omega_1|^2}}^{d-1}\right\}
=\left\{(\omega_1,\omega_2)\in B_r^d\times B_r^d:\omega_1\in\mathbb{S}_{\sqrt{r^2-|\omega_2|^2}}^{d-1}\right\}\label{representation of sphere for fixed omega_1}
\end{align}
Representation \eqref{representation of sphere for fixed omega_1} is very useful when one wants to estimate the intersection of $\mathbb{S}_r^{2d-1}$ with sets of the form $S\times\mathbb{R}^d$ or $\mathbb{R}^d\times S$, where $S\subseteq\mathbb{R}^d$ is of small measure.

\subsection{Cylinder-Sphere estimates} Here, we present certain estimates based on  the intersection of a sphere with a given solid cylinder. These estimates were used in \cite{ternary} as well. Similar estimates can be found in \cite{denlinger,gallagher}.

\begin{lemma}\label{Ryan's lemma} Let $\rho,r>0$ and $K_\rho^d\subseteq\mathbb{R}^d$ be a solid cylinder. Then  the following estimate holds for the $(d-1)$-spherical measure:
 $$\int_{\mathbb{S}_r^{d-1}}\mathds{1}_{K_\rho^d}\,d\omega\lesssim r^{d-1}\min\left\{1,\left(\displaystyle\frac{\rho}{r}\right)^{\frac{d-1}{2}}\right\}.$$
\end{lemma}
\begin{proof}
After re-scaling we may clearly assume that $r=1$. Then, we refer to the work of R. Denlinger \cite{denlinger}, p.30, for the rest of the proof.
\end{proof}

Applying Lemma \ref{Ryan's lemma}, we obtain the following geometric estimate, which will be crucially used in Section \ref{sec:stability}.
\begin{corollary}\label{spherical estimate} Given $0<\rho\leq 1\leq R$, the following estimate holds:
$$|B_{R}^{d}\cap K_\rho^d|_{d}\lesssim R^{d}\rho^{\frac{d-1}{2}}.$$
\end{corollary}
\begin{proof}
The co-area formula and  Lemma \ref{Ryan's lemma} imply
\begin{equation}\label{estimate with min}
\begin{aligned}
|B_R^d\cap K_\rho^d|_d&= \int_0^R \int_{\mathbb{S}_r^{d-1}}\mathds{1}_{K_\rho^d}\,d\omega\,dr\\
&\lesssim \int_0^Rr^{d-1}\min\left\{1,(\frac{\rho}{r})^{\frac{d-1}{2}}\right\}\,dr\\
&\leq\int_0^\rho r^{d-1}\,dr+\rho^{\frac{d-1}{2}}\int_0^R r^{\frac{d-1}{2}}\,dr\\
&\simeq \rho^d+\rho^{\frac{d-1}{2}}R^{\frac{d+1}{2}},\quad\text{since }d\geq 2\\
&\lesssim R^{d}\rho^{\frac{d-1}{2}},\quad\text{since } 0<\rho\leq 1\leq R.
\end{aligned}
\end{equation}
\end{proof}

\subsection{Estimates relying on the $(2d-1)$-sphere representation}\label{subsec:relying} Here we present certain geometric estimates relying on the representation \eqref{representation of sphere for fixed omega_1}. In particular, up to our knowledge, Lemma \ref{estimate of cubes}, Lemma \ref{estimate of difference in shell}, Lemma \ref{estimate on annulus I_1} and most importantly Lemma \ref{lemma on I_1,2} are novel results. Lemma \ref{strip lemma} is a special case of a result  proved in  \cite{ternary}.
\subsubsection{Truncation of  impact directions}\label{subsubsec:ball} We first estimate the intersection of $\mathbb{S}_1^{2d-1}$ with sets of the form $B_\rho^d\times\mathbb{R}^d$ or $\mathbb{R}^d\times B_\rho^d$.
\begin{lemma}\label{estimate of cubes} Consider $\rho>0$. We define the sets
\begin{align}
M_{1}(\rho)&=B_\rho^d\times\mathbb{R}^d=\left\{(\omega_1,\omega_2)\in\mathbb{R}^{2d}:|\omega_1|\leq\rho\right\}\label{cube parameters 1},\\
M_{2}(\rho)&=\mathbb{R}^d\times B_\rho^d=\left\{(\omega_1,\omega_2)\in\mathbb{R}^{2d}:|\omega_2|\leq\rho\right\}\label{cube parameters 2}
\end{align}
Then, the following holds
\begin{equation*}
\int_{\mathbb{S}_1^{2d-1}}\mathds{1}_{M_1(\rho)}\,d\omega_1\,d\omega_2=\int_{\mathbb{S}_1^{2d-1}}\mathds{1}_{M_2(\rho)}\,d\omega_1\,d\omega_2\lesssim\min\{1,\rho^d\}.
\end{equation*}
\end{lemma}
\begin{proof} By symmetry it suffices to estimate the first term.
Using  \eqref{cube parameters 1} and representation \eqref{representation of sphere for fixed omega_1}, 
we obtain
\begin{align*}
\int_{\mathbb{S}_1^{2d-1}}\mathds{1}_{M_1(\rho)}\,d\omega_1\,d\omega_2&=\int_{\mathbb{S}_1^{2d-1}}\mathds{1}_{B_{\rho}^d\times \mathbb{R}^d}\,d\omega_1\,d\omega_2\lesssim \int_{B_{\rho}^d\cap B_1^d}\int_{\mathbb{S}_{\sqrt{1-|\omega_1|^2}}^{d-1}}\,d\omega_2\,d\omega_1\lesssim \min\{1,\rho^d\}.
\end{align*}
\end{proof}
The following result is a special case of  Lemma 8.4. from \cite{ternary}. For the proof, see Lemma 9.5. in \cite{thesis}.
\begin{lemma}\label{strip lemma} Consider $\rho>0$. Let us define the strip 
\begin{equation}\label{strip}
W_\rho^{2d}=\{(\omega_1,\omega_2)\in\mathbb{R}^{2d}:|\omega_1-\omega_2|\leq\rho\}.
\end{equation}
Then, the following estimate holds:
$$\int_{\mathbb{S}_1^{2d-1}}\mathds{1}_{W_\rho^{2d}}\,d\omega_1\,d\omega_2\lesssim\min\left\{1,\rho^{\frac{d-1}{2}}\right\}.$$
\end{lemma}
\begin{proof}
For the proof, see Lemma 9.5. in \cite{thesis}. The main idea is to first  use representation \eqref{representation of sphere for fixed omega_1} and then apply Lemma \ref{Ryan's lemma}.
\end{proof}
\subsubsection{Conic estimates}\label{subsubsec:conic}
Now we establish estimates related to conic regions. We first present a well-known spherical cap estimate.
\begin{lemma}\label{shell estimate} Consider $0\leq\alpha\leq 1$ and $\nu\in\mathbb{R}^{d}\setminus\{0\}$. Let us define 
\begin{equation}\label{shell parameters}
S(\alpha,\nu)=\left\{\omega\in\mathbb{R}^d:|\langle\omega,\nu\rangle|\geq\alpha |\omega||\nu|\right\}.
\end{equation}
Then, for  $\rho>0$, the following estimate holds:
$$\int_{\mathbb{S}_r^{d-1}}\mathds{1}_{S(\alpha,\nu)}\,d\omega= r^{d-1}|\mathbb{S}_1^{d-2}|\int_0^{2\arccos\alpha}\sin^{d-2}(\theta)\,d\theta\lesssim r^{d-1}\arccos\alpha.$$
\end{lemma}
\begin{proof}
After re-scaling, it suffices to prove the result for $r=1$.  Notice that $\mathbb{S}_1^{d-1}\cap S(\alpha,\nu)$ is a spherical cap of angle $2\arccos\alpha$ and direction $\nu\neq 0$ on the unit sphere. Therefore, integrating in spherical coordinates, we obtain
$$\int_{\mathbb{S}_1^{d-1}}\mathds{1}_{S(\alpha,\nu)}\,d\omega=|\mathbb{S}_1^{d-2}| \int_0^{2\arccos\alpha}\sin^{d-2}\theta\,d\theta\lesssim\arccos\alpha.$$
\end{proof}

We apply Lemma \ref{shell estimate} to obtain the following result:

\begin{lemma}\label{estimate of difference in shell} Consider $0\leq\alpha\leq 1$ and $\nu\in\mathbb{R}^{d}\setminus\{0\}$. Let us define
\begin{equation}\label{difference shell parameters}
N(\alpha,\nu)=\left\{(\omega_1,\omega_2)\in\mathbb{R}^{2d}:\langle\omega_1-\omega_2,\nu\rangle\geq\alpha|\omega_1-\omega_2||\nu|\right\}.
\end{equation}
Then, we have the estimate:
\begin{equation*}
\int_{\mathbb{S}_1^{2d-1}}\mathds{1}_{N(\alpha,\nu)}\,d\omega_1\,d\omega_2\lesssim\arccos\alpha.
\end{equation*}
\end{lemma}
\begin{proof}
Recalling \eqref{shell parameters}-\eqref{difference shell parameters},  we have 
\begin{equation}\label{equation on integrals 1 W}
N(\alpha,\nu)=\{(\omega_1,\omega_2)\in\mathbb{R}^{2d}:\omega_1-\omega_2\in S(\alpha,\nu)\}.
\end{equation}
Let us define the linear map $T:\mathbb{R}^{2d}\to\mathbb{R}^{2d}$ by
\begin{equation*}
(u_1,u_2)=T(\omega_1,\omega_2):=(\omega_1+\omega_2,\omega_1-\omega_2).
\end{equation*}
Clearly  
$$|u_1|^2+|u_2|^2=|\omega_1+\omega_2|^2+|\omega_1-\omega_2|^2=2|\omega_1|^2+2|\omega_2|^2=2,\quad\forall(\omega_1,\omega_2)\in\mathbb{S}_1^{2d-1},$$
hence $T:\mathbb{S}_1^{2d-1}\to\mathbb{S}_{\sqrt{2}}^{2d-1}$.
Therefore, using \eqref{equation on integrals 1 W} and changing variables under $T$, we have
\begin{align}
\int_{\mathbb{S}_1^{2d-1}}\mathds{1}_{N(\alpha,\nu)}(\omega_1,\omega_2)\,d\omega_1\,d\omega_2&=\int_{\mathbb{S}_1^{2d-1}}\mathds{1}_{S(\alpha,\nu)}(\omega_1-\omega_2)\,d\omega_1\,d\omega_2\nonumber\\
&\simeq\int_{\mathbb{S}_2^{2d-1}}\mathds{1}_{S(\alpha,\nu)}(u_2)\,du_1\,du_2\nonumber\\
&=\int_{B_{\sqrt{2}}^d}\int_{\mathbb{S}_{\sqrt{2-|u_1|^2}}^{d-1}}\mathds{1}_{S(\alpha,\nu)}(u_2)\,du_2\,du_1\label{decomposition sqrt}\\
&\lesssim\arccos\alpha,\label{use of spherical shell lemma}
\end{align}
where to obtain \eqref{decomposition sqrt} we use the representation of the sphere \eqref{representation of sphere for fixed omega_1},
and to obtain \eqref{use of spherical shell lemma} we use Lemma \ref{shell estimate}.
\end{proof}

\subsubsection{Annuli estimates}\label{subsubsec:annulus} We present estimates based on the intersection of the unit sphere some appropriate annuli.
\begin{lemma}\label{estimate on annulus I_1} Let $0<\beta<1/2$, and consider the sets
\begin{align}
I_1&=\left\{(\omega_1,\omega_2)\in\mathbb{R}^{2d}: \left|1-2\left|\omega_1\right|^2\right|\leq 2\beta\right\},\label{annulus I1}\\
I_2&=\left\{(\omega_1,\omega_2)\in\mathbb{R}^{2d}: \left|1-2\left|\omega_2\right|^2\right|\leq 2\beta\right\}\label{annulus I2}.
\end{align}
There hold the estimates:
\begin{equation*}
\int_{\mathbb{S}_1^{2d-1}}\mathds{1}_{I_1}\,d\omega_1\,d\omega_2=\int_{\mathbb{S}_1^{2d-1}}\mathds{1}_{I_2}\,d\omega_1\,d\omega_2\lesssim\beta.
\end{equation*}
\end{lemma}
\begin{proof}
By symmetry, it suffices to prove the estimate for $I_1$.
Since $0<\beta<1/2$, we may write
$$I_1=\left\{(\omega_1,\omega_2)\in\mathbb{S}_1^{2d-1}:\sqrt{\frac{1}{2}-\beta}\leq|\omega_1|\leq\sqrt{\frac{1}{2}+\beta}\right\}.
$$
Using  the representation \eqref{representation of sphere for fixed omega_1} of the $(2d-1)$-unit sphere,
we obtain
\begin{align*}
\int_{\mathbb{S}_1^{2d-1}}\mathds{1}_{I_1}\,d\omega_1\,d\omega_2&\leq \int_{\sqrt{\frac{1}{2}-\beta}\leq|\omega_1|\leq\sqrt{\frac{1}{2}+\beta}}\int_{\mathbb{S}_{\sqrt{1-|\omega_1|^2}}^{d-1}}\,d\omega_2\,d\omega_1\\
&\lesssim(\frac{1}{2}+\beta)^{d/2}-(\frac{1}{2}-\beta)^{d/2}\\
&\overset{d\geq 2}=\left(\sqrt{\frac{1}{2}+\beta}-\sqrt{\frac{1}{2}-\beta}\right)\sum_{j=0}^{d-1}\left(\frac{1}{2}+\beta\right)^{j/2}\left(\frac{1}{2}-\beta\right)^{\frac{d-1-j}{2}}\\
&=\frac{2\beta}{\sqrt{\frac{1}{2}+\beta}+\sqrt{\frac{1}{2}-\beta}}\sum_{j=0}^{d-1}\left(\frac{1}{2}+\beta\right)^{j/2}\left(\frac{1}{2}-\beta\right)^{\frac{d-1-j}{2}}\\
&\leq 2\sqrt{2}\beta\sum_{j=0}^{d-1}\left(\frac{1}{2}+\beta\right)^{j/2}\left(\frac{1}{2}-\beta\right)^{\frac{d-1-j}{2}}\\
&\lesssim \beta,
\end{align*}
since $0<\beta<1/2$. The proof is complete.
\end{proof}
\begin{lemma}\label{lemma on I_1,2}
Consider $0<\beta<1/4$. Let us define the hemispheres
\begin{align}
\mathcal{S}_{1,2}&=\{(\omega_1,\omega_2)\in\mathbb{S}_1^{2d-1}:|\omega_1|<|\omega_2|\},\label{sphere 2<1}\\
\mathcal{S}_{2,1}&=\{(\omega_1,\omega_2)\in\mathbb{S}_1^{2d-1}:|\omega_2|<|\omega_1|\}\label{sphere 1<2}.
\end{align}
and the annuli
\begin{align}
I_{1,2}&=\{(\omega_1,\omega_2)\in \mathbb{R}^{2d}:\left|\left|\omega_1\right|^2+2\langle\omega_1,\omega_2\rangle\right|\leq\beta\}\label{I_1,2},\\
I_{2,1}&=\{(\omega_1,\omega_2)\in\mathbb{R}^{2d}:\left|\left|\omega_2\right|^2+2\langle\omega_1,\omega_2\rangle\right|\leq\beta\}\label{I_2,1}.
\end{align}
Then, there holds
$$\int_{\mathcal{S}_{1,2}}\mathds{1}_{I_{1,2}}\,d\omega_1\,d\omega_2=\int_{\mathcal{S}_{2,1}}\mathds{1}_{I_{2,1}}\,d\omega_1\,d\omega_2\lesssim\beta.$$
\end{lemma}
\begin{proof}
By symmetry, it suffices to prove 
\begin{equation}\label{sufficient condition lemma annulus}
\int_{\mathcal{S}_{2,1}}\mathds{1}_{I_{2,1}}\,d\omega_1\,d\omega_2\lesssim\beta.
\end{equation}
Recalling notation from \eqref{cube parameters 1}-\eqref{cube parameters 2}, let us define
$$U_\beta=M_1^c(2\sqrt{\beta})\cap M_2^c(2\sqrt{\beta})=\{(\omega_1,\omega_2)\in\mathbb{R}^{2d}: |\omega_1|>2\sqrt{\beta}\text{ and }|\omega_2|> 2\sqrt{\beta}\}.$$
Clearly
$U_\beta^c=M_1(2\sqrt{\beta})\cup M_2(2\sqrt{\beta}).$
Writing
$
A:=I_{2,1}\cap U_\beta,
$
we have
\begin{align}
\int_{\mathcal{S}_{2,1}}\mathds{1}_{I_{2,1}}\,d\omega_1\,d\omega_2
&\leq\int_{\mathcal{S}_{2,1}}\mathds{1}_{ U_\beta^c}\,d\omega_1\,d\omega_2+\int_{\mathcal{S}_{2,1}}\mathds{1}_{A}\,d\omega_1\,d\omega_2\lesssim\beta^{d/2}+\int_{\mathcal{S}_{2,1}}\mathds{1}_{A}\,d\omega_1\,d\omega_2\label{truncated lemma annulus},
\end{align}
where to obtain \eqref{truncated lemma annulus}, we used Lemma \ref{estimate of cubes}. 
Notice that we may write
\begin{equation}\label{equivalent condition for I_21}
A=\{(\omega_1,\omega_2)\in\mathbb{R}^{2d}: |\omega_1|>2\sqrt{\beta},\text{ }|\omega_2|>2\sqrt{\beta}\text{ and }\sqrt{|\omega_1|^2-\beta}\leq|\omega_1+\omega_2|\leq\sqrt{|\omega_1|^2+\beta}\}.
\end{equation}
By \eqref{truncated lemma annulus}, the representation of the sphere \eqref{representation of sphere for fixed omega_1} and \eqref{equivalent condition for I_21}, we have
\begin{equation}\label{reduction to single integral}
\int_{\mathcal{S}_{2,1}}\mathds{1}_{I_{2,1}}\,\omega_1\,d\omega_2
\lesssim\beta^{d/2}+\int_{2\sqrt{\beta}<|\omega_1|\leq 1}\int_{\mathcal{S}_{2,1,\omega_1}}\mathds{1}_{A_{\omega_1}}(\omega_2)\,d\omega_2\,d\omega_1,
\end{equation}
where given $2\sqrt{\beta}<|\omega_1|\leq 1$, we denote
\begin{align}
\mathcal{S}_{2,1,\omega_1}&=\{\omega_2\in\mathbb{S}_{\sqrt{1-|\omega_1|^2}}^{d-1}:|\omega_2|<|\omega_1|\},\label{S_21 projection}\\
A_{\omega_1}&=\{\omega_2\in \mathbb{R}^d:(\omega_1,\omega_2)\in A\}=\{\omega_2\in\mathbb{R}^{d}:|\omega_2|>2\sqrt{\beta}\text{ and }\sqrt{|\omega_1|^2-\beta}\leq|\omega_1+\omega_2|\leq\sqrt{|\omega_1|^2+\beta}\}.\label{A omega1}
\end{align}
Since $\beta<1/4$, it suffices to control the term:
\begin{equation}\label{iterated integral I'}
I'=\int_{2\sqrt{\beta}<|\omega_1|\leq 1}\int_{\mathcal{S}_{2,1,\omega_1}}\mathds{1}_{A_{\omega_1}}(\omega_2)\,d\omega_2\,d\omega_1.
\end{equation}
Now we shall prove that, in fact
\begin{equation}\label{ordered I'}
I'=\int_{2\sqrt{\beta}<\sqrt{1-|\omega_1|^2}<|\omega_1|\leq 1}\int_{\mathbb{S}_{\sqrt{1-|\omega_1|^2}}^{d-1}}\mathds{1}_{A_{\omega_1}}(\omega_2)\,d\omega_2\,d\omega_1.
\end{equation}
Indeed, assume $\omega_1$ does not satisfy 
\begin{equation}\label{geometric condition_1}
2\sqrt{\beta}<\sqrt{1-|\omega_1|^2}< |\omega_1|.
\end{equation}
Since we are integrating in the region $2\sqrt{\beta}<|\omega_1|\leq 1$, exactly one of the following holds:
\begin{align}|\omega_1|\leq \sqrt{1-|\omega_1|^2},&\label{bad condition 1}\\
\sqrt{1-|\omega_1|^2}\leq 2\sqrt{\beta}.&\label{bad condition 2}
\end{align}
Recalling \eqref{S_21 projection},  condition \eqref{bad condition 1} implies that $\mathcal{S}_{2,1,\omega_1}=\emptyset$, while recalling \eqref{A omega1}, condition \eqref{bad condition 2} implies $\mathcal{S}_{2,1,\omega_1}\cap A_{\omega_1}=\emptyset$. Therefore
$$I'=\int_{2\sqrt{\beta}<\sqrt{1-|\omega_1|^2}<|\omega_1|\leq 1}\int_{\mathcal{S}_{2,1,\omega_1}}\mathds{1}_{A_{\omega_1}}(\omega_2)\,d\omega_2\,d\omega_1,$$
and \eqref{ordered I'} follows from \eqref{S_21 projection}.

Fix any $\omega_1$ satisfying \eqref{geometric condition_1}. We first estimate the inner integral:
\begin{equation}\label{wanted integral}
\int_{\mathbb{S}_{\sqrt{1-|\omega_1|^2}}^{d-1}}\mathds{1}_{A_{\omega_1}}(\omega_2)\,d\omega_2.
\end{equation}
Notice that \eqref{geometric condition_1} also yields
\begin{equation}\label{geometric condition_2}
|\omega_1|-\sqrt{|\omega_1|^2-\beta}=\frac{\beta}{|\omega_1|+\sqrt{|\omega_1|^2-\beta}}<\frac{\beta}{|\omega_1|}\leq\frac{1}{2}\sqrt{\beta}\leq\frac{1}{4}\sqrt{1-|\omega_1|^2}.
\end{equation}
  Condition \eqref{geometric condition_1} guarantees that the vector\footnote{understood as a point in $\mathbb{R}^d$} $-\omega_1$ lays outside of the sphere $\mathbb{S}_{\sqrt{1-|\omega_1|^2}}^{d-1}$, while condition \eqref{geometric condition_2} guarantees that the sphere is not contained in the annulus $A_{\omega_1}$.
 Therefore, the projection of $\mathbb{S}_{\sqrt{1-|\omega_1|^2}}^{d-1}\cap A_{\omega_1}$ on any plane containing the origin and the vector $-\omega_1$  can be visualized as follows:
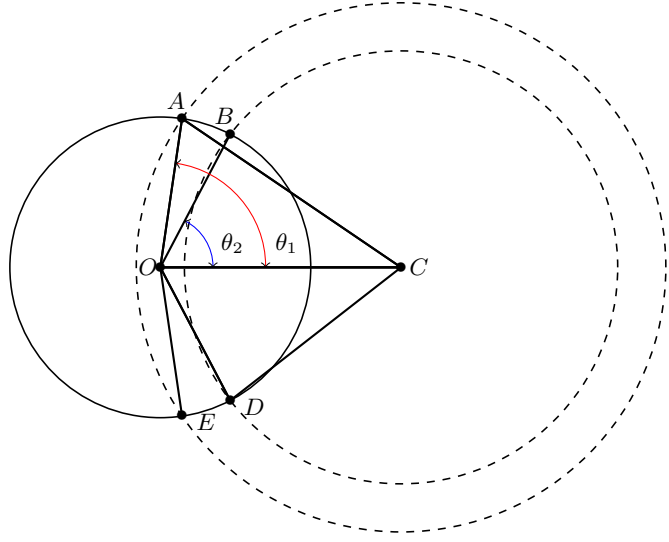
\begin{center}
\begin{tikzpicture}[scale=0.8]
\tkzDefPoint(0,0){A}
\tkzDefPoint(4,0){B}
\tkzInterCC[R](A,2.5 cm)(B,4.4cm) \tkzGetPoints{M1}{M2}
\tkzInterCC[R](A,2.5 cm)(B,3.6cm) \tkzGetPoints{N1}{N2}
\tkzDrawCircle[R](A,2.5cm) 
\tkzDrawCircle[dashed,R](B,4.4cm) 
\tkzDrawCircle[dashed,R](B,3.6cm)
\draw [thick] (A)--(M1);
\draw [thick] (A)--(N2);
\draw[thick](A)--(B);
\draw[thick](B)--(M1);
\draw [thick] (A)--(N2)--(B)--cycle;
\draw [thick] (A)--(M1)--(B)--cycle;
\draw [thick](A)--(M2);
\draw [thick](A)--(N1);
\node at (A) {$\bullet$ };
\node at (B) {$\bullet$ };
\node at (M1) {$\bullet$ };
\node at (M2) {$\bullet$ };
\node at (N1) {$\bullet$ };
\node at (N2) {$\bullet$ };
\node at ($(A) +(-0.2,0)$){$O$};
\node at ($(M1) +(-0.1,0.3)$){$A$};
\node at ($(N1) +(-0.1,0.3)$){$B$};
\node at ($(B) +(0.3,0)$){$C$};
\node at ($(N2) +(0.4,-0.1)$){$D$};
\node at ($(M2) +(0.4,-0.1)$){$E$};
 
  \draw
    (B) coordinate (a) node[right] {}
     (A) coordinate (b) node[left] {}
    (M1) coordinate (c) node[above right] {}
    pic["$$", draw=red, <->, angle eccentricity=1.2, angle radius=1.4cm]
    {angle=a--b--c};
    
    \draw
    (B) coordinate (a) node[right] {}
     (A) coordinate (b) node[left] {}
    (N1) coordinate (c) node[above right] {}
    pic["$$", draw=blue, <->, angle eccentricity=1.2, angle radius=0.7cm]
    {angle=a--b--c};

    \coordinate (a2) at (1.2,0.4);
\node at (a2){$\theta_2$};
\coordinate (a1) at (2.1,0.4);
\node at (a1){$\theta_1$};
\end{tikzpicture}
\begin{align*}
(OA)&=(OB)=\sqrt{1-|\omega_1|^2},\quad \overrightarrow{OC}=-\omega_1,\\
(AC)&=\sqrt{|\omega_1|^2+\beta},\quad (CD)=\sqrt{|\omega_1|^2-\beta}.
\end{align*}

\end{center}
We conclude that 
\begin{equation}\label{description as difference of shells}
\mathbb{S}_{\sqrt{1-|\omega|^2}}^{d-1}\cap A_{\omega_1}=\mathbb{S}_{\sqrt{1-|\omega_1|^2}}^{d-1}\cap\left( S(\cos\theta_1,-\omega_1)\setminus S(\cos\theta_2,-\omega_1)\right),
\end{equation}
where recalling the notation introduced in \eqref{shell parameters}, 
$$\mathbb{S}_{\sqrt{1-|\omega_1|^2}}^{d-1}\cap S(\cos\theta_1,-\omega_1),\quad \mathbb{S}_{\sqrt{1-|\omega_1|^2}}^{d-1}\cap S(\cos\theta_2,-\omega_1),$$
are the spherical shells on $\mathbb{S}_{\sqrt{1-|\omega_1|^2}}^{d-1}$, of direction $-\omega_1$ and angles $2\theta_1$, $2\theta_2$ respectively where
$$\theta_1=\widehat{AOC},\quad\theta_2=\widehat{BOC}.$$
Therefore, by \eqref{description as difference of shells}, we have
\begin{align}
\int_{\mathbb{S}_{\sqrt{1-|\omega_1|^2}}^{d-1}}\mathds{1}_{A_{\omega_1}}(\omega_2)\,d\omega_2
&=\int_{\mathbb{S}_{\sqrt{1-|\omega_1|^2}}^{d-1}}\mathds{1}_{S(\cos\theta_1,-\omega_1)\setminus S(\cos\theta_2,-\omega_1)}(\omega_2)\,d\omega_2\nonumber\\
&=(1-|\omega_1|^2)^{\frac{d-1}{2}}|\mathbb{S}_1^{d-2}|\int_{2\theta_2}^{2\theta_1}\sin^{d-2}\theta\,d\theta\label{use of shell estim annulus I21}\\
&\lesssim\theta_1-\theta_2.\label{estimate of wanted integral}
\end{align}
where to obtain \eqref{use of shell estim annulus I21}, we use Lemma \ref{shell estimate}, and to obtain \eqref{estimate of wanted integral} we use the fact that $d\geq 2$.

Let us calculate $\alpha_1=\cos\theta_1$, $\alpha_2=\cos\theta_2$. By the cosine law on the triangle $AOC$, we obtain
 \begin{equation}\label{x1}
 \alpha_1=\cos\theta_1=\frac{(OA)^2+(OC)^2-(AC)^2}{2(OA)(OC)}=\frac{1-|\omega_1|^2-\beta}{2|\omega_1|\sqrt{1-|\omega_1|^2}},
 \end{equation}
and by the cosine law on the triangle $BOC$, we obtain
\begin{equation}\label{x2}
\alpha_2=\cos\theta_2=\frac{(OB)^2+(OC)^2-(CB)^2}{2(OB)(OC)}=\frac{1-|\omega_1|^2+\beta}{2|\omega_1|\sqrt{1-|\omega_1|^2}}.
\end{equation}
 Then,   expression \eqref{x1} implies
\begin{equation}\label{bound on alpha_1}
|\alpha_1|\leq \frac{\sqrt{1-|\omega_1|^2}}{2|\omega_1|}+\frac{\beta}{2|\omega_1|\sqrt{1-|\omega_1|^2}}<\frac{5}{8},
\end{equation}
since by \eqref{geometric condition_1} we have  $|\omega_1|>\sqrt{1-|\omega_1|^2}>2\sqrt{\beta}$.
In the same spirit, expression \eqref{x2} yields 
\begin{equation}\label{bound on alpha_2}
|\alpha_2|<\frac{5}{8}.
\end{equation}
The inverse cosine  is smooth in $(-1,1)$, so it is Lipschitz in $[-\frac{5}{8},\frac{5}{8}]$, thus by \eqref{bound on alpha_1}-\eqref{bound on alpha_2} and \eqref{geometric condition_1}, we have
\begin{equation*}
|\arccos \alpha_1-\arccos \alpha_2|\lesssim |\alpha_1-\alpha_2|=\frac{\beta}{|\omega_1|\sqrt{1-|\omega_1|^2}}.
\end{equation*}
Therefore \eqref{estimate of wanted integral} implies
\begin{equation}\label{final estimate truncated sphere}
\int_{\mathbb{S}_{\sqrt{1-|\omega_1|^2}}^{d-1}}\mathds{1}_{A_{\omega_1}}(\omega_2)\,d\omega_2\lesssim \theta_1-\theta_2=\arccos \alpha_1-\arccos \alpha_2\lesssim \frac{\beta}{|\omega_1|\sqrt{1-|\omega_1|^2}}.
\end{equation}
Using \eqref{final estimate truncated sphere}, and recalling \eqref{ordered I'}, we have
\begin{align}
I'&=\int_{2\sqrt{\beta}<\sqrt{1-|\omega_1|^2}|\omega_1|<1}\int_{\mathbb{S}_{\sqrt{1-|\omega_1|^2}}^{d-1}}\mathds{1}_{A_{\omega_1}}(\omega_2)\,d\omega_2\,d\omega_1\nonumber\\
&\lesssim \beta\int_{B_1^d}\frac{1}{|\omega_1|\sqrt{1-|\omega_1|^2}}\,d\omega_1\nonumber\\
&\simeq\beta\int_{0}^1\frac{r^{d-2}}{\sqrt{1-r^2}}\,dr\label{integration in polar coordinates}\\
&\leq\beta\int_0^1\frac{1}{\sqrt{1-r^2}}\,dr\label{use of d>2}\\
&=\frac{\pi}{2}\beta,\label{bound on auxiliary integral I21}
\end{align}
where to obtain \eqref{integration in polar coordinates} we use integration in polar coordinates, and to obtain \eqref{use of d>2} we use the fact that $d\geq 2$.
Using \eqref{reduction to single integral} and \eqref{bound on auxiliary integral I21}, we obtain
\begin{equation*}
\int_{\mathcal{S}_{2,1}}\mathds{1}_{I_{2,1}}\,d\omega_1\,d\omega_2\lesssim \beta^{d/2}+\beta\lesssim\beta,
\end{equation*}
since $\beta<1/4$. The proof is complete.
\end{proof}

\section{Good configurations and stability}\label{sec:stability}
\subsection{Adjunction of new particles} In this section, we investigate stability of good configurations under adunctions of collisional particles. Subsection \ref{subsec:binary} investigates binary adjunctions, while Subsection \ref{subsec:ternary} investigates ternary adjunctions. To perform the measure estimates needed, we will strongly rely on the results of Section \ref{sec:geometric}.

We start with some definitions on the configurations we are using. Consider $m\in\mathbb{N}$ and $\theta>0$, and recall from \eqref{separated space data}-\eqref{separated data} the set of well-separated configurations 
\begin{equation}\label{separated conf}
\begin{aligned}
\Delta_m(\theta)&=\{\widetilde{Z}_m=(\widetilde{X}_m,\widetilde{V}_m)\in\mathbb{R}^{2dm}: |\widetilde{x}_i-\widetilde{x}_j|>\theta,\quad\forall 1\leq i<j\leq m\},\quad m\geq 2,\quad \Delta_1(\theta)=\mathbb{R}^{2d}.
\end{aligned}
\end{equation}
Roughly speaking, a good configuration is a configuration which remains well-separated under backwards time evolution. More precisely, given $\theta>0$, $t_0>0$, we define the set of good configurations as:
\begin{equation}\label{good conf def}
G_m(\theta,t_0)=\left\{Z_m=(X_m,V_m)\in\mathbb{R}^{2dm}:Z_m(t)\in\Delta_m(\theta),\quad\forall t\geq t_0\right\},
\end{equation}
where $Z_m(t)$ denotes the backwards in time free flow of $Z_m=(X_m,V_m)$, given by:
\begin{equation}\label{back-wards flow}
Z_m(t)=\left((X_m\left(t\right),V_m\left(t\right)\right):=(X_m-tV_m,V_m),\quad t\geq 0.
\end{equation}
Notice that $Z_m$ is the initial point of the trajectory i.e. $Z_m(0)=Z_m$.
In other words for $m\geq 2$, we have
\begin{equation}\label{good conf def m>=2}
\begin{aligned}
G_m(\theta,t_0)=\left\{Z_m=(X_m,V_m)\in\mathbb{R}^{2dm}:|x_i(t)-x_j(t)|>\theta,\quad\forall t\geq t_0,\quad\forall i<j\in \left\{1,...,m\right\}\right\}.\end{aligned}
\end{equation}

From  now on, we consider parameters 
$R>>1$ and $0< \delta,\eta,\epsilon_0,\alpha<<1$ satisfying:
\begin{equation}\label{choice of parameters}
 \alpha<<\epsilon_0<<\eta\delta,\quad R\alpha<<\eta\epsilon_0.
\end{equation}
For convenience we choose the parameters in \eqref{choice of parameters} in the very end of the paper, see \eqref{first parameter}-\eqref{final parameter}. Throughout this section, we will write $K_\eta^d$ for a cylinder of radius $\eta$ in $\mathbb{R}^d$.

The following Lemma is useful for the adjunction of particles to a given configuration. For  the proof, see Lemma 12.2.1 from \cite{gallagher} or Lemma 10.2. from \cite{thesis}.
 \begin{lemma}\label{adjuction of 1}
Consider parameters $\alpha,\epsilon_0,R,\eta,\delta$ as in \eqref{choice of parameters} and $\epsilon_3<<\alpha$. Let $\bar{y}_1,\bar{y}_2\in\mathbb{R}^d$, with $|\bar{y}_1-\bar{y}_2|>\epsilon_0$ and $v_1\in B_R^d$.  Then there is a $d$-cylinder $K_\eta^d\subseteq\mathbb{R}^d$ such that for any $y_1\in B_\alpha^d(\bar{y}_1)$, $y_2\in B_\alpha^d(\bar{y}_2)$ and $v_2\in B_R^d\setminus K_\eta^d$, we have
\begin{enumerate}[(i)]
\item $(y_1,y_2,v_1,v_2)\in G_2(\sqrt{2}\epsilon_3,0)$,\vspace{0.2cm}
\item $(y_1,y_2,v_1,v_2)\in G_2(\epsilon_0,\delta).$
\end{enumerate}
\end{lemma}

\subsection{Stability under binary adjunction}\label{subsec:binary}  The main results of  this subsection are stated in  Proposition \ref{bad set double} which will be the inductive step of adding a colliding particle, and  Proposition \ref{bad set double measure}, which presents the measure estimate of the bad set that appears in this process. The proofs of the Propositions presented below are in part inspired by arguments in \cite{gallagher} and \cite{ternary} with a caveat that the new scenario needs to be addressed, in the case when the binary collisional configuration formed runs to a ternary interaction under time evolution.

\subsubsection{Binary adjunction}
For convenience, given $v\in\mathbb{R}^d$, let us denote
\begin{equation}\label{post notation double}
\left(\mathbb{S}_1^{d-1}\times B_R^{d}\right)^+(v)=\big\{(\omega_1, v_1)\in\mathbb{S}_{1}^{d-1}\times B_R^{d}:b_2(\omega_1, v_1-v)>0\big\},
\end{equation}
where $b_2(\omega)1,v_1-v)=\langle\omega_1,v_1-v\rangle.$ 
Recall from \eqref{back-wards flow} that given $m\in\mathbb{N}$ and $Z_m=(X_m,V_m)\in\mathbb{R}^{2dm}$, we denote the backwards in time free flow as $Z_m(t)=(X_m-tV_m,V_m)$,  $t\geq 0.$
Recall also the notation from \eqref{interior phase space} 
\begin{align*}\mathring{\mathcal{D}}_{m+1,\epsilon_2,\epsilon_3}=\big\{Z_{m+1}=(X_{m+1},V_{m+1})\in\mathbb{R}^{2d(m+1)}:\text{ } d_2(x_i,x_j)>\epsilon_2,\quad\forall (i,j)\in\mathcal{I}_{m+1}^2,&\\
\text{ and }d_3(x_i;x_j,x_k)>\sqrt{2}\epsilon_3,\quad\forall (i,j,k)\in\mathcal{I}_{m+1}^3\big\}&,
\end{align*}
where $\mathcal{I}_{m+1}^2,\mathcal{I}_{m+1}^3$ are given by \eqref{index 2}-\eqref{index 3} respectively.
\begin{proposition}\label{bad set double} Consider parameters $\alpha,\epsilon_0,R,\eta,\delta$ as in \eqref{choice of parameters} and $\epsilon_2<<\epsilon_3<<\alpha$. Let $m\in\mathbb{N}$, $\bar{Z}_m=(\bar{X}_m,\bar{V}_m)\in G_m(\epsilon_0,0)$, $\ell\in\{1,...,m\}$ and $X_m\in B_{\alpha/2}^{dm}(\bar{X}_m)$. Then there is a subset $\mathcal{B}_{\ell}^2(\bar{Z}_m)\subseteq (\mathbb{S}_1^{d-1}\times B_R^{d})^+(\bar{v}_\ell)$ such that:
\begin{enumerate}[(i)]
\item For any $(\omega_1,v_{m+1})\in (\mathbb{S}_1^{d-1}\times B_R^{d})^+(\bar{v}_{\ell})\setminus\mathcal{B}_{\ell}^2(\bar{Z}_m)$, one has:
\begin{align}
Z_{m+1}(t)&\in\mathring{\mathcal{D}}_{m+1,\epsilon_2,\epsilon_3},\quad\forall t\geq 0,\label{pre-0-double}\\
Z_{m+1}&\in G_{m+1}(\epsilon_0/2,\delta),\label{pre-delta-double}\\
\bar{Z}_{m+1}&\in G_{m+1}(\epsilon_0,\delta),\label{pre-delta-double-bar} 
\end{align}
where 
\begin{equation}\label{pre-notation double}
\begin{aligned}
&Z_{m+1}=(x_1,...,x_\ell,...,x_m,x_{m+1},\bar{v}_1,...,\bar{v}_{\ell},...,\bar{v}_m,v_{m+1}),\\
&x_{m+1}=x_{\ell}-\epsilon_2\omega_1,\\
&\bar{Z}_{m+1}=(\bar{x}_1,...,\bar{x}_{\ell},...,\bar{x}_m,\bar{x}_{m},\bar{v}_1,...,\bar{v}_{\ell},...,\bar{v}_m,v_{m+1}),\\
\end{aligned}
\end{equation}
\item For any $(\omega_1,v_{m+1})\in (\mathbb{S}_1^{d-1}\times B_R^{d})^+(\bar{v}_{\ell})\setminus\mathcal{B}_{\ell}^2(\bar{Z}_m)$, one has:
\begin{align}
&Z_{m+1}'(t)\in\mathring{\mathcal{D}}_{m+1,\epsilon_2,\epsilon_3},\quad\forall t\geq 0,\label{post-0-double}\\
&Z_{m+1}'\in G_{m+1}(\epsilon_0/2,\delta),\label{post-delta-double}\\
&\bar{Z}_{m+1}'\in G_{m+1}(\epsilon_0,\delta)\label{post-delta-double-bar},
\end{align}
where 
\begin{equation}\label{post-notation double}
\begin{aligned}
&Z_{m+1}'=(x_1,...,x_\ell,...,x_m,x_{m+1},\bar{v}_1,...,\bar{v}_{\ell}',...,\bar{v}_m,v_{m+1}'),\\
&x_{m+1}=x_{\ell}+\epsilon_2\omega_1,\\
&\bar{Z}_{m+1}'=(\bar{x}_1,...,\bar{x}_\ell,...,\bar{x}_m,\bar{x}_{m},\bar{v}_1,...,\bar{v}_{\ell}',...,\bar{v}_m,v_{m+1}'),\\
&(\bar{v}_{\ell}',v_{m+1}')=T_{\omega_1}(\bar{v}_{\ell},v_{m+1}).
\end{aligned}
\end{equation}
\end{enumerate}
\end{proposition}
\begin{proof}
By symmetry, we may assume that $\ell=m$. 
For convenience, let us define the set
$$\mathcal{F}_{m+1}=\left\{(i,j)\in\{1,...,m+1\}\times\{1,...,m+1\}:i<\min\{j,m\}\right\}.$$

 \textbf{Proof of \it{(i)}:} Here we use notation from \eqref{pre-notation double}. 
 We start by formulating the following claim, which will imply \eqref{pre-0-double}.
 
 \begin{lemma}\label{aux lemma pre-0} Under the assumptions of Proposition \ref{bad set double}, there is a subset $\mathcal{B}_m^{2,0,-}(\bar{Z}_m)\subseteq \mathbb{S}_1^{d-1}\times B_R^d$ such that for any $(\omega_1,v_{m+1})\in (\mathbb{S}_1^{d-1}\times B_R^d)^+(\bar{v}_m)\setminus\mathcal{B}_m^{2,0,-}(\bar{Z}_m)$, there holds:
 \begin{align}
 d_2\left(x_i\left(t\right),x_j\left(t\right)\right)&>\sqrt{2}\epsilon_3,\quad\forall t\geq 0,\quad\forall (i,j)\in\mathcal{F}_{m+1},\label{pre lemma i<m}\\
 d_2\left(x_m\left(t\right),x_{m+1}\left(t\right)\right)&>\epsilon_2,\quad\forall t\geq 0.\label{pre lemma i=m}
 \end{align}
 \end{lemma}
 Notice that \eqref{pre lemma i<m}-\eqref{pre lemma i=m} trivially imply \eqref{pre-0-double}, since $\epsilon_2<<\epsilon_3$.
 
 \textit{Proof of Lemma \ref{aux lemma pre-0}} 
 
 {\it{Step 1: The proof of \eqref{pre lemma i<m}}}:
 We distinguish the following cases:
 
 $\bullet$ $j\leq m$: Since $\bar{Z}_m\in G_m(\epsilon_0,0)$ and $j\leq m$, we have
 $|\bar{x}_i(t)-\bar{x}_j(t)|>\epsilon_0,$ for all $t\geq 0.$
 Therefore, triangle inequality implies that
 \begin{equation}\label{e/2 pre}
 \begin{aligned}
 |x_i(t)-x_j(t)|&=|x_i-x_j-t(\bar{v}_i-\bar{v}_j)|\geq |\bar{x}_i-\bar{x}_j-t(\bar{v}_i-\bar{v}_j)|-\alpha\geq\epsilon_0-\alpha>\frac{\epsilon_0}{2}>\sqrt{2}\epsilon_3,
 \end{aligned}
 \end{equation}
 since $\epsilon_3<<\alpha<<\epsilon_0$.
 
$\bullet$ $j=m+1$: Since $(i,m+1)\in\mathcal{F}_{m+1}$, we have $i\leq m-1$. Since $\bar{Z}_m\in G_m(\epsilon_0,0)$ and $X_m\in B_{\alpha/2}^{dm}(\bar{X}_m)$, we conclude
 \begin{align*}
 &|\bar{x}_i-\bar{x}_m|>\epsilon_0,\quad |x_i-\bar{x}_i|\leq\frac{\alpha}{2}<\alpha,\quad |x_{m+1}-\bar{x}_m|\leq |x_m-\bar{x}_m|+\epsilon_2|\omega_1|\leq\frac{\alpha}{2}+\epsilon_2<\alpha,\quad\text{since }\epsilon_2<<\alpha.\\
 \end{align*}
   Applying part $\textit{(i)}$ of Lemma \ref{adjuction of 1} for $\bar{y}_1=\bar{x}_i$, $\bar{y}_2=\bar{x}_m$, $y_1=x_i$, $y_2=x_{m+1}$, we may find a cylinder $K_\eta^{d,i}$ such that for any $v_{m+1}\in B_R^d\setminus K_\eta^{d,i}$, we have
   $|x_i(t)-x_{m+1}(t)|>\sqrt{2}\epsilon_3,$ for all $t\geq 0.$
   Hence the inequality in \eqref{pre lemma i<m} holds for any $(\omega_1,v_{m+1})\in (\mathbb{S}_1^{d-1}\times B_R^d)^+(\bar{v}_m)\setminus V_{m+1}^i$, where
   \begin{equation}\label{V_m+1 i}
   V_{m+1}^i=\mathbb{S}_1^{d-1}\times K_\eta^{d,i}.
   \end{equation}
 
 We conclude that \eqref{pre lemma i<m} holds for any $(\omega_1,v_{m+1})\in (\mathbb{S}_1^{d-1}\times B_R^{d})\setminus\bigcup_{i=1}^{m-1}V_{m+1}^i.$
 
 \textit{Step 2: The proof of \eqref{pre lemma i=m}:}  We recall notation from \eqref{pre-notation double}.
 Considering $t\geq 0$ and $(\omega_1,v_{m+1})\in(\mathbb{S}_1^{d-1}\times B_R^{d})^+(\bar{v}_m)$. Using the fact that $(\omega_1,v_{m+1})\in(\mathbb{S}_1^{d-1}\times B_R^d)^+(\bar{v}_m)$, we obtain
\begin{align}
|x_{m}(t)-x_{m+1}(t)|^2=|\epsilon_2\omega_1-t(\bar{v}_m-v_{m+1})|^2\geq\epsilon_2^2|\omega_1|^2+2\epsilon_2 tb_2(\omega_1,v_{m+1}-\bar{v}_{m})
>\epsilon_2^2.\label{use of pre adjuction}
\end{align}
Therefore, \eqref{pre lemma i=m} holds for any $(\omega_1,v_{m+1})\in(\mathbb{S}_1^{d-1}\times B_R^{d})^+(\bar{v}_m)$.

Defining 
\begin{equation}\label{B-pre-0}
\mathcal{B}_{m}^{2,0,-}(\bar{Z}_m)=\bigcup_{i=1}^{m-1} V_{m+1}^i,
\end{equation}
the claim of Lemma \ref{aux lemma pre-0} follows.

 Now we go back to the proof of part \textit{(i)} of Proposition \ref{bad set double}. We will find a set $\mathcal{B}_m^{2,\delta,-}(\bar{Z}_m)\subseteq \mathbb{S}_1^{d-1}\times B_R^{d}$ such that \eqref{pre-delta-double} holds for any $(\omega_1,v_{m+1})\in(\mathbb{S}_1^{d-1}\times B_R^{d})\setminus \mathcal{B}_m^{2,\delta,-}(\bar{Z}_m).$
 
Let us fix $i,j\in\{1,...,m+1\}$ with $i<j$. We distinguish the following cases:

$\bullet$ $j\leq m$: We use the same argument as in \eqref{e/2 pre}, to obtain 
$|x_i(t)-x_j(t)|>\frac{\epsilon_0}{2},$ for all $t\geq 0.$

$\bullet$ $(i,j)\in\mathcal{F}_{m+1}$, $j=m+1$: Since $(i,m+1)\in\mathcal{F}_{m+1}$, we have $i\leq m-1$. Applying a similar argument to the corresponding case in the proof of \eqref{pre lemma i<m}, using part \textit{(ii)} of Lemma \ref{adjuction of 1} instead, we obtain that the inequality
$
|x_i(t)-x_{m+1}(t)|>\epsilon_0,
$
for all $t\geq\delta$,
 holds for any 
$(\omega_1,v_{m+1})\in(\mathbb{S}_1^{d-1}\times B_R^d)\setminus V_{m+1}^i,$ where $V_{m+1}^i$ is given by \eqref{V_m+1 i}. Notice that the lower bound is in fact $\epsilon_0$.

$\bullet$ $i=m$, $j=m+1$: Triangle inequality and the fact that $\epsilon_2<<\epsilon_0<<\eta\delta$ imply that for any $t\geq\delta$ and $(\omega_1,v_{m+1})\in\mathbb{S}_1^{d-1}\times B_R^d$ with $|v_{m+1}-\bar{v}_m|>\eta$, we have
\begin{align*}
|x_m(t)-x_{m+1}(t)|&=|\epsilon_2\omega_1-t(\bar{v}_m-v_{m+1})|\geq |\bar{v}_m-v_{m+1}|\delta-\epsilon_2>\eta\delta-\epsilon_2\nonumber>\epsilon_0.
\end{align*}
 Therefore, the inequality
$
|x_m(t)-x_{m+1}(t)|>\epsilon_0,$
for all $t\geq\delta$,
 holds for any 
$(\omega_1,v_{m+1})\in(\mathbb{S}_1^{d-1}\times B_R^d)\setminus V_{m,m+1},$
where
\begin{equation}\label{V_m,m+1}
V_{m,m+1}=\mathbb{S}_1^{d-1}\times B_\eta^{d}(\bar{v}_m).
\end{equation}
Notice that the lower bound is $\epsilon_0$ again.

Defining
\begin{equation}\label{B-pre-delta}
\mathcal{B}_m^{2,\delta,-}(\bar{Z}_m)=\mathcal{B}_m^{2,0,-}(\bar{Z}_m)\cup V_{m,m+1},
\end{equation}
we conclude that  \eqref{pre-delta-double} holds for any 
$(\omega_1,v_{m+1})\in(\mathbb{S}_1^{d-1}\times B_R^d)\setminus\mathcal{B}_m^{2,\delta,-}(\bar{Z}_m).$

Let us note that the only case which prevents us from having $Z_{m+1}\in G_{m+1}(\epsilon_0,\delta)$ is the case $1\leq i<j\leq m$, where we obtain a lower bound of $\epsilon_0/2$. In all other cases we can obtain lower bound $\epsilon_0$. 

More precisely, for $(\omega_1,v_{m+1})\in(\mathbb{S}_1^{d-1}\times B_R^{d})\setminus \mathcal{B}_{m}^{2,\delta,-}(\bar{Z}_m)$, the inequality
$
|\bar{x}_i(t)-\bar{x}_j(t)|>\epsilon_0,
$
for all $t\geq\delta$,
holds for all $1\leq i<j\leq m+1$ except the case $1\leq i<j\leq m$. However in this case, for any $1\leq i<j\leq m$, we have
$
|\bar{x}_i(t)-\bar{x}_j(t)|>\epsilon_0,
$
for all $t>0$,
since $\bar{Z}_m\in G_m(\epsilon_0,0)$. Therefore, \eqref{pre-delta-double-bar} holds for 
$(\omega_1,v_{m+1})\in(\mathbb{S}_1^{d-1}\times B_R^d)\setminus\mathcal{B}_m^{2,\delta,-}(\bar{Z}_m)$.

We conclude that the set
\begin{equation}\label{B-double-pre}
\mathcal{B}_{m}^{2,-}(\bar{Z}_m)=(\mathbb{S}_{1}^{d-1}\times B_R^{d})^+(\bar{v}_m)\cap\left(\mathcal{B}_{m}^{2,0,-}\left(\bar{Z}_m\right)\cup\mathcal{B}_{m}^{2,\delta,-}\left(\bar{Z}_m\right)\right),
\end{equation}
is the set we need for the precollisional case.

\textbf{Proof of \it{(ii)}:} Here we use the notation from \eqref{post-notation double}. 
The proof follows the steps of the precollisional case, but we replace the velocities $(\bar{v}_m,v_{m+1})$ by the transformed velocities $(\bar{v}_{m}',v_{m+1}')$ and then pull-back. It is worth mentioning that the $m$-th particle  needs special treatment since its velocity is transformed to $\bar{v}_{m}'$. Following similar arguments to the precollisional case, we conclude that the appropriate set for the postcollisional case is given by
\begin{equation}\label{B-double-post}
\mathcal{B}_m^{2,+}(\bar{Z}_m):=(\mathbb{S}_1^{d-1}\times B_R^d)^+(\bar{v}_m)\cap\left[ V_{m,m+1}\cup\bigcup_{i=1}^{m-1}\left(V_{m}^{i'}\cup V_{m+1}^{i'}\right)\right],
\end{equation}
where
\begin{align}
   V_{m}^{i'}&=\left\{(\omega_1,v_{m+1})\in\mathbb{S}_1^{d-1}\times B_R^d: \bar{v}_{m}'\in K_\eta^{d,i}\right\},\label{V_m i'}\\
   V_{m+1}^{i'}&=\left\{(\omega_1,v_{m+1})\in\mathbb{S}_1^{d-1}\times B_R^d: v_{m+1}'\in K_\eta^{d,i}\right\},\label{V_m+1 i'}\\
V_{m,m+1}&=\mathbb{S}_1^{d-1}\times B_\eta^{d}(\bar{v}_m). \label{V_m,m+1-post}
   \end{align}

The set 
\begin{equation}\label{B double}
\mathcal{B}_m^2(\bar{Z}_m)=\mathcal{B}_m^{2,-}(\bar{Z}_m)\cup\mathcal{B}_m^{2,+}(\bar{Z}_m),
\end{equation}
is the one we need to conclude the proof.
\end{proof}
\subsubsection{Measure estimate for binary adjunction} We now estimate the measure of the pathological set $\mathcal{B}_\ell^2(\bar{Z}_m)$ appearing in Proposition \ref{bad set double}. To control postcollisional configurations, we will strongly rely on the binary transition map introduced in the Appendix (see Proposition \ref{transition prop}).
\begin{proposition}\label{bad set double measure} Consider parameters $\alpha,\epsilon_0,R,\eta,\delta$ as in \eqref{choice of parameters} and $\epsilon_2<<\epsilon_3<<\alpha$. Let $m\in\mathbb{N}$, $\bar{Z}_m\in G_m(\epsilon_0,0)$, $\ell\in\{1,...,m\}$ and $\mathcal{B}_{\ell}^2(\bar{Z}_m)$ the set given in the statement of Proposition \ref{bad set double}. Then the following measure estimate holds:
\begin{equation*}
\left|\mathcal{B}_{\ell}^2(\bar{Z}_m)\right|\lesssim mR^{d}\eta^{\frac{d-1}{2d+2}},
\end{equation*}
where $|\cdot|$ denotes the product measure on $\mathbb{S}_1^{d-1}\times B_R^{d}$.
\end{proposition}
\begin{proof}
Without loss of generality, we may assume that $\ell=m$. By \eqref{B double} it suffices to estimate the measure of $\mathcal{B}_m^{2,-}(\bar{Z}_m)$ and $\mathcal{B}_m^{2,+}(\bar{Z}_m)$.

\textbf{Estimate of $\mathcal{B}_m^{2,-}(\bar{Z}_m)$:} Recalling \eqref{post notation double}, \eqref{B-double-pre}, \eqref{B-pre-delta}, \eqref{B-pre-0}, we have
\begin{equation}\label{B- double measure representation}
\mathcal{B}_m^{2,-}(\bar{Z}_m)=(\mathbb{S}_1^{d-1}\times B_R^d)^+(\bar{v}_m)\cap\left[ V_{m,m+1}\cup\bigcup_{i=1}^{m-1} V_{m+1}^{i}\right],
\end{equation}
where $V_{m,m+1}$ is given by \eqref{V_m,m+1} and $V_{m+1}^i$ are given by \eqref{V_m+1 i}. By sub-additivity, it suffices to estimate the measure of each term in \eqref{B- double measure representation}.

$\bullet$ Estimate of the term corresponding to $V_{m,m+1}$: By \eqref{V_m,m+1}, we have 
$V_{m,m+1}=\mathbb{S}_1^{d-1}\times B_\eta^d(\bar{v}_m),$
therefore
\begin{align}
|(\mathbb{S}_1^{d-1}\times B_R^d)^+(\bar{v}_m)\cap V_{m,m+1}|&\leq |\mathbb{S}_1^{d-1}\times (B_R^d\cap B_\eta^d(\bar{v}_m))|\leq|\mathbb{S}_1^{d-1}|_{\mathbb{S}_1^{d-1}} |B_\eta^d(\bar{v}_m)|_d\lesssim \eta^d.\label{measure V_m,m+1}
\end{align}
$\bullet$ Estimate of the term corresponding to $V_{m+1}^i$: By \eqref{V_m+1 i}, we have
$V_{m+1}^i=\mathbb{S}_1^{d-1}\times K_\eta^{d,i},$
therefore by Corollary \ref{spherical estimate}, we obtain
\begin{align}
|(\mathbb{S}_1^{d-1}\times B_R^d)^+(\bar{v}_m)\times V_{m+1}^i|&\leq |\mathbb{S}_1^{d-1}\times (B_R^d\cap K_\eta^{d,i})|\simeq|\mathbb{S}_{1}^{d-1}|_{\mathbb{S}_1^{d-1}}|B_R^d\cap K_{\eta}^{d,i}|_d\lesssim R^d\eta^{\frac{d-1}{2}}.\label{measure V_m+1,i}
\end{align}
Using \eqref{B- double measure representation}-\eqref{measure V_m+1,i}, subadditivity, and the fact that $\eta<<1$, $m\geq 1$, we obtain
\begin{equation}\label{B- measure}
|\mathcal{B}_m^{2,-}(\bar{Z}_m)|\lesssim m R^d\eta^{\frac{d-1}{2}}.
\end{equation}

\textbf{Estimate of $\mathcal{B}_m^{2,+}(\bar{Z}_m)$:} Recalling \eqref{B-double-post}, we have
\begin{equation}\label{B+ double measure representation}
\mathcal{B}_m^{2,+}(\bar{Z}_m)=(\mathbb{S}_1^{d-1}\times B_R^d)^+(\bar{v}_m)\cap\left[ V_{m,m+1}\cup\bigcup_{i=1}^{m-1}\left(V_{m}^{i'}\cup V_{m+1}^{i'}\right)\right],
\end{equation}
where $V_{m,m+1}$ is given by \eqref{V_m,m+1} and $V_{m}^{i'}$, $V_{m+1}^{i'}$ are given by \eqref{V_m i'}-\eqref{V_m+1 i'}. By subadditivity, it suffices to estimate the measure of each term in \eqref{B+ double measure representation}.
 The term corresponding to $V_{m,m+1}$ has already benn estimated in \eqref{measure V_m,m+1}. We have
\begin{equation}\label{measure V_m,m+1'}
|(\mathbb{S}_1^{d-1}\times B_R^d)^+(\bar{v}_m)\cap V_{m,m+1}|\lesssim\eta^d.
\end{equation}
To estimate the measure of the remaining terms, we will strongly rely on the properties of the binary transition map defined in Proposition \ref{transition prop}. 
We first introduce some notation. 
Given $0<r\leq 2R$, let us define the $r$-sphere, centered at $\bar{v}_m$:
\begin{equation*}S_r^{d-1}(\bar{v}_{m})=\left\{v_{m+1}\in\mathbb{R}^{d}:|\bar{v}_{m}-v_{m+1}|=r\right\}.\end{equation*}
Also, given $v_{m+1}\in\mathbb{R}^d$, we define the set
\begin{equation}\label{S-shell}
\begin{aligned}
\mathcal{S}_{\bar{v}_{m},v_{m+1}}^+&=\left\{\omega_1\in\mathbb{S}_1^{d-1}:b_2(\omega_1,v_{m+1}-\bar{v}_m)>0\right\}=\left\{\omega_1\in\mathbb{S}_1^{d-1}:(\omega_1,v_{m+1})\in(\mathbb{S}_1^{d-1}\times B_R^{d})^+(\bar{v}_m)\right\}.
\end{aligned}
\end{equation}
Since $\bar{v}_m\in B_R^d$, triangle inequality  implies
$
B_R^{d}\subseteq B_{2R}^d(\bar{v}_{m}).
$
Under this  notation,  Fubini's Theorem, the co-area formula, and relations \eqref{B+ double measure representation}-\eqref{measure V_m,m+1'} yield 
\begin{equation}\label{integral expression for post}
\begin{aligned}
|\mathcal{B}_{m}^{2+}(\bar{Z}_m)|&=\int_{(\mathbb{S}_1^{d-1}\times B_R^{d})^+(\bar{v}_m)}\mathds{1}_{\mathcal{B}_{m}^{2+}(\bar{Z}_m)}\,d\omega_1\,dv_{m+1}\\
&=\int_{B_R^d}\int_{\mathcal{S}_{\bar{v}_{m},v_{m+1}}^+}\mathds{1}_{\mathcal{B}_{m}^{2+}(\bar{Z}_m)}\,d\omega_1\,dv_{m+1}\\
&\lesssim\eta^d+\int_0^{2R}\int_{S_r^{d-1}(\bar{v}_{m})}\int_{\mathcal{S}_{\bar{v}_{m},v_{m+1}}^+}\mathds{1}_{\bigcup_{i=1}^{m-1}(V_{m}^{i'}\cup V_{m+1}^{i'})}(\omega_1)\,d\omega_1\,dv_{m+1}\,dr.
\end{aligned}
\end{equation}

Let us  estimate the integral:
$$\int_{\mathcal{S}_{\bar{v}_{m},v_{m+1}}^+}\mathds{1}_{\bigcup_{i=1}^{m-1}(V_{m}^{i,'}\cup V_{m+1}^{i,'})}(\omega_1)\,d\omega_1,$$
for fixed $0<r\leq 2R$ and $v_{m+1}\in S_r^{d-1}(\bar{v}_{m})$. We introduce a  parameter $0<\beta<<1$, which will be chosen later in terms of $\eta$, and decompose $\mathcal{S}^+_{\bar{v}_{m},v_{m+1}}$ as follows:
\begin{equation}\label{S-decomposition}
\mathcal{S}^+_{\bar{v}_{m},v_{m+1}}=\mathcal{S}_{\bar{v}_{m},v_{m+1}}^{1,+}\cup \mathcal{S}_{\bar{v}_{m},v_{m+1}}^{2,+},
\end{equation}
where
\begin{equation}\label{S1}
\mathcal{S}_{\bar{v}_{m},v_{m+1}}^{1,+}=\left\{\omega_1\in\mathcal{S}_{\bar{v}_{m},v_{m+1}}^+:b_2(\omega_1,v_{m+1}-\bar{v}_m)>\beta |v_{m+1}-\bar{v}_m|\right\},
\end{equation}
and
\begin{equation}\label{S2}
\begin{aligned}
\mathcal{S}_{\bar{v}_{m},v_{m+1}}^{2,+}&=\left\{\omega_1\in\mathcal{S}_{\bar{v}_{m},v_{m+1}}^+:b_2(\omega_1,v_{m+1}-\bar{v}_m)\leq\beta |v_{m+1}-\bar{v}_m|\right\}.
\end{aligned}
\end{equation}
Notice that $\mathcal{S}_{\bar{v}_{m},v_{m+1}}^{2,+}$ is the union of  two unit $(d-1)$-spherical caps of angle $\pi/2-\arccos\beta$. Thus,  integrating in spherical coordinates, we may estimate its measure  as follows:
\begin{equation*}\int_{\mathbb{S}_1^{d-1}}\mathds{1}_{\mathcal{S}_{\bar{v}_{m},v_{m+1}}^{2,+}}(\omega_1)
\,d\omega_1\lesssim \int_{\arccos\beta}^{\pi/2} \sin^{d-2}(\theta)\,d\theta\leq \frac{\pi}{2}-\arccos\beta=\arcsin\beta.
\end{equation*}
Thus 
\begin{equation}\label{estimate on S2}
\int_{\mathcal{S}_{\bar{v}_{m},v_{m+1}}^{2,+}}\mathds{1}_{\bigcup_{i=1}^{m-1}(V_{m}^{i'}\cup V_{m+1}^{i'})}(\omega_1)\,d\omega_1\lesssim\arcsin\beta.
\end{equation}
We now wish to estimate
\begin{equation}\label{wanted estimate on S1}
\int_{\mathcal{S}_{\bar{v}_{m},v_{m+1}}^{1,+}}\mathds{1}_{\bigcup_{i=1}^{m-1}(V_{m}^{i'}\cup V_{m+1}^{i'})}(\omega_1)\,d\omega_1.
\end{equation}
We will use the binary transition map $\mathcal{J}_{\bar{v}_m,m_{m+1}}:\mathcal{S}_{\bar{v}_m,v_{m+1}}^+\to\mathbb{S}_1^{d-1}$, which is given by
\begin{equation}\label{transition text}
\nu_1:=\mathcal{J}_{\bar{v}_m,v_{m+1}}(\omega_1)=r^{-1}(\bar{v}_m'-v_{m+1}'),
\end{equation}
to change variables in the above integral.
For details on the transition map, see Proposition \ref{transition prop} in the Appendix.
By Proposition \ref{transition prop}, for  $\omega_1\in\mathcal{S}_{\bar{v}_{m},v_{m+1}}^+$, the Jacobian matrix of the transition map is 
\begin{equation*}
\jac(\mathcal{J}_{\bar{v}_{m},v_{m+1}})(\omega_1)\simeq r^{-d}b_2^d(\omega_1,v_{m+1}-\bar{v}_m)>0.
\end{equation*}
Therefore, for $\omega_1\in\mathcal{S}_{\bar{v}_m,v_{m+1}}^{1+}$, we have
\begin{equation}\label{estimate on inverse jacobian}
\jac^{-1}(\mathcal{J}_{\bar{v}_{m},v_{m+1}})(\omega_1)\simeq r^{d}b_2^{-d}(\omega_1,v_{m+1}-\bar{v}_m)\leq r^{d}\beta^{-d} |v_{m+1}-\bar{v}_m|^{-d}\lesssim\beta^{-d},
\end{equation}
since 
$|v_{m+1}-\bar{v}_m|=r.$

For convenience, we  express $\bar{v}_m'$, $v_{m+1}'$ in terms of the precollisional velocities $\bar{v}_m$, $v_{m+1}$ and 
$\nu_1$ given by \eqref{transition text}.  As a consequence of \eqref{binary formulas with}, we obtain
\begin{align}
\bar{v}_m'&=\frac{\bar{v}_m+v_{m+1}}{2}+\frac{r}{2}\nu_1\label{v_m' with respect to v},\\
v_{m+1}'&=\frac{\bar{v}_m+v_{m+1}}{2}-\frac{r}{2}\nu_1\label{v_m+1' with respect to v}
\end{align} 
We are now in the position to estimate the integral in \eqref{wanted estimate on S1}. We first estimate for the term corresponding to $V_m^{i'}$: Recalling \eqref{V_m i'}, we have
$
   V_{m}^{i'}=\left\{(\omega_1,v_{m+1})\in\mathbb{S}_1^{d-1}\times B_R^d: \bar{v}_{m}'\in K_\eta^{d,i}\right\}.
   $
  By \eqref{v_m' with respect to v},
 \begin{equation}\label{iff cylinder m}
 \bar{v}_m'\in K_\eta^{d,i}\Leftrightarrow\nu_1=\mathcal{J}_{\bar{v}_m,v_{m+1}}(\omega_1)\in \widetilde{K}_{2\eta/r}^{d,i},
\end{equation}  
where $\widetilde{K}_{2\eta/r}^{d,i}$ is a cylinder of radius $2\eta/r$. Therefore, we obtain
\begin{align}
\int_{\mathcal{S}_{\bar{v}_{m},v_{m+1}}^{1,+}}\mathds{1}_{V_{m}^{i'}}(\omega_1)\,d\omega_1&=
\int_{\mathcal{S}_{\bar{v}_{m},v_{m+1}}^{1,+}}\mathds{1}_{\bar{v}_m'\in K_{2\eta}^{d,i}}(\omega_1)\,d\omega_1\nonumber\\
&=\int_{\mathcal{S}_{\bar{v}_{m},v_{m+1}}^{1,+}}(\mathds{1}_{\widetilde{K}_{2\eta/r}^{d,i}}\circ\mathcal{J}_{\bar{v}_m,v_{m+1}})(\omega_1)\,d\omega_1\label{V_s^* iff inside}\\
&\lesssim\beta^{-d}\int_{\mathbb{S}_1^{d-1}}\mathds{1}_{\widetilde{K}_{2\eta/r}^{d,i}}(\nu)\,d\nu\label{jac and change V_s^*}\\
&\lesssim \beta^{-d}\min\left\{1,\left(\frac{\eta}{r}\right)^{\frac{d-1}{2}}\right\},\label{cylinder estimate}
\end{align}
where to obtain \eqref{V_s^* iff inside} we use \eqref{iff cylinder m}, to obtain \eqref{jac and change V_s^*} we use part \textit{(iv)} of Proposition \ref{transition prop} and estimate \eqref{estimate on inverse jacobian}, and to obtain \eqref{cylinder estimate} we use Lemma \ref{Ryan's lemma}. 

Hence, for fixed $v_{m+1}\in S_r^{d-1}(\bar{v}_m)$, we have
\begin{equation}\label{estimate on S1-Vm^*}
\int_{\mathcal{S}_{\bar{v}_{m},v_{m+1}}^{1,+}}\mathds{1}_{V_{m}^{i,'}}(\omega_1)\,d\omega_1\lesssim \beta^{-d}\min\left\{1,\left(\frac{\eta}{r}\right)^{\frac{d-1}{2}}\right\}.
\end{equation}

Recalling also $V_{m+1}^{i'}$ from \eqref{V_m+1 i'}, we obtain in an analogous way the estimate:
\begin{equation}\label{estimate on S1-Vm+1^*}
\int_{\mathcal{S}_{\bar{v}_{m},v_{m+1}}^{1,+}}\mathds{1}_{V_{m+1}^{i'}}(\omega_1)\,d\omega_1\lesssim \beta^{-d}\min\left\{1,\left(\frac{\eta}{r}\right)^{\frac{d-1}{2}}\right\}.
\end{equation}

Combining \eqref{estimate on S1-Vm^*}-\eqref{estimate on S1-Vm+1^*} and adding for $i=1,...,m-1$, we obtain
\begin{equation}\label{on S_1}
\int_{\mathcal{S}_{\bar{v}_{m},v_{m+1}}^{1,+}}\mathds{1}_{\bigcup_{i=1}^{m-1}(V_m^{i,'}\cup V_{m+1}^{i,'})}(\omega_1)\,d\omega_1\lesssim m \beta^{-d}\min\left\{1,\left(\frac{\eta}{r}\right)^{\frac{d-1}{2}}\right\}
\end{equation}

Therefore, recalling  \eqref{S-decomposition} and using estimates \eqref{estimate on S2}, \eqref{on S_1}, we obtain the estimate:
\begin{equation}\label{estimate on post integral}
\int_{\mathcal{S}_{\bar{v}_{m},v_{m+1}}^+}\mathds{1}_{\bigcup_{i=1}^{m-1}(V_m^{i'}\cup V_{m+1}^{i'})}(\omega_1)\,d\omega_1\lesssim \arcsin\beta+m\beta^{-d}\min\left\{1,\left(\frac{\eta}{r}\right)^{\frac{d-1}{2}}\right\}.
\end{equation}

Hence, \eqref{integral expression for post} yields
\begin{equation}\label{post collisional estimate with beta}
\begin{aligned}
|\mathcal{B}_{m}^{2+}(\bar{Z}_m)|&\lesssim \eta^d+\int_0^{2R}\int_{S_r^{d-1}(\bar{v}_{m})}\arcsin\beta+m\beta^{-d}\min\left\{1,\left(\frac{\eta}{r}\right)^{\frac{d-1}{2}}\right\}\,dv_{m+1}\,dr\\
&\lesssim \eta^d+\int_0^{2R}r^{d-1}\left(\arcsin\beta+m\beta^{-d}\min\left\{1,\left(\frac{\eta}{r}\right)^{\frac{d-1}{2}}\right\}\right)\,dr\\
&\lesssim \eta^{d}+mR^{d}\left(\arcsin\beta+\beta^{-d}\eta^{\frac{d-1}{2}}\right)\\
&\lesssim mR^{d}\left(\beta+\beta^{-d}\eta^{\frac{d-1}{2}}\right),
\end{aligned} 
\end{equation}
after using an estimate similar to \eqref{estimate with min} and the fact that $\eta<<1$, $m\geq 1$, $\beta<<1$. Choosing $\beta=\eta^{\frac{d-1}{2d+2}}$, we obtain
\begin{equation}\label{measure B+}
|\mathcal{B}_{m}^{2+}(\bar{Z}_m)|\lesssim m R^{d}\eta^{\frac{d-1}{2d+2}}.
\end{equation}

Combining \eqref{B double}, \eqref{B- measure}, \eqref{measure B+}, and the fact  $\eta<<1$, we obtain the required estimate.
\end{proof}

\subsection{Stability under ternary adjunction}\label{subsec:ternary}
Now, we prove  Proposition \ref{bad set triple} and Proposition \ref{bad set triple measure} which will be the inductive step and the corresponding measure estimate of our proof for ternary adjunction of particles. To derive Proposition  \ref{bad set triple} and Proposition \ref{bad set triple measure}, in addition to results from \cite{ternary}, we develop new algebraic and geometric techniques, thanks to which   we can treat the newly formed  ternary collisional configuration  runs to a binary collision under time evolution.
\subsubsection{Ternary adjunction}
For convenience, given $v\in\mathbb{R}^d$, let us denote
\begin{equation}\label{pre-post notation triple}
\left(\mathbb{S}_1^{2d-1}\times B_R^{2d}\right)^+(v)=\big\{(\omega_1,\omega_2,v_{1},v_{2})\in\mathbb{S}_{1}^{2d-1}\times B_R^{2d}:b_3(\omega_1,\omega_2,v_{1}-v,v_{2}-v)>0\big\},
\end{equation}
where $b_3$ is the ternary cross-section given in \eqref{cross}.

Recall from \eqref{back-wards flow} that given $m\in\mathbb{N}$ and $Z_m=(X_m,V_m)\in\mathbb{R}^{2dm}$, we denote the backwards in time free flow as $Z_m(t)=(X_m-tV_m,V_m),\quad t\geq 0.$
\begin{proposition}\label{bad set tilde triple}
Consider parameters $\alpha,\epsilon_0,R,\eta,\delta$ as in \eqref{choice of parameters} and $\epsilon_3<<\alpha$. Let $m\in\mathbb{N}$, $\bar{Z}_m=(\bar{X}_m,\bar{V}_m)\in G_m(\epsilon_0,0)$, $\ell\in\{1,...,m\}$, and $X_m\in B_{\alpha/2}^{dm}(\bar{X}_m)$. Let us denote
\begin{equation*}
\mathcal{F}_{m+2}^\ell=\left\{(i,j)\in\left\{1,...,m+2\right\} \times \left\{1,...,m+2\right\} : i\neq\ell,\text{ } i\leq\min\left\{j,m\right\}\right\}.
\end{equation*}
Then there is a subset $\widetilde{\mathcal{B}}_{\ell}^3(\bar{Z}_m)\subseteq (\mathbb{S}_1^{2d-1}\times B_R^{2d})^+(\bar{v}_m)$ such that:
\begin{enumerate}[(i)]
\item For any $(\omega_1,\omega_2,v_{m+1},v_{m+2})\in (\mathbb{S}_1^{2d-1}\times B_R^{2d})^+(\bar{v}_{m})\setminus\widetilde{\mathcal{B}}_{\ell}^3(\bar{Z}_m)$, one has:
\begin{equation}\label{pre-tilde}
\begin{aligned}
d_2(x_i(t),x_j(t))&>\sqrt{2}\epsilon_3,\quad\forall (i,j)\in\mathcal{F}_{m+2}^\ell,\quad\forall t\geq 0,\\
d_3(x_\ell(t);x_{m+1}(t),x_{m+2}(t))&>\sqrt{2}\epsilon_3,\quad\forall t\geq 0,\\
Z_{m+2}&\in G_{m+2}(\epsilon_0/2,\delta),\\
\bar{Z}_{m+2}&\in G_{m+2}(\epsilon_0,\delta).
\end{aligned}
\end{equation}
where 
\begin{equation*}
\begin{aligned}
&Z_{m+2}=(x_1,...,x_\ell,...,x_m,x_{m+1},x_{m+2},\bar{v}_1,...,\bar{v}_\ell,...,\bar{v}_m,v_{m+1},v_{m+2}),\\
&x_{m+i}=x_{\ell}+\sqrt{2}\epsilon_3\omega_i,\quad\forall i\in\{1,2\},\\
&\bar{Z}_{m+2}=(\bar{x}_1,...,\bar{x}_\ell,...,\bar{x}_m,\bar{x}_{m},\bar{x}_{m},\bar{v}_1,...,\bar{v}_\ell,...,\bar{v}_m,v_{m+1},v_{m+2}),\\
\end{aligned}
\end{equation*}
\item For any $(\omega_1,\omega_2,v_{m+1},v_{m+2})\in (\mathbb{S}_1^{2d-1}\times B_R^{2d})^+(\bar{v}_{\ell})\setminus\widetilde{\mathcal{B}}_{\ell}^3(\bar{Z}_m)$, one has:
\begin{equation}\label{post-tilde}
\begin{aligned}
d_2(x_i(t),x_j(t))&>\sqrt{2}\epsilon_3,\quad\forall (i,j)\in\mathcal{F}_{m+2}^\ell,\quad\forall t\geq 0,\\
d_3(x_\ell(t);x_{m+1}(t),x_{m+2}(t))&>\sqrt{2}\epsilon_3,\quad\forall t\geq 0,\\
Z_{m+2}^*&\in G_{m+2}(\epsilon_0/2,\delta),\\
\bar{Z}_{m+2}^*&\in G_{m+2}(\epsilon_0,\delta).
\end{aligned}
\end{equation}
where 
\begin{equation*}
\begin{aligned}
&Z_{m+2}^*=(x_1,...,x_\ell,...,x_m,x_{m+1},x_{m+2},\bar{v}_1,...,\bar{v}_\ell^*,...,\bar{v}_m,v_{m+1}^*,v_{m+2}^*),\\
&x_{m+i}=x_{\ell}+\sqrt{2}\epsilon_3\omega_i,\quad\forall i\in\{1,2\},\\
&\bar{Z}_{m+2}^*=(\bar{x}_1,...,\bar{x}_\ell,...,\bar{x}_m,\bar{x}_{m},\bar{x}_{m},\bar{v}_1,...,\bar{v}_\ell^*,...,\bar{v}_m,v_{m+1}^*,v_{m+2}^*),\\
&(\bar{v}_{\ell}^*,v_{m+1}^*,v_{m+2}^*)=T_{\omega_1,\omega_2}(\bar{v}_{\ell},v_{m+1},v_{m+2}).
\end{aligned}
\end{equation*}
\end{enumerate}

There also holds the measure estimate:
\begin{equation}\label{measure estimate tilde}
|\widetilde{\mathcal{B}}_\ell^3(\bar{Z}_m)|\lesssim mR^{2d}\eta^{\frac{d-1}{4d+2}},
\end{equation}
where $|\cdot|$ denotes the product measure on $\mathbb{S}_1^{2d-1}\times B_R^{2d}$.
\end{proposition}
\begin{proof}
This Proposition follows from the statement and the proof  of Proposition 9.2. and the statement of Proposition 9.4. from \cite{ternary}.
\end{proof}
We rely on Proposition \ref{bad set tilde triple} to derive Proposition \ref{bad set triple} and Proposition \ref{bad set triple measure}. Recall  the notation from \eqref{interior phase space} 
\begin{align*}\mathring{\mathcal{D}}_{m+2,\epsilon_2,\epsilon_3}=\big\{Z_{m+2}=(X_{m+2},V_{m+2})\in\mathbb{R}^{2d(m+2)}:\text{ } d_2(x_i,x_j)>\epsilon_2,\quad\forall (i,j)\in\mathcal{I}_{m+2}^2,&\\
\text{and } d_3(x_i;x_j,x_k)>\sqrt{2}\epsilon_3,\quad\forall(i,j,k)\in\mathcal{I}_{m+2}^3\big\}&,
\end{align*}
where $\mathcal{I}_{m+2}^2,\mathcal{I}_{m+2}^3$ are given by \eqref{index 2}-\eqref{index 3} respectively.
\begin{proposition}\label{bad set triple} Consider parameters $\alpha,\epsilon_0,R,\eta,\delta$ as in \eqref{choice of parameters} and $\epsilon_2<<\eta^2\epsilon_3<<\alpha$. Let $m\in\mathbb{N}$, $\bar{Z}_m=(\bar{X}_m,\bar{V}_m)\in G_m(\epsilon_0,0)$, $\ell\in\{1,...,m\}$ and $X_m\in B_{\alpha/2}^{dm}(\bar{X}_m)$. Then there is a subset $\mathcal{B}_{\ell}^3(\bar{Z}_m)\subseteq (\mathbb{S}_1^{2d-1}\times B_R^{2d})^+(\bar{v}_\ell)$ such that:
\begin{enumerate}[(i)]
\item For any $(\omega_1,\omega_2,v_{m+1},v_{m+2})\in (\mathbb{S}_1^{2d-1}\times B_R^{2d})^+(\bar{v}_{\ell})\setminus\mathcal{B}_{\ell}^3(\bar{Z}_m)$, one has:
\begin{align}
Z_{m+2}(t)&\in\mathring{\mathcal{D}}_{m+2,\epsilon_2,\epsilon_3},\quad\forall t\geq 0,\label{in phase pre}\\
Z_{m+2}&\in G_{m+2}(\epsilon_0/2,\delta)\label{epsilon/2 pre}\\
\bar{Z}_{m+2}&\in G_{m+2}(\epsilon_0,\delta),\label{epsilon pre}
\end{align}
where 
\begin{equation}\label{pre-collisional notation ternary}
\begin{aligned}
&Z_{m+2}=(x_1,...,x_\ell,...,x_m,x_{m+1},x_{m+2},\bar{v}_1,...,\bar{v}_\ell,...,\bar{v}_m,v_{m+1},v_{m+2}),\\
&x_{m+i}=x_{\ell}-\sqrt{2}\epsilon_3\omega_i,\quad\forall i\in\{1,2\},\\
&\bar{Z}_{m+2}=(\bar{x}_1,...,\bar{x}_\ell,...,\bar{x}_m,\bar{x}_{m},\bar{x}_{m},\bar{v}_1,...,\bar{v}_\ell,...,\bar{v}_m,v_{m+1},v_{m+2}),\\
\end{aligned}
\end{equation}
\item For any $(\omega_1,\omega_2,v_{m+1},v_{m+2})\in (\mathbb{S}_1^{2d-1}\times B_R^{2d})^+(\bar{v}_{\ell})\setminus\mathcal{B}_{\ell}^3(\bar{Z}_m)$, one has:
\begin{align}
Z_{m+2}^*(t)&\in\mathring{\mathcal{D}}_{m+2,\epsilon_2,\epsilon_3},\quad\forall t\geq 0,\label{in phase post}\\
Z_{m+2}^*&\in G_{m+2}(\epsilon_0/2,\delta),\label{epsilon/2 post}\\
\bar{Z}_{m+2}^*&\in G_{m+2}(\epsilon_0,\delta),\label{epsilon post}
\end{align}
where 
\begin{equation}\label{post-collisional notation ternary}
\begin{aligned}
&Z_{m+2}^*=(x_1,...,x_\ell,...,x_m,x_{m+1},x_{m+2},\bar{v}_1,...,\bar{v}_\ell^*,...,\bar{v}_m,v_{m+1}^*,v_{m+2}^*),\\
&x_{m+i}=x_{\ell}+\sqrt{2}\epsilon_3\omega_i,\quad\forall i\in\{1,2\},\\
&\bar{Z}_{m+2}^*=(\bar{x}_1,...,\bar{x}_\ell,...,\bar{x}_m,\bar{x}_{m},\bar{x}_{m},\bar{v}_1,...,\bar{v}_\ell^*,...,\bar{v}_m,v_{m+1}^*,v_{m+2}^*),\\
&(\bar{v}_{\ell}^*,v_{m+1}^*,v_{m+2}^*)=T_{\omega_1,\omega_2}(\bar{v}_{\ell},v_{m+1},v_{m+2}).
\end{aligned}
\end{equation}
\end{enumerate}
\end{proposition}
\begin{proof}
By symmetry we may assume that $\ell=m$.  Recall the set $\widetilde{\mathcal{B}}_m^3(\bar{Z}_m)$ from Proposition \ref{bad set tilde triple} satisfying \eqref{pre-tilde}-\eqref{post-tilde}.

We will construct a set $\mathcal{A}_m(\bar{Z}_m)\subseteq(\mathbb{S}_1^{2d-1}\times B_R^{2d})^+(\bar{v}_m)$, such that for any $(\omega_1,\omega_2,v_{m+1},v_{m+2})\in (\mathbb{S}_1^{2d-1}\times B_R^{2d})^+(\bar{v}_m)\setminus \mathcal{A}_m(\bar{Z}_m)$:
\begin{itemize}
\item Using notation from \eqref{pre-collisional notation ternary} for the precollisional case, we have
\begin{equation}\label{pre-collisional claim}
|x_i(t)-x_j(t)|>\epsilon_2,\quad\forall t\geq 0,\quad \forall i,j\in\left\{m,m+1,m+2\right\} \text{ with } i<j .
\end{equation}
\item Using notation from \eqref{post-collisional notation ternary} for the postcollisional case, we have
\begin{equation}\label{post-collisional claim}
|x_i(t)-x_j(t)|>\epsilon_2,\quad\forall t\geq 0,\quad \forall i,j\in\left\{m,m+1,m+2\right\} \text{ with } i<j .
\end{equation}
\end{itemize}
Then thanks to Proposition \ref{bad set tilde triple} and \eqref{pre-collisional claim}-\eqref{post-collisional claim}, the set
$$\mathcal{B}_m^3(\bar{Z}_m):=\widetilde{\mathcal{B}}_m^3(\bar{Z}_m)\cup \mathcal{A}_m(\bar{Z}_m),$$
will satisfy \eqref{in phase pre}-\eqref{epsilon pre}, \eqref{in phase post}-\eqref{epsilon post}. 
Let us introduce the following notation:
\begin{equation}\label{gamma def}
\gamma:=\frac{\epsilon_2}{\epsilon_3}<<\eta^2,\quad\text{since }\epsilon_2<<\eta^2\epsilon_3, \text{ by assumption,}
\end{equation}
and
\begin{equation}\label{gamma'}
\gamma'=\left(1-\frac{\gamma}{2}\right)^{1/2}<1.
\end{equation}

\textbf{Construction of the set satisfying \eqref{pre-collisional claim}}: Here we use notation from \eqref{pre-collisional notation ternary}. We distinguish the following cases:

$\bullet$ Case $(i,j)=(m,m+1)$: Consider $t\geq 0$. We have
\begin{equation}\label{expr for m,m+1}
\begin{aligned}
|x_i(t)-x_j(t)|^2&=|x_{m}(t)-x_{m+1}(t)|^2\\
&=|\sqrt{2}\epsilon_3\omega_1+(v_{m+1}-\bar{v}_m)t|^2\\
&=2\epsilon_3^2|\omega_1|^2+2\sqrt{2}\epsilon_3\langle\omega_1,v_{m+1}-\bar{v}_m\rangle t+|v_{m+1}-\bar{v}_m|^2t^2.
\end{aligned}
\end{equation}
We define the sets
\begin{align}
\Omega_{1}&=\{(\omega_1,\omega_2,v_{m+1},v_{m+2})\in\mathbb{S}_1^{2d-1}\times B_R^{2d}:|\omega_1|\leq\sqrt{\gamma}\}\label{omega_1},\\
A_{m,m+1}&=\{(\omega_1,\omega_2,v_{m+1},v_{m+2})\in\mathbb{S}_1^{2d-1}\times B_R^{2d}:\left|\langle\omega_1,v_{m+1}-\bar{v}_m\rangle\right|\geq\gamma'|\omega_1||v_{m+1}-\bar{v}_m|\}\label{A_m,m+1}.
\end{align}
Consider the second degree polynomial in $t$:
\begin{equation}
P(t)=(2-\gamma)\epsilon_3^2|\omega_1|^2+2\sqrt{2}\epsilon_3\langle\omega_1,v_{m+1}-\bar{v}_m\rangle t+|v_{m+1}-\bar{v}_m|^2t^2
\end{equation}
Let $(\omega_1,\omega_2,v_{m+1},v_{m+2})\in(\mathbb{S}_1^{2d-1}\times B_R^{2d})\setminus(\Omega_1\cup A_{m,m+1})$. The polynomial $P$ has discriminant
\begin{align*}
\Delta&=8\epsilon_3^2|\langle\omega_1,v_{m+1}-\bar{v}_m\rangle|^2-4(2-\gamma)\epsilon_3^2|\omega_1|^2|v_{m+1}-\bar{v}_m|^2\\
&=8\epsilon_3^2|\langle\omega_1,v_{m+1}-\bar{v}_m\rangle|^2-8\gamma'^2\epsilon_3^2|\omega_1|^2|v_{m+1}-\bar{v}_m|^2\\
&=8\epsilon_3^2\left(|\langle\omega_1,v_{m+1}-\bar{v}_m\rangle|^2-\gamma'^2|\omega_1|^2|v_{m+1}-\bar{v}_m|^2\right)\\
&<0
\end{align*}
since $(\omega_1,\omega_2,v_{m+1},v_{m+2})\notin A_{m,m+1}$.
Since $\gamma<<1$, we obtain 
$P(t)>0$, for all $ t\geq 0,$
or in other words
\begin{equation}\label{before omega}
 2\epsilon_3^2|\omega_1|^2+2\sqrt{2}\epsilon_3\langle\omega_1,v_{m+1}-\bar{v}_m\rangle t+|v_{m+1}-\bar{v}_m|^2t^2>\gamma\epsilon_3^2|\omega_1|^2.
\end{equation}
Since $(\omega_1,\omega_2,v_{m+1},v_{m+2})\notin \Omega_1$, expressions \eqref{expr for m,m+1}, \eqref{before omega} yield
\begin{equation}
|x_m(t)-x_{m+1}(t)|^2>\gamma\epsilon_3^2|\omega_1|^2>\gamma^2\epsilon_3^2=\epsilon_2^2.
\end{equation}
Therefore for any $(\omega_1,\omega_2,v_{m+1},v_{m+2})\in(\mathbb{S}_1^{2d-1}\times B_R^{2d})\setminus(\Omega_1\cup A_{m,m+1})$, we have
$$|x_m(t)-x_{m+1}(t)|>\epsilon_2,\quad\forall t\geq 0.$$

$\bullet$ Case $(i,j)=(m,m+2)$: We follow a similar argument using the sets
\begin{align}
\Omega_{2}&=\{(\omega_1,\omega_2,v_{m+1},v_{m+2})\in\mathbb{S}_1^{2d-1}\times B_R^{2d}:|\omega_2|\leq\sqrt{\gamma}\}\label{omega_2},\\
A_{m,m+2}&=\{(\omega_1,\omega_2,v_{m+1},v_{m+2})\in\mathbb{S}_1^{2d-1}\times B_R^{2d}:\left|\langle\omega_2,v_{m+2}-\bar{v}_m\rangle\right|\geq\gamma'|\omega_2||v_{m+2}-\bar{v}_m|\}\label{A_m,m+2},
\end{align}
to conclude that for all $(\omega_1,\omega_2,v_{m+1},v_{m+2})\in(\mathbb{S}_1^{2d-1}\times B_R^{2d})\setminus (\Omega_2\cup A_{m,m+2})$, we have
$$|x_{m+2}(t)-x_{m}(t)|>\epsilon_2,\quad\forall t\geq 0.
$$

$\bullet$ Case $(i,j)=(m+1,m+2)$: We follow a similar argument using the sets 
\begin{align}
\Omega_{1,2}&=\{(\omega_1,\omega_2,v_{m+1},v_{m+2})\in\mathbb{S}_1^{2d-1}\times B_R^{2d}:|\omega_1-\omega_2|\leq\sqrt{\gamma}\},\label{omega_1,2}\\
B_{m+1,m+2}&=\{(\omega_1,\omega_2,v_{m+1},v_{m+2})\in\mathbb{S}_1^{2d-1}\times B_R^{2d}:\nonumber\\
&\hspace{0.8cm}\left|\langle\omega_1-\omega_2,v_{m+1}-v_{m+2}\rangle\right|\geq\gamma'|\omega_1-\omega_2||v_{m+1}-v_{m+2}|\}\label{A_m+1,m+2}.
\end{align}
to conclude that for all $(\omega_1,\omega_2,v_{m+1},v_{m+2})\in(\mathbb{S}_1^{2d-1}\times B_R^{2d})\setminus (\Omega_{1,2}\cup B_{m+1,m+2})$, we have
$$|x_{m+1}(t)-x_{m+2}(t)|>\epsilon_2,\quad\forall t\geq 0.$$

Defining 
\begin{equation}\label{A-}
\mathcal{A}_m^-(\bar{Z}_m)=\Omega_1\cup\Omega_2\cup\Omega_{1,2}\cup A_{m,m+1}\cup A_{m,m+2}\cup B_{m+1,m+2},
\end{equation}
we obtain that \eqref{pre-collisional claim} holds for $(\omega_1,\omega_2,v_{m+1},v_{m+2})\in(\mathbb{S}_1^{2d-1}\times B_R^{2d})\setminus\mathcal{A}_m^-(\bar{Z}_m)$.

\textbf{Construction of the set satisfying \eqref{post-collisional claim}}: Here we use notation from \eqref{post-collisional notation ternary}. We distinguish the following cases:

$\bullet$ Case $(i,j)=(m,m+1)$: We follow a similar argument to the precollisional case, using the set $\Omega_1$, defined in \eqref{omega_1}, and the set
\begin{align}
A_{m,m+1}^*&=\{(\omega_1,\omega_2,v_{m+1},v_{m+2})\in\mathbb{S}_1^{2d-1}\times B_R^{2d}:\left|\langle\omega_1,v_{m+1}^*-\bar{v}_{m}^*\rangle\right|\geq\gamma'|\omega_1||v_{m+1}^*-\bar{v}_{m}^*|\}\label{A_m,m+1*},
\end{align}
to conclude that for all $(\omega_1,\omega_2,v_{m+1},v_{m+2})\in(\mathbb{S}_1^{2d-1}\times B_R^{2d})\setminus (\Omega_2\cup A_{m,m+1}^*),$ we have
$$|x_{m+1}(t)-x_{m}(t)|>\epsilon_2,\quad\forall t\geq 0.$$
$\bullet$ Case $(i,j)=(m,m+2)$: We follow a similar argument to the precollisional case, using the set $\Omega_2$, defined in \eqref{omega_2}, and the set
\begin{align}
A_{m,m+2}^*&=\{(\omega_1,\omega_2,v_{m+1},v_{m+2})\in\mathbb{S}_1^{2d-1}\times B_R^{2d}:\left|\langle\omega_2,v_{m+2}^*-\bar{v}_{m}^*\rangle\right|\geq\gamma'|\omega_2||v_{m+2}^*-\bar{v}_{m}^*|\}\label{A_m,m+2*},
\end{align}
to conclude that for all $(\omega_1,\omega_2,v_{m+1},v_{m+2})\in(\mathbb{S}_1^{2d-1}\times B_R^{2d})\setminus (\Omega_2\cup A_{m,m+2}^*)$, we have
$$|x_{m+2}(t)-x_{m}(t)|>\epsilon_2,\quad\forall t\geq 0.$$

$\bullet$ Case $(i,j)=(m+1,m+2)$: We follow a similar argument to the precollisional case, using the set $\Omega_{1,2}$, defined in \eqref{omega_1,2}, and the set
\begin{align}
B_{m+1,m+2}^*&=\{(\omega_1,\omega_2,v_{m+1},v_{m+2})\in\mathbb{S}_1^{2d-1}\times B_R^{2d}:\\
&\hspace{0.8cm}\left|\langle\omega_1-\omega_2,v_{m+1}^*-v_{m+2}^*\rangle\right|\geq\gamma'|\omega_1-\omega_2||v_{m+1}^*-v_{m+2}^*|\},\label{A_m+1,m+2*}
\end{align}
to conclude that for all $(\omega_1,\omega_2,v_{m+1},v_{m+2})\in(\mathbb{S}_1^{2d-1}\times B_R^{2d})\setminus (\Omega_2\cup B_{m+1,m+2}^*)$, we have
$$|x_{m+1}(t)-x_{m+2}(t)|>\epsilon_2,\quad\forall t\geq 0.$$

Defining 
\begin{equation}\label{A+}
\mathcal{A}_m^+(\bar{Z}_m)=\Omega_1\cup\Omega_2\cup\Omega_{1,2}\cup A_{m,m+1}^*\cup A_{m,m+2}^*\cup B_{m+1,m+2}^*,
\end{equation}
we obtain that \eqref{post-collisional claim} holds for $(\omega_1,\omega_2,v_{m+1},v_{m+2})\in(\mathbb{S}_1^{2d-1}\times B_R^{2d})\setminus\mathcal{A}_m^+(\bar{Z}_m)$.

Defining 
\begin{equation}\label{A}
\mathcal{A}_m(\bar{Z}_m)=(\mathbb{S}_1^{2d-1}\times B_R^{2d})^+(\bar{v}_m)\cap\left(\mathcal{A}_m^-(\bar{Z}_m)\cup \mathcal{A}_m^+(\bar{Z}_m)\right),
\end{equation}
\eqref{pre-collisional claim}-\eqref{post-collisional claim} hold for any $(\omega_1,\omega_2,v_{m+1},v_{m+2})\in(\mathbb{S}_1^{2d-1}\times B_R^{2d})^+(\bar{v}_m)\setminus\mathcal{A}_m(\bar{Z}_m)$.

The set 
\begin{equation}
\mathcal{B}_m^3(\bar{Z}_m)=\widetilde{\mathcal{B}}_m^3(\bar{Z}_m)\cup\mathcal{A}_m(\bar{Z}_m),
\end{equation}
satisfies \eqref{in phase pre}-\eqref{epsilon pre}, \eqref{in phase post}-\eqref{epsilon post}, thus it is the set we need to conclude the proof.

\end{proof}

\subsubsection{Measure estimate for ternary adjunction}
We now provide the corresponding measure estimate for the set $\mathcal{B}_\ell^3(\bar{Z}_m)$ appearing in Proposition \ref{bad set triple}. To estimate the measure of this set, we will strongly rely on the results of Section \ref{sec:geometric}.
\begin{proposition}\label{bad set triple measure} Consider parameters $\alpha,\epsilon_0,R,\eta,\delta$ as in \eqref{choice of parameters} and $\epsilon_2<<\eta^2\epsilon_3<<\alpha$. Let $m\in\mathbb{N}$, $\bar{Z}_m\in G_m(\epsilon_0,0)$, $\ell\in\{1,...,m\}$ and $\mathcal{B}_{\ell}^3(\bar{Z}_m)$ be the set appearing in the statement of Proposition \ref{bad set triple}. Then the following measure estimate holds:
\begin{equation*}
\left|\mathcal{B}_{\ell}^3(\bar{Z}_m)\right|\lesssim mR^{2d}\eta^{\frac{d-1}{4d+2}},
\end{equation*}
where $|\cdot|$ denotes the product measure on $\mathbb{S}_1^{2d-1}\times B_R^{2d}$.
\end{proposition}
\begin{proof}
By symmetry, we may assume $\ell=m$. Recall that 
\begin{equation}\label{B}
\mathcal{B}_m^3(\bar{Z}_m)=\widetilde{\mathcal{B}}_m^3(\bar{Z}_m)\cup\mathcal{A}_m(\bar{Z}_m),
\end{equation}
where $\widetilde{\mathcal{B}}_m^3(\bar{Z}_m)$ is given by Proposition \ref{bad set tilde triple} and $\mathcal{A}_m(\bar{Z}_m)$ is given by \eqref{A}. Estimate \eqref{measure estimate tilde} yields
\begin{equation}\label{estimate of wide tilde}
|\widetilde{\mathcal{B}}_{m}^3(\bar{Z}_m)|\lesssim mR^{2d}\eta^{\frac{d-1}{4d+2}},
\end{equation}
so it suffices to estimate the measure of $\mathcal{A}_m(\bar{Z}_m)$. By \eqref{A}, it suffices to estimate the measure of $\mathcal{A}_m^-(\bar{Z}_m)$ and $\mathcal{A}_m^+(\bar{Z}_m)$ which are given by \eqref{A-}, \eqref{A+}  respectively.

Let us recall the notation from \eqref{gamma def}-\eqref{gamma'}:
$$\gamma=\frac{\epsilon_2}{\epsilon_3}<<\eta^2,\quad\gamma'=\sqrt{1-\frac{\gamma}{2}}.$$
\textbf{Estimate of $\mathcal{A}_m^-(\bar{Z}_m)$:} Recall from \eqref{A-} that
\begin{equation}\label{A- measure representation}
\mathcal{A}_m^-(\bar{Z}_m)=\Omega_1\cup\Omega_2\cup\Omega_{1,2}\cup A_{m,m+1}\cup A_{m,m+2}\cup B_{m+1,m+2},
\end{equation}
where $\Omega_1,A_{m,m+1}$ are given by \eqref{omega_1}-\eqref{A_m,m+1}, $\Omega_2,A_{m,m+2}$ by \eqref{omega_2}-\eqref{A_m,m+2} and $\Omega_{1,2}, B_{m+1,m+2}$ are given by \eqref{omega_1,2}-\eqref{A_m+1,m+2}.

$\bullet$ Estimate for $\Omega_1, \Omega_2$: Without loss of generality, it suffices to   estimate the measure of $\Omega_1$. Recalling notation from \eqref{cube parameters 1}, Fubini's Theorem and Lemma \ref{estimate of cubes} yield
\begin{equation}\label{measure omega_1}
|\Omega_1|=\int_{B_R^{2d}}\int_{\mathbb{S}_1^{2d-1}}\mathds{1}_{M_1(\sqrt{\gamma})}\,d\omega_1\,d\omega_2\,dv_{m+1}\,dv_{m+2}\lesssim R^{2d}\gamma^{d/2},
\end{equation}
A symmetric argument yields 
\begin{equation}\label{measure omega_2}
|\Omega_2|\lesssim R^{2d}\gamma^{d/2},
\end{equation}

$\bullet$ Estimate for $\Omega_{1,2}$: Recalling notation from \eqref{strip},  \eqref{omega_1,2} yields
$$\Omega_{1,2}=(\mathbb{S}_1^{2d-1}\cap W_{\sqrt{\gamma}}^{2d})\times B_R^{2d},$$
Therefore, Fubini's Theorem and Lemma \ref{strip lemma} imply
\begin{equation}\label{measure omega_1,2}
|\Omega_{1,2}|= \int_{B_R^{2d}}\int_{\mathbb{S}_1^{2d-1}}\mathds{1}_{W_{\sqrt{\gamma}}^{2d}}\,d\omega_1\,d\omega_2\,dv_{m+1}\,dv_{m+2}\lesssim R^{2d}\gamma^{\frac{d-1}{4}}.
\end{equation}

$\bullet$ Estimate for $A_{m,m+1}$: Recalling notation from \eqref{shell parameters}, the set $A_{m,m+1}$, which was defined in \eqref{A_m,m+1}, can be written as
\begin{equation*}
A_{m,m+1}=\left\{(\omega_1,\omega_2,v_{m+1},v_{m+2})\in\mathbb{S}_1^{2d-1}\times B_R^{2d}:\omega_1\in S(\gamma',v_{m+1}-\bar{v}_m)\right\}
\end{equation*}
Therefore, the representation of the $(2d-1)$- unit sphere \eqref{representation of sphere for fixed omega_1} and Lemma \ref{shell estimate} yield
\begin{align}
|A_{m,m+1}|&\leq\int_{B_R^{2d}}\int_{B_1^d}\int_{\mathbb{S}_{\sqrt{1-|\omega_2|^2}}^{d-1}}\mathds{1}_{S(\gamma',v_{m+1}-\bar{v}_m)}\,d\omega_1\,d\omega_2\,dv_{m+1}\,dv_{m+2}\nonumber\\
&\lesssim R^{2d}\arccos\gamma'\nonumber\\
&=R^{2d}\arccos\sqrt{1-\frac{\gamma}{2}}.\label{measure A_m,m+1}
\end{align}

$\bullet$ Estimate for $A_{m,m+2}$: We follow a similar argument as in the previous case to obtain
\begin{equation}\label{measure A_m,m+2}
|A_{m,m+2}|\lesssim R^{2d}\arccos\sqrt{1-\frac{\gamma}{2}}.
\end{equation}

$\bullet$ Estimate for $B_{m+1,m+2}$: Recalling notation from \eqref{difference shell parameters},
 \eqref{A_m+1,m+2} yields
\begin{equation*}B_{m+1,m+2}=\{(\omega_1,\omega_2,v_{m+1},v_{m+2})\in\mathbb{S}_1^{2d-1}\times B_R^{2d}:(\omega_1,\omega_2)\in N(\gamma',v_{m+1}-v_{m+2})\}.
\end{equation*}
Therefore, using Lemma \ref{estimate of difference in shell}, we obtain
\begin{align}
|B_{m+1,m+2}|&=\int_{B_R^{2d}}\int_{\mathbb{S}_1^{2d-1}}\mathds{1}_{N(\gamma',v_{m+1}-v_{m+2})}(\omega_1,\omega_2)\,d\omega_1\,d\omega_2\,dv_{m+1}\,dv_{m+2}\nonumber\\
&\lesssim R^{2d}\arccos\gamma'\nonumber\\
&=R^{2d}\arccos\sqrt{1-\frac{\gamma}{2}}.\label{measure for A_m+1,m+2}
\end{align}
Using \eqref{A- measure representation} and estimates \eqref{measure omega_1}-\eqref{measure for A_m+1,m+2}, we obtain
\begin{equation}\label{measure A-}
|\mathcal{A}_m^-(\bar{Z}_m)|\lesssim R^{2d}\left(\gamma^{d/2}+\gamma^{\frac{d-1}{4}}+\arccos\sqrt{1-\frac{\gamma}{2}}\right).
\end{equation}

\textbf{Estimate of $\mathcal{A}_m^+(\bar{Z}_m)$:} Recall from \eqref{A+} that
\begin{equation}\label{A+ measure representation}
\mathcal{A}_m^+(\bar{Z}_m)=\Omega_1\cup\Omega_2\cup\Omega_{1,2}\cup A_{m,m+1}^*\cup A_{m,m+2}^*\cup B_{m+1,m+2}^*,
\end{equation}
where $\Omega_1$, $\Omega_2$, $\Omega_{1,2}$, $A_{m,m+1}^*$, $A_{m,m+2}^*$, $B_{m+1,m+2}^*$ are given by \eqref{omega_1}, \eqref{omega_2}, \eqref{omega_1,2}, \eqref{A_m,m+1*}-\eqref{A_m+1,m+2*} respectively.
We already have estimates for $\Omega_1$, $\Omega_2$, $\Omega_{1,2}$ from \eqref{measure omega_1}-\eqref{measure omega_1,2}, hence it suffices to derive estimates for $A_{m,m+1}^*$, $A_{m,m+2}^*$, $B_{m+1,m+2}^*$.

For the rest of the proof we consider a parameter $0<\beta<<1$ which will be chosen later in terms of $\eta$, see \eqref{choice of beta in terms of eta}.

$\bullet$ Estimate for $A_{m,m+1}^*$: Recall from \eqref{A_m,m+1*} the set
\begin{equation}\label{A_m,m+1 measure representation}
 A_{m,m+1}^*=\left\{(\omega_1,\omega_2,v_{m+1},v_{m+2})\in\mathbb{S}_1^{2d-1}\times B_R^{2d}:|\langle\omega_1,v_{m+1}^*-\bar{v}_m^*\rangle|\geq\gamma'|\omega_1||v_{m+1}^*-\bar{v}_m^*|\right\}.
 \end{equation}
But for any $(\omega_1,\omega_2,v_{m+1},v_{m+2})\in\mathbb{S}_1^{2d-1}\times B_R^{2d}$, the ternary collisional law \eqref{formulas ternary} implies
$$v_{m+1}^*-\bar{v}_m^*=v_{m+1}-\bar{v}_m-2c_{\omega_1,\omega_2,\bar{v}_m,v_{m+1},v_{m+2}}\omega_1-c_{\omega_1,\omega_2,\bar{v}_m,v_{m+1},v_{m+2}}\omega_2,$$
where
\begin{equation}\label{c A_m,m+1*} c_{\omega_1,\omega_2,\bar{v}_m,v_{m+1},v_{m+2}}=\frac{\langle\omega_1,v_{m+1}-\bar{v}_m\rangle+\langle\omega_2,v_{m+2}-\bar{v}_m\rangle}{1+\langle\omega_1,\omega_2\rangle}.
\end{equation}
For convenience, we denote
$$c:=c_{\omega_1,\omega_2,\bar{v}_m,v_{m+1},v_{m+2}}.$$
Therefore, by \eqref{A_m,m+1 measure representation}, we may write
\begin{align*}
A_{m,m+1}^*&=\{(\omega_1,\omega_2,v_{m+1},v_{m+2})\in\mathbb{S}_1^{2d-1}\times B_R^{2d}:\\
&\hspace{0.8cm}|\langle\omega_1,v_{m+1}-\bar{v}_m-2c\omega_1-c\omega_2\rangle|\geq\gamma'|\omega_1||v_{m+1}-\bar{v}_m-2c\omega_1-c\omega_2|\}.
\end{align*}
By Fubini's Theorem we have
\begin{equation}\label{total integral A_m,m+1*}
|A_{m,m+1}^*|\leq\int_{\mathbb{S}_1^{2d-1}\times B_R^d}\int_{B_R^d}\mathds{1}_{V_{\omega_1,\omega_2,v_{m+2}}^{m,m+1}}(v_{m+1})\,dv_{m+1}\,d\omega_1\,d\omega_2\,dv_{m+2}
\end{equation}
where given $(\omega_1,\omega_2,v_{m+2})\in\mathbb{S}_1^{2d-1}\times B_R^d$ we write
\begin{equation}\label{V_m,m+1 proof}
V_{\omega_1,\omega_2,v_{m+2}}^{m,m+1}=\left\{v_{m+1}\in B_R^d:(\omega_1,\omega_2,v_{m+1},v_{m+2})\in A_{m,m+1}^*\right\}.
\end{equation}
Recall from \eqref{annulus I1} the set
\begin{equation}\label{bad annulus 1}
I_{1}=\left\{(\omega_1,\omega_2)\in\mathbb{R}^{2d}\left|1-2\left|\omega_1\right|^2\right|\leq 2\beta\right\}.
\end{equation}
Using \eqref{total integral A_m,m+1*}, we obtain
\begin{equation}\label{decomposed integral A_m,m+1*}
|A_{m,m+1}^*|= \widetilde{I}_1+\widetilde{I}_1',
\end{equation}
where
\begin{align}
\widetilde{I}_1&=\int_{(\mathbb{S}_1^{2d-1}\cap I_1)\times B_R^d}\int_{B_R^d}\mathds{1}_{V_{\omega_1,\omega_2,v_{m+2}}^{m,m+1}}(v_{m+1})\,dv_{m+1}\,d\omega_1\,d\omega_2\,dv_{m+2},\label{I_1, A_m,m+1}\\
\widetilde{I}_1'&=\int_{(\mathbb{S}_1^{2d-1}\setminus I_1)\times B_R^d}\int_{B_R^d}\mathds{1}_{V_{\omega_1,\omega_2,v_{m+2}}^{m,m+1}}(v_{m+1})\,dv_{m+1}\,d\omega_1\,d\omega_2\,dv_{m+2}.\label{I, A_m,m+1*}
\end{align}

We treat each of the terms in \eqref{decomposed integral A_m,m+1*} separately.

\textit{Estimate for} $\widetilde{I}_1$: By \eqref{I_1, A_m,m+1}, Fubini's Theorem and Lemma \ref{estimate on annulus I_1}, we obtain
\begin{equation}\label{estimate for I_1}
\widetilde{I}_1\lesssim R^{2d}\int_{\mathbb{S}_1^{2d-1}}\mathds{1}_{I_1}\,d\omega_1\,d\omega_2\lesssim R^{2d}\beta.
\end{equation}

\textit{Estimate for} $\widetilde{I}_1'$: Let us fix $(\omega_1,\omega_2,v_{m+2})\in (\mathbb{S}_1^{2d-1}\setminus I_1)\times B_R^d$. We define the smooth map $F_{\omega_1,\omega_2,v_{m+2}}^1:B_R^d\to\mathbb{R}^d$, by:
\begin{equation}\label{definition F1}
F_{\omega_1,\omega_2,v_{m+2}}^1(v_{m+1}):=v_{m+1}^*-\bar{v}_{m}^*=v_{m+1}-\bar{v}_m-2c\omega_1-c\omega_2,
\end{equation}
where $c$ is given by \eqref{c A_m,m+1*}.

We are showing that we may change variables under $F_{\omega_1,\omega_2,v_{m+2}}^1$, as long as $(\omega_1,\omega_2,v_{m+1})\in(\mathbb{S}_1^{2d-1}\setminus I_1)\times B_R^d$ i.e. we are showing that $F_{\omega_1,\omega_2,v_{m+2}}^1$ has non-zero Jacobian and is injective. In particular we will see that the Jacobian is bounded from below by $\beta$.

 We first show the Jacobian has a lower bound $\beta$ . Differentiating with respect to $v_{m+1}$, we obtain
$$\frac{\partial F_{\omega_1,\omega_2,v_{m+2}}^1 }{\partial v_{m+1}}=I_d+(-2\omega_1-\omega_2)\nabla_{v_{m+1}}^Tc.$$
Recalling \eqref{c A_m,m+1*}, we have
$$\nabla_{v_{m+1}}^Tc=\frac{1}{1+\langle\omega_1,\omega_2\rangle}\omega_1^T.$$
Using Lemma \ref{linear algebra lemma} from the Appendix, we get
\begin{align*}
\jac F_{\omega_1,\omega_2,v_{m+2}}^1(v_{m+1})&=\det\left( I_d+\frac{1}{1+\langle\omega_1,\omega_2\rangle}(-2\omega_1-\omega_2)\omega_1^T\right)\\
&=1+\frac{-2|\omega_1|^2-\langle\omega_1,\omega_2\rangle}{1+\langle\omega_1,\omega_2\rangle}\\
&=\frac{1-2|\omega_1|^2}{1+\langle\omega_1,\omega_2\rangle}.
\end{align*}
Since $(\omega_1,\omega_2)\notin I_1$, we have $\left|1-2\left|\omega_1\right|^2\right|>2\beta$, hence
\begin{equation}
\left|\jac F_{\omega_1,\omega_2,v_{m+2}}^1(v_{m+1})\right|=\frac{\left|1-2\left|\omega_1\right|^2\right|}{1+\langle\omega_1,\omega_2\rangle}>\frac{2\beta}{1+\langle\omega_1,\omega_2\rangle}\geq\frac{4\beta}{3}>\beta,
\end{equation}
since $\displaystyle\frac{1}{2}\leq 1+\langle\omega_1,\omega_2\rangle\leq\displaystyle\frac{3}{2}$, by \eqref{bound on inverse quotient}.
Thus
\begin{equation}\label{bound on inverse Jacobian A_m,m+1*}
\left|\jac F_{\omega_1,\omega_2,v_{m+2}}^1(v_{m+1})\right|^{-1}<\beta^{-1},\quad\forall v_{m+1}\in B_R^d.
\end{equation}
We now show that $F_{\omega_1,\omega_2,v_{m+2}}^1$ is injective.
For this purpose consider $v_{m+1},\xi_{m+1}\in B_R^d$ such that
\begin{align}
F_{\omega_1,\omega_2,v_{m+2}}^1(v_{m+1})&=F_{\omega_1,\omega_2,v_{m+2}}^1(\xi_{m+1})\nonumber\\
\Leftrightarrow v_{m+1}-\xi_{m+1}&=\frac{\langle v_{m+1}-\xi_{m+1},\omega_1\rangle}{1+\langle\omega_1,\omega_2\rangle}(2\omega_1+\omega_2),\label{one-one}
\end{align}
thanks to \eqref{c A_m,m+1*}.
Therefore, there is $\lambda\in\mathbb{R}$ such that 
\begin{equation}\label{diff v in terms of omega}
v_{m+1}-\xi_{m+1}=\lambda(2\omega_1+\omega_2),
\end{equation}
so replacing $v_{m+1}-\xi_{m+1}$  in \eqref{one-one} with the  right hand side of \eqref{diff v in terms of omega}, we obtain
$$
\lambda(1-2|\omega_1|^2)=0,
$$
which yields $\lambda=0$, since we have assumed $(\omega_1,\omega_2)\notin I_1$. Therefore $v_{m+1}=\xi_{m+1},$
thus $F_{\omega_1,\omega_2,v_{m+2}}^1$ is injective.

Since $(\omega_1,\omega_2,v_{m+2})\in\mathbb{S}_1^{2d-1}\times B_R^d$ and $\bar{v}_m\in B_R^d$, Cauchy-Schwartz inequality yields that, for any $v_{m+1}\in B_R^d$, we have
\begin{align*}|F_{\omega_1,\omega_2,v_{m+2}}(v_{m+1})|&\leq |v_{m+1}|+|\bar{v}_m|+\frac{|\omega_1|(|v_{m+1}|+|\bar{v}_m|)+|\omega_2|(|\bar{v}_m|+|v_{m+2}|)}{1+\langle\omega_1,\omega_2\rangle}(2|\omega_1|+|\omega_2|)\leq 26R,
\end{align*}
since $\frac{1}{2}\leq 1+\langle\omega_1,\omega_2\rangle\leq\frac{3}{2}$, by \eqref{bound for b relative to c}, and $(\omega_1,\omega_2,v_{m+2})\in\mathbb{S}_1^{2d-1}\times B_R^d$. Therefore
\begin{equation}\label{inclusion 1 A_m,m+1*}
F_{\omega_1,\omega_2,v_{m+2}}^1[B_R^d]\subseteq B_{26R}^d.
\end{equation}
Additionally, recalling \eqref{V_m,m+1 proof}, \eqref{A_m,m+1 measure representation} and \eqref{definition F1}, we have 
$$V_{\omega_1,\omega_2,v_{m+2}}^{m,m+1}=\{v_{m+1}\in B_R^d: \langle\omega_1,F_{\omega_1,\omega_2,v_{m+2}}(v_{m+1})\rangle\geq\beta|\omega_1||F_{\omega_1,\omega_2,v_{m+2}}(v_{m+1})|\},$$
thus
\begin{equation}\label{V in terms of U}
v_{m+1}\in V_{\omega_1,\omega_2,v_{m+2}}^{m,m+1}\Leftrightarrow F_{\omega_1,\omega_2,v_{m+2}}^1(v_{m+1})\in U_{\omega_1},
\end{equation}
where 
\begin{equation}\label{U, A_m,m+1*}
U_{\omega_1}=\left\{\nu\in\mathbb{R}^d:\langle\omega_1,\nu\rangle\geq\gamma'|\omega_1||\nu|\right\}.
\end{equation}
Hence 
\begin{equation}\label{inclusion 2 A_m,m+1*}
\mathds{1}_{V_{\omega_1,\omega_2,v_{m+2}}^{m,m+1}}(v_{m+1})=\mathds{1}_{U_{\omega_1}}(F_{\omega_1,\omega_2,v_{m+2}}^1(v_{m+1})),\quad\forall v_{m+1}\in B_R^d.
\end{equation}
Therefore, performing the substitution  $\nu:=F_{\omega_1,\omega_2,v_{m+2}}^1(v_{m+1})$, and using  \eqref{bound on inverse Jacobian A_m,m+1*}, we obtain
\begin{equation}\label{bound on velocity integral}
\int_{B_R^d}\mathds{1}_{V_{\omega_1,\omega_2,v_{m+2}}^{m,m+1}}(v_{m+1})\,dv_{m+1}
=\int_{B_R^d}\mathds{1}_{U_{\omega_1}}(F_{\omega_1,\omega_2,v_{m+2}}^1(v_{m+1}))\,dv_{m+1}\nonumber\leq \beta^{-1}\int_{B_{26R}^d}\mathds{1}_{U_{\omega_1}}(\nu)\,d\nu.
\end{equation}
Recalling notation from \eqref{shell parameters} and \eqref{U, A_m,m+1*}, we have
\begin{equation}\label{equality of chars U,A, A_m,m+1*}
\mathds{1}_{U_{\omega_1}}(\nu)=\mathds{1}_{S(\gamma',\nu)}(\omega_1),\quad\forall \omega_1\in B_1^d,\quad \forall\nu\in B_{26R}^d.
\end{equation}
Therefore, using \eqref{I, A_m,m+1*}, \eqref{bound on velocity integral}, Fubini's Theorem and \eqref{equality of chars U,A, A_m,m+1*}, we obtain
\begin{align}
I_1'&\leq\beta^{-1}\int_{(\mathbb{S}_1^{2d-1}\setminus I_1)\times B_R^d}\int_{B_{26R}^d}\mathds{1}_{U_{\omega_1}}(\nu)\,d\nu\,d\omega_1\,d\omega_2\,dv_{m+2}\nonumber\\
&\leq\beta^{-1}\int_{B_{26R}^d\times B_R^d}\int_{B_1^d}\int_{\mathbb{S}_{\sqrt{1-|\omega_2|^2}}^{d-1}}\mathds{1}_{S(\gamma',\nu)(\omega_1)}\,d\omega_1\,d\omega_2\,d\nu\,dv_{m+2}\nonumber\\
&\lesssim R^{2d}\beta^{-1}\arccos\gamma'\label{use of shell lemma A_m,m+1*}\\
&=R^{2d}\beta^{-1}\arccos\sqrt{1-\frac{\gamma}{2}},\label{estimate on I A_m,m+1*}
\end{align}
where  to obtain \eqref{use of shell lemma A_m,m+1*} we use Lemma \ref{shell estimate}.
Combining \eqref{decomposed integral A_m,m+1*}, \eqref{estimate for I_1}, \eqref{estimate on I A_m,m+1*}, we obtain
\begin{equation}\label{measure estimate of A_m,m+1*}
|A_{m,m+1}^*|\leq R^{2d}\left(\beta+\beta^{-1}\arccos\sqrt{1-\frac{\gamma}{2}}\right).
\end{equation}

$\bullet$ Estimate for $A_{m,m+2}^*$: The argument is entirely symmetric, using the set 
\begin{equation*}
V_{\omega_1,\omega_2,v_{m+1}}^{m,m+2}=\left\{v_{m+2}\in B_R^d:(\omega_1,\omega_2,v_{m+1},v_{m+2})\in A_{m,m+2}^*\right\},
\end{equation*}
for fixed $(\omega_1,\omega_2,v_{m+1})\in\mathbb{S}_1^{2d-1}\times B_R^d$ and the map
$$F^2_{\omega_1,\omega_2,v_{m+1}}(v_{m+2})=v_{m+2}-\bar{v}_m-c\omega_1-2c\omega_2.$$
We obtain the estimate
\begin{equation}\label{measure A_m,m+2*}
|A_{m,m+2}^*|\lesssim R^{2d}\left(\beta+\beta^{-1}\arccos\sqrt{1-\frac{\gamma}{2}}\right),
\end{equation}

$\bullet$ Estimate for $B_{m+1,m+2}^*$: The estimate for $B_{m+1,m+2}^*$ is in the same spirit as the previous estimates, however we will need to distinguish cases depending on the size of the impact directions. The reason for that is that we rely on Lemma \ref{lemma on I_1,2} from Section \ref{sec:geometric} which provides estimates on hemispheres of the $(2d-1)$-unit sphere.

 Recall from \eqref{A_m+1,m+2*} the set
\begin{align}
 B_{m+1,m+2}^*&=\{(\omega_1,\omega_2,v_{m+1},v_{m+2})\in\mathbb{S}_1^{2d-1}\times B_R^{2d}:|\langle\omega_1-\omega_2,v_{m+1}^*-v_{m+2}^*\rangle|\geq\gamma'|\omega_1-\omega_2||v_{m+1}^*-v_{m+2}^*|\}.\label{A_m+1,m+2* measure representation}
 \end{align}
 The ternary collisional law \eqref{formulas ternary} yields
 $v_{m+1}^*-v_{m+2}^*=v_{m+1}-v_{m+2}-c(\omega_1-\omega_2),$
 where $c$ is given by \eqref{c A_m,m+1*}.
Thus we may write
\begin{align*}B_{m+1,m+2}^*=&\{(\omega_1,\omega_2,v_{m+1},v_{m+2})\in\mathbb{S}_1^{2d-1}\times B_R^{2d}:\\
&|\langle\omega_2-\omega_1,v_{m+2}-v_{m+1}-c(\omega_2-\omega_1)\rangle|\geq \gamma'|\omega_2-\omega_1||v_{m+2}-v_{m+1}-c(\omega_2-\omega_1)|\}.
\end{align*}
Recall from \eqref{sphere 2<1}-\eqref{sphere 1<2}, the sets
\begin{align*}
\mathcal{S}_{1,2}&=\left\{(\omega_1,\omega_2)\in\mathbb{S}_1^{2d-1}:|\omega_1|<|\omega_2|\right\},\quad\mathcal{S}_{2,1}=\left\{(\omega_1,\omega_2)\in\mathbb{S}_1^{2d-1}:|\omega_2|<|\omega_1|\right\}.
\end{align*}
We also recall from \eqref{I_1,2}-\eqref{I_2,1} the sets 
\begin{align*}
I_{1,2}&=\{(\omega_1,\omega_2)\in\mathbb{R}^{2d}\left|\left|\omega_1\right|^2+2\langle\omega_1,\omega_2\rangle\right|\leq\beta\},\quad I_{2,1}=\{(\omega_1,\omega_2)\in\mathbb{R}^{2d}\left|\left|\omega_2\right|^2+2\langle\omega_1,\omega_2\rangle\right|\leq\beta\}.
\end{align*}
We clearly have
\begin{align}
|&B_{m+1,m+2}^*|=\int_{\mathbb{S}_1^{2d-1}\times B_R^{2d}}\mathds{1}_{B_{m+1,m+2}^*}\,d\omega_1\,d\omega_2\,dv_{m+1}\,dv_{m+2}\nonumber\\
&=\int_{\mathcal{S}_{1,2}\times B_R^{2d}}\mathds{1}_{B_{m+1,m+2}^*}\,d\omega_1\,d\omega_2\,dv_{m+1}\,dv_{m+2}+\int_{\mathcal{S}_{2,1}\times B_R^{2d}}\mathds{1}_{B_{m+1,m+2}^*}\,d\omega_1\,d\omega_2\,dv_{m+1}\,dv_{m+2}\nonumber\\
&= \widetilde{I}_{1,2}+\widetilde{I}_{1,2}'+\widetilde{I}_{2,1}+\widetilde{I}_{2,1}',\label{decomposed integral A_m+1,m+2*}
\end{align}
where
\begin{align}
\widetilde{I}_{1,2}&=\int_{(\mathcal{S}_{1,2}\cap I_{1,2})\times B_R^{2d}}\mathds{1}_{B_{m+1,m+2}^*}\,d\omega_1\,d\omega_2\,dv_{m+1}\,dv_{m+2},\label{I_1,2 proof}\\
\widetilde{I}_{1,2}'&=\int_{(\mathcal{S}_{1,2}\setminus I_{1,2})\times B_R^{2d}}\mathds{1}_{B_{m+1,m+2}^*}\,d\omega_1\,d\omega_2\,dv_{m+1}\,dv_{m+2},\label{I_1,2'}\\
\widetilde{I}_{2,1}&=\int_{(\mathcal{S}_{2,1}\cap I_{2,1})\times B_R^{2d}}\mathds{1}_{B_{m+1,m+2}^*}\,d\omega_1\,d\omega_2\,dv_{m+1}\,dv_{m+2}\label{I_2,1' proof},\\
I_{2,1}'&=\int_{(\mathcal{S}_{2,1}\setminus I_{2,1})\times B_R^{2d}}\mathds{1}_{B_{m+1,m+2}^*}\,d\omega_1\,d\omega_2\,dv_{m+1}\,dv_{m+2}\label{I_2,1'}.
\end{align}
We treat each of the terms in \eqref{decomposed integral A_m+1,m+2*} separately.

\textit{Estimate for} $\widetilde{I}_{1,2}$: By \eqref{I_1,2 proof}, Fubini's Theorem and Lemma \ref{lemma on I_1,2}, we obtain
\begin{equation}\label{estimate on I_1,2}
\widetilde{I}_{1,2}\lesssim R^{2d}\int_{\mathcal{S}_{1,2}}\mathds{1}_{I_{1,2}}\,d\omega_1\,d\omega_2\lesssim R^{2d}\beta.
\end{equation}

\textit{Estimate for} $\widetilde{I}_{2,1}$: Similarly, we obtain
\begin{equation}\label{estimate on I_2,1}
\widetilde{I}_{2,1}\lesssim R^{2d}\beta.
\end{equation}

\textit{Estimate for} $I_{1,2}'$: From \eqref{I_1,2'}, we obtain
\begin{equation}\label{I_1,2' with V}
I_{1,2}'\leq\int_{\mathcal{S}_{1,2}\setminus I_{1,2}}\int_{B_R^d}\int_{B_R^d}\mathds{1}_{V_{\omega_1,\omega_2,v_{m+1}}^{m+1,m+2}}(v_{m+2})\,dv_{m+2}\,dv_{m+1}\,d\omega_1\,d\omega_2,
\end{equation}
where given $(\omega_1,\omega_2,v_{m+1})\in(\mathcal{S}_{1,2}\setminus I_{1,2})\times B_R^d$, we denote
\begin{equation}\label{V A_m+1,m+2*}
V_{\omega_1,\omega_2,v_{m+1}}^{m+1,m+2}=\left\{v_{m+2}\in B_R^d:(\omega_1,\omega_2,v_{m+1},v_{m+2})\in B_{m+1,m+2}^*\right\}.
\end{equation}
Let us fix $(\omega_1,\omega_2,v_{m+1})\in (\mathcal{S}_{1,2}\setminus I_{1,2})\times B_R^d$. We define the map $F_{\omega_1,\omega_2,v_{m+1}}^{1,2}:B_R^d\to\mathbb{R}^d$ by
$$F_{\omega_1,\omega_2,v_{m+1}}^{1,2}(v_{m+2})=v_{m+2}-v_{m+1}-c(\omega_2-\omega_1),$$
where $c$ is given by \eqref{c A_m,m+1*}.
In a similar way as in the estimate of of $|A_{m,m+1}^*|$,  for any
 $(\omega_1,\omega_2)\notin I_{1,2}$, we have
\begin{equation}
\left|\jac F_{\omega_1,\omega_2,v_{m+1}}^{1,2}(v_{m+2})\right|=\frac{\left|\left|\omega_1\right|^2+2\langle\omega_1,\omega_2\rangle\right|}{1+\langle\omega_1,\omega_2\rangle}>\frac{\beta}{1+\langle\omega_1,\omega_2\rangle}\geq\frac{2\beta}{3},
\end{equation}
Thus
\begin{equation}\label{bound on inverse Jacobian A_m+1,m+2*}
\left|\jac F_{\omega_1,\omega_2,v_{m+1}}^{1,2}(v_{m+2})\right|^{-1}\leq \frac{3\beta^{-1}}{2},\quad\forall v_{m+2}\in B_R^d.
\end{equation}
Similarly to the estimate for $|A_{m,m+1}^*|$, we show also that $F_{\omega_1,\omega_2,v_{m+1}}^{1,2}$ is injective.

Since $(\omega_1,\omega_2,v_{m+1})\in\mathbb{S}_1^{2d-1}\times B_R^d$ and $\bar{v}_m\in B_R^d$, Cauchy-Schwartz inequality yields that, for any $v_{m+2}\in B_R^d$, we have
\begin{align*}|F_{\omega_1,\omega_2,v_{m+1}}^{1,2}(v_{m+2})|&\leq |v_{m+2}|+|v_{m+1}|+\frac{|\omega_1|(|v_{m+1}|+|\bar{v}_m|)+|\omega_2|(|v_{m+2}|+|\bar{v}_m|)}{1+\langle\omega_1,\omega_2\rangle}(|\omega_2|+|\omega_1|)\leq 18R,
\end{align*}
since $\frac{1}{2}\leq 1+\langle\omega_1,\omega_2\rangle\leq\frac{3}{2}$. Therefore
\begin{equation}\label{inclusion 1 A_m+1,m+2*}
F_{\omega_1,\omega_2,v_{m+1}}^{1,2}[B_R^d]\subseteq B_{18R}^d.
\end{equation}
Additionally
\begin{equation*}
v_{m+2}\in V_{\omega_1,\omega_2,v_{m+1}}^{m+1,m+2}\Leftrightarrow F_{\omega_1,\omega_2,v_{m+1}}^{1,2}(v_{m+2})\in U_{\omega_1,\omega_2},
\end{equation*}
where 
\begin{equation}\label{U, A_m+1,m+2*}
U_{\omega_1,\omega_2}=\left\{\nu\in\mathbb{R}^d:\langle\omega_2-\omega_1,\nu\rangle\geq\gamma'|\omega_2-\omega_1||\nu|\right\}.
\end{equation}
Hence
\begin{equation}\label{inclusion 2 A_m+1,m+2*}
\mathds{1}_{V_{\omega_1,\omega_2,v_{m+1}}^{m+1,m+2}}(v_{m+2})=\mathds{1}_{U_{\omega_1,\omega_2}}(F_{\omega_1,\omega_2,v_{m+1}}^{1,2}(v_{m+2})),\quad\forall v_{m+2}\in B_R^d.
\end{equation}
Therefore, performing the substitution $\nu:=F_{\omega_1,\omega_2,v_{m+1}}^{1,2}(v_{m+2})$, and using \eqref{bound on inverse Jacobian A_m+1,m+2*}, we obtain
\begin{equation}\label{bound on velocity integral A_m+1,m+2*}
\int_{B_R^d}\mathds{1}_{V_{\omega_1,\omega_2,v_{m+1}}^{m+1,m+2}}(v_{m+2})\,dv_{m+2}
=\int_{B_R^d}\mathds{1}_{U_{\omega_1,\omega_2}}(F_{\omega_1,\omega_2,v_{m+1}}^{1,2}(v_{m+2}))\,dv_{m+2}\leq \beta^{-1}\int_{B_{18R}^d}\mathds{1}_{U_{\omega_1,\omega_2}}(\nu)\,d\nu.
\end{equation}
Recalling the set
$
N(\gamma',\nu)=\left\{(\omega_1,\omega_2)\in\mathbb{R}^{2d}:\langle\omega_1-\omega_2,\nu\rangle\geq\gamma'|\omega_1-\omega_2||\nu|\right\},
$
  from \eqref{difference shell parameters} and \eqref{U, A_m+1,m+2*}, we have
\begin{equation}\label{equality of chars U,A, A_m+1,m+2*}
\mathds{1}_{U_{\omega_1,\omega_2}}(\nu)=\mathds{1}_{N(\gamma',\nu)}(\omega_1,\omega_2),\quad\forall (\omega_1,\omega_2)\in\mathbb{S}_1^{2d-1},\quad\forall \nu\in B_{18R}^d.
\end{equation}
Therefore, using \eqref{I_1,2' with V}, \eqref{bound on velocity integral A_m+1,m+2*}, Fubini's Theorem and \eqref{equality of chars U,A, A_m+1,m+2*}, we obtain
\begin{align}
I_{1,2}'&\leq\beta^{-1}\int_{(\mathcal{S}_{1,2}\setminus I_{1,2})\times B_R^d}\int_{B_{18R}^d}\mathds{1}_{U_{\omega_1,\omega_2}}(\nu)\,d\nu\,d\omega_1\,d\omega_2\,dv_{m+1}\nonumber\\
&\leq\beta^{-1}\int_{ B_R^d\times B_{18R}^d}\int_{\mathbb{S}_1^{2d-1}}\mathds{1}_{N(\gamma',\nu)}(\omega_1,\omega_2)\,d\omega_1\,d\omega_2\,d\nu\,dv_{m+1}\nonumber\\
&\lesssim R^{2d}\beta^{-1}\arccos\gamma'\label{use of strip lemma A_m+1,m+2*}\\
&=R^{2d}\beta^{-1}\arccos\sqrt{1-\frac{\gamma}{2}},\nonumber
\end{align}
where to obtain \eqref{use of strip lemma A_m+1,m+2*}, we use Lemma \ref{estimate of difference in shell}.
Therefore,
\begin{equation}\label{estimate I_12'}
I_{12}'\leq R^{2d}\beta^{-1}\arccos\sqrt{1-\frac{\gamma}{2}}.
\end{equation}

\textit{Estimate for} $I_{2,1}'$: The argument is entirely symmetric, using the set
$$V_{\omega_1,\omega_2,v_{m+2}}^{m+1,m+2}=\left\{v_{m+1}\in B_R^d:(\omega_1,\omega_2,v_{m+1},v_{m+2})\in B_{m+1,m+2}^*\right\}.$$
for given $(\omega_1,\omega_2,v_{m+2})\in(\mathcal{S}_{2,1}\setminus I_{2,1})\times B_R^d$ and the map
$F_{\omega_1,\omega_2,v_{m+2}}^{2,1}(v_{m+1})=v_{m+1}-v_{m+2}-c(\omega_1-\omega_2).$
We obtain
\begin{equation}\label{estimate I_21'}
I_{21}'\leq R^{2d}\beta^{-1}\arccos\sqrt{1-\frac{\gamma}{2}}.
\end{equation}
Recalling \eqref{decomposed integral A_m+1,m+2*} and using \eqref{estimate on I_1,2}-\eqref{estimate on I_2,1}, \eqref{estimate I_12'}-\eqref{estimate I_21'}, we obtain
\begin{equation}\label{measure A_m+1,m+2*}
|B_{m+1,m+2}^*|\lesssim  R^{2d}\left(\beta+\beta^{-1}\arccos\sqrt{1-\frac{\gamma}{2}}\right)
\end{equation}

Recalling \eqref{A+ measure representation} and using \eqref{measure omega_1}-\eqref{measure omega_1,2}, \eqref{measure estimate of A_m,m+1*}, \eqref{measure A_m,m+2*}, \eqref{measure A_m+1,m+2*}, 
we obtain
\begin{equation}\label{measure A+}
|\mathcal{A}_m^+(\bar{Z}_m)|\lesssim R^{2d}\left(\gamma^{d/2}+\gamma^{\frac{d-1}{4}}+\beta+\beta^{-1}\arccos\sqrt{1-\frac{\gamma}{2}}\right).
\end{equation}
Recalling \eqref{A}, using \eqref{measure A-}, \eqref{measure A+} and using the fact that $\gamma<<1$, we obtain
\begin{equation}\label{measure A with gamma beta}
|\mathcal{A}_m(\bar{Z}_m)|\lesssim R^{2d}\left(\gamma^{\frac{d-1}{4}}+\beta+\beta^{-1}\arccos\sqrt{1-\frac{\gamma}{2}}.\right)
\end{equation}
\textit{Choice of $\beta$}: 
Let us now choose $\beta$ in terms of $\eta$. Recalling  that $\epsilon_2<<\eta^2\epsilon_3$ and \eqref{gamma def}, we have
\begin{equation}\label{gamma bounded by eta}
\gamma^{\frac{d-1}{4}}<<\eta^{\frac{d-1}{2}}.
\end{equation}
Moreover, since $\eta<<1$, we may assume
\begin{equation}\label{eta simeq sin eta}
\frac{\eta}{\sqrt{2}}\leq\sin\eta\leq\eta,
\end{equation}
Since $\gamma<<\eta^2$, \eqref{eta simeq sin eta} implies
\begin{equation}\label{arc gamma bounded by eta}
\gamma<<2\sin^2\eta\Rightarrow\arccos\sqrt{1-\frac{\gamma}{2}}<\eta.
\end{equation}
Choosing 
\begin{equation}\label{choice of beta in terms of eta}
\beta=\eta^{1/2}<<1,
\end{equation}
estimates \eqref{measure A with gamma beta}-\eqref{gamma bounded by eta}, \eqref{arc gamma bounded by eta} imply
\begin{equation}\label{final estimate on A}
|\mathcal{A}_m(\bar{Z}_m)|\lesssim R^{2d}\left(\eta^{\frac{d-1}{2}}+\eta^{1/2}\right)\lesssim R^{2d}\eta^{\frac{d-1}{4d+2}},
\end{equation}
since $\eta<<1$ and $d\geq 2$.
The claim comes from \eqref{B}-\eqref{estimate of wide tilde} and \eqref{final estimate on A}. 
\end{proof}
\section{Elimination of recollisions}\label{sec_elimination}
In this section we reduce the convergence proof to comparing truncated elementary observables. We first restrict to good configurations and provide the corresponding measure estimate. This is happening in Proposition \ref{restriction to initially good conf}. We then inductively apply Proposition \ref{bad set double} and Proposition \ref{bad set double measure} or Proposition \ref{bad set triple} and Proposition \ref{bad set triple measure} (depending on whether the adjunction is binary or ternary) to reduce the convergence proof to truncated elementary observables. The convergence proof, completed in Section \ref{sec:convergence proof}, will then follow naturally, since the backwards $(\epsilon_2,\epsilon_3)$-flow and the backwards free flow will be comparable out of a small measure set. Throughout this section $s\in\mathbb{N}$ will be fixed, $(N,\epsilon_2,\epsilon_3)$ are given in the scaling \eqref{scaling} with $N$ large enough such that $\epsilon_2<<\epsilon_3$, and the parameters $n,R,\epsilon_0,\alpha,\eta,\delta$ satisfy \eqref{choice of parameters}. 
\subsection{Restriction to good configurations}
Inductively using Lemma \ref{adjuction of 1} we are able to reduce the convergence proof to good configurations, up to a small measure set. The measure of the complement will  be negligible in the limit.

For convenience, given $m\in\mathbb{N}$, let us define the set
\begin{equation}\label{both epsilon-epsilon_0}
G_m(\epsilon_3,\epsilon_0,\delta):=G_m(\epsilon_3,0)\cap G_m(\epsilon_0,\delta).
\end{equation}
For $s\in\mathbb{N}$, we also recall from \eqref{separated space data} the set $\Delta_s^X(\epsilon_0)$ of well-separated spatial configurations.
\begin{lemma}\label{initially good configurations}
Let $s\in\mathbb{N}$. Let $s\in\mathbb{N}$, $\alpha,\epsilon_0,R,\eta,\delta$ be parameters as in \eqref{choice of parameters} and $\epsilon_2<<\epsilon_3<<\alpha$. Then for any $X_s\in\Delta_s^X(\epsilon_0)$, there is a subset of velocities $\mathcal{M}_s(X_s)\subseteq B_R^{ds}$ of measure
\begin{equation}\label{measure of initialization}
\left|\mathcal{M}_s\left(X_s\right)\right|_{ds}\leq C_{d,s} R^{ds}\eta^{\frac{d-1}{2}},
\end{equation}
such that
\begin{equation}\label{initialization}
Z_s\in G_s(\epsilon_3,\epsilon_0,\delta),\quad\forall V_s\in B_R^{ds}\setminus \mathcal{M}_s(X_s).
\end{equation}
\end{lemma} 
\begin{proof}
We use Proposition 11.2. from \cite{thesis} for $\epsilon=\epsilon_3$.
\end{proof}

For $s\in\mathbb{N}$ and $X_s\in\Delta_s^X(\epsilon_0)$, let us denote $\mathcal{M}_s^c(X_s)=B_R^{ds}\setminus\mathcal{M}_s(X_s).$
Consider $1\leq k\leq n$ and let us recall the observables $I_{s,k,R,\delta}^N$, $I_{s,k,R,\delta}^\infty$ defined in \eqref{bbgky truncated time}-\eqref{boltzmann truncated time}. We restrict the domain of integration to velocities giving good configurations.

In particular, we define
\begin{align}
\widetilde{I}_{s,k,R,\delta}^N(t)(X_s)&=\int_{\mathcal{M}_s^c(X_s)}\phi_s(V_s)f_{N,R,\delta}^{(s,k)}(X_s,V_s)\,dV_s\label{good observables BBGKY },\\
\widetilde{I}_{s,k,R,\delta}^\infty(t)(X_s)&=\int_{\mathcal{M}_s^c(X_s)}\phi_s(V_s)f_{R,\delta}^{(s,k)}(X_s,V_s)\,dV_s\label{good observables Boltzmann}.
\end{align}

Let us apply Proposition \ref{initially good configurations} to restrict to initially good configurations. To keep track of all the possible adjuctions we recall  recall the notation from \eqref{S_k}-\eqref{sigma tilde}: given $k\in\mathbb{N}$, we write

$$S_k=\{\sigma=(\sigma_1,...,\sigma_k):\sigma_i\in\{1,2\}\},$$
and given $\sigma\in S_k$, we write
\begin{align*}
\widetilde{\sigma}_\ell&=\sum_{i=1}^\ell\sigma_i,\quad 1\leq \ell\leq k,\quad\widetilde{\sigma}_0=0.
\end{align*}
\begin{proposition}\label{restriction to initially good conf} Let $s,n\in\mathbb{N}$,   $\alpha,\epsilon_0,R,\eta,\delta$ be parameters as in \eqref{choice of parameters},  $(N,\epsilon_2,\epsilon_3)$  in the scaling \eqref{scaling} with $\epsilon_2<<\epsilon_3<<\alpha$, and $t\in[0,T]$. Then, the following estimates hold:
\begin{equation*}
\sum_{k=1}^n\|I_{s,k,R,\delta}^N(t)-\widetilde{I}_{s,k,R,\delta}^N(t)\|_{L^\infty\left(\Delta_s^X\left(\epsilon_0\right)\right)}
\leq C_{d,s,\mu_0,T}R^{ds}\eta^{\frac{d-1}{2}}\|F_{N,0}\|_{N,\beta_0,\mu_0},
\end{equation*}
\begin{equation*}
\sum_{k=1}^n\|I_{s,k,R,\delta}^\infty (t)-\widetilde{I}_{s,k,R,\delta}^\infty (t)\|_{L^\infty\left(\Delta_s^X\left(\epsilon_0\right)\right)}\leq C_{d,s,\mu_0,T}R^{ds}\eta^{\frac{d-1}{2}}\|F_{0}\|_{\infty,\beta_0,\mu_0}.
\end{equation*}
\end{proposition}
\begin{proof}
We present the proof for the BBGKY hierarchy case only. The proof for the Boltzmann hierarchy case is similar. Let us fix $X_s\in\Delta_s^X(\epsilon_0)$.

We first assume that $k\in\left\{1,...,n\right\}$. Triangle inequality, an inductive application of estimate \eqref{a priori binary bound F_N}, estimate \eqref{a priori bound F_N,0} and part $(ii)$ of Proposition \ref{remark for initial} yield
\begin{align}
|I_{s,k,R,\delta}^N(t)(X_s)-&\widetilde{I}_{s,k,R,\delta}^N(t)(X_s)|\leq\sum_{\sigma\in S_k}\int_{\mathcal{M}_s(X_s)}|\phi_s(V_s)f_{N,R,\delta}^{(s,k,\sigma)}(t,X_s,V_s)|\,dV_s\nonumber\\
&\leq 2T\|\phi_s\|_{L^\infty_{V_s}}e^{-s\bm{\mu}(T)}\left(\frac{1}{8}\right)^{k-1}\|F_{N,0}\|_{N,\beta_0,\mu_0}\int_{\mathcal{M}_s(X_s)}e^{-\bm{\beta}(T)E_s(Z_s)}\,dV_s\label{use of card reduction}\\
&\leq 2T\|\phi_s\|_{L^\infty_{V_s}}e^{-s\bm{\mu}(T)}\left(\frac{1}{8}\right)^{k-1}|\mathcal{M}_s(X_s)|_{ds}\|F_{N,0}\|_{N,\beta_0,\mu_0},\label{reduction to elem 1<k<n}
\end{align}
where to obtain \eqref{use of card reduction}, we use \eqref{cardinality of S_k}.

For $k=0$, part $(i)$ of Proposition \ref{remark for initial} and Remark \ref{T_N definition} similarly  yield
\begin{equation}\label{reduction to elem k=0}
|I_{s,0,R,\delta}^N(t)(X_s)-\widetilde{I}_{s,0,R,\delta}^N(t)(X_s)|\leq\|\phi_s\|_{L^\infty_{V_s}}e^{-s\bm{\mu}(T)}|\mathcal{M}_s(X_s)|_{ds}\|F_{N,0}\|_{N,\beta_0,\mu_0}.
\end{equation}
The claim comes after using \eqref{reduction to elem 1<k<n}-\eqref{reduction to elem k=0}, adding over $k=0,...,n$, and using the measure estimate of Proposition \ref{initially good configurations}.
\end{proof}
\begin{remark}\label{no need for k=0}
Given $s\in\mathbb{N}$ and $X_s\in\Delta_s^X(\epsilon_0)$, the definition of $\mathcal{M}_s(X_s)$ implies that $$\widetilde{I}_{s,0,R,\delta}^N(t)(X_s)=\widetilde{I}_{s,0,R,\delta}^\infty(t)(X_s).$$
Therefore, by Proposition \ref{restriction to initially good conf},  convergence reduces  to controlling the differences $\widetilde{I}_{s,k,R,\delta}^N(t)-\widetilde{I}_{s,k,R,\delta}^\infty(t),$ for $k=1,...,n$, in the scaled limit.
\end{remark}
\subsection{Reduction to elementary observables}\label{reduction to elementary observables}
Here, given $s\in\mathbb{N}$ and $1\leq k\leq n$, we express the observables $\widetilde{I}_{s,k,R,\delta}^N(t)$, $\widetilde{I}_{s,k,R,\delta}^\infty(t)$, defined in \eqref{good observables BBGKY }-\eqref{good observables Boltzmann}, as a superposition of elementary observables.

 For this purpose, given $\ell\in\mathbb{N}$, and recalling \eqref{velocity truncation of operators}, \eqref{BBGKY operator binary},  we decompose the BBGKY hierarchy binary  truncated collisional operator as:
 \begin{equation*}
\mathcal{C}_{\ell,\ell+1}^{N,R}=\sum_{i=1}^\ell\mathcal{C}_{\ell,\ell+1}^{N,R,+,i}-\sum_{i=1}^\ell\mathcal{C}_{\ell,\ell+1}^{N,R,-,i},
\end{equation*}
where
\begin{equation*}
\begin{aligned}
\mathcal{C}_{\ell,\ell+1}^{N,R,+,i}g_{\ell+1}&(Z_\ell)=A_{N,\epsilon_2,\ell}^2\int_{\mathbb{S}_1^{d-1}\times B_R^{d}}b_2^+(\omega_1,v_{\ell+1}-v_i) g_{\ell+1}(Z_{\ell+1,\epsilon_2}^{i'})\,d\omega_1\,dv_{\ell+1},
\end{aligned}
\end{equation*} 
\begin{equation*}
\begin{aligned}
\mathcal{C}_{\ell,\ell+1}^{N,R,-,i}g_{\ell+1}&(Z_\ell)=A_{N,\epsilon_2,\ell}^2\int_{\mathbb{S}_1^{d-1}\times B_R^{d}}b_2^+(\omega_1,v_{\ell+1}-v_i) g_{\ell+1}(Z_{\ell+1,\epsilon_2}^{i})\,d\omega_1\,dv_{\ell+1}.
\end{aligned}
\end{equation*} 
and the ternary truncated collisional operator as:
\begin{equation*}
\mathcal{C}_{\ell,\ell+2}^{N,R}=\sum_{i=1}^\ell\mathcal{C}_{\ell,\ell+2}^{N,R,+,i}-\sum_{i=1}^\ell\mathcal{C}_{\ell,\ell+2}^{N,R,-,i},
\end{equation*}
where
\begin{equation*}
\begin{aligned}
\mathcal{C}_{\ell,\ell+2}^{N,R,+,i}g_{\ell+2}&(Z_\ell)=A_{N,\epsilon_3,\ell}^3\int_{\mathbb{S}_1^{2d-1}\times B_R^{2d}}\frac{b_3^+(\omega_1,\omega_2,v_{\ell+1}-v_i,v_{\ell+2}-v_i)}{\sqrt{1+\langle\omega_1,\omega_2\rangle}} g_{\ell+2}(Z_{\ell+2,\epsilon_3}^{i*})\,d\omega_1\,d\omega_2\,dv_{\ell+1}\,dv_{\ell+2},
\end{aligned}
\end{equation*} 
\begin{equation*}
\begin{aligned}
\mathcal{C}_{\ell,\ell+2}^{N,R,-,i}g_{\ell+2}&(Z_\ell)=A_{N,\epsilon_3,\ell}^3\int_{\mathbb{S}_1^{2d-1}\times B_R^{2d}}\frac{b_3^+(\omega_1,\omega_2,v_{\ell+1}-v_i,v_{\ell+2}-v_i)}{\sqrt{1+\langle\omega_1,\omega_2\rangle}} g_{\ell+2}(Z_{\ell+2,\epsilon_3}^{i})\,d\omega_1\,d\omega_2\,dv_{\ell+1}\,dv_{\ell+2}.
\end{aligned}
\end{equation*} 
In order to expand the observable $\widetilde{I}_{s,k,R,\delta}^N(t)$ to elementary observables, we need to take into account all the possible particle adjuctions occurring by adding one or two particles to the system in each step. More precisely, given $\sigma\in S_k$, and $i\in\{1,...,k\}$, we are adding $\sigma_i\in\{1,2\}$ particle(s) to the existing $s+\widetilde{\sigma}_{i-1}$ particles in either  precollisional or postcollisional way.  In order to keep track of this process, given $1\leq k\leq n$, $\sigma\in S_k$, we introduce the notation
\begin{align}
\mathcal{M}_{s,k,\sigma}&=\left\{M=(m_1,...,m_k)\in\mathbb{N}^k:m_i\in\left\{1,...,s+\widetilde{\sigma}_{i-1}\right\},\quad\forall i\in\left\{1,...,k\right\}\right\}\label{M_k},\\
\mathcal{J}_{s,k,\sigma}&=\left\{J=(j_1,...,j_k)\in\mathbb{N}^k:j_i\in\left\{-1,1\right\},\quad\forall i\in\left\{1,...,k\right\}\right\}\label{J_k}.\\
\mathcal{U}_{s,k,\sigma}&=\mathcal{J}_{s,k,\sigma}\times\mathcal{M}_{s,k,\sigma}.
\end{align}

Under this notation, the BBGKY hierarchy observable functional $\widetilde{I}_{s,k,R,\delta}^N(t)$ can be expressed, for $1\leq k\leq n$, as a superposition of elementary observables
\begin{equation}\label{superposition BBGKY}
\widetilde{I}_{s,k,R,\delta}^N(t)(X_s)=\sum_{\sigma\in S_k}\sum_{(J,M)\in\mathcal{U}_{s,k,\sigma}}\left(\prod_{i=1}^kj_i\right)\widetilde{I}_{s,k,R,\delta,\sigma}^N(t,J,M)(X_s),
\end{equation}
where the elementary observables are defined by
\begin{equation}\label{elementary observable BBGKY}
\begin{aligned}
\widetilde{I}_{s,k,R,\delta,\sigma}^N(t,J,M)(X_s)&=\int_{\mathcal{M}_s^c(X_s)}\phi_s(V_s)\int_{\mathcal{T}_{k,\delta}(t)}T_s^{t-t_1}\mathcal{C}_{s,s+\widetilde{\sigma}_1}^{N,R,j_1,m_1} T_{s+\widetilde{\sigma}_1}^{t_1-t_2}...\\
&...T_{s+\widetilde{\sigma}_{k-1}}^{t_{k-1}-t_k}\mathcal{C}_{s+\widetilde{\sigma}_{k-1},s+\widetilde{\sigma}_k}^{N,R,j_k,m_k} T_{s+\widetilde{\sigma}_k}^{t_m}f_{0}^{(s+\widetilde{\sigma}_k)}(Z_s)\,dt_k...\,dt_{1}dV_s.
\end{aligned}
\end{equation}

Similarly, given $\ell\in\mathbb{N}$, and recalling \eqref{boltzmann notation binary}, \eqref{boltzmann notation triary}, we decompose the Boltzmann hierarchy binary and ternary collisional operators as:
 \begin{equation*}
\mathcal{C}_{\ell,\ell+1}^{\infty,R}=\sum_{i=1}^\ell\mathcal{C}_{\ell,\ell+1}^{\infty,R,+,i}-\sum_{i=1}^\ell\mathcal{C}_{\ell,\ell+1}^{\infty,R,-,i},
\end{equation*}
where
\begin{equation*}
\begin{aligned}
\mathcal{C}_{\ell,\ell+1}^{\infty,R,+,i}g_{\ell+1}&(Z_\ell)=\int_{\mathbb{S}_1^{d-1}\times B_R^{d}}b_2^+(\omega_1,v_{\ell+1}-v_i) g_{\ell+1}(Z_{\ell+1}^{i'})\,d\omega_1\,dv_{\ell+1},
\end{aligned}
\end{equation*} 
\begin{equation*}
\begin{aligned}
\mathcal{C}_{\ell,\ell+1}^{\infty,R,-,i}g_{\ell+1}&(Z_\ell)=\int_{\mathbb{S}_1^{d-1}\times B_R^{d}}b_2^+(\omega_1,v_{\ell+1}-v_i) g_{\ell+1}(Z_{\ell+1}^{i})\,d\omega_1\,dv_{\ell+1},
\end{aligned}
\end{equation*} 

\begin{equation*}
\mathcal{C}_{\ell,\ell+2}^{\infty,R}=\sum_{i=1}^\ell\mathcal{C}_{\ell,\ell+2}^{\infty,R,+,i}-\sum_{i=1}^\ell\mathcal{C}_{\ell,\ell+2}^{\infty,R,-,i},
\end{equation*}
where
\begin{equation*}
\begin{aligned}
\mathcal{C}_{\ell,\ell+2}^{\infty,R,+,i}g_{\ell+2}(Z_\ell)&=\int_{\mathbb{S}_1^{2d-1}\times B_R^{2d}}\frac{b_3^+(\omega_1,\omega_2,v_{\ell+1}-v_i,v_{\ell+2}-v_i)}{\sqrt{1+\langle\omega_1,\omega_2\rangle}} g_{\ell+2}(Z_{\ell+2}^{i*})\,d\omega_1\,d\omega_2\,dv_{\ell+1}\,dv_{\ell+2},
\end{aligned}
\end{equation*}
\begin{equation*}
\begin{aligned}
\mathcal{C}_{\ell,\ell+2}^{\infty,R,-,i}g_{\ell+2}(Z_\ell)&=\int_{\mathbb{S}_1^{2d-1}\times B_R^{2d}}\frac{b_3^+(\omega_1,\omega_2,v_{\ell+1}-v_i,v_{\ell+2}-v_i)}{\sqrt{1+\langle\omega_1,\omega_2\rangle}}g_{\ell+2}(Z_{\ell+2}^i)\,d\omega_1\,d\omega_2\,dv_{\ell+1}\,dv_{\ell+2}.
\end{aligned}
\end{equation*}

Under this notation, the Boltzmann hierarchy observable functional $\widetilde{I}_{s,k,R,\delta}^\infty(t)$ can be expressed, for $1\leq k\leq n$, as a superposition of elementary observables
\begin{equation}\label{superposition Boltzmann}
\widetilde{I}_{s,k,R,\delta}^\infty(t)(X_s)=\sum_{\sigma\in S_k}\sum_{(J,M)\in\mathcal{U}_{s,k,\sigma}}\left(\prod_{i=1}^kj_i\right)\widetilde{I}_{s,k,R,\delta,\sigma}^\infty(t,J,M)(X_s),
\end{equation}
where the elementary observables are defined by
\begin{equation}\label{elementary observable Boltzmann}
\begin{aligned}
\widetilde{I}_{s,k,R,\delta,\sigma}^\infty(t,J,M)(X_s)&=\int_{\mathcal{M}_s^c(X_s)}\phi_s(V_s)\int_{\mathcal{T}_{k,\delta}(t)}S_s^{t-t_1}\mathcal{C}_{s,s+\widetilde{\sigma}_1}^{\infty,R,j_1,m_1} S_{s+\widetilde{\sigma}_1}^{t_1-t_2}...\\
&...S_{s+\widetilde{\sigma}_{k-1}}^{t_{k-1}-t_k}\mathcal{C}_{s+\widetilde{\sigma}_{k-1},s+\widetilde{\sigma}_k}^{\infty,R,j_k,m_k} S_{s+\widetilde{\sigma}_k}^{t_m}f_{0}^{(s+\widetilde{\sigma}_k)}(Z_s)\,dt_k...\,dt_{1}dV_s.
\end{aligned}
\end{equation}

\subsection{Boltzmann hierarchy pseudo-trajectories}\label{subsec Boltzmann pseudo}

We introduce the following notation which we will be constantly using from now on. Let $s\in\mathbb{N}$, $Z_s=(X_s,V_s)\in\mathbb{R}^{2ds}$, $1\leq k\leq n$, $\sigma\in S_k$ and $t\in[0,T]$. Let us recall from \eqref{collision times} the set
$$\mathcal{T}_k(t)=\left\{(t_1,...,t_k)\in\mathbb{R}^k:0=t_{k+1}<t_k<...<t_1<t_0=t\right\},\quad t_0=t,\text{ }t_{k+1}=0.$$
Consider $(t_1,...,t_k)\in\mathcal{T}_k(t)$, $J=(j_1,...,j_k)$, $M=(m_1,...,m_k)$, $(J,M)\in\mathcal{U}_{s,k,\sigma}$. For each $i=1,...,k$, we distignuish two possible situation:
\begin{align}
&\text{If }\sigma_i=1,\text{ we consider } (\omega_{s+\widetilde{\sigma}_i},v_{s+\widetilde{\sigma}_i})\in\mathbb{S}_1^{d-1}\times B_R^d.\label{pseudo binary}\\
&\text{If }\sigma_i=2,\text{ we consider } (\omega_{s+\widetilde{\sigma}_i-1},\omega_{s+\widetilde{\sigma}_i},v_{s+\widetilde{\sigma}_i-1},v_{s+\widetilde{\sigma}_i})\in\mathbb{S}_1^{2d-1}\times B_R^{2d}.\label{pseudo triary}
\end{align}
 For convenience, for each $i=1,...,k$, we will  write $(\bm{\omega}_{\sigma_i,i},\bm{v}_{\sigma_i,i})\in\mathbb{S}_1^{d\sigma_i-1}\times B_R^{d\sigma_i}$ where $(\bm{\omega}_{\sigma_i,i},\bm{v}_{\sigma_i,i})$ is of the form \eqref{pseudo binary} if $\sigma_i=1$ and of the form \eqref{pseudo triary} if $\sigma_i=2$. 

We inductively define the Boltzmann hierarchy pseudo-trajectory of $Z_s$. Roughly speaking, the Boltzmann hierarchy pseudo-trajectory forms the configurations on which particles are adjusted during backwards in time evolution.

Intuitively, assume we are given a configuration $Z_s=(X_s,V_s)\in\mathbb{R}^{2ds}$ 
at time $t_0=t$. $Z_s$ evolves under backwards free flow until the time $t_1$ when the configuration $(\bm{\omega}_{\sigma,1},\bm{v}_{\sigma,1})$ is added, neglecting positions, to the $m_1$-particle, the adjunction being precollisional if $j_1=-1$ and postcollisional if $j_1=1$. We then form an $(s+\widetilde{\sigma}_1)$-configuration and continue this process inductively until time $t_{k+1}=0$. 
More precisely, we inductively construct the Boltzmann hierarchy pseudo-trajectory of $Z_s=(X_s,V_s)\in\mathbb{R}^{2ds}$
as follows:

{\bf{Time $t_0=t$}:} We initially define 
$
Z_s^\infty(t_{0}^-)=\left(x_1^\infty(t_0^-),...,x_s^\infty(t_0^-),v_1^\infty(t_0^-),...,v_s^\infty(t_0^-)\right):=Z_s.
$

{\bf{Time $t_i$}, $i\in\{1,...,k\}$:} Consider $i\in\left\{1,...,k\right\}$ and assume we know 
$$Z_{s+\widetilde{\sigma}_{i-1}}^\infty (t_{i-1}^-)=\left(x_1^\infty(t_{i-1}^-),...,x_{s+\widetilde{\sigma}_{i-1}}^\infty(t_{i-1}^-),v_1^\infty(t_{i-1}^-),...,v_{s+\widetilde{\sigma}_{i-1}}^\infty(t_{i-1}^-)\right).$$
We define $Z_{s+\widetilde{\sigma}_{i-1}}^\infty (t_{i}^+)=\left(x_1^\infty(t_{i}^+),...,x_{s+\widetilde{\sigma}_{i-1}}^\infty(t_{i}^+),v_1^\infty(t_{i}^+),...,v_{s+\widetilde{\sigma}_{i-1}}^\infty(t_{i}^+)\right)$ as:
\begin{equation*}
Z_{s+\widetilde{\sigma}_{i-1}}^\infty(t_i^+):=\left(X_{s+\widetilde{\sigma}_{i-1}}^\infty\left(t_{i-1}^-\right)-\left(t_{i-1}-t_i\right)V_{s+\widetilde{\sigma}_{i-1}}^\infty\left(t_{i-1}^-\right),
V_{s+\widetilde{\sigma}_{i-1}}^\infty\left(t_{i-1}^-\right)\right).
\end{equation*}
We also define $Z_{s+\widetilde{\sigma}_{i}}^\infty(t_i^-)=\left(x_1^\infty(t_{i}^-),...,x_{s+\widetilde{\sigma}_{i}}^\infty(t_{i}^-),v_1^\infty(t_{i}^-),...,v_{s+\widetilde{\sigma}_{i}}^\infty(t_{i}^-)\right)$ as:
\begin{equation*}
\left(x_j^\infty(t_i^-),v_j^\infty(t_i^-)\right):=(x_j^\infty(t_i^+),v_j^\infty(t_i^+)),\quad\forall j\in\{1,...,s+\widetilde{\sigma}_{i-1}\}\setminus\left\{m_i\right\},
\end{equation*}
For the rest of the particles, we distiguish the following cases, depending on $\sigma_i$:
\begin{itemize}
\item $\sigma_i=1$: If $j_i=-1$:
\begin{equation*}
\begin{aligned}
\left(x_{m_i}^\infty(t_i^-),v_{m_i}^\infty(t_i^-)\right)&:=\left(x_{m_{i}}^\infty(t_{i}^+),v_{m_{i}}^\infty(t_{i}^+)\right),\\
\left(x_{s+\widetilde{\sigma}_{i}}^\infty(t_i^-),v_{s+\widetilde{\sigma}_{i}}^\infty(t_i^-)\right)&:=\left(x_{m_{i}}^\infty(t_{i}^+),v_{s+\widetilde{\sigma}_{i}}\right),
\end{aligned}
\end{equation*}
while if $j_i=1$:
\begin{equation*}
\begin{aligned}
\left(x_{m_i}^\infty(t_i^-),v_{m_i}^\infty(t_i^-)\right)&:=\left(x_{m_{i}}^\infty(t_{i}^+),v_{m_{i}}^{\infty'}(t_{i}^+)\right),\\
\left(x_{s+\widetilde{\sigma}_{i}}^\infty(t_i^-),v_{s+\widetilde{\sigma}_{i}}^\infty(t_i^-)\right)&:=\left(x_{m_{i}}^\infty(t_{i}^+),v_{s+\widetilde{\sigma}_{i}}'\right),
\end{aligned}
\end{equation*}
where 
$
(v_{m_{i}}^{\infty'}(t_{i}^-),v_{s+\widetilde{\sigma}_{i}}')=T_{\omega_{s+\widetilde{\sigma}_{i}}}\left(v_{m_{i}}^\infty(t_{i}^+),v_{s+\widetilde{\sigma}_{i}}\right).
$
\item $\sigma_i=2$: If $j_i=-1$:
\begin{equation*}
\begin{aligned}
\left(x_{m_i}^\infty(t_i^-),v_{m_i}^\infty(t_i^-)\right)&:=\left(x_{m_{i}}^\infty(t_{i}^+),v_{m_{i}}^\infty(t_{i}^+)\right),\\
\left(x_{s+\widetilde{\sigma}_{i}-1}^\infty(t_i^-),v_{s+\widetilde{\sigma}_{i}-1}^\infty(t_i^-)\right)&:=\left(x_{m_{i}}^\infty(t_{i}^+),v_{s+\widetilde{\sigma}_{i}-1}\right),\\
\left(x_{s+\widetilde{\sigma}_{i}}^\infty(t_i^-),v_{s+\widetilde{\sigma}_{i}}^\infty(t_i^-)\right)&:=\left(x_{m_{i}}^\infty(t_{i}^+),v_{s+\widetilde{\sigma}_{i}}\right),
\end{aligned}
\end{equation*}
while if $j_i=1$:
\begin{equation*}
\begin{aligned}
\left(x_{m_i}^\infty(t_i^-),v_{m_i}^\infty(t_i^-)\right)&:=\left(x_{m_{i}}^\infty(t_{i}^+),v_{m_{i}}^{\infty*}(t_{i}^+)\right),\\
\left(x_{s+\widetilde{\sigma}_{i}-1}^\infty(t_i^-),v_{s+\widetilde{\sigma}_{i}-1}^\infty(t_i^-)\right)&:=\left(x_{m_{i}}^\infty(t_{i}^+),v_{s+\widetilde{\sigma}_{i}-1}^*\right),\\
\left(x_{s+\widetilde{\sigma}_{i}}^\infty(t_i^-),v_{s+\widetilde{\sigma}_{i}}^\infty(t_i^-)\right)&:=\left(x_{m_{i}}^\infty(t_{i}^+),v_{s+\widetilde{\sigma}_{i}}^*\right),
\end{aligned}
\end{equation*}
where 
$
(v_{m_{i}}^{\infty*}(t_{i}^-),v_{s+\widetilde{\sigma}_{i}-1}^*,v_{s+\widetilde{\sigma}_{i}}^*)=T_{\omega_{s+\widetilde{\sigma}_{i}-1},\omega_{s+\widetilde{\sigma}_{i}}}\left(v_{m_{i}}^\infty(t_{i}^+),v_{s+\widetilde{\sigma}_{i}-1},v_{s+\widetilde{\sigma}_{i}}\right).
$
\end{itemize}

{\bf{Time $t_{k+1}=0$}:}
We finally obtain 
$$Z_{s+\widetilde{\sigma}_{k}}^\infty(0^+)=Z_{s+\widetilde{\sigma}_{k}}^\infty(t_{k+1}^+)=\left(X_{s+\widetilde{\sigma}_{k}}^\infty\left(t_{k}^-\right)-t_kV_{s+\widetilde{\sigma}_{k}}^\infty\left(t_k^-\right),V_{s+\widetilde{\sigma}_{k}}^\infty\left(t_k^-\right)\right).$$
 The process is illustrated in the following diagram:

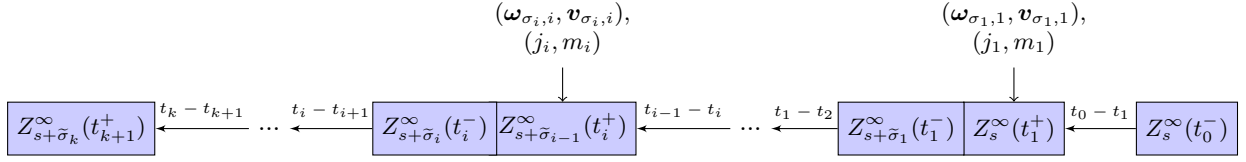
\begin{center}
\begin{tikzpicture}[node distance=2.5cm,auto,>=latex']\label{boltzmann pseudo diagram}
\node[int](0-){\small$ Z_s^\infty(t_0^-)$};
\node[int,pin={[init]above:\small$\begin{matrix}(\bm{\omega}_{\sigma_1,1},\bm{v}_{\sigma_1,1}),\\(j_1,m_1)\end{matrix}$}](1+)[left of=0-,node distance=2.3cm]{\small$Z_s^\infty(t_1^+)$};
\node[int](1-)[left of=1+,node distance=1.5cm]{$Z_{s+\widetilde{\sigma}_1}^\infty(t_1^-)$};
\node[](intermediate1)[left of=1-,node distance=2cm]{...};
\node[int,pin={[init]above:\small$\begin{matrix}(\bm{\omega}_{\sigma_i,i},\bm{v}_{\sigma_i,i}),\\(j_i,m_i)\end{matrix}$}](i+)[left of=intermediate1,node distance=2.5cm]{\small$Z_{s+\widetilde{\sigma}_{i-1}}^\infty(t_i^+)$};
\node[int](i-)[left of=i+,node distance=1.7cm]{\small$Z_{s+\widetilde{\sigma}_i}^\infty(t_i^-)$};
\node[](intermediate2)[left of=i-,node distance=2.2cm]{...};
\node[int](end)[left of=intermediate2,node distance=2.5cm]{\small$Z_{s+\widetilde{\sigma}_k}^\infty(t_{k+1}^+)$};

\path[<-] (1+) edge node {\tiny$t_{0}-t_1$} (0-);
\path[<-] (intermediate1) edge node {\tiny$t_{1}-t_2$} (1-);
\path[<-] (i+) edge node {\tiny$t_{i-1}-t_i$} (intermediate1);
\path[<-] (intermediate2) edge node {\tiny$t_{i}-t_{i+1}$} (i-);
\path[<-] (end) edge node {\tiny$t_{k}-t_{k+1}$} (intermediate2);
\end{tikzpicture}
\end{center}
We give the following definition:
\begin{definition}\label{Boltzmann pseudo}
Let $s\in\mathbb{N}$, $Z_s=(X_s,V_s)\in\mathbb{R}^{2ds}$, $(t_1,...,t_k)\in\mathcal{T}_k(t)$, $J=(j_1,...,j_k)$, $M=(m_1,...,m_k)$, $(J,M)\in\mathcal{U}_{s,k}$ and for each $i=1,...,k$, $\sigma\in S_k$, we consider  $(\bm{\omega}_{\sigma_i,i},\bm{v}_{\sigma_i,i})\in\mathbb{S}_{1}^{d\sigma_i-1}\times B_R^{d\sigma_i}.$ The sequence $\{Z_{s+\widetilde{\sigma}_{i-1}}^\infty(t_i^+)\}_{i=0,...,k+1}$ constructed above is called the Boltzmann hierarchy pseudo-trajectory of $Z_s$.
\end{definition}
\subsection{Reduction to truncated elementary observables}\label{par_reduction to truncated}
We will now use the Boltzmann hierarchy pseudo-trajectory to define the BBGKY hierarchy and Boltzmann hierarchy truncated observables. The convergence proof will then be reduced to the convergence of  the corresponding truncated elementary observables. 

Given $\ell\in\mathbb{N}$, recall the notation from \eqref{both epsilon-epsilon_0}:
$$G_\ell(\epsilon_3,\epsilon_0,\delta)=G_\ell(\epsilon_3,0)\cap G_\ell(\epsilon_0,\delta).$$
 Given $t\in[0,T]$, we also recall from \eqref{separated collision times} the set $\mathcal{T}_{k,\delta}(t)$ of separated collision times:
\begin{equation*}
\mathcal{T}_{k,\delta}(t):=\left\{(t_1,...,t_k)\in\mathcal{T}_k(t):\quad 0\leq t_{i+1}\leq t_i-\delta,\quad\forall i\in [0,k]\right\},\quad t_{k+1}=0,\quad t_0=t.
\end{equation*}

Consider $t\in[0,T]$, $X_s\in\Delta_s^X(\epsilon_0)$, $1\leq k\leq n$, $\sigma\in S_k$ and $(J,M)\in\mathcal{U}_{s,k,\sigma}$ and $(t_1,...,t_k)\in\mathcal{T}_{k,\delta}$. By Proposition \ref{initially good configurations}, for any $V_s\in\mathcal{M}_s^c(X_s)$, we have $Z_s=(X_s,V_s)\in G_s(\epsilon_3,\epsilon_0,\delta)$ which in turn implies $ Z_s^\infty(t_1^+)\in G_s(\epsilon_0,0)$ since $t_0-t_1>\delta$. 
Now we observe that either  \eqref{pre-delta-double-bar}, \eqref{post-delta-double-bar} from Proposition \ref{bad set double} (if the adjunction is binary),  or  \eqref{epsilon pre}, \eqref{epsilon post}  from Proposition \ref{bad set triple} (if the adjunction is ternary), yield that there is a set $\mathcal{B}_{m_1}\left(Z_s^{\infty}\left(t_1^+\right)\right)\subseteq \mathbb{S}_1^{d\sigma_1-1}\times B_R^{d\sigma_1}$  such that :
$$Z_{s+\widetilde{\sigma}_1}^\infty(t_2^+)\in G_{s+\widetilde{\sigma}_1}(\epsilon_0,0),\quad\forall (\bm{\omega}_{\sigma_1,1},\bm{v}_{\sigma_1,1})\in\mathcal{B}_{m_1}^c\left(Z_s^{\infty}\left(t_1^+\right)\right),$$
$$\mathcal{B}_{m_1}^c\left(Z_s^{\infty}\left(t_1^+\right)\right):=(\mathbb{S}_1^{d\sigma_i-1}\times B_R^{d\sigma_i})^+\left(v_{m_1}^\infty\left(t_1^+\right)\right)\setminus \mathcal{B}_{m_1}\left(Z_s^{\infty}\left(t_1^+\right)\right).$$
Clearly this process can be iterated. In particular, given $i\in\left\{2,...,k\right\}$, we have 
  $$Z_{s+\widetilde{\sigma}_{i-1}}^\infty(t_{i}^+)\in G_{s+\widetilde{\sigma}_{i-1}}(\epsilon_0,0),$$ 
  so there exists a set $\mathcal{B}_{m_i}\left(Z_{s+\widetilde{\sigma}_{i-1}}^\infty\left(t_{i}^+\right)\right)\subseteq \mathbb{S}_1^{d\sigma_i-1}\times B_R^{d\sigma_i}$ such that: 
  \begin{equation}\label{pseudo applicable} Z_{s+\widetilde{\sigma}_{i}}^\infty(t_{i+1}^+)\in G_{s+\widetilde{\sigma}_{i}}(\epsilon_0,0),\quad\forall (\bm{\omega}_{\sigma_i,i},\bm{v}_{\sigma_i,i})\in \mathcal{B}_{m_i}^c\left(Z_{s+\widetilde{\sigma}_{i-1}}^\infty\left(t_{i}^+\right)\right),
  \end{equation}
  where
  $$\mathcal{B}_{m_i}^c\left(Z_s^{\infty}\left(t_i^+\right)\right):=(\mathbb{S}_1^{d\sigma_i-1}\times B_R^{d\sigma_i})^+\left(v_{m_i}^\infty\left(t_i^+\right)\right)\setminus \mathcal{B}_{m_i}\left(Z_{s+\widetilde{\sigma}_i}^{\infty}\left(t_i^+\right)\right).$$ 
  We finally  obtain $Z_{s+\widetilde{\sigma}_{k}}^\infty(0^+)\in G_{s+\widetilde{\sigma}_{k}}(\epsilon_0,0)$.

Let us now define the truncated elementary observables. Heuristically we will truncate the domains of adjusted particles in the definition of the observables $\widetilde{I}_{s,k,R,\delta}^N$, $\widetilde{I}_{s,k,R,\delta}^\infty$, defined in \eqref{good observables BBGKY }-\eqref{good observables Boltzmann}.

More precisely, consider $1\leq k\leq n$, $\sigma\in S_k$, $(J,M)\in\mathcal{U}_{s,k,\sigma}$ and $t\in [0,T]$. For $X_s\in\Delta_s^X(\epsilon_0)$, Proposition \ref{initially good configurations} implies there is a set of velocities $\mathcal{M}_s(X_s)\subseteq B_R^{2d}$ such that $Z_s=(X_s,V_s)\in G_s(\epsilon_3,\epsilon_0,\delta),\quad\forall V_s\in\mathcal{M}_s^c(X_s)$. Following the reasoning above, we define the BBGKY hierarchy truncated observables as:
\begin{equation}\label{truncated BBGKY}
\begin{aligned}
J_{s,k,R,\delta,\sigma}^N(t,J,M)(X_s)&=\int_{\mathcal{M}_s^c(X_s)}\phi_s(V_s)\int_{\mathcal{T}_{k,\delta}(t)}T_s^{t-t_1}\widetilde{\mathcal{C}}_{s,s+\widetilde{\sigma}_{1}}^{N,R,j_1,m_1} T_{s+\widetilde{\sigma}_{1}}^{t_1-t_2}...\\
&...\widetilde{\mathcal{C}}_{s+\widetilde{\sigma}_{k-1},s+\widetilde{\sigma}_{k}}^{N,R,j_k,m_k} T_{s+\widetilde{\sigma}_{k}}^{t_m}f_{0}^{(s+\widetilde{\sigma}_{k})}(Z_s)\,dt_k,...\,dt_{1}dV_s,
\end{aligned}
\end{equation}
where for each $i=1,...,k$, we denote
\begin{equation*}
\begin{aligned}
\widetilde{\mathcal{C}}_{s+\widetilde{\sigma}_{i-1},s+\widetilde{\sigma}_{i}}^{N,R,j_i,m_i}&g_{N,s+\widetilde{\sigma}_{i}}=\mathcal{C}_{s+\widetilde{\sigma}_{i-1},s+\widetilde{\sigma}_{i}}^{N,R,j_i,m_i}\left[g_{N,s+\widetilde{\sigma}_{i}}\mathds{1}_{(\bm{\omega}_{\sigma_i,i},\bm{v}_{\sigma_i,i})\in
\mathcal{B}^c_{m_i}\left(Z_{s+\widetilde{\sigma}_{i-1}}^\infty\left(t_i^+\right)\right)}\right].
\end{aligned}
\end{equation*}

In the same spirit, for $X_s\in\Delta_s^X(\epsilon_0)$, we define the Boltzmann hierarchy truncated elementary observables as:
\begin{equation}\label{truncated Boltzmann}
\begin{aligned}
J_{s,k,R,\delta,\sigma}^\infty(t,J,M)(X_s)&=\int_{\mathcal{M}_s^c(X_s)}\phi_s(V_s)\int_{\mathcal{T}_{k,\delta}(t)}S_s^{t-t_1}\widetilde{\mathcal{C}}_{s,s+\widetilde{\sigma}_{1}}^{\infty,R,j_1,m_1} S_{s\widetilde{\sigma}_{1}}^{t_1-t_2}...\\
&...\widetilde{\mathcal{C}}_{s+\widetilde{\sigma}_{k-1},s+\widetilde{\sigma}_{k}}^{\infty,R,j_k,m_k} S_{s+\widetilde{\sigma}_{k}}^{t_m}f_{0}^{(s+\widetilde{\sigma}_{k})}(Z_s)\,dt_k,...\,dt_{1}dV_s,
\end{aligned}
\end{equation}
where for each $i=1,...,k$, we denote
\begin{equation*}\widetilde{\mathcal{C}}_{s+\widetilde{\sigma}_{i-1},s+\widetilde{\sigma}_{i}}^{\infty,R,j_i,m_i}g_{s+\widetilde{\sigma}_{i}}=\mathcal{C}_{s+\widetilde{\sigma}_{i-1},s+\widetilde{\sigma}_{i}}^{\infty,R,j_i,m_i}\left[g_{s+\widetilde{\sigma}_{i}}\mathds{1}_{(\bm{\omega}_{\sigma_i,i},\bm{v}_{\sigma_i,i})\in\mathcal{B}^c_{m_i}\left(Z_{s+\widetilde{\sigma}_{i-1}}^\infty\left(t_i^+\right)\right)}\right].
\end{equation*}
 
 Recalling the observables $\widetilde{I}_{s,k,R,\delta,\sigma}^N$, $\widetilde{I}_{s,k,R,\delta,\sigma}^\infty$ from \eqref{elementary observable BBGKY}, \eqref{elementary observable Boltzmann} and using Proposition \ref{bad set double measure} or Proposition \ref{bad set triple measure}, we obtain:
 \begin{proposition}\label{truncated element estimate} Let $s,n\in\mathbb{N}$,  $\alpha,\epsilon_0,R,\eta,\delta$ be parameters as in \eqref{choice of parameters},  $(N,\epsilon_2,\epsilon_3)$   in the scaling \eqref{scaling} with $\epsilon_2<<\epsilon_3<<\alpha$ and $t\in[0,T]$.  Then the following estimates hold:
 \begin{equation*}
 \begin{aligned}\sum_{k=1}^n\sum_{\sigma\in S_k}\sum_{(J,M)\in\mathcal{U}_{s,k,\sigma}}&\|\widetilde{I}_{s,k,R,\delta,\sigma}^N(t,J,M)-J_{s,k,R,\delta,\sigma}^N(t,J,M)\|_{L^\infty\left(\Delta_s^X\left(\epsilon_0\right)\right)}\leq \\
 &\leq C_{d,s,\mu_0,T}^n\|\phi_s\|_{L^\infty_{V_s}} R^{d(s+3n)}\eta^{\frac{d-1}{4d+2}}\|F_{N,0}\|_{N,\beta_0,\mu_0},
 \end{aligned}
 \end{equation*}
 \begin{equation*}
 \begin{aligned}\sum_{k=1}^n\sum_{\sigma\in S_k}\sum_{(J,M)\in\mathcal{U}_{s,k,\sigma}}&\|\widetilde{I}_{s,k,R,\delta,\sigma}^\infty(t,J,M)-J_{s,k,R,\delta,\sigma}^\infty(t,J,M)\|_{L^\infty\left(\Delta_s^X\left(\epsilon_0\right)\right)}\leq \\
 &\leq C_{d,s,\mu_0,T}^n\|\phi_s\|_{L^\infty_{V_s}} R^{d(s+3n)}\eta^{\frac{d-1}{4d+2}}\|F_{0}\|_{\infty,\beta_0,\mu_0}.
 \end{aligned}
 \end{equation*}
 \end{proposition}
 \begin{proof}
 As usual, it suffices to prove the estimate for the BBGKY hierarchy case and the Boltzmann hierarchy case follows similarly. Fix $k\in\left\{1,...,n\right\}$, $\sigma\in S_k$ and $(J,M)\in\mathcal{U}_{s,k,\sigma}$. We first estimate the difference:
 \begin{equation}\label{estimated difference}
 \widetilde{I}_{s,k,R,\delta}^N(t,J,M)(X_s)-J_{s,k,R,\delta}^N(t,J,M)(X_s).
 \end{equation}
 
 Cauchy-Schwartz inequality and triangle inequality  imply
\begin{align} 
|\langle\omega_1,v_1-v\rangle|&\leq 2R,\quad\forall \omega_1\in\mathbb{S}_1^{d-1},\quad\forall v,v_1\in B_R^d,\label{triangle on cross binary}\\
 \big|b_3(\omega_1,\omega_2,v_1-v,v_2-v)\big|&\leq 4R,\quad\forall(\omega_1,\omega_2)\in\mathbb{S}_1^{2d-1}, \quad\forall v,v_1,v_2\in B_R^{d},\label{triangle on cross ternary}
 \end{align}
 so
 \begin{align}
 \int_{\mathbb{S}_1^{d-1}\times B_R^{d}}|\langle\omega_1,v_1-v\rangle|\,d\omega_1\,dv_1&\leq C_d R^{d+1}\leq C_dR^{3d},\quad\forall v\in B_R^d,\label{estimate on the rest of the terms binary}\\
 \int_{\mathbb{S}_1^{2d-1}\times B_R^{2d}}|b_3(\omega_1,\omega_2,v_1-v,v_2-v_2)|\,d\omega_1\,d\omega_2\,dv_1\,dv_2&\leq C_d R^{2d+1}\leq C_dR^{3d},\quad\forall v\in B_R^d,\label{estimate on the rest of the terms ternary}
 \end{align}
 since $R>>1$.
 But in order to estimate the difference \eqref{estimated difference}, we integrate at least once over $\mathcal{B}_{m_i}\left(Z_{s+2i-2}^\infty\left(t_{i}^+\right)\right)$ for some $i\in\left\{1,...,k\right\}$. For convenience, given $v\in\mathbb{R}^d$, let us write
 \begin{equation}\label{mixed crossection}
 b_{\sigma_i}(\bm{\omega}_{\sigma_i,i},\bm{v}_{\sigma_i,i},v):=\begin{cases}
b_2(\omega_{s+\widetilde{\sigma}_i},v_{s+\widetilde{\sigma}_i}-v),\quad\text{if }\sigma_i=1,\\
b_3(\omega_{s+\widetilde{\sigma}_i-1},\omega_{s+\widetilde{\sigma}_i},v_{s+\widetilde{\sigma}_i-1}-v,v_{s+\widetilde{\sigma}_i}-v),\quad\text{if }\sigma_i=2.
 \end{cases}
 \end{equation}
 Under this notation, \eqref{triangle on cross binary}-\eqref{triangle on cross ternary} together with Proposition \ref{bad set double measure} or Proposition \ref{bad set triple measure}, depending on whether the adunction is binary or ternary,   yield the estimate 
 \begin{equation}\label{exclusion bad set 1}
 \begin{aligned}
 \int_{\mathcal{B}_{m_i}\left(Z_{s+\widetilde{\sigma}_{i-1}}^\infty\left(t_{i}^+\right)\right)}|b_{\sigma_i}(\bm{\omega}_{\sigma_i,i},\bm{v}_{\sigma_i,i},v)|\,d\bm{\omega}_{\sigma_i,i}\bm{v}_{\sigma_i,i}&\leq C_d(s+\widetilde{\sigma}_{i-1})R^{d\sigma_i+1}\eta^{\frac{d-1}{2d\sigma_i+2}}\\
 &\leq C_d(s+2k)R^{3d}\eta^{\frac{d-1}{4d+2}} ,\quad\forall v\in B_R^d,
 \end{aligned}
 \end{equation}
 since $R>>1$ and $\eta<<1$.
 
 Moreover, we have the elementary inequalities:
 \begin{align}
 \|f_{N,0}^{(s+\widetilde{\sigma}_k)}\|_{L^\infty}&\leq e^{-(s+\widetilde{\sigma}_k)\mu_0}\|F_{N,0}\|_{N,\beta_0,\mu_0}\leq e^{-(s+k)\mu_0}\|F_{N,0}\|_{N,\beta_0,\mu_0}\label{exclusion bad set 2 norms},\\
 \int_{T_{k,\delta}(t)}\,dt_1...\,dt_k&\leq\int_0^t\int_0^{t_1}...\int_0^{t_{k-1}}\,dt_1...\,dt_k=\frac{t^k}{k!}\leq\frac{T^k}{k!}\label{exclusion bad set 2 time}.
 \end{align}
 Therefore, \eqref{estimate on the rest of the terms binary}-\eqref{exclusion bad set 2 time} imply
 \begin{equation*}
 \begin{aligned}
 \big|&\widetilde{I}_{s,k,R,\delta,\sigma}^N(t,J,M)(X_s)-J_{s,k,R,\delta,\sigma}^N(t,J,M)(X_s)\big|\\
 &\leq \|\phi_s\|_{L^\infty_{V_s}}e^{-(s+k)\mu_0}\|F_{N,0}\|_{N,\beta_0,\mu_0}C_d R^{ds}C_d^{k-1}R^{3d(k-1)}(s+2k) C_dR^{3d}\eta^{\frac{d-1}{4d+2}}\frac{T^k}{k!}\\
 &\leq C_{d,s,\mu_0,T}^k\|\phi_s\|_{L^\infty_{V_s}}\frac{(s+2k)}{k!}R^{d(s+3k)}\eta^{\frac{d-1}{4d+2}} \|F_{N,0}\|_{N,\beta_0,\mu_0}.
 \end{aligned}
 \end{equation*}
 Adding for all $(J,M)\in \mathcal{U}_{s,k,\sigma}$ we have  $2^ks(s+\widetilde{\sigma}_1)...(s+\widetilde{\sigma}_{k-1})\leq 2^k(s+2k)^k$ contributions, thus
 \begin{equation*}
 \begin{aligned}
 &\sum_{(J,M)\in\mathcal{U}_{s,k,\sigma}}\|\widetilde{I}_{s,k,R,\delta,\sigma}^N(t,J,M)-J_{s,k,R,\delta,\sigma}^N(t,J,M)\|_{L^\infty\left(\Delta_s^X\left(\epsilon_0\right)\right)}\\
 &\leq C_{d,s,\mu_0,T}^k\|\phi_s\|_{L^\infty_{V_s}}R^{d(s+3k)}\frac{(s+2k)^{k+1}}{k!}\eta^{\frac{d-1}{4d+2}}\|F_{N,0}\|_{N,\beta_0,\mu_0}\\
 &\leq C_{d,s,\mu_0,T}^k\|\phi_s\|_{L^\infty_{V_s}}R^{d(s+3k)}\eta^{\frac{d-1}{4d+2}}\|F_{N,0}\|_{N,\beta_0,\mu_0},
 \end{aligned}
 \end{equation*}
since 
$$\frac{(s+2k)^{k+1}}{k!}\leq\frac{2^{k+1}(s+k)(s+k)^{k}}{k!}\leq 2^{k+1}(s+k)e^{s+k}\leq C_s^k ,$$ 
 Summing over $\sigma\in S_k$, $k=1,...,n$, we get the required estimate.
 \end{proof}
In the next section, in order to conclude the convergence proof, we will estimate the differences of the corresponding BBGKY hierarchy and Boltzmann hierarchy truncated elementary observables in the scaled limit.
\section{Convergence proof}\label{sec:convergence proof}
Recall from Subsection \ref{par_reduction to truncated} that given $s\in\mathbb{N}$, $t\in[0,T]$, and parameters satisfying \eqref{choice of parameters},  we have reduced the convergence proof to controlling the differences:  $$J_{s,k,R,\delta}^N(t,J,M)-J_{s,k,R,\delta}^\infty (t,J,M)$$ for given $1\leq k\leq n$ and  $(J,M)\in\mathcal{U}_{s,k}$, where $J_{s,k,R,\delta}^N(t,J,M)$, $J_{s,k,R,\delta}^\infty (t,J,M)$ are given by \eqref{truncated BBGKY}, \eqref{truncated Boltzmann}. This will be the aim of this section. 

Throughout this section $s\in\mathbb{N}$, $\phi_s\in C_c(\mathbb{R}^{ds})$ will be fixed, $(N,\epsilon_2,\epsilon_3)$ are in the scaling \eqref{scaling}, $\beta_0>0$, $\mu_0\in\mathbb{R}$, $T>0$ are given by the statements of Theorem \ref{well posedness BBGKY} and Theorem \ref{well posedness boltzmann}, and the parameters $n,\delta,R,\eta,\epsilon_0,\alpha$ satisfy \eqref{choice of parameters}.

\subsection{BBGKY hierarchy pseudo-trajectories and proximity to the Boltzmann hierarchy pseudo-trajectories}
In the same spirit as in Subsection \ref{subsec Boltzmann pseudo}, we may define the BBGKY hierarchy pseudo-trajectory. Consider $s\in\mathbb{N}$, $(N,\epsilon_2,\epsilon_3)$ in the scaling \eqref{scaling}, $k\in\mathbb{N}$ and $t\in[0,T]$. Let us recall from \eqref{collision times} the set
$$\mathcal{T}_k(t)=\left\{(t_1,...,t_k)\in\mathbb{R}^k:0=t_{k+1}<t_k<...<t_1<t_0=t\right\},$$
where we use the convention $t_0=t$ and $t_{k+1}=0$.
Consider $(t_1,...,t_k)\in\mathcal{T}_k(t)$, $\sigma\in S_k$, $J=(j_1,...,j_k)$, $M=(m_1,...,m_k)$, $(J,M)\in\mathcal{U}_{s,k,\sigma}$ and for each $i=1,...,k$, we consider  $(\bm{\omega}_{\sigma_i,i},\bm{v}_{\sigma_i,i})\in\mathbb{S}_{1}^{d\sigma_i-1}\times B_R^{d\sigma_{i}}.$

The process followed is similar to the construction of the Boltzmann hierarchy pseudo-trajectory. The only difference is that we take into account the diameter $\epsilon_2$ or the interaction zone $\epsilon_3$ of the adjusted particles in each step. 

More precisely, we inductively construct the BBGKY hierarchy pseudo-trajectory of $Z_s=(X_s,V_s)\in\mathbb{R}^{2ds}$
as follows:

{\bf{Time $t_0=t$}:} We initially define 
$
Z_s^N(t_{0}^-)=\left(x_1^N(t_0^-),...,x_s^N(t_0^-),v_1^N(t_0^-),...,v_s^N(t_0^-)\right):=Z_s.
$

{\bf{Time $t_i$}, $i\in\{1,...,k\}$:} Consider $i\in\left\{1,...,k\right\}$ and assume we know 
$$Z_{s+\widetilde{\sigma}_{i-1}}^N (t_{i-1}^-)=\left(x_1^N(t_{i-1}^-),...,x_{s+\widetilde{\sigma}_{i-1}}^N(t_{i-1}^-),v_1^N(t_{i-1}^-),...,v_{s+\widetilde{\sigma}_{i-1}}^N(t_{i-1}^-)\right).$$
We define $Z_{s+\widetilde{\sigma}_{i-1}}^N (t_{i}^+)=\left(x_1^N(t_{i}^+),...,x_{s+\widetilde{\sigma}_{i-1}}^N(t_{i}^+),v_1^N(t_{i}^+),...,v_{s+\widetilde{\sigma}_{i-1}}^N(t_{i}^+)\right)$ as:
\begin{equation*}
Z_{s+\widetilde{\sigma}_{i-1}}^N(t_i^+):=\left(X_{s+\widetilde{\sigma}_{i-1}}^N\left(t_{i-1}^-\right)-\left(t_{i-1}-t_i\right)V_{s+\widetilde{\sigma}_{i-1}}^N\left(t_{i-1}^-\right),
V_{s+\widetilde{\sigma}_{i-1}}^N\left(t_{i-1}^-\right)\right).
\end{equation*}
We also define $Z_{s+\widetilde{\sigma}_{i}}^N(t_i^-)=\left(x_1^N(t_{i}^-),...,x_{s+\widetilde{\sigma}_{i}}^N(t_{i}^-),v_1^N(t_{i}^-),...,v_{s+\widetilde{\sigma}_{i}}^N(t_{i}^-)\right)$ as:
\begin{equation*}
\left(x_j^N(t_i^-),v_j^N(t_i^-)\right):=(x_j^N(t_i^+),v_j^N(t_i^+))\quad\forall j\in\{1,...,s+\widetilde{\sigma}_{i-1}\}\setminus\left\{m_i\right\},
\end{equation*}
For the rest of the particles, we distiguish the following cases, depending on $\sigma_i$:
\begin{itemize}
\item $\sigma_i=1$: If $j_i=-1$:
\begin{equation*}
\begin{aligned}
\left(x_{m_i}^N(t_i^-),v_{m_i}^N(t_i^-)\right)&:=\left(x_{m_{i}}^N(t_{i}^+),v_{m_{i}}^N(t_{i}^+)\right),\\
\left(x_{s+\widetilde{\sigma}_{i}}^N(t_i^-),v_{s+\widetilde{\sigma}_{i}}^N(t_i^-)\right)&:=\left(x_{m_{i}}^N(t_{i}^+)-\epsilon_2\omega_{s+\widetilde{\sigma}_{i}},v_{s+\widetilde{\sigma}_{i}}\right),
\end{aligned}
\end{equation*}
while if $j_i=1$:
\begin{equation*}
\begin{aligned}
\left(x_{m_i}^N(t_i^-),v_{m_i}^N(t_i^-)\right)&:=\left(x_{m_{i}}^N(t_{i}^+),v_{m_{i}}^{N'}(t_{i}^+)\right),\\
\left(x_{s+\widetilde{\sigma}_{i}}^N(t_i^-),v_{s+\widetilde{\sigma}_{i}}^N(t_i^-)\right)&:=\left(x_{m_{i}}^N(t_{i}^+)+\epsilon_2\omega_{s+\widetilde{\sigma}_{i}},v_{s+\widetilde{\sigma}_{i}}'\right),
\end{aligned}
\end{equation*}
where 
$
(v_{m_{i}}^{N'}(t_{i}^-),v_{s+\widetilde{\sigma}_{i}}')=T_{\omega_{s+\widetilde{\sigma}_{i}}}\left(v_{m_{i}}^N(t_{i}^+),v_{s+\widetilde{\sigma}_{i}}\right).
$
\item $\sigma_i=2$: If $j_i=-1$:
\begin{equation*}
\begin{aligned}
\left(x_{m_i}^N(t_i^-),v_{m_i}^N(t_i^-)\right)&:=\left(x_{m_{i}}^N(t_{i}^+),v_{m_{i}}^N(t_{i}^+)\right),\\
\left(x_{s+\widetilde{\sigma}_{i}-1}^N(t_i^-),v_{s+\widetilde{\sigma}_{i}-1}^N(t_i^-)\right)&:=\left(x_{m_{i}}^N(t_{i}^+)-\sqrt{2}\epsilon_3\omega_{s+\widetilde{\sigma}_{i}-1},v_{s+\widetilde{\sigma}_{i}-1}\right),\\
\left(x_{s+\widetilde{\sigma}_{i}}^N(t_i^-),v_{s+\widetilde{\sigma}_{i}}^N(t_i^-)\right)&:=\left(x_{m_{i}}^N(t_{i}^+)-\sqrt{2}\epsilon_3\omega_{s+\widetilde{\sigma}_{i}},v_{s+\widetilde{\sigma}_{i}}\right),
\end{aligned}
\end{equation*}
while if $j_i=1$:
\begin{equation*}
\begin{aligned}
\left(x_{m_i}^N(t_i^-),v_{m_i}^N(t_i^-)\right)&:=\left(x_{m_{i}}^N(t_{i}^+),v_{m_{i}}^{N*}(t_{i}^+)\right),\\
\left(x_{s+\widetilde{\sigma}_{i}-1}^N(t_i^-),v_{s+\widetilde{\sigma}_{i}-1}^N(t_i^-)\right)&:=\left(x_{m_{i}}^N(t_{i}^+)+\sqrt{2}\epsilon_3\omega_{s+\widetilde{\sigma}_{i}-1},v_{s+\widetilde{\sigma}_{i}-1}^*\right),\\
\left(x_{s+\widetilde{\sigma}_{i}}^N(t_i^-),v_{s+\widetilde{\sigma}_{i}}^N(t_i^-)\right)&:=\left(x_{m_{i}}^N(t_{i}^+)+\sqrt{2}\epsilon_3\omega_{s+\widetilde{\sigma}_{i}},v_{s+\widetilde{\sigma}_{i}}^*\right),
\end{aligned}
\end{equation*}
where 
$
(v_{m_{i}}^{N*}(t_{i}^-),v_{s+\widetilde{\sigma}_{i}-1}^*,v_{s+\widetilde{\sigma}_{i}}^*)=T_{\omega_{s+\widetilde{\sigma}_{i}-1},\omega_{s+\widetilde{\sigma}_{i}}}\left(v_{m_{i}}^N(t_{i}^+),v_{s+\widetilde{\sigma}_{i}-1},v_{s+\widetilde{\sigma}_{i}}\right).
$
\end{itemize}

{\bf{Time $t_{k+1}=0$}:}
We finally obtain 
$$Z_{s+\widetilde{\sigma}_{k}}^N(0^+)=Z_{s+\widetilde{\sigma}_{k}}^N(t_{k+1}^+)=\left(X_{s+\widetilde{\sigma}_{k}}^N\left(t_{k}^-\right)-t_kV_{s+\widetilde{\sigma}_{k}}^N\left(t_k^-\right),V_{s+\widetilde{\sigma}_{k}}^N\left(t_k^-\right)\right).$$
 The process is illustrated in the following diagram:

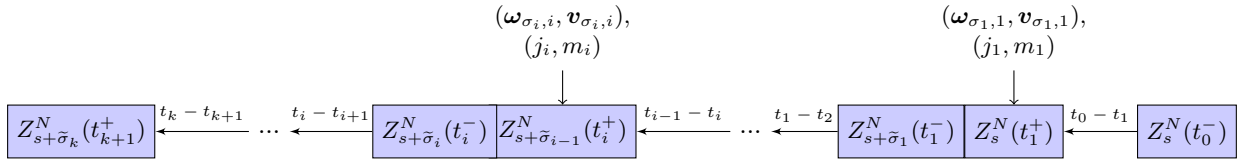
\begin{center}
\begin{tikzpicture}[node distance=2.5cm,auto,>=latex']\label{bbgjy pseudo diagram}
\node[int](0-){\small$ Z_s^N(t_0^-)$};
\node[int,pin={[init]above:\small$\begin{matrix}(\bm{\omega}_{\sigma_1,1},\bm{v}_{\sigma_1,1}),\\(j_1,m_1)\end{matrix}$}](1+)[left of=0-,node distance=2.3cm]{\small$Z_s^N(t_1^+)$};
\node[int](1-)[left of=1+,node distance=1.5cm]{$Z_{s+\widetilde{\sigma}_1}^N(t_1^-)$};
\node[](intermediate1)[left of=1-,node distance=2cm]{...};
\node[int,pin={[init]above:\small$\begin{matrix}(\bm{\omega}_{\sigma_i,i},\bm{v}_{\sigma_i,i}),\\(j_i,m_i)\end{matrix}$}](i+)[left of=intermediate1,node distance=2.5cm]{\small$Z_{s+\widetilde{\sigma}_{i-1}}^N(t_i^+)$};
\node[int](i-)[left of=i+,node distance=1.7cm]{\small$Z_{s+\widetilde{\sigma}_i}^N(t_i^-)$};
\node[](intermediate2)[left of=i-,node distance=2.2cm]{...};
\node[int](end)[left of=intermediate2,node distance=2.5cm]{\small$Z_{s+\widetilde{\sigma}_{k}}^N(t_{k+1}^+)$};

\path[<-] (1+) edge node {\tiny$t_{0}-t_1$} (0-);
\path[<-] (intermediate1) edge node {\tiny$t_{1}-t_2$} (1-);
\path[<-] (i+) edge node {\tiny$t_{i-1}-t_i$} (intermediate1);
\path[<-] (intermediate2) edge node {\tiny$t_{i}-t_{i+1}$} (i-);
\path[<-] (end) edge node {\tiny$t_{k}-t_{k+1}$} (intermediate2);
\end{tikzpicture}
\end{center}

We give the following definition:
\begin{definition}\label{BBGKY pseudo}
Let $s\in\mathbb{N}$, $Z_s=(X_s,V_s)\in\mathbb{R}^{2ds}$, $(t_1,...,t_k)\in\mathcal{T}_k(t)$, $J=(j_1,...,j_k)$, $M=(m_1,...,m_k)$, $(J,M)\in\mathcal{U}_{s,k}$ and for each $i=1,...,k$, $\sigma\in S_k$, we consider  $(\bm{\omega}_{\sigma_i,i},\bm{v}_{\sigma_i,i})\in\mathbb{S}_{1}^{d\sigma_i-1}\times B_R^{d\sigma_i}.$ The sequence $\{Z_{s+\widetilde{\sigma}_{i-1}}^N(t_i^+)\}_{i=0,...,k+1}$ constructed above is called the BBGKY hierarchy pseudo-trajectory of $Z_s$.
\end{definition}
We now state the following elementary proximity result of the corresponding BBGKY hierarchy and Boltzmann hierarchy pseudo-trajectories.
 \begin{lemma}\label{proximity}
  Let $s\in\mathbb{N}$, $Z_s=(X_s,V_s)\in\mathbb{R}^{2ds}$, $1\leq k\leq n$, $\sigma\in S_k$, $(J,M)\in
\mathcal{U}_{s,k,\sigma}$, $t\in[0,T]$ and $(t_1,...,t_k)\in\mathcal{T}_{k}(t)$. For each $i=1,...,k$, consider $(\bm{\omega}_{\sigma_i,i},\bm{v}_{\sigma_i,i})\in\mathbb{S}_{1}^{d\sigma_i-1}\times\mathbb{R}^{d\sigma_i}$. Then for all $i=1,...,k$ and $\ell=1,...,s+\widetilde{\sigma}_{i-1}$, we have
 \begin{equation}\label{proximity claim}
 |x_{\ell}^N(t_{i}^+)-x_{\ell}^\infty(t_{i}^+)|\leq \sqrt{2}\epsilon_3 (i-1),\quad v_\ell^N(t_{i}^+)=v_\ell^\infty(t_{i}^+).
 \end{equation}
 Moreover, if $s<n$, then for each $i\in\{1,...,k\}$, there holds:
 \begin{equation}\label{total proximity}
 \left|X_{s+\widetilde{\sigma}_{i-1}}^N(t_i^+)-X_{s+\widetilde{\sigma}_{i-1}}^\infty(t_i^+)\right|\leq n^{3/2}\epsilon_3.
 \end{equation}
 \end{lemma}
 \begin{proof}
 We first prove \eqref{proximity claim} by induction on $i\in\left\{1,...,k\right\}$. For $i=1$ the result is trivial since the pseudo-trajectories initially coincide  by construction. Assume the conclusion holds for $i\in\left\{1,...,k-1\right\}$ i.e. for all $\ell\in\left\{1,...,s+\widetilde{\sigma}_{i-1}\right\}$, there holds: 
 \begin{equation}\label{proximity induction}
 |x_{\ell}^N(t_{i}^+)-x_\ell^\infty(t_{i}^+)|\leq \sqrt{2}\epsilon_3(i-1)\quad\text{and}\quad v_\ell^N(t_{i}^+)=v_{\ell}^\infty(t_{i}^+).
 \end{equation}
 We prove the conclusion holds for $(i+1)\in\{2,...,k\}$.
 We need to take different cases for $j_i\in\{-1,1\}$ and $\sigma_i\in\{1,2\}$. 

$\bullet\quad \sigma_i=1, j_i=-1$: For the Boltzmann pseudo-trajectory we get
 \begin{equation*}
 \begin{aligned}
 & x_\ell^\infty(t_{i+1}^+)=x_\ell^\infty(t_{i}^+)-(t_{i}-t_{i+1})v_{\ell}^\infty(t_{i}^+),\quad v_{\ell}^\infty(t_{i+1}^+)=v_\ell^\infty(t_{i}^+),\quad \forall \ell\in\{1,..., s+\widetilde{\sigma}_{i-1}\}\setminus\{m_i\},\\
 & x_{m_i}^\infty(t_{i+1}^+)=x_{m_i}^\infty(t_{i}^+)-(t_{i}-t_{i+1})v_{\ell}^\infty(t_{i}^+),\quad v_{m_i}^\infty(t_{i+1}^+)=v_{m_i}^\infty(t_{i}^+),\\
 &x_{s+\widetilde{\sigma}_i}^\infty(t_{i+1}^+)=x_{m_i}^\infty(t_{i}^+)-(t_{i}-t_{i+1})v_{s+\widetilde{\sigma}_i},\quad v_{s+\widetilde{\sigma}_i}^\infty(t_{i+1}^+)=v_{s+\widetilde{\sigma}_i},
 \end{aligned}
 \end{equation*}
 while for the BBGKY hierarchy pseudo-trajectory we get
 \begin{equation*}
 \begin{aligned}
 & x_\ell^N(t_{i+1}^+)=x_\ell^N(t_{i}^+)-(t_{i}-t_{i+1})v_{\ell}^N(t_{i}^+),\quad v_{\ell}^N(t_{i+1}^+)=v_\ell^N(t_{i}^-),\quad \forall \ell\in\{1,..., s+\widetilde{\sigma}_{i-1}\}\setminus\{m_i\},\\ 
 & x_{m_i}^N(t_{i+1}^+)=x_{m_i}^N(t_{i}^+)-(t_{i}-t_{i+1})v_{m_i}^N(t_{i}^+),\quad v_{m_i}^N(t_{i+1}^+)=v_{m_i}^N(t_{i}^-),\\
  &x_{s+\widetilde{\sigma}_i}^N(t_{i+1}^+)=x_{m_i}^N(t_{i}^+)-(t_{i}-t_{i+1})v_{s+\widetilde{\sigma}_i}-\epsilon_2\omega_{s+\widetilde{\sigma}_i},\quad v_{s+\widetilde{\sigma}_i}^N(t_{i+1}^+)=v_{s+\widetilde{\sigma}_i}.
 \end{aligned}
 \end{equation*}
 So, for any $\ell\in\{1,..., s+\widetilde{\sigma}_{i-1}\}$, the induction assumption \eqref{proximity induction} implies
 \begin{align*}
 &v_\ell^N(t_{i+1}^+)=v_\ell^N(t_{i}^+)=v_\ell^\infty(t_{i}^+)=v_\ell^\infty(t_{i+1}^+),\\
 &|x_\ell^N(t_{i+1}^+)-x_\ell^\infty(t_{i+1}^+)|=|x_\ell^N(t_{i}^+)-x_\ell^\infty(t_{i}^+)|\leq\sqrt{2}\epsilon_3(i-1).
 \end{align*}
Moreover, since $\epsilon_2<<\epsilon_3$, for $\ell=s+\widetilde{\sigma}_i$ we get 
 \begin{equation*}
 \begin{aligned}
 &v_{s+\widetilde{\sigma}_i}^N(t_{i+1}^+)=v_{s+\widetilde{\sigma}_i}=v_{s+\widetilde{\sigma}_i}^\infty(t_{i+1}^+),\\
& |x_{s+\widetilde{\sigma}_i}^N(t_{i+1}^+)-x_{s+\widetilde{\sigma}_i}^\infty(t_{i+1}^+)|\leq |x_{m_i}^N(t_{i}^+)-x_{m_i}^\infty(t_{i}^+)|+\epsilon_2|\omega_{s+\widetilde{\sigma}_i}|\leq \sqrt{2}\epsilon_3(i-1)+\epsilon_2<\sqrt{2}\epsilon_3 i.
 \end{aligned}
\end{equation*}
$\bullet\quad  \sigma_i=1, j_i=1$: For the Boltzmann hierarchy pseudo-trajectory we get
  \begin{equation*}
 \begin{aligned}
 &x_\ell^\infty(t_{i+1}^+)=x_\ell^\infty(t_{i}^+)-(t_{i}-t_{i+1})v_{\ell}^\infty(t_{i}^+),\quad v_{\ell}^\infty(t_{i+1}^+)=v_\ell^\infty(t_{i}^+),\quad\forall \ell\in\{1,..., s+\widetilde{\sigma}_{i-1}\}\setminus\{m_i\},\\
 &x_{m_i}^\infty(t_{i+1}^+)=x_{m_i}^\infty(t_{i}^+)-(t_{i}-t_{i+1})v_{m_i}^{\infty'}(t_{i}^+),\quad v_{m_i}^\infty(t_{i+1}^+)=v_{m_i}^{\infty'}(t_{i}^+),\\
 &x_{s+\widetilde{\sigma}_i}^\infty(t_{i+1}^+)=x_{m_i}^\infty(t_{i}^+)-(t_{i}-t_{i+1})v_{s+\widetilde{\sigma}_i}',\quad v_{s+\widetilde{\sigma}_i}^\infty(t_{i+1}^+)=v_{s+\widetilde{\sigma}_i}'.
 \end{aligned}
 \end{equation*}
 and for the BBGKY hierarchy  pseudo-trajectory we obtain
 \begin{equation*}
 \begin{aligned}
 &x_\ell^N(t_{i+1}^+)=x_\ell^N(t_{i}^+)-(t_{i}-t_{i+1})v_{\ell}^N(t_{i}^+),\quad v_{\ell}^N(t_{i+1}^+)=v_\ell^N(t_{i}^+),\quad \forall \ell\in\{1,..., s+\widetilde{\sigma}_{i-1}\}\setminus\{m_i\},\\
 &x_{m_i}^N(t_{i+1}^+)=x_{m_i}^N(t_{i}^+)-(t_{i}-t_{i+1})v_{m_i}^{N'}(t_{i}^+),\quad v_{m_i}^N(t_{i+1}^+)=v_{m_i}^{N'}(t_{i}^+),\\
 &x_{s+\widetilde{\sigma}_i}^N(t_{i+1}^+)=x_{m_i}^N(t_{i}^+)-(t_{i}-t_{i+1})v_{s+\widetilde{\sigma}_i}'+\epsilon_2\omega_{s+\widetilde{\sigma}_i}, \quad v_{s+\widetilde{\sigma}_i}^N(t_{i+1}^+)=v_{s+\widetilde{\sigma}_i}'.
 \end{aligned}
 \end{equation*}
 For $\ell\in\left\{1,...,s+\widetilde{\sigma}_{i-1}\right\}\setminus\left\{m_i\right\}$, the induction assumption \eqref{proximity induction} yields
 \begin{equation*}
 \begin{aligned}
&v_\ell^N(t_{i+1}^+)=v_\ell^N(t_{i}^+)=v_{\ell}^\infty(t_{i}^+)=v_{\ell}^\infty(t_{i+1}^+),\\
&|x_\ell^N(t_{i+1}^+)-x_\ell^\infty(t_{i+1}^+)|=|x_\ell^N(t_{i}^+)-x_\ell^\infty(t_{i}^+)|\leq\sqrt{2}\epsilon_3(i-1).
\end{aligned}
\end{equation*}
and for $\ell=m_i$, it yields
\begin{equation*}
 \begin{aligned}
&v_{m_i}^N(t_{i+1}^+)=v_{m_i}^{N'}(t_{i}^+)=v_{m_i}^{\infty'}(t_{i}^+)=v_{\ell}^\infty(t_{i+1}^+),\\
&|x_{m_i}^N(t_{i+1}^+)-x_{m_i}^\infty(t_{i+1}^+)|=|x_{m_i}^N(t_{i}^+)-x_{m_i}^\infty(t_{i}^+)|\leq\sqrt{2}\epsilon_3(i-1).
\end{aligned}
\end{equation*}
Moreover, since $\epsilon_2<<\epsilon_3$, for $\ell=s+\widetilde{\sigma}_i$, we obtain
\begin{equation*}
 \begin{aligned}
&v_{s+\widetilde{\sigma}_i}^N(t_{i+1}^+)=v_{s+\widetilde{\sigma}_i}'=v_{s+\widetilde{\sigma}_i}^\infty(t_{i+1}^+),\\
&|x_{s+\widetilde{\sigma}_i}^N(t_{i+1}^+)-x_{s+\widetilde{\sigma}_i}^\infty(t_{i+1}^+)|\leq |x_{m_i}^N(t_{i}^+)-x_{m_i}^\infty(t_{i}^+)|+\epsilon_2|\omega_{s+\widetilde{\sigma}_i}|\leq \sqrt{2}\epsilon_3(i-1)+\epsilon_2<\sqrt{2}\epsilon_3 i.
\end{aligned}
\end{equation*}
$\bullet\quad\sigma_i=2,  j_i=-1$: For the Boltzmann hierarchy pseudo-trajectory we get
 \begin{equation*}
 \begin{aligned}
 &x_\ell^\infty(t_{i+1}^+)=x_\ell^\infty(t_{i}^+)-(t_{i}-t_{i+1})v_{\ell}^\infty(t_{i}^+),\quad v_{\ell}^\infty(t_{i+1}^+)=v_\ell^\infty(t_{i}^+),\quad \forall \ell\in\{1,...,s+\widetilde{\sigma}_{i-1}\}\setminus\{m_i\},\\
 &x_{m_i}^\infty(t_{i+1}^+)=x_{m_i}^\infty(t_{i}^+)-(t_{i}-t_{i+1})v_{m_i}^\infty(t_{i}^+),\quad v_{m_i}^\infty(t_{i+1}^+)=v_{m_i}^\infty(t_{i}^+),\\
 & x_{s+\widetilde{\sigma}_i-1}^\infty(t_{i+1}^+)=x_{m_i}^\infty(t_{i}^+)-(t_{i}-t_{i+1})v_{\widetilde{\sigma}_i-1},\quad v_{\ell}^\infty(t_{i+1}^+)=v_{s+\widetilde{\sigma}_i-1},\\
 &x_{s+\widetilde{\sigma}_i}^\infty(t_{i+1}^+)=x_{m_i}^\infty(t_{i}^+)-(t_{i}-t_{i+1})v_{s+\widetilde{\sigma}_i},\quad v_{s+\widetilde{\sigma}_i}^\infty(t_{i+1}^+)=v_{s+\widetilde{\sigma}_i},
 \end{aligned}
 \end{equation*}
 while for the BBGKY hierarchy pseudo-trajectory we get
 \begin{equation*}
 \begin{aligned}
 & x_\ell^N(t_{i+1}^+)=x_\ell^N(t_{i}^+)-(t_{i}-t_{i+1})v_{\ell}^N(t_{i}^+),\quad v_{\ell}^N(t_{i+1}^+)=v_\ell^N(t_{i}^-),\quad\forall \ell\in\{1,..., s+\widetilde{\sigma}_{i-1}\}\setminus\{m_i\},\\ 
 & x_{m_i}^N(t_{i+1}^+)=x_{m_i}^N(t_{i}^+)-(t_{i}-t_{i+1})v_{m_i}^N(t_{i}^+),\quad v_{m_i}^N(t_{i+1}^+)=v_{m_i}^N(t_{i}^-),\\ 
 &x_{s+\widetilde{\sigma}_i-1}^N(t_{i+1}^+)=x_{m_i}^N(t_{i}^+)-(t_{i}-t_{i+1})v_{s+\widetilde{\sigma}_i-1}-\sqrt{2}\epsilon_3\omega_{s+\widetilde{\sigma}_i-1},\quad v_{s+\widetilde{\sigma}_i-1}^N(t_{i+1}^+)=v_{s+\widetilde{\sigma}_i-1},\\
  &x_{s+\widetilde{\sigma}_i}^N(t_{i+1}^+)=x_{m_i}^N(t_{i}^+)-(t_{i}-t_{i+1})v_{s+\widetilde{\sigma}_i}-\sqrt{2}\epsilon_3\omega_{s+\widetilde{\sigma}_i},\quad v_{s+\widetilde{\sigma}_i}^N(t_{i+1}^+)=v_{s+\widetilde{\sigma}_i}.
 \end{aligned}
 \end{equation*}
 So, for any $\ell\in\{1,..., s+\widetilde{\sigma}_{i-1}\}$, the induction assumption \eqref{proximity induction} implies
 \begin{align*}
 &v_\ell^N(t_{i+1}^+)=v_\ell^N(t_{i}^+)=v_\ell^\infty(t_{i}^+)=v_\ell^\infty(t_{i+1}^+),\\
 &|x_\ell^N(t_{i+1}^+)-x_\ell^\infty(t_{i+1}^+)|=|x_\ell^N(t_{i}^+)-x_\ell^\infty(t_{i}^+)|\leq\sqrt{2}\epsilon_3(i-1),
 \end{align*}
Moreover, for $\ell=s+\widetilde{\sigma}_i-1$ we get 
 \begin{equation*}
 \begin{aligned}
 &v_{s+\widetilde{\sigma}_i-1}^N(t_{i+1}^+)=v_{s+\widetilde{\sigma}_i-1}=v_{s+\widetilde{\sigma}_i-1}^\infty(t_{i+1}^+),\\
 &|x_{s+\widetilde{\sigma}_i-1}^N(t_{i+1}^+)-x_{s+\widetilde{\sigma}_i-1}^\infty(t_{i+1}^+)|\leq |x_{m_i}^N(t_{i}^+)-x_{m_i}^\infty(t_{i}^+)|+\sqrt{2}\epsilon_3|\omega_{s+\widetilde{\sigma}_i-1}|\leq \sqrt{2}\epsilon_3(i-1)+\sqrt{2}\epsilon_3=\sqrt{2}\epsilon_3 i,
 \end{aligned}
\end{equation*}
and for $\ell=s+\widetilde{\sigma}_i$ we get 
 \begin{equation*}
 \begin{aligned}
 &v_{s+\widetilde{\sigma}_i}^N(t_{i+1}^+)=v_{s+\widetilde{\sigma}_i}=v_{s+\widetilde{\sigma}_i}^\infty(t_{i+1}^+),\\
 &|x_{s+\widetilde{\sigma}_i}^N(t_{i+1}^+)-x_{s+\widetilde{\sigma}_i}^\infty(t_{i+1}^+)|\leq |x_{m_i}^N(t_{i}^+)-x_{m_i}^\infty(t_{i}^+)|+\sqrt{2}\epsilon_3|\omega_{s+\widetilde{\sigma}_i}|\leq \sqrt{2}\epsilon_3(i-1)+\sqrt{2}\epsilon_3=\sqrt{2}\epsilon_3 i.
 \end{aligned}
\end{equation*}

$\bullet\quad \sigma_i=2,j_i=1:$ For the Boltzmann hierarchy pseudo-trajectory we get
  \begin{equation*}
 \begin{aligned}
 &x_\ell^\infty(t_{i+1}^+)=x_\ell^\infty(t_{i}^+)-(t_{i}-t_{i+1})v_{\ell}^\infty(t_{i}^+),\quad v_{\ell}^\infty(t_{i+1}^+)=v_\ell^\infty(t_{i}^+),\quad\forall \ell\in\{1,..., s+\widetilde{\sigma}_{i-1}\}\setminus\{m_i\},\\
 &x_{m_i}^\infty(t_{i+1}^+)=x_{m_i}^\infty(t_{i}^+)-(t_{i}-t_{i+1})v_{m_i}^{\infty*}(t_{i}^+),\quad v_{m_i}^\infty(t_{i+1}^+)=v_{m_i}^{\infty*}(t_{i}^+),\\
 & x_{s+\widetilde{\sigma}_{i}-1}^\infty(t_{i+1}^+)=x_{m_i}^\infty(t_{i}^+)-(t_{i}-t_{i+1})v_{s+\widetilde{\sigma}_{i}-1}^*,\\&v_{s+\widetilde{\sigma}_{i}-1}^\infty(t_{i+1}^+)=v_{s+\widetilde{\sigma}_{i}-1}^*,\\
 &x_{s+\widetilde{\sigma}_i}^\infty(t_{i+1}^+)=x_{m_i}^\infty(t_{i}^+)-(t_{i}-t_{i+1})v_{s+\widetilde{\sigma}_i}^*,\quad v_{s+\widetilde{\sigma}_i}^\infty(t_{i+1}^+)=v_{s+\widetilde{\sigma}_i}^*.
 \end{aligned}
 \end{equation*}
 and for the BBGKY hierarchy pseudo-trajectory we obtain
 \begin{equation*}
 \begin{aligned}
 &x_\ell^N(t_{i+1}^+)=x_\ell^N(t_{i}^+)-(t_{i}-t_{i+1})v_{\ell}^N(t_{i}^+),\quad v_{\ell}^N(t_{i+1}^+)=v_\ell^N(t_{i}^+),\quad \forall \ell\in\{1,..., s+\widetilde{\sigma}_{i-1}\}\setminus\{m_i\},\\
 &x_{m_i}^N(t_{i+1}^+)=x_{m_i}^N(t_{i}^+)-(t_{i}-t_{i+1})v_{m_i}^{N*}(t_{i}^+),\quad v_{m_i}^N(t_{i+1}^+)=v_{m_i}^{N*}(t_{i}^+),\\
 & x_{s+\widetilde{\sigma}_i-1}^N(t_{i+1}^+)=x_{m_i}^N(t_{i}^+)-(t_{i}-t_{i+1})v_{s+\widetilde{\sigma}_i-1}^*+\sqrt{2}\epsilon_3\omega_{s+\widetilde{\sigma}_i-1},\\
 &v_{s+\widetilde{\sigma}_i-1}^N(t_{i+1}^+)=v_{s+\widetilde{\sigma}_i-1}^*,\\
 &x_{s+\widetilde{\sigma}_i}^N(t_{i+1}^+)=x_{m_i}^N(t_{i}^+)-(t_{i}-t_{i+1})v_{s+\widetilde{\sigma}_i}^*+\sqrt{2}\epsilon_3\omega_{s+\widetilde{\sigma}_i},\\ &v_{s+\widetilde{\sigma}_i}^\infty(t_{i+1}^+)=v_{s+\widetilde{\sigma}_i}^*.
 \end{aligned}
 \end{equation*}
 For $\ell\in\left\{1,...,\widetilde{\sigma}_{i-1}\right\}\setminus\left\{m_i\right\}$, the induction assumption \eqref{proximity induction} yields
 \begin{equation*}
 \begin{aligned}
&v_\ell^N(t_{i+1}^+)=v_\ell^N(t_{i}^+)=v_{\ell}^\infty(t_{i}^+)=v_{\ell}^\infty(t_{i+1}^+),\\
&|x_\ell^N(t_{i+1}^+)-x_\ell^\infty(t_{i+1}^+)|=|x_\ell^N(t_{i}^+)-x_\ell^\infty(t_{i}^+)|\leq\sqrt{2}\epsilon_3(i-1).
\end{aligned}
\end{equation*}
Thus, for $\ell=m_i$
\begin{equation*}
 \begin{aligned}
&v_{m_i}^N(t_{i+1}^+)=v_{m_i}^{N*}(t_{i}^+)=v_{m_i}^{\infty*}(t_{i}^+)=v_{\ell}^\infty(t_{i+1}^+),\\
&|x_{m_i}^N(t_{i+1}^+)-x_{m_i}^\infty(t_{i+1}^+)|=|x_{m_i}^N(t_{i}^+)-x_{m_i}^\infty(t_{i}^+)|\leq\sqrt{2}\epsilon_3(i-1),
\end{aligned}
\end{equation*}
for $\ell=s+\widetilde{\sigma}_i-1$
\begin{equation*}
 \begin{aligned}
&v_{s+\widetilde{\sigma}_i-1}^N(t_{i+1}^+)=v_{s+\widetilde{\sigma}_i-1}^*=v_{s+\widetilde{\sigma}_i-1}^\infty(t_{i+1}^+),\\
&|x_{s+\widetilde{\sigma}_i-1}^N(t_{i+1}^+)-x_{s+\widetilde{\sigma}_i-1}^\infty(t_{i+1}^+)|\leq |x_{m_i}^N(t_{i}^+)-x_{m_i}^\infty(t_{i}^+)|+\sqrt{2}\epsilon_3|\omega_{s+\widetilde{\sigma}_i-1}|\leq \sqrt{2}\epsilon_3(i-1)+\sqrt{2}\epsilon_3=\sqrt{2}\epsilon_3 i,
\end{aligned}
\end{equation*}
and for $\ell=s+\widetilde{\sigma}_i$
\begin{equation*}
 \begin{aligned}
&v_{s+\widetilde{\sigma}_i}^N(t_{i+1}^+)=v_{s+\widetilde{\sigma}_i}^*=v_{s+\widetilde{\sigma}_i}^\infty(t_{i+1}^+),\\
&|x_{s+\widetilde{\sigma}_i}^N(t_{i+1}^+)-x_{m_i}^\infty(t_{i+1}^+)|\leq |x_{m_i}^N(t_{i}^+)-x_{s+\widetilde{\sigma}_i}^\infty(t_{i}^+)|+\sqrt{2}\epsilon_3|\omega_{s+\widetilde{\sigma}_i}|\leq \sqrt{2}\epsilon_3(i-1)+\sqrt{2}\epsilon_3=\sqrt{2}\epsilon_3 i.
\end{aligned}
\end{equation*}
Combining all cases, \eqref{proximity claim} is proved by induction. 

To prove \eqref{total proximity}, it suffices to add for $\ell=1,...,s+\widetilde{\sigma}_{i-1},$ and use the facts $1\leq i\leq k-1$, $\widetilde{\sigma}_{i-1}<\widetilde{\sigma}_i\leq\widetilde{\sigma}_{k-1}<2k\leq 2n$, from \eqref{bound on sigma}, and the assumption $s<n$.
 \end{proof} 
 \subsection{Reformulation in terms of pseudo-trajectories}
 We will now re-write the BBGKY hierarchy and Boltzmann hierarchy truncated elementary observables in terms of pseudo-trajectories. 
 
Let $s\in\mathbb{N}$ and assume $s<n$. For the Boltzmann hierarchy case, there is always free flow between the collision times. Therefore, recalling \eqref{truncated Boltzmann} and \eqref{mixed crossection}, for $X_s\in\Delta_s^X(\epsilon_0)$,  $1\leq k\leq n$, $\sigma\in S_k$, $(J,M)\in\mathcal{U}_{s,k,\sigma}$, $t\in [0,T]$ and $(t_1,...,t_k)\in\mathcal{T}_{k,\delta}(t)$, the Boltzmann hierarchy truncated elementary observable can be equivalently written as:
\begin{equation}\label{truncated elementary boltzmann}
\begin{aligned}
&J_{s,k,R,\delta,\sigma}^\infty(t,J,M)(X_s)=\int_{\mathcal{M}_s^c(X_s)}\phi_s(V_s)\int_{\mathcal{T}_{k,\delta}(t)}\int_{\mathcal{B}_{m_1}^c\left(Z_{s}^\infty\left(t_1^+\right)\right)}...\int_{\mathcal{B}_{m_k}^c\left(Z_{s+\widetilde{\sigma}_{k-1}}^\infty\left(t_{k}^+\right)\right)}\\
&\times\prod_{i=1}^{k}b_{\sigma_i}^+\left(\bm{\omega}_{\sigma_i,i},\bm{v}_{\sigma_i,i},v_{m_{i}}^\infty\left(t_i^+\right)
\right)
f_{0}^{(s+\widetilde{\sigma}_{k})}\left(Z_{s+\widetilde{\sigma}_{k}}^\infty
\left(0^+\right)\right)\prod_{i=1}^{k}\left(\,d\bm{\omega}_{\sigma_i,i}\,d\bm{v}_{\sigma_i,i}\right)\,dt_k...\,dt_1\,dV_s.
\end{aligned}
\end{equation}

Now we shall see that due to Lemma \ref{proximity}, it is possible to make a similar expansion for the BBGKY hierarchy truncated elementary observables as well. 

More precisely, fix $X_s\in \Delta_s^X(\epsilon_0)$,  $1\leq k\leq n$, $\sigma\in S_k$, $(J,M)\in\mathcal{U}_{s,k,\sigma}$, $t\in[0,T]$ and $(t_1,...,t_k)\in\mathcal{T}_{k,\delta}(t)$. Consider $(N,\epsilon_2,\epsilon_3)$ in the scaling \eqref{scaling} such that $\epsilon_2<<\eta^2\epsilon_3$ and $n^{3/2}\epsilon_3<<\alpha$. By Lemma \ref{initially good configurations}, given $V_s\in\mathcal{M}_s^c(X_s)$, we have $Z_s\in G_s(\epsilon_3,\epsilon_0,\delta).$ By the definition of the set $G_s(\epsilon_3,\epsilon_0,\delta)$, see \eqref{both epsilon-epsilon_0}, and the fact that $\epsilon_2<<\epsilon_3$, we have
\begin{equation*}
Z_s\in G_s(\epsilon_3,\epsilon_0,\delta)\Rightarrow Z_s(\tau)\in\mathring{\mathcal{D}}_{s,\epsilon_2,\epsilon_3},\quad\forall \tau\geq 0,
\end{equation*}
thus
\begin{equation}\label{equality of the flows k=0}
 \Psi_{s}^{\tau-t_0}Z_{s}^N\left(t_0^-\right)=\Phi_{s}^{\tau-t_0}Z_{s}^N\left(t_0^-\right),\quad\forall \tau\in [t_{1},t_{0}]
\end{equation}
 where $\Psi_{s}$, given in \eqref{liouville operator},  denotes the  $s$-particle $(\epsilon_2,\epsilon_3)$-interaction zone  flow   and  $\Phi_{s}$, given in \eqref{free flow operator}, denotes  the  $s$-particle free flow  respectively. We also have
$$Z_s=(X_s,V_s)\in G_s(\epsilon_3,\epsilon_0,\delta)\Rightarrow Z_{s}^\infty(t_1^+)\in G_s(\epsilon_0,0).$$ 
For all $i\in\{1,...,k\}$ inductive application of Proposition \ref{bad set double} or Proposition \ref{bad set triple}, depending on whether the adjunction is binary or ternary, implies that 
\begin{equation}\label{remak on Boltz pseudo}
Z_{s+\widetilde{\sigma}_i}^\infty(t_{i+1}^+)\in G_{s+\widetilde{\sigma}_i}(\epsilon_0,0),\quad\forall (\bm{\omega}_{\sigma_i,i},\bm{v}_{\sigma_i,i})\in \mathcal{B}_{m_i}^c(Z_{s+\widetilde{\sigma}_{i-1}}^\infty(t_i^+)).
\end{equation}
 Since we have assumed $n^{3/2}\epsilon_3<<\alpha$ and $s<n$, \eqref{total proximity} from Lemma \ref{proximity} implies
\begin{equation}\label{N-positions}
\left|X_{s+\widetilde{\sigma}_{i-1}}^N(t_{i}^+)-X_{s+\widetilde{\sigma}_{i-1}}^\infty(t_{i}^+)\right|\leq\frac{\alpha}{2},\quad\forall i=1,...,k.
\end{equation}
 Then,  \eqref{pre-0-double}, \eqref{post-0-double} from Proposition \ref{bad set double}, or \eqref{in phase pre}, \eqref{in phase post} from Proposition \ref{bad set triple}, depending on whether the adjunction is binary or ternary, yield that for any $i=1,...,k$, we have
$$\Psi_{s+\widetilde{\sigma}_i}^{\tau-t_i}Z_{s+\widetilde{\sigma}_i}^N\left(t_i^-\right)=\Phi_{s+\widetilde{\sigma}_i}^{\tau-t_i}Z_{s+\widetilde{\sigma}_i}^N\left(t_i^-\right),\quad\forall \tau\in [t_{i+1},t_{i}],$$
where $\Psi_{s+\widetilde{\sigma}_i}$ and $\Phi_{s+\widetilde{\sigma}_i}$ denote the $(s+\widetilde{\sigma}_i)$-particle $(\epsilon_2,\epsilon_3)$-flow and the $(s+\widetilde{\sigma}_i)$-particle free flow, given in \eqref{liouville operator} and \eqref{free flow operator} respectively.
In other words the  backwards $(\epsilon_2,\epsilon_3)$-flow coincides with the free flow in $[t_{i+1},t_i]$. Finally, Lemma \ref{proximity} also implies  that
$$v_{m_i}^N(t_i^+)=v_{m_i}^\infty(t_i^+),\quad\forall i=1,...,k.$$
 Therefore, for $X_s\in\Delta_s^X(\epsilon_0)$, and $(N,\epsilon_2,\epsilon_3)$ in the scaling \eqref{scaling} with $n\epsilon_3^{3/2}<<\alpha$ and $\epsilon_2<<\eta^2\epsilon_3$, the BBGKY hierarchy truncated elementary observable can be equivalently written as:
\begin{equation}\label{truncated elementary bbgky}
\begin{aligned}
J_{s,k,R,\delta,\sigma}^N(t,J,M)(X_s)&=\bm{A_{N,\epsilon_2,\epsilon_3}^{s,k}}\int_{\mathcal{M}_s^c(X_s)}\phi_s(V_s)\int_{\mathcal{T}_{k,\delta}(t)}\int_{\mathcal{B}_{m_1}^c\left(Z_{s}^\infty\left(t_1^+\right)\right)}...\int_{\mathcal{B}_{m_k}^c\left(Z_{s+\widetilde{\sigma}_{k-1}}^\infty\left(t_k^+\right)\right)}\\
&\hspace{0.2cm}\times\prod_{i=1}^{k}b_{\sigma_i}^+\left(\bm{\omega}_{\sigma_i,i},\bm{v}_{\sigma_i,i},v_{m_i}^\infty\left(t_i^+\right)\right)
f_{N,0}^{(s+\widetilde{\sigma}_{k})}\left(Z_{s+\widetilde{\sigma}_{k}}^N
\left(0^+\right)\right)\\
&\hspace{0.6cm}\times\prod_{i=1}^{k}\left(
\,d\bm{\omega}_{\sigma_i,i}\,d\bm{v}_{\sigma_i,i}\right)\,dt_k...\,dt_1\,dV_s,
\end{aligned}
\end{equation}
where, recalling \eqref{A binary}, \eqref{A triary}, we denote
\begin{equation}\label{k-A}
\bm{A_{N,\epsilon_2,\epsilon_3}^{s,k,\sigma}}=\prod_{i\in\{1,...,k\}:\sigma_i=1}A_{N,\epsilon_2,s+\widetilde{\sigma}_{i-1}}^2\prod_{i\in\{1,...,k\}:\sigma_i=2}A_{N,\epsilon_3,s+\widetilde{\sigma}_{i-1}}^3.
\end{equation}
\begin{remark}\label{inductive scaled limit}Notice that for fixed $s\in\mathbb{N}$ and $k\geq 1$ and $\sigma\in S_k$, the scaling \eqref{scaling} implies 
\begin{equation*}
\bm{A_{N,\epsilon_2,\epsilon_3}^{s,k,\sigma}}\to 1,\quad\text{as }N\to\infty.
\end{equation*}
\end{remark}
Let us approximate the BBGKY hierarchy truncated elementary observables by Boltzmann hierarchy truncated elementary observables defining some auxiliary functionals. Let $s\in\mathbb{N}$ and $X_s\in\Delta_s^X(\epsilon_0)$. For $ 1\leq k\leq n$, $\sigma\in S_k$ and $(J,M)\in\mathcal{U}_{s,k,\sigma}$, we define
\begin{equation}\label{auxiliary functionals}
\begin{aligned}
&\widehat{J}_{s,k,R,\delta,\sigma}^N(t,J,M)(X_s)=
\int_{\mathcal{M}_s^c(X_s)}\phi_s(V_s)\int_{\mathcal{T}_{k,\delta}(t)}\int_{\mathcal{B}_{m_1}^c\left(Z_{s}^\infty\left(t_1^+\right)\right)}...\int_{\mathcal{B}_{m_k}^c\left(Z_{s+\widetilde{\sigma}_{k-1}}^\infty\left(t_k^+\right)\right)}\\
&\times\prod_{i=1}^{k}b_{\sigma_i}^+\left(\bm{\omega}_{\sigma_i,i},\bm{v}_{\sigma_i,i},v_{m_{i}}^\infty\left(t_i^+\right)\right)
f_{0}^{(s+\widetilde{\sigma}_{k})}\left(Z_{s+\widetilde{\sigma}_{k}}^N
\left(0^+\right)\right)\prod_{i=1}^{k}\left(\,d\bm{\omega}_{\sigma_i,i}\,d\bm{v}_{\sigma_i,i}\right)\,dt_k...\,dt_1\,dV_s,
\end{aligned}
\end{equation}
{\text{red} explain what it is}
We conclude that the auxiliary functionals approximate the BBGKY hierarchy truncated elementary observables $J_{s,k,R,\delta}^N$, defined in \eqref{truncated elementary bbgky}
 \begin{proposition}\label{aux estimate 1} Let $s,n\in\mathbb{N}$, with $s<n$,  $\alpha,\epsilon_0,R,\eta,\delta$ be parameters as in \eqref{choice of parameters},  and $t\in[0,T]$. Then for any
  $\zeta>0$, there is $N_1=N_1(\zeta,n,\alpha,\eta,\epsilon_0
  )\in\mathbb{N}$, such that for all $(N,\epsilon_2,\epsilon_3)$ in the scaling \eqref{scaling} with $N>N_1$, there holds:
 \begin{equation*}
\sum_{k=1}^n\sum_{\sigma\in S_k}\sum_{(J,M)\in\mathcal{U}_{s,k}}\|J_{s,k,R,\delta,\sigma}^N(t,J,M)-\widehat{J}_{s,k,R,\delta,\sigma}^N(t,J,M)\|_{L^\infty\left(\Delta_s^X\left(\epsilon_0\right)\right)}\leq
 C_{d,s,\mu_0,T}^n\|\phi_s\|_{L^\infty_{V_s}}R^{d(s+3n)}\zeta ^2.
 \end{equation*}
 \end{proposition}
 \begin{proof}
 Fix $1\leq k\leq n$, $\sigma\in S_k$ and $(J,M)\in\mathcal{U}_{s,k,\sigma}$. Consider $(N,\epsilon_2,\epsilon_3)$ in the scaling \eqref{scaling}. Remark \eqref{remark for epsilons} guarantees that we can consider $N$ large enough such that$\epsilon_2<<\eta^2\epsilon_3$ and $n^{3/2}\epsilon_3<<\alpha$. Triangle inequality  and the inclusion $\Delta_s^X(\epsilon_0)\subseteq\Delta_s^X(\epsilon_0/2)$ yield
 \begin{align}
 &\|J_{s,k,R,\delta,\sigma}^N(t,J,M)-\widehat{J}_{s,k,R,\delta,\sigma}^N(t,J,M)\|_{L^\infty\left(\Delta_s^X\left(\epsilon_0\right)\right)}\nonumber\\
 &\leq \|J_{s,k,R,\delta,\sigma}^N(t,J,M)-\bm{A_{N,\epsilon_2,\epsilon_3}^{s,k,\sigma}}\widehat{J}_{s,k,R,\delta,\sigma}^N(t,J,M)\|_{L^\infty\left(\Delta_s^X\left(\epsilon_0/2\right)\right)}\label{final estimate 1 term 1}\\
 &\hspace{1cm}+|\bm{A_{N,\epsilon_2,\epsilon_3}^{s,k,\sigma}}-1|\|\widehat{J}_{s,k,R,\delta,\sigma}^N(t,J,M)\|_{L^\infty\left(\Delta_s^X\left(\epsilon_0\right)\right)}.\label{final estimate 1 term 2}
 \end{align}
 
 We estimate each of the terms  \eqref{final estimate 1 term 1}-\eqref{final estimate 1 term 2} separately. 
 
 \textbf{Term \eqref{final estimate 1 term 1}}: Let us fix $(t_1,...,t_k)\in\mathcal{T}_{k,\delta}(t)$. Applying \eqref{pseudo applicable} for  $i=k-1$, we obtain 
 $$Z_{s+\widetilde{\sigma}_{k-1}}^\infty(t_k^+)\in G_{s+\widetilde{\sigma}_{k-1}}(\epsilon_0,0).$$
Since $s<n$ and $n^{3/2}\epsilon_3<<\alpha$, \eqref{total proximity}, applied for $i=k$, implies
 $$|X_{s+\widetilde{\sigma}_{k-1}}^N(t_k^+)-X_{s+\widetilde{\sigma}_{k-1}}^\infty(t_k^+)|\leq\frac{\alpha}{2}.$$
 Therefore, \eqref{pre-delta-double}, \eqref{post-delta-double} from Proposition \ref{bad set double}, or \eqref{epsilon/2 pre}, \eqref{epsilon/2 post} from Proposition \ref{bad set triple}, depending on whether the adjunction is binary or ternary, imply 
 \begin{equation}\label{G halfs}
 Z_{s+\widetilde{\sigma}_k}^N(0^+)\in G_{s+\widetilde{\sigma}_k}(\epsilon_0/2,0)\subseteq\Delta_{s+\widetilde{\sigma}_k}(\epsilon_0/2).
 \end{equation}
 Thus  \eqref{estimate on the rest of the terms binary}-\eqref{estimate on the rest of the terms ternary}, \eqref{exclusion bad set 2 time}, \eqref{truncated elementary bbgky}-\eqref{auxiliary functionals} and crucially \eqref{G halfs} imply that for $N$ large enough, we have
 \begin{equation}\label{final estimate 2}
 \begin{aligned}
 \|J_{s,k,R,\delta,\sigma}^N(t,J,M)&-\bm{A_{N,\epsilon_2,\epsilon_3}^{s,k,\sigma}}\widehat{J}_{s,k,R,\delta,\sigma}^N(t,J,M)\|_{L^\infty\left(\Delta_s^X\left(\epsilon_0/2\right)\right)}\leq \\
 &\leq \frac{C_{d,s,T}^k}{k!}\|\phi_s\|_{L^\infty_{V_s}}R^{d(s+3k)}\|f_{N,0}^{(s+\widetilde{\sigma}_k)}-f_0^{(s+\widetilde{\sigma}_k)}\|_{L^\infty(\Delta_{s+\widetilde{\sigma}_k}(\epsilon_0/2))}.
 \end{aligned}
 \end{equation}
 
\textbf{Term \eqref{auxiliary estimate 2}}: By \eqref{exclusion bad set 2 norms}, we have
$\|f_0^{(s+\widetilde{\sigma}_k)}\|_{L^\infty}\leq e^{-(s+k)\mu_0}\|F_0\|_{\infty,\beta_0,\mu_0}.$
 Therefore, using  \eqref{estimate on the rest of the terms binary}- \eqref{estimate on the rest of the terms ternary} and \eqref{exclusion bad set 2 time},
 we obtain
 \begin{equation}\label{final estimate 3}
 \begin{aligned}
 \|\widehat{J}_{s,k,R,\delta,\sigma}^N(t,J,M)\|_{L^\infty\left(\Delta_s^X\left(\epsilon_0\right)\right)}\leq \frac{C_{d,s,\mu_0,T}^k}{k!}\|\phi_s\|_{L^\infty_{V_s}}R^{d(s+3k)}\|F_0\|_{\infty,\beta_0,\mu_0}.
 \end{aligned}
 \end{equation}
 Adding over all $(J,M)\in\mathcal{U}_{s,k,\sigma}$, $\sigma\in S_k$, $k=1,...,n$ and using \eqref{final estimate 2}-\eqref{final estimate 3}, we obtain the estimate
 \begin{equation*}
 \begin{aligned}
\sum_{k=1}^n\sum_{\sigma\in S_k}\sum_{(J,M)\in\mathcal{U}_{s,k,\sigma}}\|&J_{s,k,R,\delta,\sigma}^N(t,J,M)-\widehat{J}_{s,k,R,\delta,\sigma}^N(t,J,M)\|_{L^\infty\left(\Delta_s\left(\epsilon_0\right)\right)}\leq
C_{d,s,\mu_0,T}^n\|\phi_s\|_{L^\infty_{V_s}}R^{d(s+3n)}\\
&\hspace{-3cm}\times\left(\sup_{k\in\{1,...,n\}}\sup_{\sigma\in S_k}\|f_{N,0}^{(s+\widetilde{\sigma}_{k})}-f_0^{(s+\widetilde{\sigma}_{k})}\|_{L^\infty(\Delta_{s+\widetilde{\sigma}_{k}}(\epsilon_0))}+\|F_0\|_{\infty,\beta_0,\mu_0}\sup_{k\in\{1,...,n\}}\sup_{\sigma\in S_k}|\bm{A_{N,\epsilon_2,\epsilon_3}}^{s,k,\sigma}-1|\right).
 \end{aligned}
 \end{equation*}
 But since $n\in\mathbb{N}$, $\epsilon_0>0$ are fixed, \eqref{initial convergence to 0} implies
 \begin{equation*}
 \lim_{N\to\infty}\sup_{k\in \{1,...,n\}}\sup_{\sigma\in S_k}\|f_{N,0}^{(s+\widetilde{\sigma}_{k})}-f_0^{(s+\widetilde{\sigma}_{k})}\|_{L^\infty\left(\Delta_{s+\widetilde{\sigma}_{k}}\left(\epsilon_0\right)\right)}=0.
 \end{equation*}
 Moreover, Remark \ref{inductive scaled limit} yields
 \begin{equation*}
 \lim_{N\to\infty}\sup_{k\in \{1,...,n\}}\sup_{\sigma\in S_k}|\bm{A_{N,\epsilon_2,\epsilon_3}}^{s,k,\sigma}-1|= 0,
 \end{equation*}
 and the result follows.
 \end{proof}
 
 By the uniform continuity assumption, we also obtain the following estimate:
 \begin{proposition}\label{auxiliary estimate 2}  Let $s,n\in\mathbb{N}$ with $s<n$,  $\alpha,\epsilon_0,R,\eta,\delta$ be parameters as in \eqref{choice of parameters} and $t\in[0,T]$. Then for any  $\zeta>0$, there is $N_2=N_2(\zeta,n)\in\mathbb{N}$, such that for all $(N,\epsilon_2,\epsilon_3)$ in the scaling \eqref{scaling} with
 $N>N_2$, there holds
 \begin{equation*}
\sum_{k=1}^n\sum_{\sigma\in S_k}\sum_{(J,M)\in\mathcal{U}_{s,k,\sigma}}\|\widehat{J}_{s,k,R,\delta,\sigma}^N(t,J,M)-J_{s,k,R,\delta,\sigma}^\infty(t,J,M)\|_{L^\infty\left(\Delta_s^X\left(\epsilon_0\right)\right)}\leq C_{d,s,\mu_0,T}^n\|\phi_s\|_{L^\infty_{V_s}}R^{d(s+3n)}\zeta^2.
 \end{equation*}
 \end{proposition}
\begin{proof}
Let $\zeta>0$. Fix $1\leq k\leq n$, $\sigma\in S_k$ and $(J,M)\in\mathcal{U}_{s,k,\sigma}$. Since $s<n$, Lemma \ref{proximity} yields
\begin{equation}\label{continuity for integral}
|Z_{s+\widetilde{\sigma}_k}^N(0^+)-Z_{s+\widetilde{\sigma}_k}^\infty(0^+)|\leq \sqrt{6}n^{3/2}\epsilon_3,\quad\forall Z_s\in\mathbb{R}^{2ds}.
\end{equation}
Thus the continuity assumption \eqref{continuity assumption} on $F_0$, \eqref{continuity for integral}, the scaling \eqref{scaling}, and \eqref{epsilon with respect to N} from Remark \ref{remark for epsilons} imply that there exists $N_2=N_2(\zeta,n)\in\mathbb{N}$, such that for all $N>N_2$, we have
\begin{equation}\label{continuity satisfied}
|f_0^{(s+\widetilde{\sigma}_k)}(Z_{s+\widetilde{\sigma}_k}^N(0^+))-f_0^{(s+\widetilde{\sigma}_k)}(Z_{s+\widetilde{\sigma}_k}^\infty(0^+))|\leq C^{s+\widetilde{\sigma}_k-1}\zeta^2\leq C^{s+2k-1}\zeta^2 ,\quad\forall Z_s\in\mathbb{R}^{2ds}.
\end{equation}
In the same spirit as in the proof of Proposition \ref{aux estimate 1}, using \eqref{continuity satisfied}, \eqref{estimate on the rest of the terms binary}-\eqref{estimate on the rest of the terms ternary}, \eqref{exclusion bad set 2 time}, and summing over $(J,M)\in\mathcal{U}_{s,k,\sigma}$, $\sigma\in S_k$, $k=1,...,n$,  we obtain the result.
 \end{proof}
 \subsection{Proof of Theorem \ref{convergence theorem}}
 
 We are now in the position to prove Theorem \ref{convergence theorem}. Fix $s\in\mathbb{N}$, $\phi_s\in C_c(\mathbb{R}^{ds})$ and $t\in[0,T]$. Consider $n\in\mathbb{N}$ with $s<n$, and assume there exist parameters $\alpha,\epsilon_0,R,\eta,\delta$ satisfying \eqref{choice of parameters} . Let $\zeta>0$ small enough. Triangle inequality, Propositions \ref{reduction}, \ref{restriction to initially good conf}, \ref{truncated element estimate}, \ref{aux estimate 1}, \ref{auxiliary estimate 2}, Remark \ref{no need for k=0} and part \textit{(i)} of Proposition  \ref{approximation proposition}, yield that there is $N_0(\zeta,n,\alpha,\eta,\epsilon_0)\in\mathbb{N}$ such that for all $N>N_0$,  we have
 \begin{equation}\label{final final bounds}
 \begin{aligned}
 & \|I_s^N(t)-I_s^\infty(t)\|_{L^\infty\left(\Delta_s^X\left(\epsilon_0\right)\right)}\leq C\left(2^{-n}+e^{-\frac{\beta_0}{3}R^2}+ \delta C^n\right)+C^n R^{4dn}\eta^{\frac{d-1}{4d+2}}+C^nR^{4dn}\zeta^2,
 \end{aligned}
 \end{equation}
 where 
\begin{equation}\label{final C}
C:=C_{d,s,\beta_0,\mu_0,T}\|\phi_s\|_{L^\infty_{V_s}}\max\left\{1,\|F_0\|_{\infty,\beta_0,\mu_0}\right\}>1,
\end{equation}
 is an appropriate  constant.

Let us fix $\theta>0$.  Recall that we have  also fixed $s\in\mathbb{N}$ and $\phi_s\in C_c(\mathbb{R}^{ds})$.  We will now choose  parameters satisfying \eqref{choice of parameters}, depending only on $\zeta$, such that the right hand side of \eqref{final final bounds} becomes less than $\zeta$. 

\textit{Choice of parameters}: We choose $n\in\mathbb{N}$ and the parameters $\delta,\eta,R,\epsilon_0,\alpha $ in the following order:
\begin{align}
\bullet\hspace{0.2cm}& \max\left\{s,\log_2(C\zeta^{-1})\right\}<< n,& \text{(this implies $s<n$, $C2^{-n}<<\zeta$)},\\
\bullet\hspace{0.2cm}&  \delta<< \zeta C^{-(n+1)},&\text{(this implies $C^{n+1}\delta<<\zeta$)},\label{first parameter}\\
\bullet\hspace{0.2cm}& \eta<<\zeta^{\frac{8d+4}{d-1}},\hspace{0.2cm} R<<\zeta^{-1/4dn}C^{-1/4d},&\text{(those imply $C^nR^{4dn}\eta^{\frac{d-1}{4d+2}}<<\zeta$ and $C^nR^{4dn}\zeta^2<<\zeta)$,}\\
\bullet\hspace{0.2cm}&  \max\left\{1,\sqrt{3}\beta_0^{-1/2}\ln^{1/2}(C\zeta^{-1})\right\}< <R,&\text{(this implies $Ce^{-\frac{\beta_0 }{3}R^2}<<\zeta$),}\label{prefinal parameter}\\
\bullet\hspace{0.2cm}&\epsilon_0<<\eta\delta,\quad \epsilon_0<\theta,& \label{choice of epsilon0}\\
\bullet\hspace{0.2cm}&\alpha<<\epsilon_0\min\{1,R^{-1}\eta\}.&\label{final parameter}
\end{align}
 Clearly \eqref{first parameter}-\eqref{final parameter} imply the parameters chosen satisfy \eqref{choice of parameters} and depend only on $\zeta$.
Then,  \eqref{final final bounds} and the choice of parameters imply that we may find $N_0(\zeta)\in\mathbb{N}$, such that for all $N>N_0$, there holds:
\begin{equation*}
\|I_s^N(t)-I_s^\infty(t)\|_{L^\infty\left(\Delta_s^X\left(\epsilon_0\right)\right)}<\zeta.
\end{equation*}
But by \eqref{choice of epsilon0}, we have $\epsilon_0<\theta$, therefore we obtain
\begin{equation*}
\|I_s^N(t)-I_s^\infty(t)\|_{L^\infty\left(\Delta_s^X\left(\theta\right)\right)}\leq\|I_s^N(t)-I_s^\infty(t)\|_{L^\infty\left(\Delta_s^X\left(\epsilon_0\right)\right)}<\zeta,
\end{equation*}
and Theorem \ref{convergence theorem} is proved.

\section{Appendix}
In this appendix, we present some auxiliary results which are used throughout the paper.
\subsection{Calculation of Jacobians}
We first present an elementary Linear Algebra result, which will be useful throughout the manuscript for the calculation of Jacobians. For a proof see Lemma A.1. from \cite{thesis}.
\begin{lemma}\label{linear algebra lemma} Let $n\in\mathbb{N}$, $\lambda\neq 0$ and $w,u\in\mathbb{R}^n$. Then
\begin{equation*}
\det(\lambda I_n+w u^T)=\lambda^n(1+\lambda^{-1}\langle w,u\rangle),
\end{equation*}
where $I_n$ is the $n\times n$ identity matrix.
\end{lemma}
\subsection{The binary transition map}  Here, we introduce the binary transition map, which will enable us to control binary postcollisional configurations.
 Recall from \eqref{binary cross} the binary cross-section:
$$b_2(\omega_1,\nu_1)=\langle\omega,v_1\rangle,\quad(\omega_1,\nu_1)\in\mathbb{S}_1^{d-1}\times \mathbb{R}^d.$$
Given $v_1,v_2\in\mathbb{R}^d$, we define the domain\footnote{we trivially extend the binary cross-section for any $\omega\in\mathbb{R}^d.$}
$
\Omega:=\big\{\omega_1\in\mathbb{R}^d:|\omega_1|\leq 2,\text{ and }b_2(\omega_1,v_2-v_1)>0\big\},
$
and the set
$
\mathcal{S}_{v_1,v_2}^+=\{\omega_1\in\mathbb{S}_1^{d-1}: b_2(\omega_1,v_2-v_1)>0\}\subseteq\Omega.
$
We also define the smooth map $\Psi:\mathbb{R}^d\to\mathbb{R}$ by
$
\Psi(\omega_1):=|\omega_1|^2.
$
Notice that the unit $(d-1)$-sphere is given by  level sets of $\Psi$ i.e.
$
\mathbb{S}_1^{d-1}=[\Psi=1].
$
\begin{proposition}\label{transition prop}
 Consider $v_1,v_2\in\mathbb{R}^d$ and $r>0$ such that 
$
|v_1-v_2|=r.
$
We define the binary transition map $\mathcal{J}_{v_1,v_2}:\Omega\to\mathbb{R}^d$ as follows\footnote{we trivially extend the binary collisional operator for any $\omega\in\Omega$.}:
\begin{equation}\label{definition of v}
\mathcal{J}_{v_1,v_2}(\omega_{1}):=r^{-1}(v_1'-v_2'),\quad\omega\in\Omega.
\end{equation}
The map $\mathcal{J}_{v_1,v_2}$ has the following properties:
\begin{enumerate}[(i)]
\item $\mathcal{J}_{v_1,v_2}$ is smooth in $\Omega$ with bounded derivative uniformly in $r$  i.e.
\begin{equation}\label{matrix derivative lemma}
\|D\mathcal{J}_{v_1,v_2}(\omega_1)\|_\infty\leq C_d,\quad\forall\omega_1\in\Omega,
\end{equation}
where $\|\cdot\|_\infty$ denotes the maximum element matrix norm of $D\mathcal{J}_{v_1,v_2,v_3}(\omega_1)$.
\vspace{0.2cm}
\item The Jacobian of $\mathcal{J}_{v_1,v_2}$ is given by:
\begin{equation}\label{equality for jacobian}\jac(\mathcal{J}_{v_1,v_2})(\omega_1)\simeq r^{-d}b_2^d(\omega_1,v_2-v_1)>0,\quad\forall\omega_1\in\Omega.
\end{equation}
\item The map
 $\mathcal{J}_{v_1,v_2}:\mathcal{S}_{v_1,v_2}^+\to\mathbb{S}_1^{d-1}\setminus\{r^{-1}(v_1-v_2)\}$
 is bijective. Morever, there holds
 \begin{equation}\label{S=level sets}
 \mathcal{S}_{v_1,v_2}^+=[\Psi\circ\mathcal{J}_{v_1,v_2}=1].
 \end{equation}
\item For any measurable $g:\mathbb{R}^{d}\to[0+\infty]$, there holds the change of variables estimate:
\begin{equation}\label{substitution estimate}
\int_{\mathcal{S}_{v_1,v_2}^+}(g\circ\mathcal{J}_{v_1,v_2}(\omega_1)|\jac\mathcal{J}_{v_1,v_2}(\omega_1)|\,d\omega_1\lesssim\int_{\mathbb{S}_1^{d-1}}g(\nu_1)\,d\nu_1.
\end{equation}
\end{enumerate}
\end{proposition}
\begin{proof}
The proof is the binary analogue of the proof of Proposition 8.5. in \cite{ternary}.
\end{proof}


\begin{thebibliography}{99}
\bibitem{alexander} R.  Alexander,  {\em The  Infinite  Hard  Sphere  System, Ph.D.  dissertation}, Dept.  Mathematics,  Univ. California,  Berkeley,  1975.
\bibitem{thesis} I. Ampatzoglou, {\em Higher order extensions of the Boltzmann equation, Ph.D. dissertation},  Dept. Mathematics, UT Austin, (2020).
\bibitem{gwp} I. Ampatzoglou, I.M. Gamba, N. Pavlović, M. Tascović, \textit{Global well-posedness for a binary-ternary Boltzmann equation}, Submitted for publication (2020)
\bibitem{ternary} I. Ampatzoglou, N. Pavlovi\'c, {\em Rigorous derivation of a ternary Boltzmann equation for a classical system of particles}, Submitted for publication (2020).
\bibitem{multiple gamba 2} A.V. Bobylev, I.M. Gamba, C. Cercignani, {\em On the self-similar asymptotics for generalized non-linear kinetic Maxwell models}, Commun. Mathematical Physics 291, 599 - 644 (2009).
\bibitem{multiple gamba} A.V. Bobylev, I.M. Gamba, C. Cercignani, {\em Generalized kinetic Maxwell type models of granular gases}, Mathematical models of granular matter Series: Lecture Notes in Mathematics Vol.1937, Springer, G. Capriz, P. Giovine and P. M. Mariano (Eds.) (2008) ISBN: 978-3-540-78276-6.
\bibitem{bobylev} A.V. Bobylev, M. Pulvirenti, C. Saffirio, {\em From Particle Systems to the Landau Equation: A Consistency Result}, Commun. Math. Phys. 319, 683–702 (2013), Digital Object Identifier (DOI) 10.1007/s00220-012-1633-6.
\bibitem{brownian} T. Bodineau, I. Gallagher, L. Saint-Raymond,
{\em The Brownian motion as the limit of a deterministic system of hard-spheres}, Inventiones mathematicae, 203 (2016), 493-553.
\bibitem{fourier} T. Bodineau, I. Gallagher, L. Saint-Raymond, {\em From hard spheres dynamics to the Stokes-Fourier equations: an $L^2$ analysis of the Boltzmann-Grad limit}, 
Annals of PDE 3 (2017), no. 1, Art. 3:2, 118 pp.
\bibitem{boltzmann} L. Boltzmann, Weitere Studien uber das Warme gleichgenicht unfer Gasmolakular. Sitzungs-
berichte der Akademie der Wissenschaften 66 (1872), 275-370. Translation : Further studies on
the thermal equilibrium of gas molecules, in Kinetic Theory 2, 88-174, Ed. S.G. Brush, Pergamon,
Oxford (1966).
\bibitem{cercignani paper} C. Cercignani, {\em On the Boltzmann equation for rigid spheres}, Transport Theory Statist. Phys. 2
(1972), no. 3, p. 211-225.
\bibitem{cercignani gases}
C.  Cercignani,  R.  Illner,  M.  Pulvirenti,
{\em The  Mathematical  Theory  of  Dilute  Gases},
Springer  Verlag, New York  NY, 1994.
\bibitem{choh-uhlenbeck} S.T. Choh, G.E. Uhlenbeck, {\em The kinetic theory of phenomena in
dense gases, Ph.D. dissertation}, University of Michigan (1958).
\bibitem{cohen1} E.G.D. Cohen, {\em Fundamental Problems in Statistical Mechanics}, Vol.
1, E.G.D. Cohen, ed. (NHPC, Amsterdam, 1968) p. 228.
\bibitem{cohen50} E.G.D. Cohen, {\em Fifty years of kinetic theory}, Physica A 194 (1993) 229-257, North Holland.
\bibitem{denlinger} R. Denlinger, {\em The propagation of chaos for a rarefied gas of hard spheres in the whole space},  Archive for Rational Mechanics and Analysis volume 229, pages 885–952(2018).
\bibitem{desvilletes} L. Desvillettes, M. Pulvirenti, {\em The linear Boltzmann equation for long-range forces: a derivation from
particles}. Math. Meth. Mod. Appl. Sci. 9 (1999), no. 8, 11231145.
\bibitem{imp3} J. Dobnikar, C. Bechinger, M. Brunner, H. H. von Gr\"undberg, {\em Three-body interactions in colloidal systems}, PHYSICAL REVIEW E 69, 031402 (2004).  
\bibitem{imp5} J. Dobnikar, Y. Chen, R. Rzehak, H. H. von Gr\"unberg,{\em Many-body interactions in colloidal suspensions}, J. Phys.: Condens. Matter 15 S263–S268 (2003).
\bibitem{dorfman+cohen} J. Dorfman, E. Cohen, {\em Difficulties in the Kinetic Theory of Dense Gases}, Journal of Mathematical Physics 8, 282 (1967); https://doi.org/10.1063/1.1705194. 
\bibitem{gallagher nonograph} I. Gallagher, {\em From Newton to Navier-Stokes, or how to connect fluid mechanics equations from microscopic to macroscopic scales}, Bulletin of the American Mathematical Society, 58 (2019), 65-85.
\bibitem{gallagher} I. Gallagher, L. Saint-Raymond, B. Texier {\em From Newton to Boltzmann : hard spheres and short-range potentials}, Zurich Advanced Lectures in Mathematics Series, 18 2014 (148 pages).
\bibitem{Grad 1}H. Grad, {\em On  the  kinetic  theory  of  rarefied  gases},
Comm.  Pure  Appl.  Math. \textbf{2} (1949),  p.  331-407.
\bibitem {Grad 2}H.  Grad, {\em Principles  of  the  kinetic  theory  of  gases}, Handbuch  der  Physik \textbf{12}, Thermodynamik  der Gase  p.  205-294  Springer-Verlag,  Berlin-Gottingen-Heidelberg, 1958.
\bibitem {green} M. Green, {\em Boltzmann Equation from the Statistical Mechanical Point of View}, J. Chem. Phys. 25, 836 (1956); https://doi.org/10.1063/1.1743132.
\bibitem{hoegy} W. Hoegy, J. Sengers, {\em Three-Particle Collisions in a Gas of Hard Spheres}, Physical Review A, Volume 2, number  6 (1970).
\bibitem{holinger} B. Hollinger, C. Curtis, {\em Kinetic Theory of Dense Gases}, J. Chem. Phys. 33, 1386 (1960); https://doi.org/10.1063/1.1731418.
\bibitem{imp4} A. P. Hynninen, M. Dijkstra, R. van Roij, {\em Effect of three-body interactions on the phase behavior of charge-stabilized colloidal suspensions}, PHYSICAL REVIEW E 69, 061407 (2004).
\bibitem{illner} R. Illner and M. Pulvirenti, {\em Global Validity of the Boltzmann Equation
for Two and Three-Dimensional Rare Gas in Vacuum:
Erratum and Improved Result}, Commun. Math. Phys. 121, 143-146 (1989).
preprint DM-388-IR, September 1985.
\bibitem{king} F. King, {\em BBGKY hierarchy for positive potentials, Ph.D. dissertation}, Dept. Mathematics, Univ. California, Berkeley, 1975.
\bibitem {lanford}O.E.  Lanford, {\em Time  evolution  of  large  classical  systems}, Lect.  Notes  in  Physics 38, J. Moser ed. 1--111,  Springer  Verlag  (1975).
\bibitem{maxwell} J. Maxwell, {\em On the dynamical theory of gases,} Philos. Trans. Roy. Soc. London Ser. A,
157:49–88, 1867.
\bibitem{pulvirenti-simonella} M. Pulvirenti, C. Saffirio, S. Simonella, {\em On the validity of the Boltzmann equation for short range potentials}, Rev. Math. Phys. 26 (2014), no 2, 1450001.
\bibitem{ref 26} C. Russ, H. H. von Gr\"unberg, {\em Three-body forces between charged colloidal particles}, PHYSICAL REVIEW E 011402 (2002).
\bibitem{spohn} H. Spohn, {\em Boltzmann hierarchy and Boltzmann equation}, in Kinetic theories and the Boltzmann
equation (Montecatini, 1981), p. 207-220.
\bibitem{spohn1} H. Spohn, {\em Large Scale Dynamics of Interacting Particles}, Texts and Monographs in Physics,
Springer Verlag, Heidelberg, 1991.
\bibitem {saint raymond} L. Saint-Raymond, {\em Hydrodynamic limits of the Boltzmann equation}, Lecture Notes in Mathematics $1971$, DOI:10.1007-978-3-540-92847-8-2, Springer-Verlag Berlin Heidelberg (2009).
\bibitem{sengers} J. Sengers, {\em The Three-Particle Collision Term in the Generalized Boltzmann Equation}, Acta Physica Austriaca, Suppl.X, 177-208 (1973).
\bibitem{uchiya} K. Uchiyama,  {\em Derivation of the Boltzmann equation from particle dynamics}, Hiroshima Math. J.
18 (1988), no. 2, p. 245-297.
\end{thebibliography}
\end{document}